\documentclass{article}
\usepackage{amsmath,amsthm,amssymb}
\usepackage[usenames,dvipsnames]{xcolor}
\usepackage{enumerate}
\usepackage{graphicx}
\usepackage{cite}
\usepackage{comment}
\usepackage{oands}
\usepackage{tikz}
\usepackage{changepage}
\usepackage{bbm}
\usepackage{mathtools}
\usepackage[margin=0.98in]{geometry}
\usepackage[pagewise,mathlines]{lineno}
\usepackage{appendix}
\usepackage{stmaryrd} 
\usepackage{multicol}
\usepackage{microtype}

\usepackage[pdftitle={Convergence of percolation on uniform quadrangulations with boundary to SLE$_6$ on sqrt{8/3}-Liouville quantum gravity},
  pdfauthor={Ewain Gwynne and Jason Miller},
colorlinks=true,linkcolor=NavyBlue,urlcolor=RoyalBlue,citecolor=PineGreen,bookmarks=true,bookmarksopen=true,bookmarksopenlevel=2,unicode=true,linktocpage]{hyperref}

\theoremstyle{plain}
\newtheorem{thm}{Theorem}[section]

\newtheorem{lem}[thm]{Lemma}
\newtheorem{prop}[thm]{Proposition}

\def\@rst #1 #2other{#1}
\newcommand\MR[1]{\relax\ifhmode\unskip\spacefactor3000 \space\fi
  \MRhref{\expandafter\@rst #1 other}{#1}}
\newcommand{\MRhref}[2]{\href{http://www.ams.org/mathscinet-getitem?mr=#1}{MR#2}}

\theoremstyle{definition}
\newtheorem{defn}[thm]{Definition}
\newtheorem{remark}[thm]{Remark}

\numberwithin{equation}{section}

\newcommand{\dsb}{\begin{adjustwidth}{2.5em}{0pt}
\begin{footnotesize}}
\newcommand{\dse}{\end{footnotesize}
\end{adjustwidth}}

\newcommand{\ssb}{\begin{adjustwidth}{2.5em}{0pt}}
\newcommand{\sse}{\end{adjustwidth}}

\newcommand{\aryb}{\begin{eqnarray*}}
\newcommand{\arye}{\end{eqnarray*}}
\def\alb#1\ale{\begin{align*}#1\end{align*}}
\def\allb#1\alle{\begin{align}#1\end{align}}
\newcommand{\eqb}{\begin{equation}}
\newcommand{\eqe}{\end{equation}}
\newcommand{\eqbn}{\begin{equation*}}
\newcommand{\eqen}{\end{equation*}}

\newcommand{\BB}{\mathbbm}
\newcommand{\ol}{\overline}
\newcommand{\ul}{\underline}
\newcommand{\op}{\operatorname}

\newcommand{\frk}{\mathfrak}
\newcommand{\eqD}{\overset{d}{=}}
\newcommand{\ep}{\epsilon}
\newcommand{\rta}{\rightarrow}

\newcommand{\wt}{\widetilde}
 
\newcommand{\mcl}{\mathcal}

\newcommand{\bdy}{\partial}
\newcommand{\rng}{\mathring}
\newcommand{\srta}{\shortrightarrow}
\newcommand{\el}{l}

\newcommand{\SLE}{{\operatorname{SLE}}}

\newcommand{\pbl}{{\operatorname{pbl}}}
\newcommand{\bcon}{{\mathbbm c}}
\newcommand{\tcon}{{\mathbbm s}}
\newcommand{\ccon}{\mathbbm c_{\op{Peel}}}
\newcommand{\lcon}{\mathbbm c_{\op{Stable}}}

\let\originalleft\left
\let\originalright\right
\renewcommand{\left}{\mathopen{}\mathclose\bgroup\originalleft}
\renewcommand{\right}{\aftergroup\egroup\originalright}

\title{Convergence of percolation on uniform quadrangulations with boundary to SLE$_{6}$ on $\sqrt{8/3}$-Liouville quantum gravity} 
\date{  }
\author{
\begin{tabular}{c} Ewain Gwynne\footnote{\url{ewain@uchicago.edu}}\\[-5pt]\small University of Chicago \end{tabular}
\begin{tabular}{c} Jason Miller\footnote{\url{jpmiller@statslab.cam.ac.uk}}\\[-5pt]\small University of Cambridge \end{tabular}
}

\begin{document}

\maketitle

\begin{abstract}
Let $Q$ be a free Boltzmann quadrangulation with simple boundary decorated by a critical ($p=3/4$) face percolation configuration.  We prove that the chordal percolation exploration path on $Q$ between two marked boundary edges converges in the scaling limit to chordal SLE$_6$ on an independent $\sqrt{8/3}$-Liouville quantum gravity disk (equivalently, a Brownian disk). The topology of convergence is the Gromov-Hausdorff-Prokhorov-uniform topology, the natural analog of the Gromov-Hausdorff topology for curve-decorated metric measure spaces.  We also obtain analogous scaling limit results for face percolation on the uniform infinite half-plane quadrangulation with simple boundary, and for site percolation on a uniform triangulation with simple boundary. Our method of proof is robust and, up to certain technical steps, extends to any percolation model on a random planar map which can be explored via peeling.
\end{abstract}

\medskip


\tableofcontents

\section{Introduction}
\label{sec-intro}

\subsection{Overview}
\label{sec-overview}

In the past several decades, a vast literature concerning statistical mechanics models in two dimensions has been developed.  This work includes models on on deterministic lattices (such as $\BB Z^2$) as well as on random planar maps, i.e., random graphs embedded in the plane, viewed modulo orientation-preserving homeomorphisms. A central focus in this field is to show that these statistical mechanics models converge, under an appropriate scaling limit, to continuum models.

In the case of critical models on deterministic lattices, the limiting objects are often described (or conjectured to be described) in terms of \emph{Schramm-Loewner evolution (SLE)}~\cite{schramm0}, a one-parameter family of random fractal curves; see, e.g.,~\cite{lsw-lerw-ust,smirnov-cardy,smirnov-ising,ss-dgff}. $\SLE$ has been conjectured to arise in this context because in his original derivation \cite{schramm0} Schramm showed that it is characterized by the fact that it is \emph{conformally invariant} and satisfies a spatial Markov property called the \emph{domain Markov property}; these two properties together are sometimes referred to as the \emph{conformal Markov property}.  Many discrete models in two dimensions satisfy an exact spatial Markov property and are conjectured to be conformally invariant in the limit and therefore be $\SLE$s.
For critical models on random planar maps, one instead gets SLE curves in a random geometry which arises as the scaling limit of the underlying random planar map.  This random geometry can be described in terms of \emph{Liouville quantum gravity (LQG)}, a one-parameter family of random fractal surfaces.  
LQG surfaces with parameter $\gamma=\sqrt{8/3}$ are especially important since such surfaces describe the scaling limits of \emph{uniform} random planar maps, i.e., planar maps where each possibility is assigned equal probability. Certain special $\sqrt{8/3}$-LQG surfaces are equivalent, as metric measure spaces, to \emph{Brownian surfaces}, such as the Brownian map~\cite{legall-uniqueness,miermont-brownian-map} or the Brownian disk~\cite{bet-mier-disk}. 

The goal of this article is to show that a certain statistical mechanics model --- namely, critical percolation --- on certain types of random planar maps converges to SLE$_6$ on a $\sqrt{8/3}$-LQG surface, or equivalently a Brownian surface. The topology of convergence is the so-called \emph{Gromov-Hausdorff-Prokhorov uniform topology}, the natural analog of the Gromov-Hausdorff topology for curve-decorated metric measure spaces. We will provide more background about our result and the relevant mathematical objects shortly, but before we do so let us first comment briefly on our proof strategy.

The proof of our main scaling limit result has three main steps.
\begin{enumerate}
\item \label{item-tight} Show that percolation on a random planar map is tight with respect to the above topology.
\item \label{item-char}  Show that the desired limiting object --- namely SLE$_6$ on a $\sqrt{8/3}$-LQG surface --- is uniquely characterized by a certain set of properties  (essentially, the topology of the curve plus a ``LQG" variant of the domain Markov property). 
\item \label{item-ssl} Show that every possible subsequential limit of our discrete objects satisfies the properties in this characterization theorem.
\end{enumerate}
This proof outline is quite different from known scaling limit proofs for models on deterministic lattices toward SLE, which typically show directly that the Loewner driving function of the discrete curve converges to a multiple of Brownian motion (as opposed to using Schramm's conformal Markov characterization of $\SLE$). Our argument is also very different from the proof of the convergence of self-avoiding walk on random planar maps to SLE$_{8/3}$ on $\sqrt{8/3}$-LQG, with respect to the same topology we consider here, from~\cite{gwynne-miller-saw}.

In this paper, we will carry out steps~\ref{item-tight} and~\ref{item-ssl}, which both involve purely discrete (random planar map) arguments. Step~\ref{item-char} is carried out in the companion paper~\cite{gwynne-miller-char}, using purely continuum (SLE/LQG) arguments which are of quite a different flavor and make use of a different mathematical toolbox in comparison to this paper. We review all of the SLE/LQG results which are needed for the proofs of our main results, including the characterization theorem from~\cite{gwynne-miller-char}, in Section~\ref{sec-lqg-prelim} below. 
We note that in the course of proving this characterization theorem,~\cite{gwynne-miller-char} also establishes characterizations for other variants of SLE$_\kappa$ curves on $\gamma$-LQG surfaces for $\gamma \in (\sqrt2 , 2)$ and $\kappa  = 16/\gamma^2 \in (4,8)$, which may have applications to proving other scaling limit results for statistical mechanics models in random geometries.

\subsubsection{Percolation}
\label{sec-overview-perc} 

Let $G$ be a graph and $p \in [0,1]$.  Recall that \emph{site (resp.\ bond) percolation} on $G$ with parameter $p\in [0,1]$ is the model in which each vertex (resp.\ edge) of $G$ is declared to be open independently with probability $p$.  A vertex (resp.\ edge) which is not open is called closed. If $G$ is a planar map (i.e.,\ a graph together with an embedding into the plane so that no two edges cross), one can also consider \emph{face percolation}, equivalently site percolation on the dual map, whereby each face is open with probability $p$ and closed with probability $1-p$.  We refer to~\cite{grimmett-book,br-perc-book} for general background on  percolation.

Suppose now that~$G$ is an infinite graph with a marked vertex~$v$. The first question that one is led to ask about percolation on~$G$, which was posed in~\cite{bh-perc-processes}, is whether there exists an infinite \emph{open cluster} containing~$v$, i.e.\ a connected set of open vertices, edges, or faces (depending on the choice of model). For $p \in [0,1]$, let $\phi(p)$ be the probability that there is such an open cluster containing~$v$ and let $p_c = \sup\{ p \in [0,1] : \phi(p) = 0\}$ be the \emph{critical probability} above (resp.\ below) which there is a positive (resp.\ zero) chance there is an infinite open cluster containing~$v$.  The value of~$p_c$ is in general challenging to determine, but has been identified in some special cases.  For example, it is known that $p_c=1/2$ for both bond percolation on $\BB Z^2$ and for site percolation on the triangular lattice \cite{kesten-perc-book}.  As we will explain below, $p_c$ has also been identified for a number of random planar map models.

The next natural question that one is led to ask is whether the percolation configuration at criticality ($p=p_c$) possesses a \emph{scaling limit}, and this is the question in which we will be interested in the present work.  For percolation on a two-dimensional lattice when $p = p_c$, the interfaces between open and closed clusters are expected to converge in the scaling limit to Schramm-Loewner evolution (SLE)-type curves~\cite{schramm0} with parameter $\kappa = 6$. The reason for this is that the scaling limits of these percolation interfaces are conjectured to be conformally invariant (attributed to Aizenman by Langlands, Pouliot, and Saint-Aubin in \cite{lpsa-conformal-invariance}) with crossing probabilities which satisfy Cardy's formula~\cite{cardy-formula}. The particular value $\kappa = 6$ is obtained since this is the only value for which SLE possesses the \emph{locality property}~\cite{lsw-bm-exponents1}, which is a continuum analog of the statement that the behavior of a percolation interface is not affected by the percolation configuration outside of a sub-graph of the underlying lattice until it exits that sub-graph.  This conjecture has been proven in the special case of site percolation on the triangular lattice by Smirnov~\cite{smirnov-cardy}; see~\cite{camia-newman-sle6} for a detailed proof of the scaling limit result and~\cite{hls-sle6} for a proof of convergence in the so-called natural parameterization.  The proof of~\cite{smirnov-cardy} relies crucially on the combinatorics of site percolation on the triangular lattice and does not generalize to other percolation models.

In this paper we will prove scaling limit results for percolation on \emph{random planar maps} and identify the limit with $\SLE_6$ on $\sqrt{8/3}$-Liouville quantum gravity, equivalently, $\SLE_6$ on a Brownian surface.  Statistical mechanics models on random planar maps and deterministic lattices are both of fundamental importance in mathematical physics. Indeed, both are well-motivated in the physics literature and both possess a rich mathematical structure. Many questions (e.g., scaling limit results for random curves toward SLE) can be asked for both random planar maps and deterministic lattices, and it is not in general clear which setting is easier. There are scaling limit results which have been proven for models on deterministic lattices but not random planar maps (e.g., the convergence of Ising model interfaces to SLE$_3$~\cite{smirnov-ising} or, prior to this paper, the convergence of percolation to SLE$_6$) or for random planar maps but not deterministic lattices (e.g., the convergence of self-avoiding walk to SLE$_{8/3}$~\cite{gwynne-miller-saw} or peanosphere scaling limit results~\cite{wedges,shef-burger,kmsw-bipolar,gkmw-burger}).

We will focus on the particular model of face percolation on a random quadrangulation.  (We will discuss the universality of the scaling limit in Section~\ref{sec-triangulation} in detail in the setting of site percolation on triangulations.) Critical probabilities for several percolation models on random planar maps are computed in~\cite{angel-curien-uihpq-perc}, building on ideas of \cite{angel-peeling,angel-uihpq-perc}; in particular, $p_c=3/4$ for face percolation on random quadrangulations.  The fact that $p_c=3/4$ and not $1/2$ is related to the asymmetry between open and closed faces: open faces are considered adjacent if they share an edge, whereas closed faces are considered adjacent if they share a vertex. See~\cite{richier-perc,menard-nolin-perc} for the computation of $p_c$ for other planar map models.

One useful feature of percolation on random planar maps is the so-called \emph{peeling procedure} which allows one to describe the conditional law of the remaining map when we explore a single face.  For face percolation with open/closed boundary conditions, the peeling process gives rise to a natural path from the root edge to the target edge which we call the \emph{percolation exploration path} (see Section~\ref{sec-intro-def-perc} for a precise definition of this path).  The peeling exploration path is closely related to, but not in general identical to, the percolation interface from the root edge to the target edge; see Section~\ref{sec-interface} for further discussion of this relationship. In the special case of site percolation on a triangulation the percolation exploration path is the same as the percolation interface.
 
\subsubsection{Limiting object: SLE$_6$ on $\sqrt{8/3}$-Liouville quantum gravity}
\label{sec-overview-rpm}

For $\gamma \in (0,2)$, a \emph{$\gamma$-Liouville quantum gravity (LQG) surface} is (formally) the random surface parameterized by a domain $D\subset \BB C$ whose Riemannian metric tensor is $e^{\gamma h(z)} \, dx \otimes dy$, where $h$ is some variant of the Gaussian free field (GFF) on $D$ and $dx\otimes dy$ is the Euclidean metric tensor. This does not make rigorous sense since $h$ is a distribution, not a function. However, it was shown in~\cite{shef-kpz} that one can make rigorous sense of the volume form associated with a $\gamma$-LQG surface, i.e.\ one can define a random measure $\mu_h$ on $D$ which is a limit of regularized versions of $e^{\gamma h(z)} \, dz$ where $dz$ is the Euclidean volume form (see~\cite{rhodes-vargas-review} and the references therein for a more general approach to constructing measures of this form). Hence a $\gamma$-LQG surface can be viewed as a random measure space together with a conformal structure.

In the special case when $\gamma =\sqrt{8/3}$, it is shown in~\cite{lqg-tbm1,lqg-tbm2,lqg-tbm3}, building on \cite{qle,tbm-characterization,sphere-constructions}, that $(D,h)$ can also be viewed as a random metric space, i.e., one can construct a metric $\frk d_h$ on $D$ which is interpreted as the distance function associated with $e^{\gamma h(z)} \, dx \otimes dy$. For certain special $\sqrt{8/3}$-LQG surfaces introduced in~\cite{wedges,shef-zipper}, the metric measure space structure of a $\sqrt{8/3}$-LQG surface is equivalent to a corresponding Brownian surface. In particular, the Brownian map, the scaling limit of the uniform quadrangulation of the sphere \cite{legall-uniqueness,miermont-brownian-map}, is equivalent to the \emph{quantum sphere}.  Also, the Brownian half-plane, the scaling limit of the uniform quadrangulation of the upper-half plane $\BB H$ in the Gromov-Hausdorff topology~\cite{gwynne-miller-uihpq,bmr-uihpq}, is equivalent to the \emph{$\sqrt{8/3}$-quantum wedge}.  Finally, the Brownian disk, the scaling limit of the uniform quadrangulation of the disk $\BB D$ \cite{bet-mier-disk}, is equivalent to the \emph{quantum disk}.

The metric measure space structure of a $\sqrt{8/3}$-LQG surface a.s.\ determines the conformal structure~\cite{lqg-tbm3}, so we have a canonical way of embedding a Brownian surface into~$\BB C$.  This enables us to define an independent $\SLE_6$ on the Brownian map, half-plane, and disk as a curve-decorated metric measure space by first embedding the Brownian surface into~$\BB C$ to get a $\sqrt{8/3}$-LQG surface and then sampling an independent $\SLE_6$ connecting two marked points. The canonical choice of parameterization is the so-called \emph{quantum natural time} with respect to this $\sqrt{8/3}$-LQG surface, a notion of time which is intrinsic to the curve decorated quantum surface~\cite{wedges}.  See Section~\ref{sec-lqg-prelim} for more on $\sqrt{8/3}$-LQG surfaces and their relationship to $\SLE_6$ and to Brownian surfaces.

In the companion paper~\cite{gwynne-miller-char}, we prove a characterization of SLE$_6$ on a Brownian surface by a set of simple properties, which is re-stated as Theorem~\ref{thm-bead-mchar}. 
This characterization is in some ways similar to Schramm's~\cite{schramm0} characterization of SLE in terms of conformal Markov property, in that it involves a continuum analog of the Markov property for percolation interfaces on a random planar map. However, the properties in our characterization theorem are different than those in Schramm's characterization, and the proof is extremely different. 
As discussed above, this characterization plays a fundamental role in the proof of our main results.

\subsubsection{Scaling limit}
\label{sec-overview-result} 

The main theorem of this paper (stated precisely as Theorem~\ref{thm-perc-conv} below) says that the exploration path associated with face percolation on a random quadrangulation with simple boundary between two marked edges converges in the scaling limit to $\SLE_6$ on the $\sqrt{8/3}$-LQG disk, equivalently $\SLE_6$ on the Brownian disk. The topology of convergence is given by the \emph{Gromov-Hausdorff-Prokhorov-uniform (GHPU) metric} introduced in~\cite{gwynne-miller-uihpq}. The GHPU metric is the natural analog of the Gromov-Hausdorff metric for curve-decorated metric measure spaces: two such spaces are close in this metric if they can be isometrically embedded into a common metric space in such a way that the spaces are close in the Hausdorff distance, the measures are close in the Prokhorov distance, and the curves are close in the uniform distance. We also deduce from this finite-volume scaling limit result an analogous infinite-volume scaling limit result for face percolation on a uniform quadrangulation of $\BB H$ toward $\SLE_6$ on the $\sqrt{8/3}$-LQG wedge, equivalently $\SLE_6$ on the Brownian half-plane.


Recall that the results of~\cite{lqg-tbm1,lqg-tbm2,lqg-tbm3} allow one to give a definition of $\SLE_6$ on a Brownian surface.  The main result of the present paper says that this definition agrees with the scaling limit of percolation on random planar maps, which implies that this is the correct definition of $\SLE_6$.  It is also not difficult to see that the conformal structure imposed on Brownian surfaces by the results of \cite{lqg-tbm1,lqg-tbm2,lqg-tbm3} is characterized by the property that it embeds the scaling limit of percolation as constructed in the present paper to $\SLE_6$ on an independent $\sqrt{8/3}$-LQG surface. Indeed, this follows because any homeomorphism which takes an $\SLE_6$ to an $\SLE_6$ must be conformal. From this perspective, the results of the present paper can be interpreted as implying that the conformal structure from \cite{lqg-tbm1,lqg-tbm2,lqg-tbm3} imposed on Brownian surfaces is the correct one.  (A similar conclusion also follows from the main results of~\cite{gwynne-miller-saw}, which show that the aforementioned embedding is the one which maps the scaling limit of self-avoiding walk on a random quadrangulation to an independent $\SLE_{8/3}$.)

\begin{figure}[ht!]
\begin{center}
\includegraphics[scale=0.85]{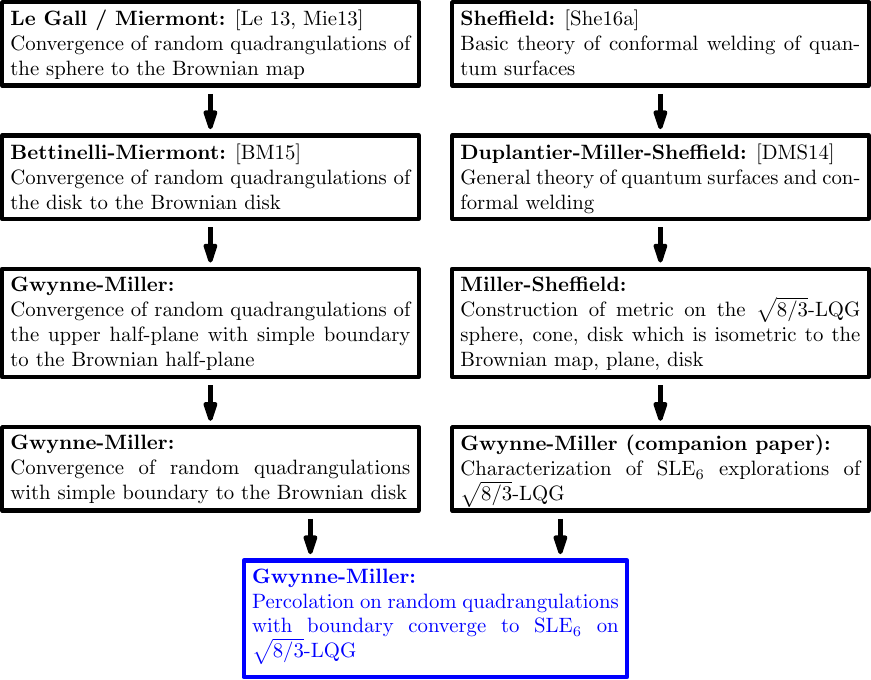}	
\end{center}
\vspace{-0.02\textheight}
\caption[Relationships between papers related to Chapter~\ref{chap-perc}]{\label{fig:chart_of_papers} An illustration of the dependencies between results which lead to the main theorem of the present paper, which is the blue box at the bottom of the figure.  See also \cite{bmr-uihpq} for another proof that the UIHPQ converges to the Brownian half-plane.}
\end{figure}

A major problem in the theory of LQG is to show that certain random planar map models conformally embedded into the plane converge in the scaling limit to $\gamma$-LQG, e.g., in the sense that the counting measure on vertices converges to the $\gamma$-LQG area measure. Some forms of this conjecture have been proved in recent years.  The works~\cite{gms-tutte,gms-poisson-voronoi} establish this type of convergence for the Tutte embedding in the case of the \emph{mated-CRT maps} and in the case of the Poisson-Voronoi approximation of the Brownian map.  Both \cite{gms-tutte,gms-poisson-voronoi} are based on understanding the embedding problem in terms of a random walk in a random environment \cite{gms-random-walk}.

The results of this paper have applications to the embedding problem in the case of uniform random planar maps. Indeed, one can define an embedding of a random planar map into the plane by matching crossing probabilities for percolation on the planar map to crossing probabilities for SLE$_6$. This embedding is called the \emph{Cardy embedding} after Cardy's formula for such crossing probabilities~\cite{cardy-formula,smirnov-cardy}. Convergence of the map under the Cardy embedding to $\sqrt{8/3}$-LQG is closely related to certain \emph{quenched} scaling limit results for percolation on a random planar map toward SLE$_6$ on $\sqrt{8/3}$-LQG (i.e., scaling limit results for the conditional law of the percolation given the map). Proving such a quenched scaling limit result amounts to showing that any finite number $N$ of independent (given the planar map) percolation explorations converge to $N$ independent $\SLE_6$ curves on the same $\sqrt{8/3}$-LQG surface.  The present paper shows that the exploration associated with a single percolation configuration converges. Subsequent work by Holden and Sun~\cite{hs-cardy-embedding}, building on the present paper as well as their various joint works with other authors~\cite{ghss-negative-moments,hlls-cut-pts,hls-sle6,aasw-type2,ghs-metric-peano,bhs-site-perc,ghss-ldp}, 
extends this result to get convergence of $N$ independent percolation configurations and deduces from this that Cardy-embedded uniform triangulations converge to $\sqrt{8/3}$-LQG.
 
\subsubsection{Remarks on proof strategy}
\label{sec-overview-proof} 

Our method of proof is robust in the sense that it does not rely on the particular random planar map or percolation model, provided one has certain technical inputs. As we will explain later, one (roughly) only needs to know that the boundary length processes associated to the percolation exploration converge to their continuum counterparts and that the corresponding planar map model with the topology of the disk converges to the Brownian disk.  The main reason that we focus on the case of face percolation on a quadrangulation of the disk is that this latter step has been carried out for quadrangulations \cite{gwynne-miller-simple-quad}.  In Section~\ref{sec-triangulation}, we will give a precise statement of a variant of our result in the setting of site percolation on a random triangulation, using the fact that \emph{triangulations} with simple boundary converge to the Brownian disk, which was recently proven in~\cite{aasw-type2}.

We only consider chordal percolation exploration paths and chordal $\SLE_6$ in this paper, but with some additional work our methods can be extended to obtain analogous scaling limit theorems for radial or whole-plane percolation exploration paths.  Likewise, we expect that our results can be extended to obtain a scaling limit statement for the full collection of percolation interfaces toward a conformal loop ensemble~\cite{shef-cle} with $\kappa=6$. See Section~\ref{sec-other-results} for more details.  

A key tool in our proof is a characterization theorem for chordal $\SLE_6$ on a Brownian disk which is proven in the companion paper~\cite{gwynne-miller-char} and re-stated as Theorem~\ref{thm-bead-mchar} below.  Roughly speaking, this result says that if $(\wt H , \wt d , \wt \mu , \wt\eta)$ is a random curve-decorated metric measure space such that $(\wt H , \wt d , \wt \mu)$ is a Brownian disk, $(\wt H , \wt\mu , \wt\eta)$ differs from an $\SLE_6$ on a Brownian disk via a curve- and measure-preserving homeomorphism, and the connected components of $\wt H \setminus \wt\eta([0,t])$ for each $t\geq 0$ equipped with their internal metrics are independent Brownian disks, then $\wt\eta$ is an independent chordal $\SLE_6$ on $(\wt H , \wt d , \wt\mu)$. 
The tools used to prove our characterization theorem are completely different from the tools used in the present paper: the proof of the characterization theorem uses SLE/LQG techniques whereas the proofs in the present paper are based on discrete methods (such as the  peeling operation for random quadrangulations). 

In addition to the characterization theorem from~\cite{gwynne-miller-char}, we will also use the scaling limit result for free Boltzmann quadrangulations with simple boundary toward the Brownian disk~\cite{gwynne-miller-simple-quad} and some properties of SLE$_6$ on the quantum disk proved in~\cite{gwynne-miller-sle6}.

To prove our main result we will prove tightness of the face percolation exploration path on a free Boltzmann quadrangulation with simple boundary in the GHPU topology (which amounts to showing equicontinuity of the percolation exploration path since we already know the scaling limit of the underlying map), then check that every subsequential limit satisfies the hypotheses of this characterization theorem.  
Although our proof relies on the theory of SLE and LQG, in the form of the characterization result Theorem~\ref{thm-bead-mchar} to prove uniqueness of subsequential limits, most of our arguments can be read without any knowledge of SLE or LQG if one takes this characterization theorem plus a few other results as a black box.

\medskip

\noindent{\bf Acknowledgements}
We thank two anonymous referees for helpful comments on an earlier version of this manuscript.
We thank Nina Holden, Scott Sheffield, and Xin Sun for helpful discussions. J.M.\ thanks Institut Henri Poincar\'e for support as a holder of the Poincar\'e chair, during which part of this work was completed.

\subsection{Preliminary definitions} 
\label{sec-intro-def}

In this subsection we recall the definitions of the objects involved in the statements of our main results. 

\subsubsection{Quadrangulations with simple boundary}
\label{sec-intro-def-quad}
 
A \emph{quadrangulation with boundary} is a (finite or infinite) planar map~$Q$ with a distinguished face~$f_\infty$, called the \emph{exterior face}, such that every face of~$Q$ other than~$f_\infty$ has degree $4$. The \emph{boundary} of $Q$, denoted by $\bdy Q$, is the smallest subgraph of~$Q$ which contains every edge of $Q$ adjacent to $f_\infty$. The \emph{perimeter} of~$Q$ is defined to be the degree of the exterior face. We note that the perimeter of a quadrangulation with boundary is necessarily even. 

We say that $Q$ has simple boundary if $\bdy Q$ is simple, i.e.\ it only has vertices of unit multiplicity. In this paper we will only consider quadrangulations with simple boundary. 
 
For such a quadrangulation $Q$, a \emph{boundary path} of $Q$ is a path $\beta$ from $[0,\#\mcl E(\bdy Q)]_{\BB Z}$ (if $\bdy Q$ is finite) or $\BB Z$ (if $\bdy Q$ is infinite) to $\mcl E(\bdy Q)$ which traces the edges $\mcl E(\bdy Q)$ of $\bdy Q$ (counted with multiplicity) in cyclic order. Choosing a boundary path is equivalent to choosing an oriented root edge on the boundary. This root edge is $\beta(0)$ and is oriented toward $\beta(1)$. In the finite boundary case the \emph{periodic boundary path} is the path obtained by extending $\beta$ to be $\#\mcl E(\bdy Q)$-periodic on $\BB Z$.

For $n \in \BB N $ and $\el \in \BB N$, we write $\mcl Q_{ \op{S} }^\shortrightarrow(n,\el)$ for the set of pairs $(Q , \BB e)$ where $Q$ is a quadrangulation with simple boundary having $2\el$ boundary edges and $n$ interior vertices and $\BB e$ is an oriented root edge in~$\bdy Q$.   By convention, we consider the trivial quadrangulation with one edge and no interior faces to be a quadrangulation with simple boundary of perimeter $2$ and define $\mcl Q_{\op{S}}^\srta(0,1)$ to be the set consisting of this single quadrangulation, rooted at its unique edge. We define $\mcl Q_{\op{S}}^\srta(0,\el) = \emptyset$ for $\el \geq 2$.

We define the \emph{free Boltzmann partition function} by  
\eqb \label{eqn-fb-partition}
\frk Z(2 \el )   :=    \frac{8^\el(3 \el-4)!}{(\el-2)! (2 \el)!}, \qquad \frk Z(2 \el+1) = 0 ,\qquad \forall \el \in \BB N ,
\eqe 
where here we set $(-1)! = 1$. 

\begin{defn} \label{def-fb}
For $\el\in\BB N$, the \emph{free Boltzmann distribution} on quadrangulations with simple boundary and perimeter $2\el$ is the probability measure on $\bigcup_{n=0}^\infty \mcl Q_{\op{S}}^\shortrightarrow (n ,  \el)$ which assigns to each element of $ \mcl Q_{\op{S}}^\shortrightarrow (n, \el)$ a probability equal to $12^{-n} \frk Z(2\el)^{-1}$.  
\end{defn}

It is shown in~\cite{bg-simple-quad} that $\frk Z(2\el) = \sum_{n=0}^\infty 12^{-n} \# \mcl Q_{\op{S}}^\shortrightarrow(n, \el) $, so that the free Boltzmann distribution is indeed a probability measure.

The \emph{uniform infinite half-plane quadrangulation with simple boundary} (UIHPQ$_{\op{S}}$) is the infinite rooted quadrangulation $(Q^\infty ,\BB e^\infty)$ with infinite simple boundary which is the Benjamini-Schramm local limit~\cite{benjamini-schramm-topology} in law of the free Boltzmann quadrangulation with simple boundary as the perimeter tends to $\infty$~\cite{curien-miermont-uihpq,caraceni-curien-uihpq}.  When we refer to a free Boltzmann quadrangulation with perimeter $2\el = \infty$, we mean the UIHPQ$_{\op{S}}$. 

\subsubsection{Critical face percolation on quadrangulations with simple boundary}
\label{sec-intro-def-perc}

In this subsection we give a brief description of the percolation exploration path of critical face percolation on a quadrangulation with simple boundary; see Section~\ref{sec-perc-peeling} for a precise definition and Figure~\ref{fig-perc-peeling-finite} for an illustration. 

Let $(Q,\BB e)$ be a quadrangulation with simple boundary of perimeter $2\el \in 2\BB N \cup \{\infty\}$.  A \emph{critical face percolation configuration} on $Q$ is a random function $\theta$ from the set of quadrilaterals $q$ of $Q$ to the set $\{\mathsf{white}, \mathsf{black}\}$ such that the values~$\theta(q)$ are i.i.d.\ Bernoulli random variables which equal $\mathsf{white}$ with probability $3/4$ and $\mathsf{black}$ with probability~$1/4$. We say that~$q$ is white or open (resp.\ closed or black) if $\theta(q) = \mathsf{white}$ (resp.\ $\theta(q) = \mathsf{black}$).

\begin{figure}[ht!]
\begin{center}
\includegraphics[scale=0.85]{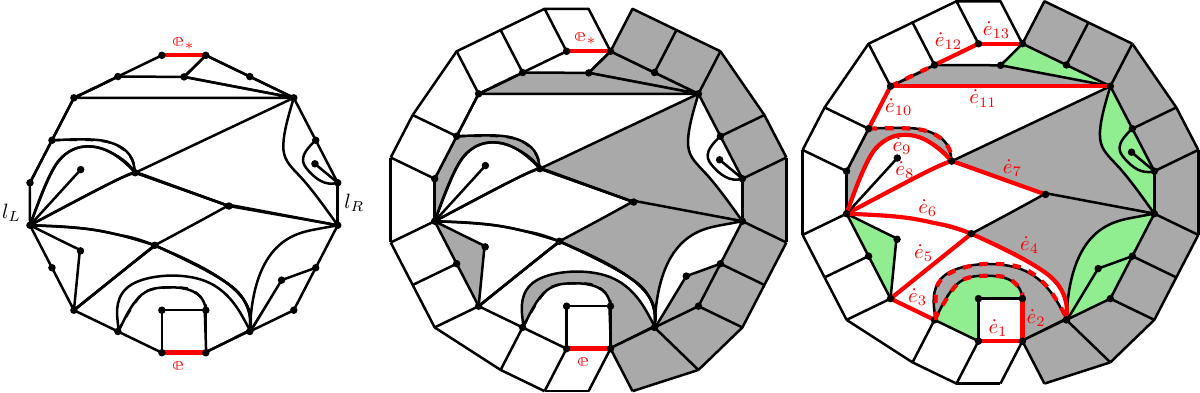}	
\end{center}
\vspace{-0.02\textheight}
\caption[Face percolation on a quadrangulation with simple boundary]{\label{fig-perc-peeling-finite} \textbf{Left:} A quadrangulation $Q$ with simple boundary with two marked boundary edges separated by boundary arcs which contain $\el_L $ and $\el_R $ vertices, respectively (here $\el_L = \el_R = 10$). \textbf{Middle:} The same quadrangulation $Q$ equipped with a face percolation configuration, with extra external quadrilaterals attached to impose white/black boundary conditions. \textbf{Right:} The edges $\dot e_j$ peeled by the percolation peeling process are shown in solid red. Quadrilaterals not revealed by this process are shown in green. The percolation exploration path $\lambda$ is obtained by adding extra edges (dashed red) to link up the $\dot e_j$'s into a path. Our main result says that $(Q,\lambda)$ converges in law to SLE$_6$ on the Brownian disk. } 
\end{figure}

Let $\el_L , \el_R \in \BB N $ with $\el_L + \el_R = 2\el$ or let $\el_L = \el_R = \infty$ in the case when $\el = \infty$.  The \emph{percolation peeling process of $(Q,\BB e,\theta)$ with $\el_L$-white/$\el_R$-black boundary conditions} is the algorithm for exploring $(Q,\BB e , \theta)$ described as follows.  If $\el < \infty$, we choose a target edge $\BB e_*$ in such a way that the clockwise (resp.\ counterclockwise) arc of $\bdy Q$ from $\BB e$ to $\BB e_*$ contains exactly $\el_L$ (resp.\ $\el_R$) vertices; or if $\el=\infty$ let $\BB e_* = \infty$. We impose white/black boundary conditions by attaching a white (resp.\ black) quadrilateral in the external face of $Q$ to each edge of $Q$ in the left (resp.\ right) arc of $\bdy Q$ from $\BB e$ to $\BB e_*$, with $\BB e$ and $\BB e_*$ included in the left, but not the right, arc.

The percolation peeling process explores $Q$ from $\BB e$ to $\BB e_*$ one quadrilateral at a time and is illustrated in the right panel of Figure~\ref{fig-perc-peeling-finite}. Let $\dot e_1$ be the root edge $ \BB e$ and at time 1, reveal the quadrilateral of $Q$ containing $\BB e$ on its boundary along with its color. We then consider the complementary connected component $\ol Q_1$ of the removed quadrilateral with $\BB e_*$ on its boundary, which is a sub-quadrangulation with simple boundary with the following property. If $\ol Q_1 \not=\emptyset$, then there are at most two edges of $\bdy \ol Q_1$ which are on the boundary of a white explored or external quadrilateral and which share a vertex with a black explored or external quadrilateral. One of these edges is equal to $\BB e_*$. We let $\dot e_2$ be the other of these two edges, or $\dot e_2 = \BB e_*$ if there is only one such edge. If $\dot e_2 \not=\BB e_*$, we can iterate the above procedure with $(\ol Q_1 , \dot e_2)$ in place of $(Q,\BB e)$ to define a quadrangulation with simple boundary $\ol Q_2 \subset \ol Q_1$ and an edge $\dot e_3 \in \bdy \ol Q_2$. We then continue inductively until we get all the way to $\BB e_*$, thereby defining edges $\dot e_j$ and quadrangulations $\ol Q_j$ for all $j\in\BB N$ (for large enough $j$, we will have $\dot e_j = \BB e_*$ and $\ol Q_j = \emptyset$).  

We now define the \emph{percolation exploration path} $\lambda : \frac12\BB N_0 \rta \mcl E(Q)$, which we will show converges to SLE$_6$. 
We set $\lambda(0) = \BB e = \dot e_1$ and for $j\in\BB N $, we set $\lambda(j):=  \dot{e}_j$. 
This does not define a path since the successive edges $\dot e_j$ might not share an endpoint. But, successive edges lie at graph distance at most 1 from each other.  We therefore extend the definition of $\lambda$ to $\frac12\BB N_0$ by taking $\lambda(j+1/2)$ to be an edge which shares an endpoint with each of $\lambda(j)$ and $\lambda(j+1/2)$ for each $j\in\BB N_0$.  One can choose $\lambda(j+1/2)$ in a variety of different ways, for example the left or rightmost edge sharing an endpoint with each of $\lambda( j)$ and $\lambda(j+1)$.  The particular choice does not affect the scaling limit, so we will not fix a convention. 
Note that we do not require that the edges of a path can be oriented in a consistent manner, so long as successive edges share an endpoint (c.f.\ Section~\ref{sec-graph-notation}).

As explained in more detail in Section~\ref{sec-interface}, the path $\lambda$ is closely related to, but not identical, to the percolation interface which goes from $\BB e$ to $\BB e_*$.  (In the case of site percolation on a triangulation, however, the analog of $\lambda$ can be taken to be the same as the percolation interface; see Section~\ref{sec-triangulation}.)

\subsubsection{Brownian disk and Brownian half-plane}
\label{sec-intro-def-disk}

For $\frk a , \frk l >0$, the \emph{Brownian disk with area $\frk a$ and perimeter $\frk l$} is the random curve-decorated metric measure space $(H , d , \mu , \xi)$ with the topology of the disk which arises as the scaling limit of uniformly random quadrangulations with boundary. The Brownian disk can be constructed as a metric space quotient of $[0,\frk a]$ via a continuum analog of the Schaeffer bijection~\cite{bet-mier-disk}; we will not need this construction here so we will not review it carefully. The area measure $\mu$ is the pushforward of Lebesgue measure on $[0,\frk a]$ under the quotient map and the path $\xi : [0,\frk l] \rta \bdy H$, called the \emph{boundary path}, parameterizes $\bdy H$ according to its natural length measure (which is the pushforward under the quotient map of the local time measure at the set of times when the encoding function attains a record minimum). The \emph{periodic boundary path} of $H$ is the path obtained by extending $\xi$ to be $\frk l$-periodic on $\BB R$.

The \emph{free Boltzmann Brownian disk with perimeter $\frk l$} is the random curve-decorated metric measure space $(H , d , \mu , \xi)$ obtained as follows: first sample a random area $\frk a$ from the probability measure $  \frac{\frk l^3}{ \sqrt{2\pi a^5 } } e^{-\frac{\frk l^2}{2 a} } \BB 1_{(a\geq 0)} \, da$, then sample a Brownian disk with boundary length $\frk l$ and area $\frk a$. The free Boltzmann Brownian disk is the scaling limit in the GHPU topology (c.f.\ Section~\ref{sec-ghpu}) of the free Boltzmann quadrangulation with simple boundary equipped with its graph metric, the measure which assigns each vertex a mass equal to its degree, and its boundary path~\cite{gwynne-miller-simple-quad}.

The \emph{Brownian half-plane} is the random-curve-decorated metric measure space $(H^\infty , d^\infty , \mu^\infty , \xi^\infty)$ with the topology of $\BB H$ which (like the Brownian disk) can be constructed via a continuum analog of the Schaeffer bijection~\cite{gwynne-miller-uihpq,bmr-uihpq}.  The path $\xi^\infty : \BB R \rta H^\infty$ is called the \emph{boundary path}.  The Brownian half-plane is the scaling limit in the local GHPU topology of the UIHPQ$_{\op{S}}$ equipped with its graph metric, the measure which assigns each vertex a mass equal to its degree, and its boundary path~\cite{gwynne-miller-uihpq}.

As alluded to in Section~\ref{sec-overview} and discussed in more detail in Section~\ref{sec-lqg-prelim}, the $\sqrt{8/3}$-LQG metric~\cite{lqg-tbm1,lqg-tbm2,lqg-tbm3} gives a natural embedding of the Brownian disk into the unit disk $\BB D$ and the Brownian half-plane into $\BB H$ which enables one to define an independent $\SLE_6$ curve between two given boundary points of either of these metric spaces.

\subsection{Main results}
\label{sec-results}

We first state the finite-volume version of our scaling limit result. 
Define the boundary length scaling constant
\eqb \label{eqn-normalizing-constant}
\bcon := \frac{2^{3/2}}{3} .
\eqe 
Also fix a time scaling constant $\tcon > 0$, which we will not compute explicitly, which depends on the random planar map model and on the scaling parameter of the $3/2$-stable process appearing in~\cite[Corollary~1.19]{wedges}, which has not been computed explicitly (the choice of $\tcon$ is made in~\eqref{eqn-tcon-choice}). 
 
Fix $\frk l_L , \frk l_R > 0$ and a sequence of pairs of positive integers $\{(\el_L^n ,\el_R^n)\}_{n\in\BB N}$ such that $ \el_L^n + \el_R^n$ is always even, $\bcon^{-1} n^{-1/2} \el_L^n  \rta \frk l_L$, and $\bcon^{-1}  n^{-1/2} \el_R^n \rta \frk l_R$.  

For $n\in\BB N$, let $(Q^n ,\BB e^n )$ be a free Boltzmann quadrangulation with simple boundary of perimeter $\el_L^n+\el_R^n$ (Definition~\ref{def-fb}), viewed as a connected metric space by replacing each edge with an isometric copy of the unit interval and let $\theta^n$ be a critical face percolation configuration on $Q^n$ (so that conditional on $Q^n$, $\theta^n$ assigns to each face of $Q^n$ the color white with probability $3/4$ and the color black with probability $1/4$).

Let $d^n$ be the graph metric on $Q^n$, thus extended, rescaled by $\bcon^{-1/2} n^{-1/4} = (9/8)^{1/4} n^{-1/4}$.  Let $\mu^n$ be the measure on~$Q^n$ which assigns to each vertex a mass equal to $(4n)^{-1}$ times its degree.  Let $\beta^n : [0,\el_L^n  + \el_R^n ] \rta \bdy Q^n$ be the counterclockwise boundary path of $Q^n$ started from the root edge $\BB e^n$, extended by linear interpolation, and define the rescaled boundary path $\xi^n(s) := \beta^n(\bcon n^{1/2}s )$ for $s \in [0, \bcon^{-1} n^{-1/2} (\el_L^n + \el_R^n)]$ where here $\bcon$ is as in~\eqref{eqn-normalizing-constant}.  Also let $\lambda^n : [0,\infty) \rta Q^n$ be the percolation exploration path of $(Q^n,\BB e^n , \theta^n)$ with $\el_L^n$-white/$\el_R^n$-black boundary conditions (Section~\ref{sec-intro-def-perc}), extended to a continuous path on $[0,\infty)$ which traces the edge $\lambda^n(j)$ during each time interval $[j-1/2,j]$ for $j\in\frac12 \BB N$; and for $t\geq 0$ let $\eta^n(t ) := \lambda^n( \tcon n^{ 3/4} t)$. 
Define the doubly-marked curve-decorated metric measure spaces
\eqbn
\frk Q^n := \left( Q^n , d^n , \mu^n , \xi^n , \eta^n \right)  .
\eqen

Let $(H ,  d , \mu , \xi)$ be a free Boltzmann Brownian disk with boundary length $\frk l_L + \frk l_R$ equipped with its natural metric, area measure, and boundary path. 
Conditional on $H$, let $\eta$ be a chordal $\SLE_6$ from $\xi(0)$ to $\xi(\frk l_R)$ in~$H$, parameterized by quantum natural time (see Section~\ref{sec-lqg-prelim} for details) and the doubly-marked curve-decorated metric measure space $\frk H = (H,d,\mu,\xi,\eta)$.

\begin{thm} \label{thm-perc-conv}
One has $\frk Q^n \rta \frk H$ in law with respect to the (two-curve) Gromov-Hausdorff-Prokhorov-uniform topology. That is, face percolation on a free Boltzmann quadrangulation with simple boundary converges to chordal $\SLE_6$ on a free Boltzmann Brownian disk. 
\end{thm}

See Section~\ref{sec-ghpu} for more on the GHPU topology.

It will be clear from our proof of Theorem~\ref{thm-perc-conv} that we actually obtain a slightly stronger statement: namely, the joint law of $\frk Q^n$ and the associated rescaled boundary length process $Z^n = (L^n,R^n)$ of Definition~\ref{def-bdy-process-rescale} converges to $\frk H$ and its associated left/right boundary length process $Z = (L,R)$ (Section~\ref{sec-lqg-bdy-process}) in the GHPU topology on the first coordinate and the Skorokhod topology on the second coordinate.

We next state a scaling limit theorem for face percolation on the UIHPQ$_{\op{S}}$, which will be an easy consequence of Theorem~\ref{thm-perc-conv}.  To state the theorem, let $(Q^\infty , \BB e^\infty )$ be a UIHPQ$_{\op{S}}$, viewed as a connected metric space by replacing each edge with an isometric copy of the unit interval as in the case of $Q^n$ above. Let $\theta^\infty$ be a critical face percolation configuration on $Q^\infty$.

For $n\in\BB N$, let $d^{\infty,n}$ be the graph metric on $Q^\infty$, thus extended, rescaled by $\bcon^{-1/2} n^{-1/4}$.  Let $\mu^{\infty,n}$ be the measure on $Q^\infty$ which assigns to each vertex a mass equal to $(4n)^{-1}$ times its degree.  Let $\beta^\infty : \BB R \rta \bdy Q^\infty$ be the boundary path of $Q^\infty$ with $\beta^\infty(0) = \BB e^\infty$, extended by linear interpolation, and let $\xi^{\infty,n}(s) := \beta^\infty(\bcon n^{1/2}s )$ for $s \in \BB R$.  Also let $\lambda^\infty : [0,\infty) \rta Q^\infty$ be the percolation exploration path of $(Q^\infty ,\BB e^\infty , \theta^\infty)$ with white/black boundary conditions, extended to $[0,\infty)$ as in the case of $\lambda^n$ above; and for $t\geq 0$ let $\eta^{\infty,n}(t ) := \lambda^\infty( \tcon n^{ 3/4} t)$ (with $\tcon$ the same time scaling constant as above).  Define the doubly-marked curve-decorated metric measure spaces
\eqbn
\frk Q^{\infty,n} := \left( Q^{\infty } , d^{\infty,n} , \mu^{\infty,n} , \xi^{\infty,n} , \eta^{\infty,n} \right).
\eqen

Let $(H^\infty ,  d^\infty , \mu^\infty , \xi^\infty )$ be a Brownian half-plane equipped with its natural metric, area measure, and boundary path. 
Conditional on $H^\infty$, let $\eta^\infty$ be a chordal $\SLE_6$ from $\xi^\infty(0)$ to $\infty$ in $H^\infty$, parameterized by quantum natural time and the doubly-marked curve-decorated metric measure space 
$\frk H^\infty = (H^\infty ,d^\infty ,\mu^\infty ,\xi^\infty ,\eta^\infty )$.
 
\begin{thm} \label{thm-perc-conv-uihpq}
One has $\frk Q^{\infty,n} \rta \frk H^\infty$ in law with respect to the local (two-curve) Gromov-Hausdorff-Prokhorov-uniform topology. That is, face percolation on the UIHPQ$_{\op{S}}$ converges to chordal $\SLE_6$ on the Brownian half-plane.
\end{thm}

As in the case of Theorem~\ref{thm-perc-conv}, our proof of Theorem~\ref{thm-perc-conv-uihpq} also yields a scaling limit for the joint law of~$\frk Q^{\infty,n}$ and its rescaled left/right boundary length process~$Z^{\infty,n}$ (Definition~\ref{def-bdy-process-rescale}) toward $\frk H^\infty$ and its associated left/right boundary length process (Section~\ref{sec-lqg-bdy-process}).

\subsection{Other scaling limit results}
\label{sec-other-results}
 
There are a number of other natural settings in which one can consider the scaling limit of face percolation on a quadrangulation (or more generally other percolation models on random planar maps which can be explored via peeling). We expect that scaling limit results in these settings can be deduced from the results of this paper modulo some technical steps.\footnote{In the case of site percolation on a triangulation, several extensions along the lines described in this subsection will be proven in~\cite{ghs-metric-peano}, building on the present paper. }

One can consider face percolation on a free Boltzmann quadrangulation with simple boundary, with its law weighted by the total number of interior vertices, and study a percolation exploration path targeted at a uniformly random interior vertex $\BB v_*$ rather than a fixed boundary edge.  In this case, the scaling limit will be radial (rather than chordal) $\SLE_6$ on a free Boltzmann Brownian disk weighted by its area, targeted at a uniformly random interior point.  This result can be extracted from the main result of the present paper as follows.  Choose a uniformly random target edge on the boundary (in addition to the interior marked vertex) and follow the chordal exploration path from the root edge towards this random target edge until the first time that it separates the target edge from $\BB v_*$.  Once this happens, choose another uniformly random marked boundary edge on the boundary of the complementary connected component containing $\BB v_*$ and then repeat the procedure. Due to the target invariance of $\SLE_6$, one can produce a radial $\SLE_6$ curve by re-targeting and concatenating chordal $\SLE_6$ curves via a continuum analog of the above construction. Hence the scaling limit result follows from Theorem~\ref{thm-perc-conv} with a little bit of extra technical work.

One can also work with a quadrangulation of the sphere and consider the percolation exploration between two uniformly random marked vertices.  In this case, the scaling limit will be whole-plane $\SLE_6$ on the Brownian map.  To see this, one notes that if one explores the percolation exploration path for a little bit, then one ends up in the radial setting above. Since our scaling limit result only applies to free Boltzmann quadrangulations with simple boundary, without specified area, one needs to work with the free Boltzmann distribution on quadrangulations of the sphere (which is obtained from the free Boltzmann distribution on quadrangulations with simple boundary of perimeter~$2$ by identifying the boundary edges) in order for the unexplored regions to be free Boltzmann quadrangulations with simple boundary. This distribution, appropriately rescaled, converges vaguely to the infinite measure on doubly marked Brownian maps considered in~\cite{tbm-characterization}, where the area is sampled from the infinite measure $a^{-3/2} \,da$. 

It is also natural to ask for convergence of all of the interfaces on a free Boltzmann quadrangulation with simple boundary, rather than a single exploration path. In this setting, the scaling limit should be a free Boltzmann Brownian disk decorated by a CLE$_6$~\cite{shef-cle}.  
In the case of site percolation on a triangulation, this convergence is proven in~\cite{ghs-metric-peano} by exploring the discrete interfaces via a discrete analog of the branching $\SLE_6$ process used to construct CLE$_6$ in~\cite{shef-cle}, and then using the results of the present paper to show the convergence of this discrete branching process to its continuum analog.

\subsection{Outline}
\label{sec-outline}

In Section~\ref{sec-prelim}, we introduce some (mostly standard) notation and review the GHPU metric, the Brownian disk and Brownian half-plane, and LQG surfaces and their relationship to SLE and to Brownian surfaces. We will also restate some results from~\cite{gwynne-miller-char,gwynne-miller-sle6} about $\SLE_6$ on a free Boltzmann Brownian disk which will be needed in the present paper, in particular the description of the law of the left/right boundary length process (Theorem~\ref{thm-bdy-process-law}); and the characterization theorem in terms of the topology of the curve-decorated metric space and the law of the internal metric spaces parameterized by the complementary connected components of the curve at each time $t\geq 0$ (Theorem~\ref{thm-bead-mchar}). 

In Section~\ref{sec-peeling}, we recall the definition of the peeling procedure for the uniform infinite half-plane quadrangulation, review some formulas and estimates for this procedure, and give a precise definition of the face percolation peeling process and the associated percolation exploration path which we will show converges to $\SLE_6$, as discussed in Section~\ref{sec-intro-def-perc}.  In Section~\ref{sec-interface} we also discuss the relationship between this path and the face percolation interface. 

In the remainder of the paper we commence with the proofs of our main theorems. See Figure~\ref{fig-outline} for a schematic map of the argument. Our main focus is on proving the finite-volume scaling limit result Theorem~\ref{thm-perc-conv}, which will imply the infinite-volume version Theorem~\ref{thm-perc-conv-uihpq} via a short local absolute continuity argument.  However, we will frequently switch back and forth between proving statements in the finite-volume and infinite-volume settings, depending on the setting in which the proof is easier. We will transfer estimates between the two settings using a Radon-Nikodym derivative estimate for peeling processes on a free Boltzmann quadrangulation with respect to peeling processes on the UIHPQ$_{\op{S}}$ (Lemma~\ref{lem-perc-rn}). 

In Section~\ref{sec-bdy-process} we introduce the \emph{boundary length processes} for the percolation peeling process. These processes encode the number of edges on the outer boundary of the percolation peeling cluster to the left and right of the tip of the curve, the number of edges to the left and right of the starting edge which this cluster disconnects from the target edge (or $\infty$, in the case of the UIHPQ$_{\op{S}}$), and the differences between these quantities at each time $j$ for the peeling process. 

We then prove that the boundary length processes for face percolation on the UIHPQ$_{\op{S}}$ and on a free Boltzmann quadrangulation converge in the scaling limit to the analogous processes for the hulls of a chordal $\SLE_6$ on the Brownian half-plane and on a free Boltzmann Brownian disk, respectively.  In the case of the UIHPQ$_{\op{S}}$, the desired limiting boundary length process is a pair of independent $3/2$-stable processes~\cite[Corollary~1.19]{wedges} (see also Section~\ref{sec-lqg-prelim}) so this convergence statement amounts to a straightforward application of the peeling estimates of Section~\ref{sec-peeling} and the heavy-tailed central limit theorem. 
In the case of a free Boltzmann quadrangulation, however, the argument is more subtle and relies on the description of the desired limiting boundary length process from~\cite{gwynne-miller-sle6} (see Theorem~\ref{thm-bdy-process-law}) as well as some estimates for peeling which will also be used in subsequent sections.

In Section~\ref{sec-ghpu-tight} we prove tightness of the curve-decorated metric measure spaces appearing in Theorem~\ref{thm-perc-conv}. Since we already know that the rescaled free Boltzmann quadrangulations $(Q^n, d^n ,\mu^n , \xi^n)$ converge in the scaling limit to the Brownian disk, this amounts to proving that the rescaled percolation exploration paths~$\eta^n$ are equicontinuous in law. The idea of the proof is to estimate the diameter of the outer boundary of a small increment $\eta^n([(k-1)\delta , k\delta])$ using the scaling limit result for the boundary length process and estimates for distances along the boundary in a quadrangulation with simple boundary (which follow from analogous estimates for the Brownian disk), then use the fact that the Brownian disk has the topology of a disk to bound the diameter of $\eta^n([(k-1)\delta , k\delta])$ in terms of the diameter of its outer boundary. 

Most of the remainder of the paper is devoted to checking the hypotheses of the characterization result Theorem~\ref{thm-bead-mchar} for a subsequential limit $\wt{\frk H} = (\wt H,\wt d  , \wt\mu , \wt\xi,\wt\eta)$ of the curve-decorated metric measure spaces $\frk Q^n$ of Theorem~\ref{thm-perc-conv}. As will be explained in the first several subsections of Section~\ref{sec-identification}, one can deduce from the convergence of the boundary length processes and the Markov property of peeling that internal metrics on the complementary connected components of $\wt\eta([0,t])$ at each time $t$ have the same law as the corresponding objects for $\SLE_6$ on a Brownian disk (i.e., condition~\ref{item-bead-mchar-wedge} in Theorem~\ref{thm-bead-mchar} is satisfied). Furthermore, one can show that $\wt\eta$ hits itself at least as often as an SLE$_6$, so that there exists a coupling of $\wt{\frk H}$ with a $\SLE_6$-decorated free Boltzmann Brownian disk $\frk H = (H,d,\mu,\xi,\eta)$ and a continuous surjective measure-preserving, curve-preserving map $\Phi : H \rta \wt H$ which maps $H\setminus \eta$ bijectively to $\wt H\setminus \wt\eta$. 

The most difficult part of the proof is showing that $\Phi$ is injective, so that the topology and consistency condition~\ref{item-bead-mchar-homeo} in Theorem~\ref{thm-bead-mchar} is satisfied.  For this purpose we will use the topological result~\cite[Main Theorem]{almost-inj}, which says that if $\Phi : M \rta N$ is a continuous map between manifolds with boundary which is \emph{almost injective}, in the sense that $\Phi^{-1}(\Phi(x)) = \{x\}$ for a dense set of points $x \in M$ and \emph{light} in the sense that $\Phi^{-1}(\{y\})$ is totally disconnected for each $y \in N$, then $\Phi|_{M\setminus \bdy M}$ is an embedding. The map $\Phi$ discussed in the preceding paragraph is almost injective. 
 
In Section~\ref{sec-crossing} we prove an estimate for the percolation exploration path $\eta^n$ which will enable us to show that the above map $\Phi$ is light. 
In particular, we will to prove an estimate for the number of times that~$\eta^n$ can cross an annulus between two metric balls in~$Q^n$ (equivalently, the number of percolation ``arms" which cross such an annulus), which will enable us to conclude in Section~\ref{sec-homeo} that a subsequential limit of the curves $\eta^n$ can hit a single point at most 7 times. This will imply that the pre-image of any point under $\Phi$ is finite, and hence that $\Phi$ is light and therefore a homeomorphism. In Section~\ref{sec-main-proof} we will conclude the proof of Theorem~\ref{thm-perc-conv} and then deduce Theorem~\ref{thm-perc-conv-uihpq} by coupling the UIHPQ$_{\op{S}}$ with a free Boltzmann quadrangulation with large simple boundary in such a way that they agree in a neighborhood of the root edge with high probability.

\begin{figure}[ht!]
\begin{center}
\includegraphics[scale=0.85]{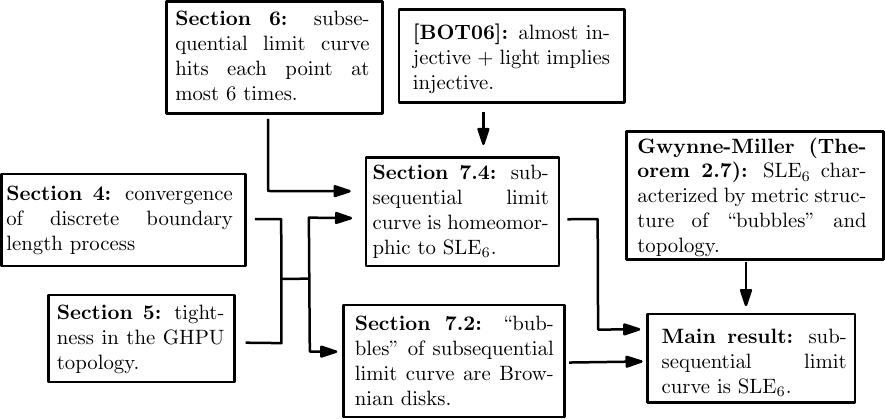}	
\end{center}
\vspace{-0.02\textheight}
\caption[Schematic outline of the proof of Theorem~\ref{thm-perc-conv}]{\label{fig-outline} Schematic illustration of the main statements involved in the paper and how they fit together.}
\end{figure}

\section{Preliminaries}
\label{sec-prelim}

\subsection{Notational conventions}
\label{sec-notation-prelim}

In this subsection, we will review some basic notation and definitions which will be used throughout the paper.
 
\subsubsection{Basic notation}
\label{sec-basic-notation}

\noindent
We write $\BB N$ for the set of positive integers and $\BB N_0 = \BB N\cup \{0\}$. 
\vspace{6pt}

\noindent
For $a,b \in \BB R$ with $a<b$ and $r > 0$, we define the discrete intervals $[a,b]_{r \BB Z} := [a, b]\cap (r \BB Z)$ and $(a,b)_{r \BB Z} := (a,b)\cap (r\BB Z)$.
\vspace{6pt}

\noindent
If $a$ and $b$ are two quantities, we write $a\preceq b$ (resp.\ $a \succeq b$) if there is a constant $C > 0$ (independent of the parameters of interest) such that $a \leq C b$ (resp.\ $a \geq C b$). We write $a \asymp b$ if $a\preceq b$ and $a \succeq b$.
\vspace{6pt}

\noindent
If $a$ and $b$ are two quantities depending on a variable $x$, we write $a = O_x(b)$ (resp.\ $a = o_x(b)$) if $a/b$ remains bounded (resp.\ tends to 0) as $x\rta 0$ or as $x\rta\infty$ (the regime we are considering will be clear from the context). We write $a = o_x^\infty(b)$ if $a = o_x(b^s)$ for every $s\in\BB R$, i.e., $a$ decays faster than any power of $b$ as $x$ tends to 0 or $\infty$ (depending on the context).  
\vspace{6pt}

\subsubsection{Graphs and maps}
\label{sec-graph-notation}
 
\noindent
For a planar map $G$, we write $\mcl V(G)$, $\mcl E(G)$, and $\mcl F(G)$, respectively, for the set of vertices, edges, and faces, respectively, of $G$.
\vspace{6pt}

\noindent
By a \emph{path} in $G$, we mean a function $ \lambda : I \rta \mcl E(G)$ for some (possibly infinite) discrete interval $I\subset \BB Z$, with the property that the edges $\lambda(i)$ and $\lambda(i+1)$ share an endpoint for each $i\in I$ other than the right endpoint of $I$. We also allow paths defined on discrete intervals $I \subset \frac12 \BB Z$ (such as the percolation exploration path). We do \emph{not} require that the edges traversed by $\lambda$ can be oriented in a consistent manner, since some of the paths we consider (such as the percolation exploration path) do not have this property.
\vspace{6pt}

\noindent
For sets $A_1,A_2$ consisting of vertices and/or edges of~$G$, we write $\op{dist}\left(A_1 , A_2 ; G\right)$ for the graph distance from~$A_1$ to~$A_2$ in~$G$, i.e.\ the minimum of the lengths of paths in $G$ whose initial edge either has an endpoint which is a vertex in $A_1$ or shares an endpoint with an edge in $A_1$; and whose final edge satisfies the same condition with $A_2$ in place of $A_1$.  If $A_1$ and/or $A_2$ is a singleton, we omit the set brackets. Note that the graph distance from an edge $e$ to a set $A$ is the minimum distance between the endpoints of $e$ and the set $A$.
We write $\op{diam}(G)$ for the maximal graph distance between vertices of $G$. 
\vspace{6pt}

\noindent
For $r>0$, we define the graph metric ball $B_r\left( A_1 ; G\right)$ to be the subgraph of $G$ consisting of all vertices of $G$ whose graph distance from $A_1$ is at most $r$ and all edges of $G$ whose endpoints both lie at graph distance at most $r$ from $A_1$.  If $A_1 = \{x\}$ is a single vertex or edge, we write $B_r\left( x ; G\right)$ for $B_r\left( \{x\} ; G\right)$.
\vspace{6pt}

\noindent
Let $Q$ be a quadrangulation with boundary and let $S\subset Q$ be a subgraph. We define its \emph{boundary} $\bdy_Q S$ of~$S$ relative to~$Q$ to be the subgraph of $S$ consisting of all vertices of $S$ which either belong to $\bdy Q$ or which are incident to an edge not in $S$; and the set of edges of $S$ which join two such vertices.  We typically drop the subscript $Q$ if the quadrangulation $Q$ we are considering is clear from the context.  If $S$ is itself a quadrangulation with boundary such that $\bdy Q$ lies in the external face of $S$ and every internal face of $S$ is a face of $Q$ (which is the case we will most often consider), then the boundary of $S$ relative to $Q$ coincides with the intrinsic boundary of $S$.

\subsubsection{Metric spaces}
\label{sec-metric-prelim}
 
Here we introduce some notation for metric spaces and recall some basic constructions.
Throughout, let $(X,d_X)$ be a metric space. 
\vspace{6pt}

\noindent
For $A\subset X$ we write $\op{diam} (A ; d_X )$ for the supremum of the $d_X$-distance between points in $A$.
\vspace{6pt}

\noindent
For $r>0$, we write $B_r(A;d_X)$ for the set of $x\in X$ with $d_X (x,A) \leq r$. We emphasize that $B_r(A;d_X)$ is closed (this will be convenient when we work with the local GHPU topology). 
If $A = \{y\}$ is a singleton, we write $B_r(y;d_X)$ for $B_r(\{y\};d_X)$.  
\vspace{6pt}

\noindent
For a curve $\gamma : [a,b] \rta X$, the \emph{$d_X$-length} of $\gamma$ is defined by 
\eqbn
\op{len}\left( \gamma ; d_X  \right) := \sup_P \sum_{i=1}^{\# P} d_X (\gamma(t_i) , \gamma(t_{i-1})) 
\eqen
where the supremum is over all partitions $P : a= t_0 < \dots < t_{\# P} = b$ of $[a,b]$. Note that the $d_X$-length of a curve may be infinite.
\vspace{6pt}

\noindent
For $Y\subset X$, the \emph{internal metric $d_Y$ of $d_X$ on $Y$} is defined by
\eqb \label{eqn-internal-def}
d_Y (x,y)  := \inf_{\gamma \subset Y} \op{len}\left(\gamma ; d_X \right) ,\quad \forall x,y\in Y 
\eqe 
where the infimum is over all curves in $Y$ from $x$ to $y$. 
The function $d_Y$ satisfies all of the properties of a metric on $Y$ except that it may take infinite values. 
\vspace{6pt}
 
\noindent
We say that $(X,d_X)$ is a \emph{length space} if for each $x,y\in X$ and each $\ep > 0$, there exists a curve of $d_X$-length at most $d_X(x,y) + \ep$ from $x$ to $y$. 
\vspace{6pt}

\subsubsection{Skorokhod topology} 
\label{sec-skorokhod}

For $k\in\BB N$, let $\mcl D_\infty^k$ be the set of cadlag functions $f : \BB R\rta \BB R^k$, i.e., those which are right continuous with left limits. Also let $\mcl D^k $ be the set of those $f \in \mcl D_\infty^k$ which extend continuously to the two-point compactification $[-\infty,\infty]$, equivalently $\lim_{t\rta-\infty} f(t)$ and $\lim_{t\rta\infty} f(t)$ exist. We view a cadlag function $f : [a,b] \rta \BB R^k$ as an element of $\mcl D^k$ by setting $f(t) = f(a)$ for $t<a$ and $f(t) =f(b)$ for $t>b$. 

We define the \emph{Skorokhod metric} on $\mcl D^k$ by
\eqbn
d^{\op{Sk}}(f,g) := \inf_\phi \max\left\{ \|f\circ \phi - g\|_\infty , \| \phi - \op{Id}\|_\infty \right\}
\eqen
where the infinimum is over all increasing homeomorphisms $\phi : \BB R\rta\BB R$, $\op{Id}$ denotes the identity function on $\BB R$, and $\|\cdot\|_\infty$ denotes the uniform norm. Note that $d^{\op{Sk}}(f,g)$ is finite for $f,g\in\mcl D^k$ since $f$ and $g$ are nearly constant outside of some compact interval. 

We define the \emph{local Skorokhod metric} on $\mcl D_\infty^k$ by
\eqbn
d^{\op{Sk}}_\infty(f,g) := \sum_{k=1}^\infty 2^{-k} \left(1 \wedge d^{\op{Sk}}(f|_{[-k,k]} , g|_{[-k,k]} ) \right) .
\eqen

\subsection{The Gromov-Hausdorff-Prokhorov-uniform metric}
\label{sec-ghpu}

In this subsection we will review the definition of the Gromov-Hausdorff-Prokhorov-uniform (GHPU) metric from~\cite{gwynne-miller-uihpq}, 
which is the metric with respect to which our scaling limit results hold. 
In fact, in this paper we will have occasion to consider the GHPU topology for metric measure spaces decorated by multiple curves, the theory of which is identical to the theory in the case of a single curve. Actually, for this paper we will only need to consider metric measure spaces with two curves, but we treat an arbitrary finite number of curves for the sake of completeness.
  
For a metric space $(X,d)$, we let $C_0(\BB R , X)$ be the space of continuous curves $\eta : \BB R\rta X$ which are ``constant at $\infty$," i.e.\ $\eta$ extends continuously to the extended real line $[-\infty,\infty]$. 
Each curve $\eta : [a,b] \rta X$ can be viewed as an element of $C_0(\BB R ,X)$ by defining $\eta(t) = \eta(a)$ for $t < a$ and $\eta(t) = \eta(b)$ for $t> b$. 
\begin{itemize}
\item Let $\BB d_d^{\op{H}}$ be the $d$-Hausdorff metric on compact subsets of $X$.
\item Let $\BB d_d^{\op{P}}$ be the $d$-Prokhorov metric on finite measures on $X$.
\item Let $\BB d_d^{\op{U}}$ be the $d$-uniform metric on $C_0(\BB R , X)$.
\end{itemize}

For $k\in \BB N$, let $\BB M_k^{\op{GHPU}}$ be the set of $3+k$-tuples $\frk X  = (X , d , \mu , \eta_1,\dots,\eta_k)$ where $(X,d)$ is a compact metric space, $\mu$ is a finite Borel measure on $X$, and $\eta_1,\dots,\eta_k \in C_0(\BB R,X)$. 

Given elements $\frk X^1 = (X^1 , d^1, \mu^1 , \eta^1_1,\dots , \eta_k^1) $ and $\frk X^2 =  (X^2, d^2,\mu^2,\eta^2_1 , \dots , \eta^2_k) $ of $ \BB M_k^{\op{GHPU}}$, a compact metric space $(W, D)$, and isometric embeddings $\iota^1 : X^1\rta W$ and $\iota^2 : X^2\rta W$, we define their \emph{GHPU distortion} by 
\begin{align}
\label{eqn-ghpu-var}
\op{Dis}_{\frk X^1,\frk X^2}^{\op{GHPU}}\left(W,D , \iota^1, \iota^2 \right)   
:=  \BB d^{\op{H}}_D \left(\iota^1(X^1) , \iota^2(X^2) \right) +   
\BB d^{\op{P}}_D \left(( (\iota^1)_*\mu^1 ,(\iota^2)_*\mu^2) \right) + 
\sum_{j=1}^k \BB d_D^{\op{U}}\left( \iota^1 \circ \eta_j^1 , \iota^2 \circ\eta_j^2 \right) .
\end{align}
We define the \emph{Gromov-Hausdorff-Prokhorov-Uniform (GHPU) distance} by
\begin{align} \label{eqn-ghpu-def}
 \BB d^{\op{GHPU}}\left( \frk X^1 , \frk X^2 \right) 
 = \inf_{(W, D) , \iota^1,\iota^2}  \op{Dis}_{\frk X^1,\frk X^2}^{\op{GHPU}}\left(W,D , \iota^1, \iota^2 \right)      ,
\end{align}
where the infimum is over all compact metric spaces $(W,D)$ and isometric embeddings $\iota^1 : X^1 \rta W$ and $\iota^2 : X^2\rta W$.
It is shown in~\cite[Proposition~1.3]{gwynne-miller-uihpq} that in the case when $k=1$, this defines a complete separable metric on~$\BB M_k^{\op{GHPU}}$ provided we identify two elements of~$\BB M_k^{\op{GHPU}}$ which differ by a measure- and curve- preserving isometry. Exactly the same proof shows that the same is true for general $k\in\BB N$.

GHPU convergence is equivalent to a closely related type of convergence which is often easier to work with, in which all of the curve-decorated metric measure spaces are subsets of a larger space. For this purpose we need to introduce the following definition, which we take from~\cite{gwynne-miller-uihpq}. 
 
\begin{defn}[HPU convergence] \label{def-hpu}
Let $(W ,D)$ be a metric space. Let $k\in\BB N$ and let $\frk X^n = (X^n , d^n , \mu^n , \eta_1^n, \dots ,\eta_k^n)$ for $n\in\BB N$ and $\frk X = (X,d,\mu,\eta_1,\dots , \eta_k)$ be elements of $\BB M^{\op{GHPU}}_k$ such that $X$ and each $X^n$ is a subset of $W$ satisfying $D|_X = d$ and $D |_{X^n} = d^n$. We say that $\frk X^n\rta \frk X$ in the \emph{$D$-Hausdorff-Prokhorov-uniform (HPU) sense} if $X^n \rta X$ in the $D$-Hausdorff metric, $\mu^n \rta \mu$ in the $D$-Prokhorov metric, and for each $j \in [1,k]_{\BB Z}$, $\eta_j^n \rta \eta_j $ in the $D$-uniform metric.  
\end{defn}

The following result, which is the variant of~\cite[Proposition~1.5]{gwynne-miller-uihpq} in the case of $k$ curves (and which is proven in exactly the same manner as in the case of one curve), will play a key role in Section~\ref{sec-identification}. 
 
\begin{prop} \label{prop-ghpu-embed}
Let $ \frk X^n = (X^n , d^n , \mu^n , \eta_1^n,\dots , \eta_k^n)$ for $ n\in\BB N $ and $\frk X = (X,d,\mu,\eta_1,\dots , \eta_k)$ be elements of $\BB M_k^{\op{GHPU}}$. Then $\frk X^n\rta \frk X$ in the GHPU topology if and only if there exists a compact metric space $(W,D)$ and isometric embeddings $X^n \rta W$ for $n\in\BB N$ and $X\rta W$ such that the following is true. If we identify $X^n$ and $X$ with their embeddings into $W$, then $\frk X^n \rta \frk X$ in the $D$-HPU sense.
\end{prop}
  
For our scaling limit result for percolation on the UIHPQ$_{\op{S}}$, we need to consider the local version of the GHPU metric.  
Following~\cite{gwynne-miller-uihpq}, for $k\in\BB N$ we let $\BB M^{\op{GHPU},\infty}_k$ be the set of $3+k$-tuples $\frk X = (X,d,\mu,\eta_1,\dots ,\eta_k)$ where $(X,d)$ is a locally compact length space, $\mu$ is a measure on $X$ which assigns finite mass to each finite-radius metric ball in $X$, and $\eta_1,\dots,\eta_k $ are curves in $X$ with the following property. For each $i\in \{1,\dots,k\}$, either (a) $\eta : \BB R\rta X$ or (b) $\eta : (a,b) \rta X$ for some open interval $(a,b) \subset \BB R$ (with possibly one of $a$ or $b$ equal to $\infty$) and $\eta$ extends to a continuous curve from the closure of $(a,b)$ to the one-point compactification $X\cup \{\infty\}$. In the latter case, we view $\eta$ as a continuous function $\BB R \rta X\cup \{\infty\}$ which is constant outside of $[a,b]$. 
Note that $\BB M_k^{\op{GHPU} }$ is not contained in $\BB M_k^{\op{GHPU},\infty}$ since elements of the former are not required to be length spaces.
 
The following definition, which is slightly modified from~\cite{gwynne-miller-uihpq}, is used to define the local GHPU metric in terms of the GHPU metric.  

\begin{defn} \label{def-ghpu-truncate}
Let $k\in\BB N$ and let $\frk X = (X , d, \mu,\eta_1,\dots,\eta_k)$ be an element of $\BB M_k^{\op{GHPU},\infty}$. For $\rho > 0$ and $j \in [1,n]_{\BB Z}$, let 
\eqb
\ul\tau_\rho^{\eta_j} := (-\rho) \vee \sup\left\{t < 0 \,:\, d(\eta_1(0) ,\eta_j(t)) = \rho\right\} \quad \op{and}\quad \ol\tau_\rho^{\eta_j} := \rho\wedge \inf\left\{t > 0 \,:\, d(\eta_1(0),\eta_j(t)) = \rho\right\} .
\eqe
The \emph{$\rho$-truncation} of $\eta_j$ is the curve $\frk B_\rho\eta \in C_0(\BB R , X)$ defined by
\eqbn
\frk B_\rho\eta_j(t) = 
\begin{cases}
\eta_j(\ul\tau_\rho^{\eta_j}) ,\quad &t\leq \ul\tau_\rho^{\eta_j}  \\
\eta_j(t) ,\quad &t\in ( \ul\tau_\rho^{\eta_j}  , \ol\tau_\rho^{\eta_j}) \\
\eta_j( \ol\tau_\rho^{\eta_j}) ,\quad &t\geq  \ol\tau_\rho^{\eta_j}  .
\end{cases}
\eqen
The \emph{$\rho$-truncation} of $\frk X$ is the curve-decorated metric measure space
\eqbn
\frk B_\rho \frk X  = \left( B_\rho(\eta_1(0) ;d) , d|_{B_\rho(\eta_1(0) ;d)} , \mu|_{B_\rho(\eta_1(0) ;d)} , \frk B_\rho\eta_1 ,\dots,\frk B_\rho\eta_k \right) .
\eqen
\end{defn}

Note that the curve $\eta_1$ plays a distinguished role in Definition~\ref{def-ghpu-truncate} since $\eta_1(0)$ is taken to be the base point. 
 
The \emph{local GHPU metric} on $\BB M_k^{\op{GHPU},\infty}$ is defined by
\eqb \label{eqn-ghpu-local-def}
\BB d^{\op{GHPU},\infty}  \left( \frk X^1,\frk X^2\right) = \int_0^\infty e^{-\rho} \left(1 \wedge \BB d^{\op{GHPU}}\left(\frk B_\rho\frk X^1, \frk B_\rho\frk X^2 \right)  \right) \, d\rho
\eqe 
where $\BB d^{\op{GHPU}}$ is as in~\eqref{eqn-ghpu-def}. It is shown in~\cite[Proposition~1.7]{gwynne-miller-uihpq} that in the case when $k=1$,~$\BB d^{\op{GHPU},\infty}$ defines a complete separable metric on~$\BB M_k^{\op{GHPU},\infty}$ provided we identify spaces which differ by a measure-preserving, curve-preserving isometry. The case of general $k\in\BB N$ is treated in exactly the same manner.

\begin{remark}[Graphs as elements of $\BB M_k^{\op{GHPU}}$] \label{remark-ghpu-graph}
In this paper we will often be interested in a graph $G$ equipped with its graph distance $d_G$.  In order to study continuous curves in $G$, we identify each edge of $G$ with a copy of the unit interval $[0,1]$.  We extend the graph metric on $G$ by requiring that this identification is an isometry. 

If $\lambda : [a,b]_{\BB Z} \rta \mcl E(G)$ is a path in $G$, we extend $\lambda$ from $[a,b]_{\BB Z}$ to $[a-1,b] $ in such a way that $\lambda$ is continuous and for each $i\in [a,b]_{\BB Z}$, $\lambda|_{[i-1,i]}$ is a path lying in the edge $\lambda(i)$. Note that there are multiple ways to do this, but different choices result in paths whose uniform distance from one another is at most~$1$.

If $G$ is a finite graph and we are given a measure $\mu$ on vertices of $G$ and curves $\lambda_1,\dots,\lambda_k$ in $G$ and we view $G$ as a connected metric space and $\lambda_1,\dots,\lambda_k$ as continuous curves as above, then $(G , d_G , \mu , \lambda_1,\dots,\lambda_k)$ is an element of $\BB M_k^{\op{GHPU}}$.  Similar considerations enable us to view infinite graphs equipped with a locally finite measure and $k$ curves as elements of $\BB M_k^{\op{GHPU},\infty}$. 
\end{remark}

\subsection{Liouville quantum gravity and SLE}
\label{sec-lqg-prelim}
 
In this subsection we review the definition of Liouville quantum gravity (LQG) surfaces (Section~\ref{sec-lqg-surface}) and explain their equivalence with Brownian surfaces in the case when $\gamma = \sqrt{8/3}$ and how this enables us to define $\SLE_6$-type curves on Brownian surfaces (Section~\ref{sec-lqg-metric}).

We also state the results about SLE and LQG from~\cite{gwynne-miller-char} which are used in our proofs (Section~\ref{sec-lqg-bdy-process}).  Our proofs do not make any explicit use of SLE or LQG outside of these results, so if the reader is willing to take the results described in this subsection as a black box, the paper can be read without any knowledge of SLE and LQG. 

\subsubsection{Liouville quantum gravity surfaces}
\label{sec-lqg-surface}
 
For $\gamma \in (0,2)$, a \emph{Liouville quantum gravity (LQG)} surface with $k\in\BB N_0$ marked points is an equivalence class of $(k+2)$-tuples $(D, h ,x_1,\dots ,x_k)$, where $D\subset \BB C$ is a domain; $h$ is a distribution on $D$, typically some variant of the Gaussian free field (GFF)~\cite{shef-kpz,shef-gff,ss-contour,shef-zipper,ig1,ig4}; and $x_1,\dots , x_k \in D\cup \bdy D$ are $k$ marked points. Two such $(k+2)$-tuples $(D, h ,x_1,\dots ,x_k)$ and $(\wt D, \wt h  ,\wt x_1,\dots , \wt x_k)$ are considered equivalent if there is a conformal map $f : \wt D \rta D$ such that
\eqb \label{eqn-lqg-coord}
f(\wt x_j) = x_j,\quad \forall j \in [1,k]_{\BB Z} \quad \op{and}\quad  \wt h = h\circ f + Q \log |f'|  \quad \op{where} \: Q = \frac{2}{\gamma} + \frac{\gamma}{2} .
\eqe  
Several specific types of $\gamma$-LQG surfaces (which correspond to particular choices of the GFF-like distribution~$h$) are studied in~\cite{wedges}.  In this paper we will only consider the special case when $\gamma =\sqrt{8/3}$ and the only quantum surfaces we will be interested in are the quantum disk and the $\sqrt{8/3}$-quantum wedge. 

It is shown in~\cite{shef-kpz} that a Liouville quantum gravity surface for general $\gamma \in (0,2)$ admits a natural area measure $\mu_h$, which can be interpreted as ``$e^{\gamma h(z)} \, dz$", where $dz$ is Lebesgue measure on $D$, and a length measure $\nu_h$ defined on certain curves in $D$, including $\bdy D$ and $\SLE_\kappa$-type curves for $\kappa = \gamma^2$~\cite{shef-zipper}.  These measures are invariant under coordinate changes of the form~\eqref{eqn-lqg-coord}, so one can think of a $\gamma$-LQG surface as an equivalence class of measure spaces modulo conformal maps. 

For $\gamma \in (0,2)$, a \emph{quantum disk} $(\BB D , h)$ is a finite-volume quantum surface typically taken to be parameterized by the unit disk, defined precisely in~\cite[Definition~4.21]{wedges}. One can consider quantum disks with fixed area or with fixed area and fixed boundary length. In this paper we will primarily be interested in the case of fixed boundary length and random area. A \emph{singly (resp.\ doubly) marked quantum disk} is a quantum disk together with one (resp.\ two) marked points sampled uniformly from its $\gamma$-quantum boundary length measure.  One can consider a doubly marked quantum disk with specified left and right boundary lengths (and possibly also area) by conditioning on the $\gamma$-quantum lengths of the two arcs between the marked points.

For $\alpha < Q$, an \emph{$\alpha$-quantum wedge} $(\BB H ,h^\infty , 0, \infty)$ is an infinite-volume (i.e., $\mu_{h^\infty}(\BB H ) =\infty$) doubly-marked quantum surface which is typically taken to be parameterized by the upper half plane.  This quantum surface is defined precisely in~\cite[Section~1.6]{shef-zipper} and in~\cite[Definition~4.5]{wedges}.  Roughly speaking, the distribution~$h^\infty$ can be obtained by starting with the distribution $\wt h - \alpha \log |\cdot|$, where $\wt h$ is a free-boundary GFF on $\BB H$, then zooming in near the origin and re-scaling to get a surface which describes the local behavior of this field when the additive constant is fixed appropriately~\cite[Proposition~4.7(ii)]{wedges}.  The case when $\alpha = \gamma$ is special because the $\gamma$-LQG boundary length measure is supported on points where the field has a $-\gamma$-log singularity, so the $\gamma$-quantum wedge can be thought of as describing the local behavior of the field at a quantum typical boundary point (see~\cite[Proposition~1.6]{shef-zipper} for a precise statement along these lines).  

It is particularly natural to consider a $\gamma$-LQG surface decorated by an independent $\SLE_{\kappa}$-type curve for $\kappa \in \{\gamma^2 , 16/\gamma^2\}$.  Such a curve admits a natural quantum parameterization with respect to the underlying field, which depends on the phase of $\kappa$:
\begin{enumerate}
\item For $\kappa \in (0,4)$, we parameterize by quantum length (which is shown to be well-defined on $\SLE_\kappa$ curves in~\cite{shef-zipper}). 
\item For $\kappa \in (4,8)$ we parameterize by \emph{quantum natural time}, which roughly speaking means that we parameterize by the ``quantum local time" of $\eta^\infty$ at the set of times when it disconnects a bubble from~$\infty$ (see~\cite[Definition~6.23]{wedges} for a precise definition). 
\item For $\kappa  \geq 8$, we parameterize by the quantum mass of the region filled in by the curve. 
\end{enumerate}

\subsubsection{The $\sqrt{8/3}$-LQG metric and $\SLE_6$ on a Brownian surface}
\label{sec-lqg-metric}

It was recently proven by Miller and Sheffield that in the special case when $\gamma =\sqrt{8/3}$, a $\sqrt{8/3}$-LQG surface admits a natural metric $\frk d_h$~\cite{lqg-tbm1,lqg-tbm2,lqg-tbm3}, building on \cite{qle}.  This metric is also invariant under coordinate changes of the form~\eqref{eqn-lqg-coord}.  Hence one can view a $\sqrt{8/3}$-LQG surface as a metric measure space. In the case of $\sqrt{8/3}$-LQG surfaces with boundary (such as quantum disks and quantum wedges), one also obtains a natural boundary path, modulo a choice of starting point, by traversing one unit of $\sqrt{8/3}$-LQG length in one unit of time. 

In particular, it follows from~\cite[Corollary~1.5]{lqg-tbm2} (combined with~\cite[Theorem 3]{legall-disk-snake}) that the quantum disk is equivalent as a  metric measure space to the Brownian disk, equipped with its natural metric and area measure. In fact,~\cite{legall-disk-bdy} shows that quantum disk and Brownian disk also agree as curve-decorated metric measure spaces when equipped also with their boundary paths parameterized according to the natural boundary length measure. 
This holds if we condition on area, boundary length, or both.  In particular, the quantum disk with boundary length $\frk l > 0$ is equivalent to the free Boltzmann Brownian disk with boundary length $\frk l$. Using this and a local comparison argument, it is shown in~\cite[Proposition~1.10]{gwynne-miller-uihpq} that the $\sqrt{8/3}$-quantum wedge is equivalent as a curve-decorated metric measure space to the Brownian half-plane. 

It is shown in~\cite{lqg-tbm3} that the metric measure space structure a.s.\ determines the embedding of the quantum surface into a subset of $\BB C$.  In particular, there is a canonical embedding of the Brownian disk (resp.\ the Brownian half-plane) into $\BB D$ (resp.\ $\BB H$). 

This embedding enables us to define a chordal $\SLE_6$ on a doubly marked Brownian disk $(H ,d , \mu , x, y)$ (with fixed area, boundary length, or both) via the following procedure:
\begin{enumerate}
\item Let $(\BB D ,h , -i,i)$ be the doubly marked quantum disk obtained by embedding our given Brownian disk into $(\BB D , -i, i)$. 
\item Let $\eta_{\BB D}$ be an independent chordal $\SLE_6$ from $-i$ to $i$ in $\BB D$, parameterized by quantum natural time with respect to $h$ (recall the discussion at the end of Section~\ref{sec-lqg-surface}), and let $\eta$ be the curve from $x$ to $y$ in $H$ which is the pre-image of $\eta_{\BB D}$ under the embedding map. 
\end{enumerate}
We note that the law of $\eta$ does not depend on the particular choice of embedding since the quantum natural time parameterization is invariant under coordinate changes as in~\eqref{eqn-lqg-coord}.  One can similarly define chordal $\SLE_6$ on the Brownian half-plane. 

The same procedure also allows one to define other variants of $\SLE_6$ or $\SLE_{8/3}$ on other Brownian surfaces, but in this paper we will only consider chordal $\SLE_6$ on the Brownian disk or the Brownian half-plane.

\subsubsection{Boundary length processes and characterization theorem}
\label{sec-lqg-bdy-process}

In this subsection we review some particular facts about $\SLE_6$ on a $\sqrt{8/3}$-quantum wedge or a doubly marked quantum disk (equivalently, on a Brownian half-plane or a doubly marked Brownian disk).

Suppose first that $(\BB H , h^\infty , 0,\infty)$ is a $\sqrt{8/3}$-quantum wedge and $\eta^\infty :[0,\infty)\rta \ol{\BB H}$ is a chordal $\SLE_6$ from~$0$ to~$\infty$ in $\BB H$ sampled independently from $h^\infty$ and then parameterized by quantum natural time with respect to~$h^\infty$.

For $t \geq 0$, let $K_t^\infty$ be the closure of the set of points disconnected from $\infty$ by $\eta^\infty([0,t])$ and let $L_t^\infty$ (resp.\ $R_t^\infty$) be equal to the $\nu_{h^\infty}$-length of the segment of $\bdy K_t^\infty\cap\BB H$ lying to the left (resp.\ right) of $\eta^\infty(t)$ minus the $\nu_{h^\infty}$-length of the segment of $(-\infty,0]$ (resp.\ $[0,\infty)$) which is disconnected from $\infty$ by $\eta^\infty(t)$.  We call $Z_t^\infty := (L_t^\infty ,R_t^\infty)$ the \emph{left/right boundary length process} of~$\eta^\infty$. See Figure~\ref{fig-bdy-process-sle} for an illustration of this definition. The process $Z^\infty$ is the continuum analog of the discrete left/right boundary length process of Definition~\ref{def-bdy-process} below as well as the so-called \emph{horodistance} process of~\cite{curien-glimpse}.

\begin{figure}[ht!]
\begin{center}
\includegraphics[scale=1]{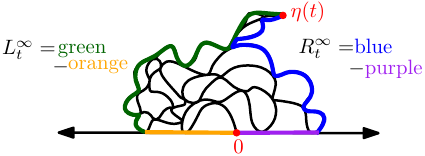} 
\caption[SLE boundary length process]{\label{fig-bdy-process-sle} Illustration of the definition of the left/right $\sqrt{8/3}$-LQG boundary length process $Z^\infty = (L^\infty,R^\infty)$ for SLE$_6$ in $\BB H$ with respect to an independent $\sqrt{8/3}$-quantum wedge. The analogous finite-volume process is defined similarly. }
\end{center}
\end{figure}

It is shown in~\cite[Corollary~1.19]{wedges} that $Z^\infty$ evolves as a pair of independent totally asymmetric $3/2$-stable processes with no upward jumps. That is, the L\'evy measure of each of $L^\infty$ and $R^\infty$ is given by a constant times $|t|^{-5/2} \BB 1_{(t < 0)} \, dt$ and
\eqb \label{eqn-levy-process-scaling}
\BB P\left[ L_1^\infty < -r \right] = \BB P\left[ R^\infty_1 <- r \right] \sim \lcon r^{-3/2} 
\eqe 
for $\lcon > 0$ a constant which is not computed explicitly. Henceforth, whenever we refer to a totally asymmetric $3/2$-stable processes with no upward jumps we mean one with this choice of scaling constant $\lcon$ (which determines the scaling constant for the L\'evy measure). 

By~\cite[Theorem~1.18]{wedges}, if we condition on $Z^\infty$ then the quantum surfaces obtained by restricting $h^\infty$ to the bubbles disconnected from $\infty$ by $\eta^\infty$, each marked by the point where $\eta^\infty$ finishes tracing its boundary, have the law of a collection of independent singly marked quantum disks indexed by the downward jumps of the two coordinates of $Z^\infty$, with boundary lengths specified by the magnitudes of the downward jumps. Furthermore, for each $t\geq 0$, the quantum surface obtained by restricting $h^\infty$ to $\BB H\setminus K_t^\infty$ is a $\sqrt{8/3}$-quantum wedge independent from the quantum surface obtained by restricting $h^\infty$ to $K_t^\infty$. 

One has similar statements in the case of a chordal $\SLE_6$ on a doubly marked quantum disk. Let $(\BB D , h , -i,i)$ be a doubly marked quantum disk with left/right boundary lengths $\frk l_L > 0$ and $\frk l_R > 0$ and let $\eta$ be an independent chordal $\SLE_6$ from $-i$ to $i$ in $\BB D$, parameterized by quantum natural time with respect to $h$.  In this case $\eta$ is only defined on some random finite time interval $[0,\sigma_0]$ but we extend the definition of $\eta$ to all of $[0,\infty)$ by setting $\eta(t) = i$ for $t > \sigma_0$.  Define the left/right boundary length process $Z = (L,R)$ in exactly the same manner as above.  Then $Z_0 = (0,0)$ and $Z_t = (-\frk l_L  , -\frk l_R)$ for each $t\geq \sigma_0$ (note that the left/right boundary length process in this paper coincides with the definition in~\cite{gwynne-miller-sle6} but is shifted by $(-\frk l_L ,-\frk l_R)$ as compared to the definition in~\cite{gwynne-miller-char}). The law of the process $Z$ is described by the following theorem, which is~\cite[Theorem~1.2]{gwynne-miller-sle6}.

\begin{thm}[\cite{gwynne-miller-sle6}] \label{thm-bdy-process-law}
The process $Z= (L,R)$ satisfies the following properties.
\begin{enumerate}
\item (Endpoint continuity) Almost surely, the terminal time $\sigma_0$ is finite. Furthermore, a.s.\ 
\eqb \label{eqn-sigma0-formula}
\sigma_0 = \inf\left\{ t \geq 0 : L_t = -\frk l_L \right\}  = \inf\left\{ t \geq 0 : R_t = -\frk l_R \right\}
\eqe 
and a.s.\ $\lim_{t\rta \sigma_0} \eta(t) = i$ and $\lim_{t \rta \sigma_0} Z_t = (-\frk l_L , -\frk l_R)$. \label{item-bead-process-time}
\item (Radon-Nikodym derivative) Let $Z^\infty = ( L^\infty, R^\infty)$ be a pair of independent totally asymmetric $3/2$-stable processes with no upward jumps, scaled as in~\eqref{eqn-levy-process-scaling},
and define
\eqb \label{eqn-S^wge-def}
\sigma_0^\infty := \inf\left\{ t\geq 0:  L_t^\infty \leq -\frk l_L \: \op{or} \:  R_t^\infty \leq -\frk l_R  \right\}   .
\eqe 
For $t\geq 0$, the law of $Z |_{[0,t]}$ restricted to the event $\{t < \sigma_0\}$ is absolutely continuous with respect to the law of $Z^\infty|_{[0,t]}$, with Radon-Nikodym derivative given by
\eqbn
 \left( \frac{ L_t^\infty  +  R_t^\infty}{\frk l_L + \frk l_R} + 1 \right)^{-5/2}\BB 1_{(t < \sigma_0^\infty) } .
\eqen
\label{item-bead-process-nat} 
\end{enumerate} 
\end{thm} 

Similarly to the $\sqrt{8/3}$-quantum wedge case, if we condition on $Z|_{[0,t]}$ then the conditional law of the quantum surfaces obtained by restricting $h $ to the bubbles disconnected from $i$ by $\eta $, each marked by the point where $\eta $ finishes tracing its boundary, has the law of a collection of independent singly marked quantum disks indexed by the downward jumps of the two coordinates of $Z $, with boundary lengths specified by the magnitudes of the downward jumps. Furthermore, the conditional law of the doubly marked quantum surface obtained by restricting $h$ to the connected component of $\BB D\setminus \eta([0,t])$ with $i$ on its boundary, with marked points $\eta(t)$ and $i$, is that of a doubly marked quantum disk with left/right boundary lengths $L_t + \frk l_L$ and $R_t + \frk l_R$~\cite[Theorem~1.1]{gwynne-miller-sle6}. 
The $\sqrt{8/3}$-LQG metric induced by the restriction of $h$ to a sub-domain coincides with the internal metric (Section~\ref{sec-metric-prelim}) of $\frk d_h$ on that sub-domain~\cite[Lemma~7.5]{gwynne-miller-char}, and a quantum disk conditioned on its boundary lengths is equivalent to a free Boltzmann Brownian disk (as defined in Section~\ref{sec-intro-def-disk}).
Hence we can re-phrase~\cite[Theorem~1.1]{gwynne-miller-sle6} in terms of Brownian disks instead of quantum disks to obtain the following theorem.

\begin{thm}[\cite{gwynne-miller-sle6}] \label{thm-sle-bead-nat}
Let $(H,d,\mu,\xi , x,y)$ be a doubly marked Brownian disk with left/right boundary lengths $\frk l_L$ and $\frk l_R$ and let $\eta$ be an independent chordal $\SLE_6$ in $H$ from $x$ to $y$, parameterized by quantum natural time. Let $Z = (L,R)$ be the corresponding left/right boundary length process.  For $t\geq 0$, let $\mcl U_t$ be the collection of singly marked metric measure spaces of the form $(U ,  d_U  ,  \mu_U  ,  x_U)$ where $U$ is a connected component of $H \setminus \eta ([0,t])$, $d_U  $ is the internal metric of $d$ on $U$, and $x_U$ is the point where $\eta$ finishes tracing $\bdy U$.  If we condition on $Z  |_{[0,t]}$, then the conditional law of $\mcl U_t$ is that of a collection of independent singly marked free Boltzmann Brownian disks with boundary lengths specified as follows. The elements of $\mcl U_t$ corresponding to the connected components of $H \setminus \eta ([0,t])$ which do not have the target point $y$ on their boundaries are in one-to-one correspondence with the downward jumps of the coordinates of $Z |_{[0,t]}$, with boundary lengths given by the magnitudes of the corresponding jump. The element of $\mcl U_t$ corresponding to the connected component of $H \setminus  \eta ([0,t])$ with $y$ on its boundary has boundary length $L_t   +   R_t  + \frk l_L  + \frk l_R  $. 
\end{thm}

It is shown in~\cite[Theorem~7.12]{gwynne-miller-char} that the Markov-type property of Theorem~\ref{thm-sle-bead-nat} together with the topological curve-decorated measure space structure of $(H,d,\mu ,\xi ,\eta)$ uniquely characterizes its law. We re-state this theorem here for the sake of reference.

\begin{thm}[\cite{gwynne-miller-char}] \label{thm-bead-mchar} 
Let $( \frk l_L, \frk l_R )\in (0,\infty)^2$ and suppose we are given a coupling of a doubly-marked free Boltzmann Brownian disk $(\wt H , \wt d , \wt \mu, \wt\xi , \wt x ,\wt y  )$ with left/right boundary lengths $\frk l_L$ and $\frk l_R$, a random continuous curve $\wt\eta: [0,\infty) \rta \wt H$ from $\wt x$ to $\wt y$, and a random process $ Z  = ( L  , R  )$ which has the law of the left/right boundary length process of a chordal $\SLE_6$ on a free Boltzmann Brownian disk with left/right boundary lengths $\frk l_L$ and $\frk l_R$, parameterized by quantum natural time.
Assume that the following conditions are satisfied.
\begin{enumerate}
\item (Laws of complementary connected components) For $t \geq 0$, let $\wt{\mcl U}_t$ be the collection of singly marked metric measure spaces of the form $(U , \wt d_U  , \wt\mu_U  , \wt x_U)$ where $U$ is a connected component of $\wt H \setminus \wt\eta  ([0,t])$, $\wt d_U $ is the internal metric of $\wt d $ on $U$, and $\wt x_U$ is the point where $\wt\eta $ finishes tracing $\bdy U$.  If we condition on $Z |_{[0,t]}$, then the conditional law of $\wt{\mcl U}_t$ is that of a collection of independent singly marked free Boltzmann Brownian disks with boundary lengths specified as follows. The elements of $\wt{\mcl U}_t$ corresponding to the connected components of $\wt H \setminus \wt\eta  ([0,t])$ which do not have the target point $\wt y$ on their boundaries are in one-to-one correspondence with the downward jumps of the coordinates of $Z |_{[0,t]}$, with boundary lengths given by the magnitudes of the corresponding jump. The element of $\wt{\mcl U}_t$ corresponding to the connected component of $\wt H \setminus \wt\eta  ([0,t])$ with $\wt y$ on its boundary has boundary length $L_t    +   R_t  +\frk l_L + \frk l_R   $. 
\label{item-bead-mchar-wedge}  
\item (Topology and consistency) The topology of $(\wt H ,\wt\eta )$ is determined by $Z $ in the same manner as the topology of a chordal $\SLE_6$ on a doubly marked free Boltzmann Brownian disk, i.e.\ there is a curve-decorated metric measure space $(H,d,\mu,\xi,\eta)$ consisting of a doubly marked Brownian disk with left/right boundary lengths $\frk l_L$ and $\frk l_R$ and an independent chordal $\SLE_6$ between its two marked points parameterized by quantum natural time and a homeomorphism $\Phi  : H \rta \wt H$ with $\Phi_* \mu =  \wt\mu $ and $\Phi \circ \eta   = \wt\eta $ (this homeomorphism is random and may depend on $(\wt H , \wt\eta)$ and $(H,d,\mu,\xi,\eta)$). Moreover, for each $t\in [0,\infty)\cap \BB Q$, $\Phi$ a.s.\ pushes forward the natural boundary length measure on the connected component of $H \setminus \eta ([0,t])$ containing the target point of $\eta$ on its boundary to the natural boundary length measure on the connected component of $\wt H \setminus \wt\eta ([0,t])$ containing the target point of $\wt\eta$ on its boundary (these boundary length measures are well-defined since we know the internal metric on the connected component is that of a Brownian disk).  \label{item-bead-mchar-homeo} 
\end{enumerate}
Then $(\wt H, \wt d  , \wt \mu , \wt\xi , \wt\eta )$ is equivalent (as a curve-decorated metric measure space) to a doubly-marked free Boltzmann Brownian disk with left/right boundary lengths $\frk l_L$ and $\frk l_R$ together with an independent chordal $\SLE_{6}$ between the two marked points, parameterized by quantum natural time. 
\end{thm}

 We emphasize that in hypothesis~\ref{item-bead-mchar-wedge} of Theorem~\ref{thm-bead-mchar}, we are taking internal metrics with respect to open sets, i.e., we are looking at the infimum of distances along paths which do not touch the boundaries of the components $U$.  

Theorem~\ref{thm-bead-mchar} is the key tool in the proof of Theorem~\ref{thm-perc-conv}; it will be used to identify the law of a subsequential limit of the curve-decorated metric measure spaces in that theorem.

Again suppose that $(\BB D ,h , -i,i)$ is a doubly-marked quantum disk with left/right boundary lengths $\frk l_L$ and $\frk l_R$ and that $\eta $ is an independent chordal $\SLE_6$ from $-i$ to $i$, parameterized by quantum natural time with respect to $h$. As above, let $Z = (L,R)$ be its left/right boundary length process and let $\sigma_0$ be the time at which $\eta$ reaches $i$.  We end this subsection by pointing out the manner in which the left/right boundary length process $Z$ encodes the topology of the pair $(\BB D , \eta)$ and the quantum length measure on the complementary connected components of $\eta$, which will be needed in Section~\ref{sec-bdy-process-id}.  
 
\begin{lem} \label{lem-inf-adjacency0}
Almost surely, the following is true.  If $t\in [0,\sigma_0]$ then $\eta$ hits the left (resp.\ right) arc of $\bdy\BB D$ from $-i$ to $i$ if and only if $L$ (resp.\ $R$) attains a running infimum at time $t$.  Furthermore, if $t_1 , t_2 \in [0,\sigma_0]$ with $t_1 < t_2$, then $\eta(t_1) = \eta(t_2)$ if and only if either
\eqb \label{eqn-inf-adjacency0}
 L_{t_1}  =  L_{t_2} = \inf_{s\in [t_1,t_2]}  L_s \quad\op{or} \quad
 R_{t_1}  =  R_{t_2} = \inf_{s\in [t_1,t_2]}   R_s .
\eqe  
\end{lem}
\begin{proof}
The first statement is immediate since, by the definition of $Z = (L,R)$, a running infimum of $L$ (resp.\ $R$) is the same as a time at which the length of the arc disconnected from $i$ by $\eta$ which lies to the left (resp.\ right) of $-i$ increases. 
For the second statement, we observe that if $\eta(t_1) = \eta(t_2)$, then $\eta(t_1)$ lies on the outer boundary of $\eta([0,t_2-\ep])$ for each $\ep \in (0,t_2-t_1)$  and at time $t_2$, $\eta$ disconnects the boundary arc of $\eta([0,t_2-\ep])$ between $\eta(t_1)$ and $\eta(t_2)$ from the target point $i$. Therefore~\eqref{eqn-inf-adjacency0} holds. The converse is obtained similarly.
\end{proof} 

One can also describe the quantum boundary length measure on the connected components of $\BB D\setminus \eta([0,t])$ in terms of $Z$. 
If we let 
\eqbn
T^L_t(u) := \sup\left\{ s \leq t : L_s \leq L_t  - u \right\} ,\quad \forall u \in   \left[ 0, L_t  -\inf_{s\in [0,t]} L_s \right]
\eqen
then it follows from the definition of the left/right boundary length process that $\eta(T^L_t(u))$ is the point on the outer boundary of $\eta([0,t])$ lying to the left of $\eta(t)$ with the property that the counterclockwise arc of the outer boundary of $\eta([0,t])$ from $\eta(T^L_t(u))$ to $\eta(t)$ is equal to $u$. A similar statement holds with ``left" in place of ``right".
Furthermore, if $\eta$ disconnects a bubble $B$ from $i$ at time $\tau$ which lies to the left of $\eta$ then 
\eqbn
\left(L_{\tau^-} - L_\tau \right) \wedge \left( L_{ \tau^- } - \inf_{s\in [0,\tau ]} L_s \right) 
\eqen
 is the quantum length of the boundary segment of $B$ which is traced by $\eta$ (the rest of $\bdy B$ is part of $\bdy\BB D$) and if we set
\eqbn
T_\tau^L(u) := \sup\left\{ s \leq t : L_s \leq L_{\tau^-}  - u \right\} ,\quad \forall u \in \left[ 0  , \left(L_{\tau^-} - L_\tau \right) \wedge \left( L_{ \tau^- } - \inf_{s\in [0,\tau ]} L_s \right)    \right] 
\eqen
then $\eta(T_\tau^L(u))$ is the point of~$\bdy B$ such that the clockwise arc of~$\bdy B$ from $\eta(\tau)$ to~$\eta(T_\tau^L(u))$ has quantum length~$u$.  Similar statements hold for bubbles disconnected from~$i$ on the right side of~$\eta$.

We note that, as explained in~\cite[Figure~1.15, Line 3]{wedges}, in the infinite-volume case the curve-decorated topological space $(\BB H , \eta^\infty)$ can be expressed as an explicit functional of the boundary length process~$Z^\infty$. By local absolute continuity one has a similar description for the topology of $(\BB D , \eta)$ in terms of~$Z$.

\section{Peeling of the UIHPQ with simple boundary}
\label{sec-peeling}

In this section, we will review the peeling procedure for the UIHPQ$_{\op{S}}$ and for free Boltzmann quadrangulations with simple boundary (also known as the spatial Markov property) and formally define the percolation peeling process for face percolation on a quadrangulation, which is the main object of study in this paper. The first rigorous use of peeling was in~\cite{angel-peeling}, in the context of the uniform infinite planar triangulation. The peeling procedure was later adapted to the case of the uniform infinite planar quadrangulation~\cite{benjamini-curien-uipq-walk}. In this paper, we will only be interested in peeling on the UIHPQ$_{\op{S}}$ and on free Boltzmann quadrangulations with simple boundary, which is also studied, e.g., in~\cite{angel-curien-uihpq-perc,angel-ray-classification,richier-perc,gwynne-miller-saw,caraceni-curien-saw,gwynne-miller-simple-quad}.  

In Section~\ref{sec-fb}, we will review some basic estimates for the free Boltzmann distribution on quadrangulations with simple boundary and given perimeter.  In Section~\ref{sec-peeling-basic}, we will review the definition of peeling and introduce notation for the objects involved (which is largely consistent with that of~\cite{gwynne-miller-saw}).  In Section~\ref{sec-peeling-estimate}, we will review some formulas and estimates for peeling probabilities.  In Section~\ref{sec-perc-peeling} we will define the face percolation peeling process, which is main peeling process which we will be interested in this paper and which is used to define the peeling exploration paths in Theorems~\ref{thm-perc-conv} and~\ref{thm-perc-conv-uihpq}.  In Section~\ref{sec-interface} we will discuss how this peeling process is related to face percolation interfaces.

\subsection{Estimates for free Boltzmann quadrangulations with simple boundary}
\label{sec-fb}

Recall the set $\mcl Q_{ \op{S} }^\shortrightarrow(n,\el)$ of rooted quadrangulations with simple boundary having $2l$ boundary edges and $n$ interior vertices, the free Boltzmann distribution on quadrangulations with simple boundary of perimeter $2\el$ from Definition~\ref{def-fb}, and the associated partition function $\frk Z$ from~\eqref{eqn-fb-partition}.
Stirling's formula implies that for each even $\el \in\BB N$,  
\eqb \label{eqn-stirling-asymp}
\frk Z(\el ) = (c_{\frk Z} +o_\el(1)) 54^{\el/2}  \el^{-5/2}  
\eqe   
for $c_{\frk Z} > 0$ a universal constant. 
By~\cite[Equation (2.11)]{bg-simple-quad} (c.f.~\cite[Section 6.2]{curien-legall-peeling}), 
\eqb \label{eqn-simple-quad-area} 
\# \mcl Q_{\op{S}}^{\srta}(n,\el) =  3^{n-1} \frac{(3\el)! (3\el-3+2n)!}{n! \el! (2\el-1)! (n+3\el-1)!}  =  C_{\op{S}}^\srta(\el) (1 + o_n(1)) 12^n n^{-5/2}  
\eqe  
where the $o_n(1)$ can possibly depend on $\el$ and
\eqb \label{eqn-simple-quad-area-coeff}
C_{\op{S}}^\srta(\el) = \frac{8^{\el-1} (3l)!}{3\sqrt\pi \el! (2l-1)!}  = \frac{1}{8\sqrt{3} \pi }   54^{\el} (1+o_\el(1)) \sqrt \el . 
\eqe 
From~\eqref{eqn-stirling-asymp} (applied with $2\el$ in place of $\el$),~\eqref{eqn-simple-quad-area}, and~\eqref{eqn-simple-quad-area-coeff}, we obtain the tail distribution of the number of interior vertices of a free Boltzmann quadrangulation with simple boundary of perimeter $2\el$: if $(Q, \BB e)$ is such a random quadrangulation, then 
\eqb \label{eqn-fb-area-tail}
\BB P\left[ \# \mcl V(Q\setminus \bdy Q) = n \right] = (c  + o_\el(1)) (1 + o_n(1)) \el^{3} n^{-5/2} 
\eqe 
where $c > 0$ is a universal constant, the rate of the $o_\el(1)$ is universal, and the rate of the $o_n(1)$ depends on $\el$.

\subsection{The peeling procedure}
\label{sec-peeling-basic}
  
\subsubsection{General definitions for peeling} 
\label{sec-general-peeling}

Let~$Q$ be a finite or infinite quadrangulation with simple boundary. For an edge $e  \in \mcl E(\bdy Q)$, we let $\frk f(Q,e)$ be the quadrilateral of~$Q$ containing $e$ on its boundary or $\frk f(Q,e) = \emptyset$ if $Q \in \mcl Q_{\op{S}}^\srta(0,1)$ is the trivial one-edge quadrangulation with no interior faces.  If $\frk f(Q,e) \not=\emptyset$, the quadrilateral $\frk f(Q,e)$ has either two, three, or four vertices in $\bdy Q$, so divides~$Q$ into at most three connected components, whose union includes all of the vertices of~$Q$ and all of the edges of~$Q$ except for~$e$. These components have a natural cyclic ordering inherited from the cyclic ordering of their intersections with~$\bdy Q$.  We write 
\eqbn
\frk P(Q,e) \in \{\emptyset\} \cup (\BB N_0 \cup \{\infty\}) \cup (\BB N_0 \cup \{\infty\})^2 \cup (\BB N_0 \cup \{\infty\})^3
\eqen
 for the vector whose elements are the number of edges of each of these components shared by $\bdy Q$, listed in counterclockwise cyclic order started from $e$, or $\frk P(Q,e) = \emptyset$ if $\frk f(Q,e) = \emptyset$. We refer to $\frk P(Q,e)$ as the \emph{peeling indicator}. 
 
The peeling indicator determines the total boundary lengths of each of the connected components of $Q\setminus \frk f(Q,e)$, not just the lengths of their intersections with $\bdy Q$. Indeed, if $i\in \{1,2,3\}$ and the $i$th component of $\frk P(Q,e)$ is $k$, then the total boundary length of the $i$th connected component of $Q\setminus \frk f(Q,e)$ in counterclockwise cyclic order is $k+3$ if there is only one such component (Figure~\ref{fig-peeling-cases}, left panel); $k+1$ if there is more than one component and $k$ is odd (Figure~\ref{fig-peeling-cases}, middle panel); $k+2$ if $k$ is even (Figure~\ref{fig-peeling-cases}, middle and right); or $\infty$ if $k=\infty$. 

The procedure of extracting $\frk f(Q,e)$ and $\frk P(Q,e)$ from $(Q,e)$ will be referred to as \emph{peeling $Q$ at $e$}. See Figure~\ref{fig-peeling-cases} for an illustration of the possible cases that can arise when peeling $Q$ at $e$.

\begin{figure}[ht!]
\begin{center}
\includegraphics[scale=1]{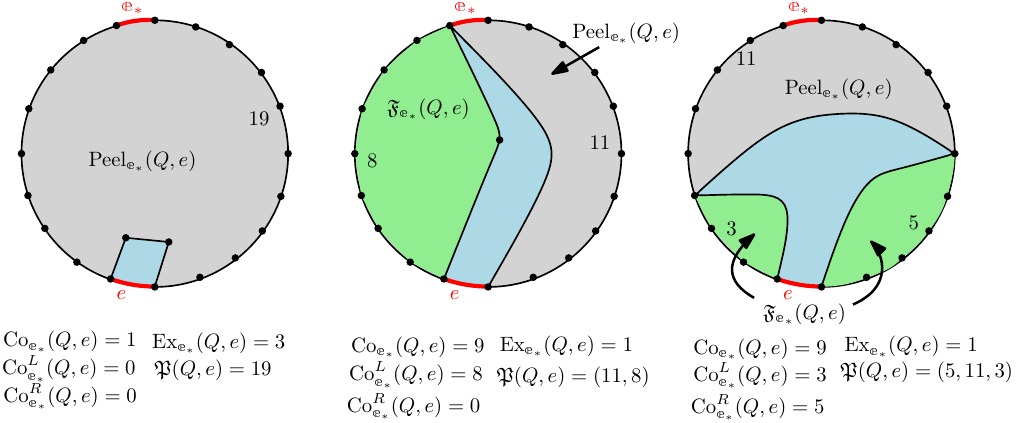} 
\caption[Peeling a free Boltzmann quadrangulation]{\label{fig-peeling-cases} A finite rooted quadrangulation with simple boundary $(Q,e) \in \mcl Q_{\op{S}}^\srta(n,20)$ together with three different possible cases for the peeled quadrilateral $\frk f(Q,e)$ (shown in light blue). In each case, the component $\op{Peel}_{\BB e_*}(Q,e)$ of $Q\setminus \frk f(Q,e)$ with $\BB e_*$ on its boundary is shown in grey and the union $\frk F_{\BB e_*}(Q,e)$ of the other components is shown in light green. The number of covered edges, left/right covered edges, and exposed edges and the peeling indicator are listed under each figure. }
\end{center}
\end{figure}

Suppose now that $\BB e_* \in \bdy Q\setminus \{e\}$ or that $\bdy Q$ is infinite and $\BB e_* = \infty$. 
\begin{itemize} 
\item Let $\op{Peel}_{\BB e_*}(Q , e)$ be the connected component of $Q\setminus \frk f(Q,e)$ with $\BB e_*$ on its boundary, or $\op{Peel}_{\BB e_*}(Q,e) = \emptyset$ if $\frk f(Q,e) =\emptyset$ (equivalently $(Q,e) \in \mcl Q_{\op{S}}^\srta(0,1)$). 
\item Let $\frk F_{\BB e_*}(Q,e)$ be the union of the components of $Q\setminus \frk f(Q,e)$ other than $\op{Peel}_{\BB e_*}(Q,e)$ or $\frk F_{\BB e_*}(Q,e) = \emptyset$ if $\frk f(Q,e) =\emptyset$.
\item Let $\op{Ex}_{\BB e_*}(Q,e)$ be the number of \emph{exposed edges} of $\frk f(Q,e)$, i.e.\ the number of edges of $ \op{Peel}_{\BB e_*}(Q,e)$ which do not belong to $\bdy Q$ (equivalently, those which are adjacent to $\frk f(Q,e)$). 
\item Let $\op{Co}_{\BB e_*}(Q,e)$ be the number of \emph{covered edges} of $\bdy Q$, i.e.\ the number of edges of $ \bdy Q $ which do not belong to $\op{Peel}_{\BB e_*}(Q,e)$ (equivalently, one plus the number of such edges which belong to $\frk F_{\BB e_*}(Q,e)$). Also let $\op{Co}_{\BB e_*}^L(Q,e)$ (resp.\ $\op{Co}_{\BB e_*}^R(Q,e)$) be the number of \emph{left (resp.\ right) covered edges}, i.e., the number of covered edges lying in the left (resp.\ right) arc of $\bdy Q$ from $e$ to $\BB e_*$, with $e$ not included in either arc and $\BB e_*$ included in the left arc, so that 
\eqb
\label{eqn-covered-identity}
\op{Co}_{\BB e_*}(Q,e) = \op{Co}_{\BB e_*}^L(Q,e) + \op{Co}_{\BB e_*}^R(Q,e) + 1 .
\eqe
\end{itemize}

\subsubsection{Markov property and peeling processes}
\label{sec-peeling-process}

If $(Q,\BB e)$ is a free Boltzmann quadrangulation with simple boundary of perimeter $2\el$ for $\el\in \BB N\cup \{\infty\}$ (recall that $\el=\infty$ corresponds to the UIHPQ$_{\op{S}}$) and we condition on $\frk P(Q,\BB e)$, then the connected components of $Q\setminus \frk f(Q,\BB e)$ are conditionally independent. 
The conditional law of each of the connected components, rooted at one of the edges of $\frk f(Q,\BB e)$ on its boundary (chosen by some deterministic convention in the case when there is more than one such edge), is the free Boltzmann distribution on quadrangulations with simple boundary and perimeter $2\wt \el$ (Definition~\ref{def-fb}), for a $\sigma(\frk P(Q,\BB e))$-measurable choice of $\wt \el \in \BB N\cup \{\infty\}$. These facts are collectively referred to as the \emph{Markov property of peeling}.

Due to the Markov property of peeling, one can iteratively peel a free Boltzmann quadrangulation with boundary to obtain a sequence of nested free Boltzmann quadrangulations with simple boundary.
To make this notion precise, let $\el \in \BB N\cup \{\infty\}$ and let $(Q, \BB e_*)$ be a free Boltzmann quadrangulation with simple boundary with perimeter $2\el$; we also allow $\BB e_*=\infty$ in the UIHPQ$_{\op{S}}$ case (when $\el=\infty$). 
A \emph{peeling process} of $Q$ targeted at $\BB e_*$ is described by a sequence of quadrangulations with simple boundary $\ol Q_j \subset Q$ for $j\in\BB N_0$, called the \emph{unexplored quadrangulations}, such that the following is true. 
\begin{enumerate}
\item We have $\ol Q_0 = Q$ and for each $j \in \BB N_0$ for which $\ol Q_{j-1}\not=\emptyset$, we have $\dot e_j \in \mcl E(\bdy \ol Q_{j-1} )$ and $\ol Q_j = \op{Peel}_{\BB e_*}(\ol Q_{j-1} , \dot e_j)$. For each $j \in \BB N_0$ for which $\ol Q_{j-1}=\emptyset$, we also have $\ol Q_j =\emptyset$. \label{item-peeling-iterate}
\item Each edge $\dot e_{j+1}$ for $j\in \BB N_0$ is chosen in a manner which is measurable with respect to the $\sigma$-algebra generated by the peeling indicator variables $\frk P(\ol Q_{i -1} , \dot e_i)$ for $i \in [1,j]_{\BB Z}$, the peeling cluster $\dot Q_j := (Q \setminus \ol Q_j) \cup (\bdy \ol Q_j \setminus \bdy Q) $, and possibly some additional random variables which are independent from $\ol Q_j$.    \label{item-peeling-msrble}
\end{enumerate} 
It follows from the Markov property of peeling that for each $j \in \BB N_0$, the conditional law of $(\ol Q_j , \dot e_{j+1})$ given the $\sigma$-algebra $\mcl F_j$ of condition~\ref{item-peeling-msrble} is that of a free Boltzmann quadrangulation with perimeter $2\wt\el$ for some $\wt \el\in\BB N_0\cup \{\infty\}$ which is measurable with respect to $\mcl F_j$ (where here a free Boltzmann quadrangulation with perimeter~$0$ is taken to be the empty set).

\subsubsection{Peeling formulas and estimates}
\label{sec-peeling-estimate}
     
Let $(Q^\infty , \BB e^\infty )$ be a UIHPQ$_{\op{S}}$.  
As explained in~\cite[Section~2.3.1]{angel-curien-uihpq-perc}, the distribution of the peeling indicator of Section~\ref{sec-general-peeling} when we peel at the root edge is described as follows, where here $\frk Z$ is the free Boltzmann partition function from~\eqref{eqn-fb-partition}. 
\begin{align} \label{eqn-uihpq-peel-prob}
\BB P\left[  \frk P(Q^\infty , \BB e^\infty) = \infty  \right]  
&= \frac38 \notag \\
\BB P\left[ \frk P(Q^\infty , \BB e^\infty) = (k , \infty) \right]  
&=  \frac{1}{12} 54^{(1-k)/2}  \frk Z(k+1)  ,\quad \text{$\forall k\in \BB N$ odd} \notag \\
\BB P\left[ \frk P(Q^\infty , \BB e^\infty) = (k , \infty) \right] 
&=   \frac{1}{12}  54^{-k/2}   \frk Z(k+2)  ,\quad \text{$\forall k\in \BB N_0$ even} \notag \\
\BB P\left[ \frk P(Q^\infty , \BB e^\infty) = (k_1, k_2 , \infty)   \right]  
&=  54^{-(k_1+k_2)/2}   \frk Z(k_1+1) \frk Z(k_2+1)  ,   \quad \text{$\forall k_1,k_2\in  \BB N$ odd} . 
\end{align} 
We get the same formulas if we replace $(k,\infty)$ with $(\infty,k)$ or $(k_1,k_2,\infty)$ with either $(\infty,k_1,k_2)$ or $(k_1,\infty,k_2)$ (which corresponds to changing which side of $\BB e^\infty$ the bounded complementary connected components of $\frk f(Q^\infty ,\BB e^\infty)$ lie on).  
 
From~\eqref{eqn-stirling-asymp}, we infer the following approximate versions of the probabilities~\eqref{eqn-uihpq-peel-prob}.
\begin{align} \label{eqn-uihpq-peel-prob-asymp} 
\BB P\left[ \frk P(Q^\infty , \BB e^\infty) = (k , \infty) \right]  
&\asymp k^{-5/2}  ,\quad  \forall k\in \BB N \notag \\ 
\BB P\left[ \frk P(Q^\infty , \BB e^\infty) = (k_1, k_2 , \infty)   \right]  
&\asymp  k_1^{-5/2} k_2^{-5/2}  ,   \quad \text{$\forall k_1,k_2\in  \BB N$ odd} . 
\end{align}
We get the same approximate formulas if we replace $(k,\infty)$ with $(\infty,k)$ or $(k_1,k_2,\infty)$ with either $(\infty,k_1,k_2)$ or $(k_1,\infty,k_2)$. 

Recall from Section~\ref{sec-general-peeling} the definitions of the number of exposed edges $\op{Ex}_{\infty}(Q^\infty,\BB e^\infty)$ and the number of covered edges $\op{Co}_{\infty}(Q^\infty,\BB e^\infty)$ and left/right covered edges $\op{Co}_\infty^L(Q^\infty,\BB e^\infty)$ and $\op{Co}_\infty^R(Q^\infty,\BB e^\infty)$ when we peel targeted at $\infty$. 

By~\cite[Proposition~3]{angel-curien-uihpq-perc},
\eqb \label{eqn-peel-mean}
\BB E\left[\op{Ex}_\infty\left(Q^\infty , \BB e^\infty \right) \right] = 2 \quad \op{and} \quad
\BB E\left[ \op{Co}_\infty^L\left(Q^\infty , \BB e^\infty \right) \right] =   \BB E\left[ \op{Co}_\infty^R\left(Q^\infty , \BB e^\infty \right) \right]  =  \frac12 , 
\eqe 
whence
\eqb \label{eqn-net-bdy-mean}
\BB E\left[  \op{Co}_\infty(Q^\infty  , \BB e^\infty ) \right] = \BB E\left[  \op{Ex}_\infty(Q^\infty  , \BB e^\infty ) \right] =2 ,
\eqe 
so in particular the expected net change in the boundary length of $Q^\infty $ under the peeling operation is~$0$. We always have $\op{Ex}_\infty(Q^\infty  , \BB e^\infty ) \in \{1,2,3\}$, but $\op{Co}_\infty(Q^\infty  , \BB e^\infty )$ can be arbitrarily large. In fact, a straightforward calculation using~\eqref{eqn-uihpq-peel-prob} shows that for $k\in\BB N$, one has 
\eqb \label{eqn-cover-tail}
\BB P\left[ \op{Co}_\infty^L(Q^\infty  , \BB e^\infty ) = k \right] = (\ccon + o_k(1)) k^{-5/2} \quad \op{and} \quad
\BB P\left[ \op{Co}_\infty^R(Q^\infty  , \BB e^\infty ) = k \right] = (\ccon + o_k(1)) k^{-5/2}
\eqe 
for $\ccon   = \frac{10\sqrt 2}{27 \sqrt{3\pi} }$.

\subsection{The face percolation peeling process}
\label{sec-perc-peeling}

In this subsection we will formally define the percolation peeling process, which is the main peeling process we will consider in this paper, and introduce some relevant notation (see Section~\ref{sec-intro-def-perc} for a less formal description of this process). 
See Figures~\ref{fig-perc-peeling-finite} and~\ref{fig-perc-peeling}, respectively, for illustrations of this process in the finite-volume and infinite-volume cases.

Let $\el_L , \el_R \in \BB N\cup \{\infty\}$ such that $\el_L$ and $\el_R$ are either both odd, both even, or both $\infty$. 
Let $\el := \el_L + \el_R$, so that $\el$ is either even or $\infty$, and let $(Q , \BB e)$ be a free Boltzmann quadrangulation with simple boundary and perimeter $\el$ (so that $(Q,\BB e)$ is a UIHPQ$_{\op{S}}$ if $\el =\infty$). Throughout this paper, when we want to refer to just the UIHPQ$_{\op{S}}$ case $\el =\infty$ we will add a superscript $\infty$ to all of the objects involved; and when we take $\el_L$ and $\el_R$ to depend on $n$, we will add a superscript $n$.

Let $\beta : \BB Z \rta \mcl E(\bdy Q)$ be the counterclockwise boundary path of $Q$ with $\beta(0) = \BB e$, extended by periodicity in the finite case, and define the target edge
\eqbn
\BB e_* := \begin{cases}
\beta(\el_R+1) ,\quad &\el < \infty \\
\infty ,\quad &\el = \infty .
\end{cases}
\eqen

Conditional on $(Q,\BB e)$, let $\theta$ be a critical face percolation configuration as in Section~\ref{sec-intro-def-perc}, so that $ \{\theta(q) : q\in\mcl F(Q ) \setminus \{f_\infty\} \}$ is a collection of i.i.d.\ Bernoulli random variables with $\BB P[\theta(q) = \mathsf{white}]= 3/4$ and $\BB P[\theta(q) = \mathsf{black}] = 1/4$. Recall that quadrilaterals $q$ with $\theta(q) = \mathsf{white}$ are said to be \emph{white} or \emph{open} and quadrilaterals with $\theta(q) = \mathsf{black}$ are \emph{black} or \emph{closed}.  

\begin{figure}[ht!]
 \begin{center}
\includegraphics[scale=1]{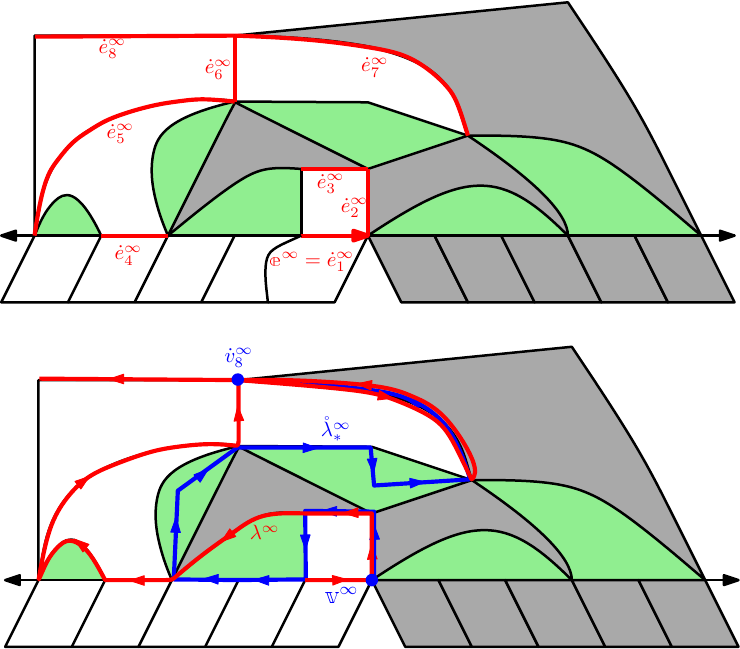} 
\caption[The percolation peeling process]{\label{fig-perc-peeling} \textbf{Top:} The percolation peeling process in the case when $\el_L = \el_R = \infty$, in which case $(Q,\BB e )= (Q^\infty,\BB e^\infty)$ is a UIHPQ$_{\op{S}}$ and we denote objects associated with the percolation peeling process by a superscript $\infty$. The external quadrilaterals used to define boundary conditions are shown below the horizontal line indicating the boundary of $Q^\infty$. The peeled edges $\dot e_j^\infty$ for $j\in [1,8]_{\BB Z}$ are shown in red, the peeled quadrilaterals are colored according to the face percolation configuration, and the regions disconnected from $\infty$ by the percolation peeling process (which, along with the peeled quadrilaterals, belong to $\dot Q_8^\infty$) are shown in light green.
\textbf{Bottom:} Same setup as in the top panel but with the percolation exploration path $\lambda^\infty$ shown in red and a possible realization of the associated percolation interface $\rng\lambda_*^\infty$ started from the right endpoint $\BB v^\infty$ of $\BB e^\infty$ shown in blue (the realization depends on the colors of the quadrilaterals inside the green regions, which are not shown). Note that $\lambda^\infty(j) = \dot e_j^\infty$ for $j\in\BB N_0$, but $\lambda^\infty(j)$ is allowed to traverse edges which are not among the $e_j^\infty$'s and/or to re-trace itself at half-integer times in order to get a continuous path. Although the percolation path $\rng\lambda_*^\infty$ does not coincide with $\lambda^\infty$, Lemma~\ref{lem-peel-interface} tells us that $\rng\lambda_*^\infty$ hits the endpoint $\dot v_j^\infty$ of $\dot e_j^\infty$ for each $j\in\BB N_0$, and does not exit the cluster $\dot Q_j^\infty$ until after it hits $\dot v_j^\infty$. 
  }
\end{center}
\end{figure}

To assign boundary conditions to $Q$, we attach a quadrilateral $q_e$ lying in the external face of $Q$ to each edge $e\in \mcl E(\bdy Q)$. We call the quadrilaterals $q_e$ \emph{external quadrilaterals} and their edges and vertices \emph{external edges and vertices}, respectively. We also define the extended quadrangulation
\eqb \label{eqn-external-quad}
Q^{\op{ext}} := Q \cup \bigcup_{e \in \mcl E(\bdy Q)} q_e .
\eqe
For each edge $e$ in the left arc of $\bdy Q$ from $\BB e$ to $\BB e_*$ (including $\BB e$ and also $\BB e_*$ if $\el < \infty$) we color the external quadrilateral $q_e$ white and for each edge $e$ in the right arc of $\bdy Q$ from $\BB e$ to $\BB e_*$ we color $q_e$ black.  We then identify any two external edges which have a common endpoint in $\bdy Q$ and which are incident to external faces of the same color.  We declare that an edge $e\in\mcl E(\bdy Q)$ is \emph{white} (resp.\ \emph{black}) if it is incident to a white (resp.\ black) external quadrilateral.  In other words, we impose $\el_L$-white/$\el_R$-black or $\el_L$-open/$\el_R$-closed boundary conditions. See Figure~\ref{fig-perc-peeling} for an illustration in the case when $\el_L = \el_R = \infty$.

We define a peeling process on $Q $ started from $\BB e$ and targeted at $\BB e_*$, which we call the \emph{percolation peeling process of $(Q,\BB e ,\theta)$ with $\el_L$-white/$\el_R$-black boundary conditions} as follows. Let $\ol Q_0 = Q$ and let $\dot Q_0 = \emptyset$.  Inductively, suppose $j\in \BB N$ and we have defined ``unexplored" and ``explored" quadrangulations with boundary $\ol Q_i$ and $\dot Q_i$ for $i\in [0,j-1]_{\BB Z}$ in such a way that the following holds.
\begin{enumerate}
\item $\ol Q_i \cup \dot Q_i = Q$ and $\ol Q_i$ and $\dot Q_i$ intersect only along their boundaries. 
\item Either $\ol Q_i = \emptyset$ and $\dot Q_i = Q$; or $\bdy \ol Q_i$ is a simple path containing $\BB e_*$ which has white-black boundary conditions, i.e.\ there is an edge $\dot e_{i+1} \in \bdy \ol Q_i$ such that $\dot e_{i+1}$ is incident to a white quadrilateral lying in the external face of $ \ol Q_i$ and each edge of $\bdy \ol Q_i$ lying in the left (resp.\ right) arc of $\bdy \ol Q_i$ from $\dot e_{i+1}$ to $\BB e_*$ is incident to a white (resp.\ black) quadrilateral lying in the external face of $\ol Q_i$.  
\end{enumerate} 
If $\ol Q_{j-1} = \emptyset$, we set $\ol Q_j = \emptyset$ and $\dot Q_j = Q$. 
Otherwise, we peel the quadrilateral of $\ol Q_{j-1}$ incident to $\dot e_j$ and add the subgraph which it disconnects from $\BB e_*$ in $\ol Q_{j-1}$ to $\dot Q_{j-1}$, i.e., in the notation of Section~\ref{sec-general-peeling} we set
\eqbn
\dot Q_j := \dot Q_{j-1} \cup \frk f(\ol Q_{j-1} , \dot e_j) \cup \frk F_{\BB e_*}(\ol Q_{j-1} , \dot e_j) \quad \op{and} \quad  \ol Q_j = \op{Peel}_{\BB e_*}( \ol Q_{j-1} , \dot e_j) .
\eqen
It is clear that $\ol Q_j$ and $\dot Q_j$ satisfy the above two conditions with $j$ in place of $i$, which completes the induction. We note that $\ol Q_j = \emptyset$ and $\dot Q_j = Q$ if and only if $\dot e_j = \BB e_*$. 

We define the terminal time 
\eqb \label{eqn-terminal-time}
\mcl J := \inf\left\{ j\in\BB N_0 : \ol Q_j = \emptyset \right\} ,
\eqe 
and note that $\mcl J = \infty$ in the UIHPQ$_{\op{S}}$ case (i.e., when $\el = \infty$).  
We also write
\eqbn
\theta_j := \theta\left( \frk f(\ol Q_{j-1} , \dot e_j) \right)  ,\quad \forall j \in \BB N
\eqen
for the random variable indicating the color of the $j$th peeled quadrilateral. 
For $j \in \BB N_0$ define the $\sigma$-algebra
\eqb \label{eqn-peel-filtration}
\mcl F_j := \sigma\left( \dot Q_{i} ,  \frk P(\ol Q_{i-1} , \dot e_i) , \theta_i \,:\, i  \in [1,j]_{\BB Z} \right) 
\eqe
where here $\frk P(\cdot,\cdot)$ is the peeling indicator (Section~\ref{sec-general-peeling}). Note that $\dot e_{j+1} \in \mcl F_j$.

We then define the percolation exploration path $\lambda : \frac12\BB N_0  \rta \mcl E(Q)$ exactly as in Section~\ref{sec-intro-def-perc}, i.e., $\lambda(0) = \BB e = \dot e_1$, $\lambda(j) = \dot e_j$ for $j\in\BB N$, and for $j\in\BB N_0$, $\lambda(j+1/2)$ is an edge which shares an endpoint with each of $\lambda(j)$ and $\lambda(j+1/2)$ (chosen by some arbitrary convention). 
 

\subsection{Face percolation interfaces}
\label{sec-interface}

Suppose we are in the setting of Section~\ref{sec-perc-peeling}, so that $(Q,\BB e,\theta)$ is a free Boltzmann quadrangulation with simple boundary and (possibly infinite) perimeter $\el_L + \el_R$ and $\lambda : \BB N_0 \rta \mcl E(Q)$ is the percolation exploration path with $\el_L$-white/$\el_R$-black boundary conditions. 
The path $ lambda$ is \emph{not} a percolation interface, but it is closely related to a certain percolation interface, as we explain in this section. 

We say that two white quadrilaterals of $Q$ are $\theta$-adjacent if they share an \emph{edge} and we say that two black quadrilaterals of $Q$ are $\theta$-adjacent if they share a \emph{vertex}; we emphasize here the asymmetry between black and white quadrilaterals (which is related to the fact that the percolation threshold $3/4$, not $1/2$).

A set of quadrilaterals $F\subset \mcl F(Q) \setminus \{f_\infty\}$ is \emph{$\theta$-connected} if any two quadrilaterals $q , q' \in F$ are the same color and can be joined by a $\theta$-adjacent path in $F$, i.e.\ a finite string of quadrilaterals $q_0,\dots ,q_n \in F$ such that $q_0 = q$, $q_n = q'$, and $q_i$ is $\theta$-adjacent to $q_{i-1}$ for each $i\in [1,n]_{\BB Z}$.  A \emph{white (resp.\ black) cluster} is a subgraph $S$ of $Q$ such that the face set $\mcl F(S)$ is a white (resp.\ black) $\theta$-connected component of $F$ and the vertex and edge sets of $F$ consist, respectively, of all vertices and edges of quadrilaterals in $\mcl F(S)$. 

A path $\rng\lambda : [a,b]_{\BB Z} \rta \mcl E(Q)$ is an \emph{interface path} if the following is true. If we orient the edges $\rng\lambda(i)$ for $i\in [a,b]_{\BB Z}$ in such a way that the terminal endpoint of $\rng\lambda(i)$ is the same as the initial endpoint of $\rng\lambda(i+1)$ for each $i\in [a,b-1]_{\BB Z}$, then each edge $\rng\lambda(i)$ has a white quadrilateral to its left and a black quadrilateral to its right and moreover neither of the endpoints of $\rng\lambda(i)$ is the corner of a black quadrilateral which lies to the left of $\rng\lambda$.  Equivalently, $\rng\lambda$ traces a segment of the boundary of some white cluster of $Q$ in the counterclockwise direction. 

It is immediate from the above definition that an interface path either has no repeated edges (it may, however, have repeated vertices) or is a sub-path of a periodic path each of whose periods have no repeated edges.  Furthermore, two distinct interface paths cannot cross (although they can share a vertex), and two interface paths which share an edge must be sub-paths of some common interface path. 
 
For each edge $e$ of $Q$ which is incident to both a white quadrilateral and a black quadrilateral, there is a unique infinite interface path $\rng\lambda : \BB Z \rta \mcl E(Q)$ satisfying $\rng\lambda(0) = e$, which traces the boundary of the white cluster containing the quadrilateral to the left of $e$ in the counterclockwise direction. This path is simple if this cluster is infinite or periodic if it is finite. 

Let~$\BB v$ (resp.\ $\BB v_*$) be the terminal endpoint of~$\BB e$ (resp.\ the initial endpoint of the target edge $\BB e_*$), so that~$\BB v$ and~$\BB v_*$ are the unique vertices of~$\bdy Q$ which are incident to both black and white external quadrilaterals.  Then there is a distinguished interface path $\rng\lambda_* : [1, N]_{\BB Z} \rta \mcl E(Q)$ in $Q^{\op{ext}} $ (notation as in~\eqref{eqn-external-quad}) from $\BB v$ to $\BB v_*$, namely the path which traces the segment of the boundary of the white cluster containing all of the white external faces from $\BB v$ to $\BB v_*$ in the counterclockwise direction.  We extend $\rng\lambda_*$ from $[1,N]_{\BB Z}$ to $[0,N+1]_{\BB Z}$ by declaring that $\rng\lambda_*(0)$ (resp.\ $\rng\lambda_*(N+1)$) is the external edge of $Q$ which is incident to $\BB v$ (resp.\ $\BB v_*$) and to a white external quadrilateral. 

The aforementioned interface path $\rng\lambda_*$ is \emph{not} the same as the percolation exploration path $\lambda$ produced via peeling. However, it is related to the percolation peeling clusters $\{\dot Q_j\}_{j\in\BB N_0}$ in the following manner. See Figure~\ref{fig-perc-peeling} for an illustration.

\begin{lem} \label{lem-peel-interface}
For $j\in\BB N_0$, let $\dot v_j $ be the right endpoint of the peeled edge $\dot e_j$. 
For each $j \in \BB N_0$, there exists a unique $s_j \in \BB N_0$ such that $\dot v_{j+1}$ is the terminal endpoint of $\rng\lambda_*(s_j)$ and $\rng\lambda_*([0,s_j]_{\BB Z})$ is contained in the peeling cluster $\dot Q_j$. 
Furthermore, $s_j \leq s_{j'}$ whenever $j\leq j'$. 
\end{lem}
\begin{proof}
Let $\dot Q_j^{\op{ext}}$ be the quadrangulation obtained by adjoining all of the external quadrilaterals of $Q$ to the cluster $\dot Q_j$. 
Then $\dot Q_j^{\op{ext}}$ is a quadrangulation with simple boundary. Furthermore, every quadrilateral of $\dot Q_j^{\op{ext}}$ incident to the clockwise (resp.\ counterclockwise) arc of $\bdy \dot Q_j^{\op{ext}}$ from $\BB v$ to $\dot v_{j+1}$ belongs to the same white (resp.\ black) cluster of $\dot Q_j^{\op{ext}}$ (this follows since we always peel at a white edge). 
Hence there is a distinguished interface path $\rng\lambda$ in $\dot Q_j^{\op{ext}}$ from $\BB v$ to $\dot v_{j+1}$, namely the right outer boundary of the white cluster of $\dot Q_j^{\op{ext}}$ containing all of the white external faces.  
Since this white cluster is the intersection with $\dot Q_j^{\op{ext}}$ of the white cluster of $Q^{\op{ext}}$ containing the white external faces, we infer that $\rng\lambda = \rng\lambda_* |_{[0,s_j]_{\BB Z}}$ for some $s_j \in \BB N_0$. This choice of $s_j$ is unique since the interface path $\rng\lambda_*$ is a simple path. The monotonicity statement for the $s_j$'s follows since $\rng\lambda_*([0,s_j]_{\BB Z}) \subset \dot Q_j$ and the clusters $\dot Q_j$ are increasing.
\end{proof}

\begin{remark}
Suppose we are in the setting of Theorem~\ref{thm-perc-conv}, so that for $n\in\BB N$, $Q^n$ is a free Boltzmann quadrangulation with simple boundary of perimeter $\el_L^n+\el_R^n$ and $d^n$ is its rescaled graph metric. Also let $\rng\lambda_*^n$ be the interface path from Lemma~\ref{lem-peel-interface}, with its domain extended to $\BB N_0$ in such a way that it is constant after its terminal time and then to $[0,\infty)$ by linear interpolation. Once Theorem~\ref{thm-perc-conv} is established, it follows from Lemma~\ref{lem-peel-interface} that the $d^n$-Hausdorff distance between $\lambda_*^n([0,\infty))$ and the range $\lambda^n([0,\infty))$ of the percolation exploration path tends to zero in law as $n\rta\infty$: indeed, this comes from the fact that $\lambda^n([0,\infty))$ has only finitely many complementary connected components of $d^n$-diameter at least $\ep$ for each $\ep > 0$, and the law of the time length of the segment of $\rng\lambda_*^n$ contained in each such connected component is typically of constant order.
We do not prove in the present paper that there is a constant $\rng{\tcon} > 0$ such that $\rng\lambda_*^n(\rng{\tcon}  n^{3/4} \cdot)$ converges uniformly to the same limiting path as $\eta^n = \lambda^n(\tcon n^{3/4} \cdot)$, but we expect this can be accomplished with some additional technical work. 
Note that in the case of site percolation on a triangulation the analogs of the paths $\lambda^n$ and $\rng\lambda_*^n$ coincide; see Section~\ref{sec-triangulation}.
\end{remark}

\section{Boundary length processes}
\label{sec-bdy-process}

Let $\el_L , \el_R \in \BB N\cup \{\infty\}$ be such that $\el_L$ and $\el_R$ are either both odd, both even, or both $\infty$. Let $(Q,\BB e, \theta )$ be a free Boltzmann quadrangulation with simple boundary of perimeter $\el_L+\el_R$ decorated by a face percolation configuration and define the clusters $\{\dot Q_j\}_{j\in \BB N_0}$ and the unexplored quadrangulations $\{\ol Q_j\}_{j\in\BB N_0}$ of the percolation peeling process of $(Q,\BB e,\theta)$ with $\el_L$-white/$\el_R$-black boundary conditions as in Section~\ref{sec-perc-peeling}. In this section we will study the boundary length processes for the percolation peeling process, which we now define. Unlike in the case of general peeling processes (see, e.g.,~\cite[Definition 3.2]{gwynne-miller-simple-quad}) it is natural to consider left and right boundary length processes since the boundary of the percolation peeling clusters have a natural notion of left and right sides.

\begin{defn} \label{def-bdy-process}
The \emph{left (resp.\ right) exposed boundary length process} $X^L_{j}$ (resp.\ $X^R_{j}$) at time $j\in\BB N_0$ is equal to the number of edges of $\bdy \ol Q_j \cap \bdy \dot Q_j$ which are adjacent to a white (resp.\ black) quadrilateral in $\dot Q_j$, equivalently the number of edges of $\bdy \ol Q_j \cap \bdy \dot Q_j$ to the left (resp.\ to the right) of $\dot e_{j+1}$, including (resp.\ not including) $\dot e_{j+1}$ itself. The \emph{left (resp.\ right) covered boundary length process} $Y^L_{j}$ (resp.\ $Y^R_{j}$) at time $j\in\BB N_0$ is the number of edges of $\bdy Q  \setminus \bdy \ol Q_j$ which are adjacent to a white (resp.\ black) quadrilateral lying in the external face of $Q$, equivalently the number of edges of $\dot Q_j\cap \bdy Q$ lying to weakly to the left (resp.\ strictly to the right) of $\BB e$. 
The \emph{left, right, and total net boundary length processes} are defined, respectively, by  
\eqbn
W_j^L := X_j^L - Y_j^L , \quad  W^R_j := X^R_j - Y^R_j ,\quad \op{and} \quad W_j := (W_j^L , W_j^R) ,\quad \forall j\in\BB N_0 .
\eqen 
We include an additional superscript $\infty$ in the notation for the above objects when we wish to discuss only the UIHPQ$_{\op{S}}$ case, when $\el_L = \el_R = \infty$; and an additional superscript $n$ when we take $\el_L$ and $\el_R$ to depend on~$n$. 
\end{defn}

See Figure~\ref{fig-bdy-process} for an illustration of Definition~\ref{def-bdy-process}.
Note that $W_j = (-\el_L , -\el_R)$ for each $j$ after the terminal time $\mcl J$ defined in~\eqref{eqn-terminal-time}.

\begin{figure}[ht!]
 \begin{center}
\includegraphics[scale=1]{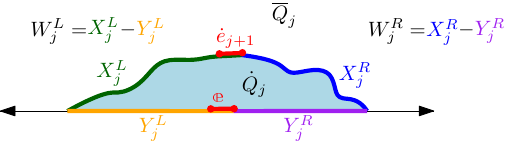} 
\caption[Definition of the left/right boundary length proceses for the percolation exploration]{\label{fig-bdy-process} Illustration of Definition~\ref{def-bdy-process}. Individual edges and vertices are not shown, except for the edges $\BB e$ and $\dot e_{j+1}$ (red). We emphasize that the definitions are not symmetric under swapping the roles of $L$ and $R$ since the two red edges are counted as part of the left boundary (this is related to the fact that the critical probability is $3/4$, not $1/2$).} 
\end{center}
\end{figure}

The process $W = (W^L , W^R)$ is the discrete analog of the left/right boundary length process for chordal $\SLE_6$ on an independent $\sqrt{8/3}$-quantum surface (Section~\ref{sec-lqg-bdy-process}), so it is natural to expect $W$ to converge in the scaling limit to this latter process both in the finite-volume and infinite-volume cases. Most of this section will be devoted to proving that this is indeed the case.

\begin{defn} \label{def-bdy-process-rescale}
For $n\in\BB N$, we define the re-normalized (net) boundary length processes
\eqb \label{eqn-bdy-process-rescale}
L_t^n :=  \bcon^{-1} n^{-1/2} W^L_{\lfloor \tcon n^{3/4} t \rfloor} , \quad 
R_t^n :=  \bcon^{-1} n^{-1/2} W^R_{\lfloor \tcon n^{3/4} t \rfloor}, \quad 
\op{and} \quad Z_t^n:= (L_t^n ,R_t^n)    
\eqe
where here $\bcon = 2^{3/2}/3$ and $\tcon$ are the normalizing constants from Section~\ref{sec-results}.
We include an additional superscript $\infty$ in the notation for the above objects when we wish to discuss only the UIHPQ$_{\op{S}}$ case.
\end{defn} 

The reason for the factor of $\bcon^{-1}$ in Definition~\ref{def-bdy-process-rescale} is that in the scaling limit results for quadrangulations with simple boundary~\cite{gwynne-miller-uihpq,gwynne-miller-simple-quad} the boundary path is pre-composed with $t\mapsto \bcon  n^{1/2} t$ so one edge along the boundary of such a quadrangulation corresponds to approximately $\bcon^{-1} n^{-1/2}$ units of boundary length in the scaling limit. 
The reason for the factor of $\tcon$ is so that the scaling of the L\'evy measure for the scaling limits of $L^{\infty,n}$ and $R^{\infty,n}$ is the same as for the $3/2$-stable processes in Section~\ref{sec-lqg-bdy-process}.

The rest of this section is devoted to proving the following two propositions, which give the scaling limits of $Z^n$ in the infinite-volume and finite-volume cases.  Along the way, we will also prove several estimates for the percolation peeling process which will be needed later. 

Recall from Section~\ref{sec-lqg-bdy-process} that the left/right boundary length process for chordal $\SLE_6$ on the Brownian half-plane is a pair of independent totally asymmetric $3/2$-stable processes with no upward jumps. As one expects, this process is the scaling limit of the discrete boundary length process in the UIHPQ$_{\op{S}}$ case.

\begin{prop} \label{prop-stable-conv}
In the UIHPQ$_{\op{S}}$ case (i.e., $\el_L = \el_R = \infty$) it holds for a suitable choice of $\tcon > 0$ that the re-normalized boundary length processes $Z^{\infty,n} = (L^{\infty,n}, R^{\infty,n})$ of Definition~\ref{def-bdy-process-rescale} converge in law as $n\rta\infty$ with respect to the local Skorokhod topology (Section~\ref{sec-skorokhod}) to a pair $Z^\infty = (L^\infty ,R^\infty)$ of independent totally asymmetric $3/2$-stable processes with no positive jumps (scaled so that~\eqref{eqn-levy-process-scaling} holds). 
\end{prop}

We now consider the finite-volume case. Suppose we are in the setting of Theorem~\ref{thm-perc-conv}. In particular, $\frk l_L , \frk l_R > 0$; $\{(\el_L^n , \el_R^n)\}_{n\in\BB N}$ is a sequence of pairs of positive integers such that $\el_L^n + \el_R^n$ is always even, $\bcon^{-1} n^{-1/2} \el_L^n \rta \frk l_L$, and $\bcon^{-1} n^{-1/2} \el_R^n \rta \frk l_R$ (where $\bcon = 2^{3/2} /3$ is the usual boundary length scaling constant); and $(Q^n, \BB e^n ,\theta^n)$ is a free Boltzmann quadrangulation with simple boundary of perimeter $\el_L^n + \el_R^n$ equipped with a critical face percolation configuration. 

Let $Z = (L,R)$ be the left/right boundary length process of a chordal $\SLE_6$ between the two marked points of a doubly marked Brownian disk with left/right boundary lengths $\frk l_L$ and $\frk l_R$, respectively, parameterized by quantum natural time. The law of the left/right boundary length process $Z = (L,R)$ for SLE$_6$ on the Brownian disk is described in~Theorem~\ref{thm-bdy-process-law}: it behaves locally like a pair of independent $3/2$-stable processes, but it is conditioned so that $L$ hits $-\frk l_L$ before $(-\infty,-\frk l_L)$, $R$ hits $-\frk l_R$ before $(-\infty,-\frk l_R)$, and these two hitting events occur at the same time. The process $Z$ remains constant after this common hitting time.

\begin{prop}
\label{prop-stable-conv-finite} 
For $n\in\BB N$, let $Z^n = (L^n, R^n)$ be the re-normalized left/right boundary length process for the percolation peeling process on $(Q^n, \BB e^n ,\theta^n)$ with $\el_L^n$-white/$\el_R^n$-black boundary conditions, as in Definition~\ref{def-bdy-process-rescale} with $Q=Q^n$ and with $\tcon$ as in Proposition~\ref{prop-stable-conv}. Then with $Z = (L,R)$ the process described just above, we have $Z^n \rta Z$ in law with respect to the Skorokhod topology. 
\end{prop}

In Section~\ref{sec-bdy-process-def}, we point out some basic properties of the processes of Definition~\ref{def-bdy-process}. In Section~\ref{sec-uihpq-bdy-conv}, we prove Proposition~\ref{prop-stable-conv}. The proof is a straightforward application of the heavy-tailed central limit theorem plus a short argument to make sure that the two coordinates of the limiting process are indeed independent. 

We then turn our attention to the proof of Proposition~\ref{prop-stable-conv-finite}, which is more challenging. We will deduce Proposition~\ref{prop-stable-conv-finite} from Proposition~\ref{prop-stable-conv} and a local absolute continuity argument. In Section~\ref{sec-fb-bdy-estimate}, we state a general lemma which allows us to compare peeling processes on free Boltzmann quadrangulations and on the UIHPQ$_{\op{S}}$ (Lemma~\ref{lem-perc-rn}) and prove some general estimates for peeling processes on free Boltzmann quadrangulations which rule out various pathologies. 

In Section~\ref{sec-sk-tight}, we establish tightness of the law of the finite-volume boundary length process $Z^n$ with respect to the Skorokhod topology. To do this we need to analyze certain stopping times for this process in order to rule out pathological behavior of $Z^n$ near the terminal time. In Section~\ref{sec-stable-conv-finite} we conclude the proof of Proposition~\ref{prop-stable-conv} by identifying a subsequential scaling limit, using a comparison of the Radon-Nikodym derivatives appearing in Theorem~\ref{thm-bdy-process-law} and Lemma~\ref{lem-perc-rn}.  


\subsection{Basic properties of the boundary length processes}
\label{sec-bdy-process-def}

Suppose we are in the setting of Definition~\ref{def-bdy-process-rescale}, so that $\el_L  , \el_R \in \BB N$ with $\el_L + \el_R$ even or $\el_L = \el_R = \infty$ and we are considering the boundary length processes of the percolation peeling process on a free Boltzmann quadrangulation with $\el_L$-white/$\el_R$-black boundary conditions. 

For each $j \in \BB N_0$, the boundary length of the unexplored quadrangulation is given by
\eqb \label{eqn-unexplored-length}
\#\mcl E(\bdy \ol Q_j) = W_j^L + W_j^R + \el_L + \el_R.
\eqe 
One can almost recover the pairs $(X^L , Y^L)$ and $(X^R , Y^R)$ from the two-dimensional process $W$ in the following manner.  If the $j$th peeled quadrilateral $\frk f(\ol Q_{j-1} , \dot e_j)$ disconnects an edge in $\bdy Q \cap \bdy \ol Q_{j-1}$ adjacent to a white quadrilateral in the external face from the target edge $\BB e_*$ in $\ol Q_{j-1}$, then $\frk f(\ol Q_{j-1} , \dot e_j)$ also disconnects every edge in $\bdy \ol Q_{j-1} \cap \bdy \dot Q_{j-1}$ adjacent to a white quadrilateral of $\dot Q_{j-1}$ from $\BB e_*$ in $\ol Q_{j-1}$. If this is the case, then $X_j^L \leq 3$ and $W_j^L$ differs from the running minimum of $W^L$ up to time $j$ by at most~$3$ (the~$3$ comes from the exposed edges of $\frk f(\ol Q_{j-1} , \dot e_{j})$ in case this quadrilateral is white). Similar considerations apply to~$W^R$. Therefore,
\eqb \label{eqn-bdy-process-inf}
\min_{ i\in [0,j ]_{\BB Z}} W_i^L  \in \left[ - Y_j^L , - Y_j^L +  3 \right]_{\BB Z}     \quad \op{and} \quad   W_j^L - \min_{ i\in [0,j]_{\BB Z}} W^L_i \in \left[  X_j^L  -3 , X_j^L \right]_{\BB Z} ;
\eqe
and similarly with ``$R$'' in place of ``$L$''.

With $\op{Ex}_{\BB e_*}(\cdot, \cdot)$, $\op{Co}_{\BB e_*}^L(\cdot ,\cdot)$, $\op{Co}_{\BB e_*}^R(\cdot ,\cdot)$, and $\op{Co}_{\BB e_*}(\cdot ,\cdot)$ the number of exposed, left/right covered, and covered edges with respect to the target edge, respectively, as in Section~\ref{sec-general-peeling}, and $\theta_j$ the color indicator from Section~\ref{sec-perc-peeling},
\allb \label{eqn-peel-inc-cases}
W_j^L - W^L_{j-1} &=    \BB 1_{(\theta_j = \mathsf{white})}  \op{Ex}_{\BB e_*} (\ol Q_{j-1}  , \dot e_j)    - \op{Co}_{\BB e_*}^L(\ol Q_{j-1}  , \dot e_j) - 1  \notag \\
W^R_j - W^R_{j-1} &=    \BB 1_{(\theta_j = \mathsf{black})}  \op{Ex}_{\BB e_*} (\ol Q_{j-1}  , \dot e_j)    - \op{Co}_{\BB e_*}^R(\ol Q_{j-1}  , \dot e_j)  ,\quad \op{and} \\
W_j^L  + W_j^R - W^L_{j-1} - W^R_{j-1} &= \op{Ex}_{\BB e_*} (\ol Q_{j-1}  , \dot e_j)    - \op{Co}_{\BB e_*} (\ol Q_{j-1}  , \dot e_j) .\notag 
\alle
 Recall~\eqref{eqn-covered-identity} for the last line. The $-1$ in the first line comes from the fact that the white edge $\dot e_j$ itself is always disconnected from $\BB e_*$ by the peeled quadrilateral $\frk f(\ol Q_{j-1}  , \dot e_j)$.  
 
The indicator variable $\theta_j$ is independent from $\mcl F_{j-1}$ and $\frk P(\ol Q_{j-1} , \dot e_j)$ and equals $\mathsf{white}$ with probability $3/4$ and $\mathsf{black}$ with probability $1/4$.  In the case when $\el = \infty$ (so that $(Q,\BB e) = (Q^\infty,\BB e^\infty)$ is a UIHPQ$_{\op{S}}$), it follows from~\eqref{eqn-peel-inc-cases} that the increments $W_j^\infty -W_{j-1}^\infty$ are i.i.d. Furthermore, by~\eqref{eqn-peel-mean}, 
\eqb  \label{eqn-peel-inc-mean}
\BB E\left[ W_j^{\infty,L} - W^{\infty,L}_{j-1}   \right] =
\BB E\left[ W^{\infty,R}_j - W^{\infty,R}_{j-1}   \right] 
= 0 
\eqe 
and by~\eqref{eqn-cover-tail}, 
\eqb \label{eqn-peel-inc-tail}
\BB P\left[ W_j^{\infty,L} - W^{\infty,L}_{j-1} = k   \right] = (\ccon + o_k(1)) k^{-5/2}  \quad \op{and} \quad
\BB P\left[ W^{\infty,R}_j - W^{\infty,R}_{j-1} = k   \right] = (\ccon + o_k(1)) k^{-5/2}  .
\eqe

\subsection{Scaling limit of the UIHPQ$_{\op{S}}$ boundary length processes}
\label{sec-uihpq-bdy-conv}
 
We will deduce Proposition~\ref{prop-stable-conv} from the heavy-tailed central limit theorem and the fact that peeling steps are i.i.d. 
We will need the following elementary estimate to show that the coordinates of the limiting process do not have simultaneous jumps (which, as we will see, will imply that they are independent).  

\begin{lem} \label{lem-simultaneous-jump}
Suppose we are in the UIHPQ$_{\op{S}}$ case (i.e., $\el_L = \el_R = \infty$).
For $m\in\BB N$ and $k\in\BB N$, let $E(m,k)$ be the event that there exist times $j_L , j_R \in [1,m]_{\BB Z}$ such that both $W^{\infty,L}_{j_L} - W^{\infty,L}_{j_L-1} \leq -k$ and $W^{\infty,R}_{j_R} - W^{\infty,R}_{j_R-1} \leq -k$. Then $\BB P[E(m,k)] \preceq k^{-3} m^2$ with universal implicit constant. 
\end{lem}
\begin{proof}
If $E(m,k)$ occurs, let $j_L$ and $j_R$ be the smallest times in $[1,m]_{\BB Z}$ satisfying the conditions in the definition of $E(m,k)$. 
We will treat the case where $j_L = j_R$ and the case where $j_L \not= j_R$ separately.

If $j_L = j_R$, then by~\eqref{eqn-peel-inc-cases} there is a $j \in [1,m]_{\BB Z}$ such that the numbers of left and right covered edges satisfy
\eqbn
\op{Co}_\infty^L(\ol Q^\infty_{j-1}  , \dot e_j) \geq k -1  \quad \op{and} \quad \op{Co}_\infty^R(\ol Q^\infty_{j-1}  , \dot e_j) \geq k .
\eqen
By~\eqref{eqn-uihpq-peel-prob-asymp} and the Markov property of peeling, the probability that this is the case for any fixed time $j$ is at most
\eqbn
\sum_{k_1=k}^\infty \sum_{k_2=k-1}^\infty \BB P\left[ \frk P_\infty\left(\ol Q^\infty_{j-1}  , \dot e_j\right) = (k_1,\infty, k_2) \right] \preceq k^{-3} .
\eqen
Taking a union bound over $j \in [1,m]_{\BB Z}$ shows that $\BB P\left[ E(m,k) ,\, j_L = j_R \right] \preceq k^{-3} m$. 

To treat the case when $j_L \not=j_R$, let $J_1$ (resp.\ $J_2$) be the first (resp.\ second) time in $[1,m]_{\BB Z}$ for which the total number of covered edges satisfies $\op{Co}_\infty(\ol Q^\infty_{j-1}  , \dot e_j) \geq k$, or $\infty$ if there are fewer than one (resp.\ two) such times. If $E(m,k)$ occurs and $j_L \not= j_R$, then both $J_1$ and $J_2$ are finite. By~\eqref{eqn-cover-tail}, for $j\in [1,m]_{\BB Z}$ 
\eqbn
\BB P\left[ \op{Co}_\infty(\ol Q^\infty_{j-1}  , \dot e_j) \geq k \right] \preceq k^{-3/2} .
\eqen
By a union bound, $\BB P[J_1 < \infty] \preceq k^{-3/2} m$. By the strong Markov property, $\BB P[J_2 < \infty \,|\, J_1 < \infty] \preceq k^{-3/2} m$. Hence $\BB P\left[ E(m,k) ,\, j_L \not= j_R \right] \preceq k^{-3} m^2$. 
\end{proof}

\begin{proof}[Proof of Proposition~\ref{prop-stable-conv}]
We take
\eqb \label{eqn-tcon-choice}
\tcon := \frac32 \bcon^{3/2} \ccon^{-1} \lcon
\eqe  
where here $\bcon = 2^{3/2}/3$, $\lcon$ is as in~\eqref{eqn-levy-process-scaling}, and $\ccon$ is as in~\eqref{eqn-cover-tail}. 
It is clear from~\eqref{eqn-peel-inc-mean},~\eqref{eqn-peel-inc-tail}, and the heavy-tailed central limit theorem (see, e.g.,~\cite{js-limit-thm}) that the processes $L^{\infty,n}$ and $R^{\infty,n}$ each converge in law separately in the local Skorokhod topology to a totally asymmetric $3/2$-stable process with no positive jumps, scaled as in~\eqref{eqn-levy-process-scaling}. We need to check that they converge jointly. 

By the Prokhorov theorem, for any sequence of positive integers tending to~$\infty$, there exists a subsequence~$\mcl N$ along which $Z^{\infty,n} \rta Z^\infty$ in law as $\mcl N\ni n \rta\infty$ with respect to the local Skorokhod topology on each coordinate, where here $Z^\infty = (L^\infty,R^\infty)$ is a coupling of two totally asymmetric $3/2$-stable processes with no positive jumps. We must show that $L^\infty$ and $R^\infty$ are independent (which in particular implies that the joint law of the limit does not depend on the choice of $\mcl N$). We claim that to prove this, it suffices to show that $L^\infty$ and $R^\infty$ a.s.\ do not have any simultaneous jumps. 
Indeed, $(L^\infty , R^\infty)$ is a L\'evy process taking values in $\BB R^2$, so the jumps of $L^\infty$ and $R^\infty$ arrive according to two Poisson point processes which are adapted w.r.t.\ a common filtration (namely, the one generated by $(L^\infty, R^\infty)$). If $L^\infty$ and $R^\infty$ have no simultaneous jumps, then these two Poisson point processes have no simultaneous jumps, so are independent~\cite[Proposition~1, Section~0]{bertoin-book}. Since $L^\infty$ and $R^\infty$ are each determined by their jumps, it then follows that $L^\infty$ and $R^\infty$ are independent. 

Fix $\ep > 0$ and $T> 0$. For $n\in\BB N$, $\delta > 0$, and $i \in [1, \lceil T/\delta \rceil]_{\BB Z}$ let $E_\delta^n(i)$ be the event that there exist times $t_L , t_R \in [(i-1)\delta  ,i \delta]$ such that both $L^{\infty,n}_{t_L} - \lim_{t\rta t_L^-} L^{\infty,n}_t \leq -\ep$ and $R^{\infty,n}_{t_R} - \lim_{t\rta t_R^-} R^{\infty,n}_t \leq -\ep$. By Lemma~\ref{lem-simultaneous-jump} (applied with $k \asymp \ep n^{1/2}$ and $m = \delta n^{3/4}$) and stationarity of the law of $Z$, $\BB P[E_\delta^n(i)] \preceq \ep^{-3/2} \delta^2$, with universal implicit constant. By a union bound, $\BB P\left[ \bigcup_{i=1}^{\lceil T/\delta \rceil} E_\delta^n(i) \right] \preceq T \ep^{-3/2} \delta$, with universal implicit constant.

Taking a limit as $\mcl N \ni n\rta\infty$ shows that as $\delta \rta 0$, the probability that there is an $i\in [1, \lceil T/\delta \rceil]_{\BB Z}$ such that $L^\infty$ and $R^\infty$ have a simultaneous jump of size at least $\ep$ in the time interval $((i-1)\delta , i\delta)$ tends to zero. For fixed $\delta>0$, a.s.\ neither $L^\infty$ nor $R^\infty$ has a jump at time $\delta i$ for any $i\in\BB Z$. Since $\ep$ and $T$ are arbitrary, we infer that $L^\infty$ and $R^\infty$ a.s.\ do not have any simultaneous jumps, so are independent. 
\end{proof}

\subsection{Estimates for free Boltzmann quadrangulations with simple boundary}
\label{sec-fb-bdy-estimate}

The goal of the remainder of this section is to prove our scaling limit result for the finite-volume boundary length process, Proposition~\ref{prop-stable-conv-finite}.
In this section we record some general estimates for peeling processes on a free Boltzmann quadrangulation with simple boundary which will be used in the remainder of this section as well as in Sections~\ref{sec-ghpu-tight} and~\ref{sec-crossing}.  
 
Throughout this subsection we consider the following setup. 
Let $\el_L , \el_R \in \BB N_0$ with $\el_L + \el_R$ even, let $(Q,\BB e , \theta)$ be a free Boltzmann quadrangulation with boundary with a critical face percolation configuration, and define the clusters $\{\dot Q_j\}_{j\in\BB N_0}$, the unexplored quadrangulations $\{\ol Q_j\}_{j\in\BB N_0}$, the peeled edges $\{\dot e_j\}_{j\in\BB N}$, the filtration $\{\mcl F_j\}_{j \in \BB N_0}$, and the terminal time $\mcl J$ as in Section~\ref{sec-perc-peeling}. 
Also define the boundary length processes $X^L , X^R,Y^L,Y^R$, and $W = (W^L , W^R)$ as in Definition~\ref{def-bdy-process}.  

We will have occasion to compare the above objects with the analogous objects associated with a UIHPQ$_{\op{S}}$. To this end, we let $(Q^\infty,\BB e^\infty,\theta^\infty)$ be a UIHPQ$_{\op{S}}$ with a critical face percolation configuration and define the objects $\{\dot Q_j^\infty\}_{j\in\BB N_0}$, $\{\ol Q_j^\infty\}_{j\in\BB N_0}$, $\{\dot e_j^\infty\}_{j\in\BB N}$, $\{\theta_j^\infty\}_{j\in\BB N_0}$, and $\{\mcl F_j^\infty\}_{j \in \BB N_0}$ as in Section~\ref{sec-perc-peeling} with $\el_L = \el_R = \infty$ (with an extra superscript $\infty$ to indicate the UIHPQ$_{\op{S}}$ case). Also define the associated boundary length processes $X^{ \infty,L} , X^{\infty,R} ,Y^{\infty,L} ,Y^{\infty,R}$, and $W^\infty = (W^{\infty,L} , W^{\infty,R})$ as in Definition~\ref{def-bdy-process}. We also define the time
\eqb \label{eqn-infty-terminal-time}
\mcl J^\infty := \inf\left\{ j \in \BB N_0 : Y_j^{\infty,L} \leq - \el_L \:\op{or} \: Y_j^{\infty,R} \leq - \el_R \right\} 
\eqe 
so that $\mcl J^\infty$ and $\mcl J$ are given by the same deterministic functional of the processes $W^\infty$ and $W$, respectively.

The following lemma is our main tool for comparing the percolation explorations in the case when $\el_L,\el_R < \infty$ and the case when $\el_L=\el_R=\infty$. We note the similarity to the Radon-Nikodym estimate in Theorem~\ref{thm-bdy-process-law} (which is no coincidence, as the estimates are proven in an analogous manner).
 
\begin{lem} \label{lem-perc-rn}
Let $\iota$ be a stopping time for $\{\mcl F_{j}\}_{j\in\BB N}$ which is less than $\mcl J$ with positive probability. The law of $\{ \dot Q_{j} , \frk P(\ol Q_{j-1} , \dot e_j)  , \theta_j \}_{j\in [1,\iota]_{\BB Z}}$ conditional on the event $  \{ \iota < \mcl J\}$ is absolutely continuous with respect to the law of $\{ \dot Q_{j}^\infty , \frk P(\ol Q_{j-1}^\infty , \dot e_j^\infty) , \theta_j^\infty  \}_{j\in [1,\iota]_{\BB Z}}$, with Radon-Nikodym derivative given by
\eqb \label{eqn-perc-rn}
 (1 + o(1))  \left( \frac{ W_\iota^{\infty,L} + W_\iota^{\infty,R} }{\el_L+\el_R}  + 1 \right)^{-5/2}  \BB 1_{(\iota < \mcl J^\infty) } 
\eqe  
where here the $o(1)$ tends to zero as $(\el_L + \el_R) \wedge (W_\iota^{\infty,L} + W_\iota^{\infty,R} + \el_L + \el_R)   $ tends to $\infty$, at a deterministic rate.
\end{lem}
\begin{proof}
This is an immediate consequence of~\cite[Lemma~3.6]{gwynne-miller-simple-quad}.
\end{proof} 

The main tool in the proof of Proposition~\ref{prop-stable-conv-finite} is the estimate of the Radon-Nikodym derivative~\eqref{eqn-perc-rn}.  We note that it blows up on the event that $W_\iota^{\infty,L} + W_\iota^{\infty,R}$ is close to $-\el_L-\el_R$, which in turn occurs when $\iota$ is chosen to be close to the terminal time of the finite-volume percolation exploration process.  We will thus need to rule out $W_\iota^{\infty,L} + W_\iota^{\infty,R}$ being \emph{too} close to $-\el_L-\el_R$ for an appropriately chosen stopping time.  This is carried out in Lemma~\ref{lem-short-bdy-time}, using Lemmas~\ref{lem-uneven-length} and~\ref{lem-no-close-jump}.  It will also be important that the boundary length does not change much after our stopping time, which is the purpose of Lemma~\ref{lem-max-bdy-length} just below.  The reader might find it helpful first to read Section~\ref{sec-stable-conv-finite}, where the proof of Proposition~\ref{prop-stable-conv-finite} is completed before reading the details of these intermediate lemmas.

The following general lemma gives us a bound for the boundary length of the unexplored region for \emph{any} peeling process on a free Boltzmann quadrangulation with simple boundary, and will be used in conjunction with~\eqref{eqn-unexplored-length} to bound the processes of Definition~\ref{def-bdy-process}.

\begin{lem} \label{lem-max-bdy-length}
Let $\el \in \BB N$ and let $(Q,\BB e)$ be a free Boltzmann quadrangulation with simple boundary and perimeter $2\el$. 
Let $\{(\ol Q_j , \dot e_{j+1})\}_{j\in\BB N_0}$ be any peeling process of $Q$, with any choice of target point. 
For $k \in \BB N$ with $k > \el$, 
\eqbn
\BB P\left[ \exists j \in \BB N_0 \: \op{with} \: \#\mcl E\left(\bdy \ol Q_j \right)  \geq k \right] \leq \left(1 + o_\el(1) \right) \left(\frac{k}{\el} \right)^{-3}
\eqen
with the rate of the $o_\el(1)$ universal. 
\end{lem}
\begin{proof}
Let $\iota_k$ be the smallest $j\in\BB N$ for which $ \#\mcl E\left(\bdy \ol Q_j \right)  \geq k $, or $\iota_k = \infty$ if no such $j$ exists.
We seek an upper bound for $\BB P[\iota_k < \infty]$. The idea of the proof is that if $\iota_k < \infty$, then the number of interior vertices of $Q$ is likely to be larger than usual, so $\BB P[\iota_k < \infty]$ cannot be too large.
For $m\in\BB N$, 
\eqb \label{eqn-area-bayes}
\BB P\left[ \#\mcl V\left(Q\setminus \bdy Q \right) \geq m \right] \geq \BB P\left[  \#\mcl V\left( \ol Q_{\iota_k} \setminus \bdy \ol Q_{\iota_k} \right) \geq m \,|\, \iota_k < \infty \right] \BB P[\iota_k < \infty] .
\eqe 
By the asymptotic formula~\eqref{eqn-fb-area-tail} for the law of the area of a free Boltzmann quadrangulation with simple boundary,
\eqb
\label{eqn-boundary-l-tail}
\BB P\left[ \#\mcl V\left(Q\setminus \bdy Q \right) \geq m \right] = (c  + o_\el(1) )(1+o_m(1)) \el^{3} m^{-3/2} , 
\eqe
with the constant $c >0$ and the rate of the $o_\el(1)$ universal; and the rate of the $o_m(1)$ depending on $\el$.  Since $ \#\mcl E\left(\bdy \ol Q_{\iota_k} \right)  \geq k > \el$ on the event $\{\iota_k <\infty\}$ and the conditional law of $\ol Q_{\iota_k}$ given $\{\iota_k<\infty\}$ and its boundary length is that of a free Boltzmann quadrangulation with simple boundary, also
\eqb
\label{eqn-boundary-k-tail}
\BB P\left[ \#\mcl V\left( \ol Q_{\iota_k} \setminus \bdy \ol Q_{\iota_k} \right) \geq m \,|\, \iota_k < \infty \right] = (c  + o_\el(1) )(1+o_m(1)) k^{3} m^{-3/2}  .
\eqe
Rearranging~\eqref{eqn-area-bayes} and inserting the bounds~\eqref{eqn-boundary-l-tail}, \eqref{eqn-boundary-k-tail} we have for each $m\in\BB N$ that
\eqbn
\BB P[\iota_k < \infty]  \leq   (1  + o_\el(1) )(1+o_m(1))\left( \frac{k}{\el} \right)^{-3} .
\eqen 
Sending $m\rta\infty$ yields the statement of the lemma.
\end{proof}

Lemma~\ref{lem-max-bdy-length} gives us a good bound for $\max_{j\in \BB N_0} |W_j|$ in the case when~$\el_L$ and~$\el_R$ are approximately the same size. If one of~$\el_L$ or~$\el_R$ is much smaller than the other, we expect that one of~$W^L$ or~$W^R$ has a single big downward jump corresponding to the time when most of~$\bdy Q$ is disconnected from the target edge~$\BB e_*$ by the percolation peeling process, and otherwise~$|W_j|$ fluctuates by at most~$\el_L\wedge\el_R$. The following lemma makes this intuition precise. The lemma will only be used in the proof of Lemma~\ref{lem-short-bdy-time} below, so the reader may wish to skip the statement and proof for now and refer back to it when it is needed.

\begin{figure}[ht!]
 \begin{center}
\includegraphics[scale=.65]{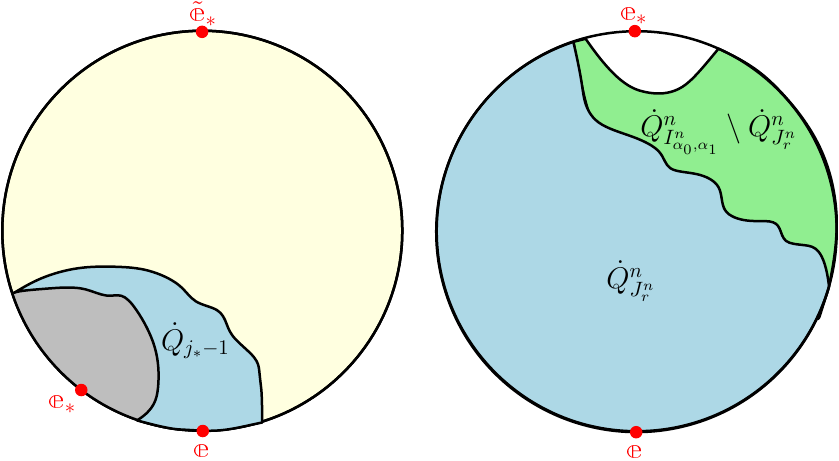} 
\caption[Stopping times for the peeling process]{\label{fig-uneven-length} \textbf{Left:} Illustration of the proof of Lemma~\ref{lem-uneven-length}. If the target edge $\BB e_*$ is close to the initial edge $\BB e$, we re-target the process at an edge $\wt{\BB e}_*$ at macroscopic boundary length away from $\BB e$. The peeling processes targeted at $\BB e_*$ and $\wt{\BB e}_*$ agree at every time strictly before the first time $j_*$ that the target edges are separated, at which time the cluster of the process targeted at $\BB e_*$ (resp.\ $\wt{\BB e}_*$) is the union of the blue and yellow (resp.\ blue and gray) regions. \textbf{Right:} The peeling clusters at the stopping times $J_r^n$ and $I_{\alpha_0,\alpha_1}^n$ considered in Lemmas~\ref{lem-no-close-jump} and~\ref{lem-short-bdy-time}, respectively. At time $J_r^n$, the left and right boundary lengths of the unexplored quadrangulation might be very different. Lemma~\ref{lem-uneven-length} is used to produce the time $I_{\alpha_0,\alpha_1}^n$, at which these two lengths are comparable.   }
\end{center}
\end{figure}

\begin{lem}
\label{lem-uneven-length}
For each $\ep \in (0,1)$ there exists $C  = C(\ep) >1$ such that for each choice of $\el_L , \el_R \in \BB N_0$ with $\el_L + \el_R$ even, there is a stopping time $j_* = j_*(\ep) \in [0,\mcl J]_{\BB Z}$ such that the following is true. Let $E(C)$ be the event that all of the following hold.
\begin{enumerate}
\item $C^{-1} (\el_L \wedge \el_R)^{3/2} \leq j_* \leq C(\el_L \wedge \el_L)^{3/2}$. \label{item-uneven-time} 
\item $Y^L_{j_*}  \leq \el_L - C^{-1} (\el_L \wedge \el_R) $ and $Y^R_{j_*}  \leq \el_R  -  C^{-1} (\el_L \wedge \el_R)$. \label{item-uneven-lower}
\item $\max_{j\in [1,j_* -1]_{\BB Z }} |W_j| + \max_{j\in [j_* , \mcl J]_{\BB Z}} |W_j - W_{j_*}|  \leq C (\el_L \wedge \el_R) $ (here $W = (W^L,W^R)$, as in Definition~\ref{def-bdy-process}).  \label{item-uneven-upper}
\end{enumerate} 
Then $\BB P\left[ E(C) \right] \geq 1- \ep $. 
\end{lem}
\begin{proof}
Throughout the proof we assume that $\el_L \leq \el_R$; the case when $\el_R \leq \el_L$ is treated similarly.  Let $b \in (0,1)$ be a small parameter to be chosen later, in a manner depending only on $\ep$. 

We will first treat the case when $\el_L$ and $\el_R$ are roughly comparable, in the sense that $\el_L \geq b \el_R$, which is a straightforward consequence of Lemma~\ref{lem-max-bdy-length}. To treat the case when $\el_L \leq b \el_R$, we will consider the percolation peeling process with a new target edge $\wt{\BB e}_*$ chosen in such a way that the two arcs separating $\BB e$ and $\wt{\BB e}_*$ have comparable lengths. We will take $j_*$ to be the time at which the original (equivalently, the re-targeted) process disconnects $\BB e_*$ from $\wt{\BB e}_*$, so that the processes agree up to time $j_*$. We will then show that the conditions in the definition of $E(C)$ are satisfied with high probability at time $j_*$ using estimates for the re-targeted process (which come from a comparison to peeling on the UIHPQ$_{\op{S}}$). See Figure~\ref{fig-uneven-length} for an illustration. 
\medskip
 
\noindent\textit{Step 1: the case when $\el_L \geq b \el_R$.} If $\el_L \geq b \el_R$, then $\el_L \wedge \el_R \geq \frac{b}{2} (\el_L  + \el_R)$. We always have $Y_j^L \leq \el_L$ and $Y_j^R \leq \el_R$. Furthermore, the boundary length of the unexplored region at time $j$ is given by $\#\mcl E(\bdy \ol Q_j) = W_j^L + W_j^R + \el_L + \el_R$. Lemma~\ref{lem-max-bdy-length} implies that if $\el_L \geq b \el_R$, then for $C>0$,
\eqb \label{eqn-similar-length}
\BB P\left[ \max_{j\in [1,\mcl J]_{\BB Z}}  \#\mcl E(\bdy \ol Q_j)  > C (\el_L \wedge \el_R) \right] \preceq b^{-3} C^{-3}  
\eqe 
with universal implicit constant. Furthermore, if we set $j_* = C_0^{-1} (\el_L \wedge \el_R)^{3/2}$ for large enough $C_0 = C_0(b,\ep) > 1$, then Proposition~\ref{prop-stable-conv} and Lemma~\ref{lem-perc-rn} together imply that $\BB P\left[ |W^L_{j_*} + \el_L| \wedge |W^R_{j_*} + \el_R| \leq  C_1^{-1} (\el_L \wedge \el_R) \right] \leq \ep/2$ for a constant $C_1 = C_1(b,\ep) > 0$. By combining this with~\eqref{eqn-similar-length}, we obtain the statement of the lemma in the case when $\el_L \geq b \el_R$ for an appropriate choice of $C \leq C_0\wedge C_1$. 
 \medskip
 
\noindent\textit{Step 2: re-targeting the peeling process.}
Now suppose that $\el_L < b \el_R$. We will show that the statement of the lemma holds in this case for sufficiently small $b = b(\ep) \in (0,1)$. We start by re-targeting the peeling process in such a way that the two boundary arcs between the starting point and the target point have the same length.  Let $\beta : [0, \el_L + \el_R]_{\BB Z} \rta \mcl E(\bdy Q)$ be the boundary path for $Q$ with $\beta(0) = \BB e$, and recall that $\BB e_* = \beta(\el_R+1)$ is the target edge for the peeling process of $(Q,\BB e,\theta)$ with $\el_L$-white/$\el_R$-black boundary conditions.  Let $\wt \el_L = \wt \el_R = (\el_L + \el_R)/2$ and consider the peeling process of $(Q,\BB e,\theta)$ with $\wt\el_L$-white/$\wt\el_R$-black boundary conditions, which is a peeling process targeted at $\wt{\BB e}_*:= \beta(\wt\el_R-1)$.  Define the boundary length processes $\wt X^L , \wt X^R ,\wt Y^L , \wt Y^R, \wt W^L, \wt W^R$, and $\wt W$ as in Definition~\ref{def-bdy-process} for this peeling process.

Let $j_* $ be the smallest $j\in\BB N$ for which the peeling cluster $\dot Q_j$ disconnects $\BB e_*$ from $\wt{\BB e}_*$.  The definition of the percolation peeling process implies that the peeling processes targeted at $\BB e_*$ and $\wt{\BB e}_*$ agree until time $j_*$.  In particular,
\eqb \label{eqn-retarget-agree}
\wt X^L|_{[0,j_*-1]} = X^L|_{[0,j_*-1]}  ,\quad \wt Y^L|_{[0,j_*-1]} = Y^L|_{[0,j_*-1]}  , \quad\op{and} \quad    \wt W^L|_{[0,j_*-1]} = W^L|_{[0,j_*-1]}  .
\eqe  
and similarly with $R$ in place of $L$. Furthermore,
\eqb \label{eqn-bubble-compare-time}
j_* = \min\left\{ j\in \BB N_0 : \wt Y^L_j \geq \el_L \right\} .
\eqe 
 \medskip
 
\noindent\textit{Step 3: regularity event for the re-targeted process.}
We will now define an event in terms of $(\wt W,\wt Y^L,\wt Y^R) |_{[0,j_*]}$ on which the event $E(C)$ in the statement of the lemma is likely to occur.  For $C_1 >0$, let $\wt E(C_1)$ be the event that the following is true.
\begin{enumerate}
\item $C_1^{-1} \el_L^{3/2} \leq   j_* \leq C_1 \el_L^{3/2}$. 
\item $\wt Y^L_{j_*-1} \leq \left(1 - C_1^{-1} \right) \el_L$ and $\wt Y^L_{j_*} \geq \left(1 + C_1^{-1} \right) \el_L$.
\item $\max_{j\in [0,j_*]} |\wt W_j|  \leq C_1 \el_L$. 
\end{enumerate}
If $\wt E(C_1)$ occurs, then condition~\ref{item-uneven-time} in the definition of $E(C_1)$ occurs.  Furthermore, by~\eqref{eqn-retarget-agree}, $Y^L_{j_*}=\wt Y^L_{j_*-1}$ and $\el_R -  Y^R_{j_*} =   \wt Y^L_{j_* } - \el_L$. Hence also condition~\ref{item-uneven-lower} in the definition of $E(C_1)$ occurs. The relation~\eqref{eqn-retarget-agree} shows also that $\max_{j\in [1,j_* -1]_{\BB Z }} |W_j| \leq C_1 \el_L$.  In other words, $\wt E(C_1)$ implies all of the conditions in the definition of $E(C_1)$ except possibly an upper bound for $\max_{j\in [j_* , \mcl J]_{\BB Z}} |W_j - W_{j_*}|$.

We now check that we have a suitable upper bound for $\max_{j\in [j_* , \mcl J]_{\BB Z}} |W_j - W_{j_*}|$ on $\wt E(C_1)$. 
The unexplored quadrangulation $\ol Q_{j_*}$ at time $j_*$ is the same as the quadrangulation disconnected from $\wt{\BB e}_*$ at time $j_*$ by the peeling process targeted at $\wt{\BB e}_*$, so on $\wt E(C_1)$, 
\eqbn
\#\mcl E(\bdy \ol Q_{j_*} )  \leq \wt W_{j_*} - \wt W_{j_*-1}  +2 \leq 2 C_1 \el_L + 2 .
\eqen
By the Markov property of peeling, the conditional law of $ \ol Q_{j_*} $ given $\mcl F_{j_*}$ is that of a free Boltzmann quadrangulation with simple boundary and given boundary length. Hence Lemma~\ref{lem-max-bdy-length} implies that we can find $C_2 = C_2(C_1, \ep) > 1$ such that 
\eqbn
\BB P\left[  \max_{j\in [j_* , \mcl J]_{\BB Z}} |W_j - W_{j_*}| \leq C_2 \el_L \,|\, \wt E(C_1) \right] \geq 1-\ep ,
\eqen
which means that 
\eqb \label{eqn-retarget-cond}
\BB P[E(C_1 + C_2) \,|\, \wt E(C_1) ] \geq 1-\ep . 
\eqe
Hence it remains only to choose $b = b(\ep) \in (0,1)$ and $C_1 = C_1(\ep) >1$ in such a way that $\BB P[\wt E(C_1)]$ is close to~$1$. 
 \medskip
 
\noindent\textit{Step 4: regularity event holds with high probability.}
We will estimate the probability of $\wt E(C_1)$ by comparison to the boundary length processes for the percolation peeling process on the UIHPQ$_{\op{S}}$,  which we recall are defined as in Definition~\ref{def-bdy-process} and are denoted with a superscript $\infty$.  Following~\eqref{eqn-bubble-compare-time}, let $j_*^\infty$ be the smallest $j\in\BB N_0$ for which $Y^{\infty,L}_j \geq \el_L$.  Also let $E^\infty(C_1)$ be defined in the same manner as $\wt E(C_1)$ as above but with $Y^{\infty,L}$, $Y^{\infty,R}$, $W^\infty$, and $j_*^\infty$ in place of $\wt Y^L,\wt Y^R,\wt W$, and $j_*$.  By Proposition~\ref{prop-stable-conv}, the process $W^\infty = (W^{\infty,L} , W^{\infty,R})$, rescaled as in~\eqref{eqn-bdy-process-rescale}, converges in law in the local Skorokhod topology to a pair of independent totally asymmetric $3/2$-stable processes with no positive jumps.  From this scaling limit result together with~\eqref{eqn-bdy-process-inf}, we find that for each $\ep > 0$ there exists $C_1  = C_1(\ep ) > 0$ such that $\BB P[E^\infty(C_1)] \geq 1 - \ep$ for every possible choice of $\el_L$. 

Choose $b = b(\ep) \in (0,1)$ in such a way that $b < \frac14 C_1^{-1} \ep$, so that $C_1 \el_L \leq \ep \wt\el_L = \ep \wt \el_R$.  Then $j_*$ is always less than the terminal time $\wt{\mcl J}$ of the peeling process targeted at $\beta(\wt \el_R + 1)$.  By Lemma~\ref{lem-perc-rn} applied with $(\wt\el_L, \wt \el_R)$ in place of $(\el_L , \el_R)$, we find that
\allb
\BB P\left[ \wt E(C_1) \right] 
\geq \BB E\left[  \left( \frac{ W^{\infty,L}_{j_*^\infty} + W^{\infty,R}_{j_*^\infty}  }{  \el_L +  \el_R} + 1     \right)^{-5/2}    \BB 1_{E^\infty(C_1)}    \right]  \geq (1 + o_{ \el_L +  \el_R} (1)) (1+\ep)^{-5/2} (1-\ep) . \label{eqn-retarget-sup}
\alle 
Here we recall that $W^{\infty,L}_{j_*^\infty} + W^{\infty,R}_{j_*^\infty} \leq C_1 \el_L \leq \ep \wt \el_L$ on $E^\infty(C_1)$ and that $\el_L+ \el_R = \wt \el_L +\wt \el_R$.  Combining~\eqref{eqn-retarget-cond} and~\eqref{eqn-retarget-sup}, possibly shrinking $C_1$ to deal with finitely many small values of $\ep$, and using that $\ep \in (0,1)$ is arbitrary concludes the proof. 
\end{proof}

\subsection{Tightness of the boundary length process in the finite boundary case}
\label{sec-sk-tight}

In the remainder of this section we assume we are in the setting of Proposition~\ref{prop-stable-conv-finite} (equivalently, the setting of Theorem~\ref{thm-perc-conv}).
With $\el_L^n , \el_R^n$ as in the discussion just above Proposition~\ref{prop-stable-conv-finite}, we set 
\eqbn
\frk l_L^n :=  \bcon^{-1} n^{-1/2} \el_L^n \quad \op{and} \quad \frk l_R^n :=  \bcon^{-1} n^{-1/2} \el_R^n , 
\eqen
so that $\frk l_L^n \rta \frk l_L$ and $\frk l_R^n \rta \frk l_R$. 

We also define the peeling clusters $\{\dot Q_j^n\}_{j\in\BB N_0}$, the unexplored quadrangulations $\{\ol Q_j^n\}_{j\in\BB N_0}$, the peeled edges $\{\dot e_j^n\}_{j\in\BB N}$, the filtration $\{\mcl F_j^n\}_{j \in \BB N_0}$, and the terminal time $\mcl J^n$ for the percolation peeling process of $(Q^n,\BB e^n,\theta^n)$ with $\el_L^n$-white/$\el_R^n$-black boundary conditions as in Section~\ref{sec-perc-peeling}. 
Also define the boundary length processes $X^{L,n} , X^{R,n}   ,Y^{L,n} ,Y^{R,n}$, and $W^n = (W^{L,n} , W^{R,n})$ as in Definition~\ref{def-bdy-process} with $(Q,\BB e, \theta) = (Q^n,\BB e^n,\theta^n)$ and the rescaled boundary length process $Z^n = (L^n,R^n)$ as in~\eqref{eqn-bdy-process-rescale}. 

We will often compare the boundary length processes for $(Q^n,\BB e^n,\theta^n)$ to the analogous processes for the percolation peeling process on the UIHPQ$_{\op{S}}$, which (as usual) we denote by an additional superscript $\infty$. 

In this subsection we will prove tightness of the law of the rescaled boundary length processes $Z^n$ in the Skorokhod topology, and in the next subsection we will complete the proof of Proposition~\ref{prop-stable-conv-finite} by identifying the law of a subsequential limit. 

\begin{lem} \label{lem-sk-tight}
The laws of the processes $Z^n = (L^n ,R^n)$ are tight in the Skorokhod topology on $[0,\infty)$. 
\end{lem}

Recall that $Z^n$ is constant on  $[\tcon^{-1} n^{-3/4} \mcl J^n,\infty)$. We do not show that the law of the time $\tcon^{-1} n^{-3/4} \mcl J^n$ is tight, but we do show that $Z^n$ is likely to be nearly constant after a time which might be much smaller than $\tcon^{-1} n^{-3/4} \mcl J^n$ but which is typically of constant order.

We will deduce Lemma~\ref{lem-sk-tight} from the scaling limit for the UIHPQ$_{\op{S}}$ boundary length processes $Z^{\infty,n}$ together with local absolute continuity in the form of Lemma~\ref{lem-perc-rn}. However, some care is needed since the Radon-Nikodym derivative in Lemma~\ref{lem-perc-rn} blows up when the boundary length $W^{L,n} + W^{R,n} + \el_L^n +\el_R^n $ of the unexplored quadrangulation is small, so this lemma does not immediately enable us to rule out pathological behavior of the process $W $ when it is close to $(-\el_L,-\el_R)$. To get around this issue, we need to analyze certain stopping times corresponding to when the percolation peeling process gets close, in some sense, to the terminal time. The stopping times introduced in this subsection and the estimates we prove for these times will also be used several times later in the paper. See Figure~\ref{fig-uneven-length}, right, for an illustration of these stopping times.
 
Let $\bcon$ and $\tcon$ be the normalizing constants from Section~\ref{sec-results} and for $r \geq 0$, let
\allb \label{eqn-discrete-hit-time}
J_r^n &:= \min\left\{ j \in \BB N_0 : Y_j^{L,n} \geq    \el_L^n -   r \bcon  n^{1/2}   \:\op{or} \: Y_j^{R,n} \geq   \el_R^n  -   r \bcon n^{1/2} \right\}   , \notag \\
J_r^{\infty,n} &:= \min\left\{ j \in \BB N_0 : Y_j^{\infty,L} \geq    \el_L^n -   r \bcon n^{1/2}    \:\op{or} \: Y_j^{\infty,R} \geq   \el_R^n  -   r \bcon  n^{1/2}   \right\} , \quad   \\
&\qquad \sigma_r^n := \tcon^{-1} n^{-3/4} J_r^n  , \quad \op{and} \quad   \sigma_r^{\infty,n} := \tcon^{-1} n^{-3/4} J_r^{\infty,n}   \notag  
\alle
We observe that $J_r^n$ is an $\{\mcl F_j^n\}_{j\in\BB N_0}$-stopping time.  For $r > 0$, $J_r^n $ is typically strictly less than the terminal time $\mcl J^n$ and for $r = 0$, $J_0^n  = \mcl J^n$.  Since $(\frk l_L^n , \frk l_R^n) \rta (\frk l_L ,\frk l_R)$ and by~\eqref{eqn-bdy-process-inf}, the time $\sigma_r^n$ is approximately the first time $t$ that either $L^n_t \leq -\frk l_L - r$ or $R^n_t \leq -\frk l_R - r$, and similarly for $\sigma_r^{\infty,n}$. 

The main fact we need about the times $J_r^n$ is the following overshoot lemma, which says that it is unlikely that either the left boundary length or the right boundary length of the unexplored quadrangulation at time $J_r^n$ is of smaller order than $n^{1/2}$. Heuristically, this means that the tip of the percolation peeling process cannot jump from an edge at macroscopic rescaled boundary length distance from the target edge to an edge at microscopic rescaled distance from the target edge. 
 
\begin{lem} \label{lem-no-close-jump}
For each $r>0$ and $\ep \in (0,1)$ there exists $\alpha = \alpha(r,\ep) \in (0,r]$, $A = A(r,\ep) >0$, and an $n_* = n_*(r,\ep) \in \BB N$ such that for $n\geq n_*$, 
\eqbn
\BB P\left[\el_L^n - Y^{ L,n}_{J_r^{ n} } \geq \alpha n^{1/2}  , \:  \el_R^n - Y^{R,n}_{J_r^{ n} } \geq \alpha n^{1/2} ,\: \op{and} \: J_r^n \leq A n^{3/4}  \right] \geq 1-\ep .
\eqen
\end{lem}
\begin{proof}
The idea of the proof is to compare $Y^{L,n}$ and $Y^{R,n}$ to the analogous processes for the UIHPQ$_{\op{S}}$ via Lemma~\ref{lem-perc-rn}. For this purpose we first need a lower bound of order $n^{1/2}$ for the total boundary length of the unexplored quadrangulation at time $J_r^n$, so that the Radon-Nikodym derivative of Lemma~\ref{lem-perc-rn} does not blow up.

Let $\beta^n$ be the counterclockwise periodic boundary path of $Q^n$ with $\beta^n(0) = \BB e^n$.  For $r > 0$, let $K_r^{L,n}$ (resp.\ $ K_r^{R,n}$) be the largest $k \in \BB N$ such that $\beta^n(- \lfloor \el_L^n - r \bcon n^{1/2} +k \rfloor )$ (resp.\ $\beta^n(\lfloor \el_R^n - r \bcon n^{1/2} + k \rfloor)$) belongs to $\dot Q^n_{J_r^n}$, or $0$ if no such $k\in\BB N$ exists. By the definition of~$J_r^n$, at least one of~$K_r^{L,n}$ or~$K_r^{R,n}$ is positive. Furthermore,
\eqb \label{eqn-overshoot-to-length}
 \el_L^n - Y^{L,n}_{J_r^n}  \geq    r \bcon n^{1/2} -  K_r^{L,n} \quad \op{and} \quad 
  \el_R^n - Y^{R,n}_{J_r^n} \geq     r \bcon n^{1/2} -   K_r^{R,n}  .
\eqe 
By~\cite[Lemma~3.9]{gwynne-miller-simple-quad} (applied with $2\el = \el_L^n + \el_R^n $, $a_L \asymp 2\el-a_R \asymp r n^{1/2}$, and $E$ the whole probability space), we can find a universal constant $c>0$ such that 
\eqb \label{eqn-use-overshoot}
\BB P\left[   r \bcon n^{1/2} - K_r^{L,n} \leq  c r n^{1/2} \,\op{and} \, r \bcon n^{1/2} -  K_r^{R,n} \leq  c r n^{1/2}   \right]  
\preceq   \frac{(\el_L^n+\el_R^n)^{5/2}  (r n^{1/2})^{3/2}  }{ (r n^{1/2})^{5/2}  (r n^{1/2})^{5/2} } 
\preceq r^{-7/2} n^{-1/2} .
\eqe 
In particular,~\eqref{eqn-overshoot-to-length} implies that we can find $n_* = n_*(r , \ep) \in \BB N$ such that for $n\geq n_*$, it holds with probability at least $1-\ep$ that
\eqb \label{eqn-total-length-min}
  (\el_L^n - Y^{ L ,n}_{J_r^{ n} }) \vee (\el_R^n - Y^{ R ,n}_{J_r^{ n} } )  \geq c r n^{1/2} .
\eqe
 
By Proposition~\ref{prop-stable-conv}, after possibly increasing $n_*$ we can find $\alpha  = \alpha(r,\ep) \in (0,r]$ and $A  =A(r,\ep) > 0$ such that for $n\geq n_*$, 
\eqb \label{eqn-big-jump-use-stable}
\BB P\left[ \el_L^n - Y^{\infty,L}_{J_r^{\infty,n} } \in \left[ 0 ,  \alpha n^{1/2}   \right]_{\BB Z} ,\:  \el_R^n - Y^{\infty,R}_{J_r^{\infty,n} } \in \left[ 0 ,  \alpha n^{1/2}   \right]_{\BB Z} ,\: \op{or} \:  J_r^{\infty,n} > A n^{3/4} \right] \leq  \frac12 \ep (c r)^{5/2} .
\eqe 
On the event that $(\el_L^n - Y^{\infty ,L }_{J_r^{\infty, n} }) \wedge (\el_R^n - Y^{\infty ,R }_{J_r^{\infty, n}} )  \geq 0$ and $ (\el_L^n - Y^{\infty ,L }_{J_r^{\infty, n} }) \vee (\el_R^n - Y^{\infty ,R }_{J_r^{\infty, n}} )  \geq c r n^{1/2} $, the Radon-Nikodym derivative of Lemma~\ref{lem-perc-rn} at time $J_r^{\infty,n}$ is bounded above by $(1+o_n(1)) (c r)^{-5/2}$.
By~\eqref{eqn-big-jump-use-stable}, after possibly increasing $n_*$ we can arrange that for $n\geq n_*$, 
\eqbn
\BB P\left[ \left\{ \el_L^n - Y^{ L,n}_{J_r^{ n} } \leq \alpha n^{1/2} , \:   \el_R^n - Y^{R,n}_{J_r^{ n} } \leq \alpha n^{1/2} ,  \: \op{or} \:  J_r^n > A n^{3/4} \right\} \cap \left\{  (\el_L^n - Y^{ L ,n}_{J_r^{ n} }) \vee (\el_R^n - Y^{ R ,n}_{J_r^{ n} } )  \geq c r n^{1/2}  \right\}  \right] \leq  \ep  
\eqen
with the implicit constant depending only on $(\frk l_L , \frk l_R)$. Since $\ep \in (0,1)$ is arbitrary, we conclude by combining this last estimate with~\eqref{eqn-total-length-min}. 
\end{proof}

Lemma~\ref{lem-no-close-jump} is not quite sufficient for our purposes since we have not ruled out the possibility that one of $W_{J_r^n}^{L,n} + \el_L^n$ or $W_{J_r^n}^{R,n} + \el_R^n$ is much larger than $r n^{1/2}$. We next consider a stopping time at which $W_{J_r^n}^{L,n} + \el_L^n$ and $W_{J_r^n}^{R,n} + \el_R^n$ are necessarily proportional to one another.

For $0\leq \alpha_0 \leq \alpha_1 $, let $I_{\alpha_0,\alpha_1}^n$ be the smallest $j\in\BB N_0$ for which 
\allb \label{eqn-discrete-close-time}
& \alpha_0 \bcon n^{1/2}   \leq \el_L^n -   Y_j^{L,n}    ,\quad  W_j^{L,n} + \el_L^n \leq   \alpha_1  \bcon  n^{1/2}    ,  \quad \notag\\
& \alpha_0 \bcon n^{1/2}  \leq   \el_R^n - Y_j^{R,n} ,\quad \op{and} \quad  W_j^{R,n} + \el_R^n \leq   \alpha_1 \bcon n^{1/2}   ,
\alle
so that on the event $\{I_{\alpha_0,\alpha_1}^n < \infty\}$ we have both a lower bound for the number of edges of $\bdy \ol Q_{I_{\alpha_0,\alpha_1}^n}^n\cap \bdy Q^n$ lying to the left and right of the target edge $\BB e_*^n$ and an upper bound for the total number of edges of $\bdy \ol Q_{I_{\alpha_0,\alpha_1}^n}^n$. 
Analogously, let $I_{\alpha_0,\alpha_1}^{\infty,n}$ be the smallest $j\in\BB N_0$ such that 
\allb \label{eqn-discrete-close-time-infty}
& \alpha_0 \bcon n^{1/2}    \leq   \el_L^n - Y_j^{\infty,L} ,\quad  W_j^{\infty,L} + \el_L^n \leq  \alpha_1  \bcon n^{1/2}  ,  \quad \notag\\
& \alpha_0 \bcon  n^{1/2}  \leq   \el_R^n  - Y_j^{\infty,R}  ,\quad \op{and} \quad  W_j^{\infty,R}  + \el_R^n \leq   \alpha_1 \bcon n^{1/2}   .
\alle
Also define the rescaled times
\eqb \label{eqn-discrete-close-time-rescale}
\tau_{\alpha_0,\alpha_1}^n := \tcon^{-1} n^{-3/4} I_{\alpha_0,\alpha_1}^n \quad \op{and} \quad \tau_{\alpha_0,\alpha_1}^{\infty,n} := \tcon^{-1} n^{-3/4} I_{\alpha_0,\alpha_1}^{\infty,n} .
\eqe 

In contrast to the times $J_r^n$ of~\eqref{eqn-discrete-hit-time}, it is possible that $I_{\alpha_0,\alpha_1}^n = \infty$ (equivalently $I_{\alpha_0,\alpha_1}^n \geq \mcl J^n$). However, as the next lemma demonstrates, this is unlikely to be the case if we make an appropriate choice of $\alpha_0$ and $\alpha_1$. 
 
\begin{lem} \label{lem-short-bdy-time}
For each $\alpha_1 >0$ and each $\ep \in (0,1)$, there exists $\ol \alpha_0 = \ol \alpha_0(\alpha_1,\ep) \in (0,\alpha_1)$, $A = A(\alpha_1 ,\ep ) > 0$, and $n_* = n_*(\alpha_1 ,\ep) \in \BB N$ 
such that for each $\alpha_0 \in (0,\ol\alpha_0]$ and each $n\geq n_*$, $\BB P[I_{\alpha_0,\alpha_1}^n \leq A n^{3/4} ]\geq 1-\ep$.
\end{lem}  
\begin{proof} 
By Lemma~\ref{lem-no-close-jump}, for each $r >0$ there exists $\alpha(r) = \alpha(r,\ep) \in (0,r]$, $\wt A = \wt A(r,\ep) >0$, and $n_* = n_*(r,\ep) \in \BB N$ such that for $n\geq n_*$, the event
\eqbn
E_r^1 := \left\{\el_L^n - Y^{ L,n}_{J_r^{ n} } \geq \alpha(r) n^{1/2} , \: \el_R^n - Y^{R,n}_{J_r^{ n} } \geq \alpha(r) n^{1/2} ,  \: \op{and} \:   J_r^n \leq \wt A n^{3/4} \right\}
\eqen
has probability at least $1-\ep$. It could be the case that $E_r^1$ occurs, but one of $W_{J_r^n}^L + \el_L^n$ or $W_{J_r^n}^R + \el_R^n$ is much larger than $r n^{1/2}$. To deal with this, we will apply Lemma~\ref{lem-uneven-length} to the process after time $J_{r}^n$ to produce a time after $J_r^n$ at which the left/right boundary lengths of the unexplored quadrangulation are comparable.

By the Markov property of peeling, if we condition on $\mcl F_{J_r^n}^n$, then the conditional law of $(\ol Q_{J_r^n}^n , \dot e_{J_r^n}^n)$ is that of a free Boltzmann quadrangulation with simple boundary and perimeter $W_{J_r^n}^{L,n} + W_{J_r^n}^{R,n} + \el_L^n + \el_R^n$. Furthermore, the conditional law of the percolation peeling process after time $J_r^n$ is that of a percolation peeling process from $\dot e_{J_r^n}$ to the target edge $\BB e_*^n$, i.e.\ a percolation peeling process in $\ol Q_{J_r^n}^n$ with $(W_{J_r^n}^{L,n} + \el_L^n)$-white/$(W_{J_r^n}^{R,n} + \el_R^n)$-black boundary conditions. By the definition~\eqref{eqn-discrete-hit-time} of $J_r^n$, either $W_{J_r^n} +\el_L^n \leq \el_L^n - Y_{J_r^n}^L + 2 \leq r \bcon  n^{1/2} + 2$ or the same holds with $R$ in place of $L$. Therefore, if $E_r^1$ occurs then 
\eqbn
 \alpha(r) n^{1/2}\leq (W_{J_r^n}^{L,n} + \el_L^n) \wedge (W_{J_r^n}^{R,n} + \el_R^n) \leq r \bcon   n^{1/2} +2 .
\eqen

By Lemma~\ref{lem-uneven-length} applied to the above conditional percolation peeling process after time $J_r^n$, there is a constant $C = C(\ep) > 0$, independent of $n$ and $r$, and a stopping time $j_*^n(r)  \geq J_r^n+1$ such that on $E_r^1$, it holds with conditional probability at least $1-\ep$ given $\mcl F_{J_r^n}^n$ that the following hold.
\begin{enumerate} 
\item \label{item-use-uneven-time} $j_*^n(r) - J_r^n \leq C r^{3/2}  n^{3/4}$. 
\item \label{item-use-uneven-lower} $\el_L^n - Y^{L,n}_{j_*^n(r) }  \geq  C^{-1}  \alpha(r)  n^{1/2} $ and $\el_R^n - Y^{R,n}_{j_*^n(r) }  \geq C^{-1} \alpha(r)  n^{1/2}$.
\item \label{item-use-uneven-upper} $ \max_{j\in [j_*^n(r) , \mcl J^n]_{\BB Z}} |W_j^n - W_{j_*^n(r)}^n|  \leq C r  n^{1/2} $.  
\end{enumerate} 
Let $E_r^2$ be the event that this is the case, so that $\BB P[E_r^1 \cap E_r^2] \geq (1-\ep)^2$. 

Given $\alpha_1 > 0$, choose $r>0$ small enough that $  C r \leq  \alpha_1 \bcon $. 
Since $W_{\mcl J^n}^n = (-\el_L^n , -\el_R^n)$, condition~\ref{item-use-uneven-upper} in the definition of $E_r^2$ implies that
\eqbn
W_{j_*^n(r)}^{L,n} + \el_L^n \leq  C r n^{1/2} \leq \alpha_1 \bcon  n^{1/2}  \quad\op{and} \quad W_{j_*^n(r)}^{R,n} + \el_R^n \leq C r  n^{1/2} \leq   \alpha_1 \bcon n^{1/2}  .
\eqen
By combining this with condition~\ref{item-use-uneven-lower}, we see that for $\alpha_0 \leq \ol \alpha_0 :=   C^{-1} \alpha(r) \bcon^{-1}$, the time $j_{*}^n(r)$ satisfies the conditions in the definition~\eqref{eqn-discrete-close-time} of $I_{\alpha_0,\alpha_1}^n$ on $E_r^1\cap E_r^2$, whence $I_{\alpha_0,\alpha_1}^n \leq j_*^n(r) \leq A n^{3/4}$ on this event for $A :=   C r^{3/2} +\wt A $. Since $\BB P[E_r^1 \cap E_r^2] \geq (1-\ep)^2$ and $\ep \in (0,1)$ is arbitrary, we obtain the statement of the lemma.
\end{proof}

\begin{proof}[Proof of Lemma~\ref{lem-sk-tight}]
Fix $\ep \in (0,1)$.  By the standard compactness criterion for the Skorokhod space, we must show that there is a $C > 0$, a $\delta>0$, and an $n_*\in\BB N$ such that for each $n\geq n_*$, it holds with probability at least $1-\ep$ that the following hold.
\begin{enumerate}
\item $\sup_{t \geq 0} |Z_t^n| \leq C$. 
\item There exists a partition $0 = t_0 < \dots < t_N = \delta^{-1}$ such that $t_k - t_{k-1} \geq \delta$ and $\sup_{s,t\in [t_{k-1} , t_k)} |Z^n_s - Z^n_t| \leq \ep$ for each $k \in [1,N]_{\BB Z}$. 
\item $|Z^n_t - Z^n_s | \leq \ep$ for each $s,t \geq \delta^{-1}$. 
\end{enumerate}

To this end, let $\alpha_1 =  \alpha_1(\ep) > 0$ to be chosen later, in a manner depending only on $\ep$. By Lemma~\ref{lem-short-bdy-time}, there exists $\alpha_0 \in (0,\alpha_1)$, $A = A(\alpha_1,\ep)>0$, and $n_0  = n_0(\alpha_1,\ep) \in\BB N$ such that for $n\geq n_*$, the stopping time defined in~\eqref{eqn-discrete-close-time} satisfies 
\eqb \label{eqn-use-short-time-sk}
\BB P\left[ I_{\alpha_0,\alpha_1}^n \leq A n^{3/4} \right] \geq 1-\ep.
\eqe 

If $I_{\alpha_0,\alpha_1}^n < \infty$, then since $\alpha_0 > 0$ necessarily $I_{\alpha_0,\alpha_1}^n  < \mcl J^n$. 
In fact, by definition, 
\eqbn
W_{I_{\alpha_0,\alpha_1}^n}^{L,n} + W_{I_{\alpha_0,\alpha_1}^n}^{R,n} +\el_L^n + \el_R^n \geq  \alpha_0  \bcon n^{1/2} 
\eqen
whenever $I_{\alpha_0,\alpha_1}^n < \infty$. Hence Lemma~\ref{lem-perc-rn} implies that there is an $n_1 = n_1(\alpha_1 ,\ep) \geq n_1$ such that for $n\geq n_1$, the law of $W^n|_{[0, I_{\alpha_0,\alpha_1}^n]_{\BB Z}}$ on the event $\{ I_{\alpha_0,\alpha_1}^n  <\infty\}$ is absolutely continuous with respect to the law of $W^\infty|_{[0, I_{\alpha_0,\alpha_1}^{\infty,n}]_{\BB Z}}$, with Radon-Nikodym derivative bounded above by 
\eqb \label{eqn-sk-tight-rn}
b \alpha_0^{-5/2} \BB 1_{  (I_{\alpha_0,\alpha_1}^{\infty,n}  < \mcl J^{\infty,n} ) } 
\eqe 
where here $b > 0$ is a constant depending only on $(\frk l_L , \frk l_R)$ and $\mcl J^{\infty,n} := \min\left\{ j \in \BB N_0 : Y_j^{\infty,L} \leq - \el_L^n \:\op{or} \: Y_j^{\infty,R} \leq - \el_R^n\right\} $ is as in~\eqref{eqn-infty-terminal-time}.  

By Proposition~\ref{prop-stable-conv}, the rescaled boundary length processes $Z^{\infty,n} = (L^{\infty,n}, R^{\infty,n})$ for the UIHPQ$_{\op{S}}$ converge in law as $n\rta\infty$ to a pair $Z^\infty= (L^\infty,R^\infty)$ of independent totally asymmetric $3/2$-stable processes with no upward jumps.  
From this, we infer that the rescaled times $\tau_{\alpha_0,\alpha_1}^{\infty,n} = \tcon^{-1} n^{-3/4} I_{\alpha_0,\alpha_1}^{\infty,n} $ from~\eqref{eqn-discrete-close-time-rescale} and $\tcon^{-1} n^{-3/4} \mcl J^{\infty,n}$ converge in law to the analogous times for the process $Z^\infty$.  

Consequently, there exists $n_2 = n_2(\alpha_1,\ep) \geq n_1$, $\wt C = \wt C(\alpha_0 ,\ep) > 0$ and $\wt\delta =\wt\delta(\alpha_0 , \ep) > 0$ such that with probability at least $1- b^{-1} \alpha_0^{ 5/2} \ep  $, the following is true.
\begin{enumerate}
\item $\sup_{t \in [0, \tau_{\alpha_0,\alpha_1}^{\infty,n} \wedge A]} |Z_t^{\infty,n}| \leq \wt C$.
\item There exists a partition $0 = t_0 < \dots < t_N = \tau_{\alpha_0,\alpha_1}^{\infty,n} \wedge A $ such that $t_k - t_{k-1} \geq \wt\delta$ and $\sup_{s,t\in [t_{k-1} , t_k)} |Z^{\infty,n}_s - Z^{\infty,n}_t| \leq \ep$ for each $k \in [1,N]_{\BB Z}$. 
\end{enumerate} 
By this,~\eqref{eqn-use-short-time-sk}, and the Radon-Nikodym derivative estimate~\eqref{eqn-sk-tight-rn}, with probability at least $1-2\ep$, 
\begin{enumerate} 
\item $\sup_{t \in [0,\tau_{\alpha_0,\alpha_1}^n]} |Z_t^n| \leq \wt C$ and $\tau_{\alpha_0,\alpha_1}^n \leq A$. \label{item-sk-tight-max}
\item There exists a partition $0 = t_0 < \dots < t_N = \tau_{\alpha_0,\alpha_1}^n$ such that $t_k - t_{k-1} \geq \wt\delta$ and $\sup_{s,t\in [t_{k-1} , t_k)} |Z^n_s - Z^n_t| \leq \ep$ for each $k \in [1,N]_{\BB Z}$. \label{item-sk-tight-partition}
\end{enumerate} 
Let $E$ be the event that the above conditions are satisfied, so that $\BB P[E]\geq 1-2\ep$. 

It remains to deal with the behavior of $Z^n$ after time $\tau_{\alpha_0,\alpha_1}^n$. 
By the Markov property of peeling, if we condition on $\mcl F_{I_{\alpha_0,\alpha_1}^n}^n$ then on the event $E$, the conditional law of the unexplored quadrangulation $(\ol Q_{I_{\alpha_0,\alpha_1}^n}^n , \dot e_{I_{\alpha_0,\alpha_1}^n+1}^n)$ is that of a free Boltzmann quadrangulation with simple boundary and perimeter at most 
\eqbn 
W_{I_{\alpha_0,\alpha_1}^n }^{L,n} + W_{I_{\alpha_0,\alpha_1}^n }^{R,n} + \el_L^n + \el_R^n \leq 2 \alpha_1 \bcon n^{1/2} ,
\eqen
where here we recall the definition~\eqref{eqn-discrete-close-time} of $I_{\alpha_0,\alpha_1}^n$. 
By Lemma~\ref{lem-max-bdy-length}, if we choose $\alpha_1 = \alpha_1(\ep)$ sufficiently small, then there is an $n_* =n_*(\alpha_1 ,\ep) \geq n_2$ such that for $n\geq n_*$, 
\eqbn
\BB P\left[ \sup_{t \geq \tau_{\alpha_0,\alpha_1}^n} |W_t^n - W_{\tau_{\alpha_0,\alpha_1}^n}^n| > \ep  \,|\, E \right] \leq \ep .
\eqen
We thus obtain the conditions required for tightness with $C = \wt C + \ep$, $\delta = \wt\delta \wedge A^{-1}$, and $3\ep$ in place of $\ep$. Since $\ep\in (0,1)$ is arbitrary this suffices. 
\end{proof}

\subsection{Identification of the limiting boundary length process}
\label{sec-stable-conv-finite}

In this subsection we will prove Proposition~\ref{prop-stable-conv-finite}.  By Lemma~\ref{lem-sk-tight}, in the setting of that proposition, for any sequence of positive integers tending to $\infty$, there is a subsequence $\mcl N$ along which $Z^n$ converges in law in the Skorokhod topology to a process $\wt Z = (\wt L , \wt R): [0,\infty)\rta\BB R^2$ as $\mcl N \ni n \rta\infty$.  Henceforth fix such a subsequence $\mcl N$ and such a process $\wt Z$. We must show that $\wt Z$ has the same law as the process $Z = (L,R)$ in Proposition~\ref{prop-stable-conv-finite}.

We will compare the laws of the processes $\wt Z$ and $Z$ to the law of a pair $Z^\infty=(L^\infty,R^\infty)$ of independent totally asymmetric $3/2$-stable processes with no upward jumps (which we recall from Proposition~\ref{prop-stable-conv} is the limit of the laws of the processes $Z^{\infty,n}$).  In analogy with~\eqref{eqn-discrete-hit-time}, for $r \geq 0$ let
\allb \label{eqn-continuum-hit-time}
 \wt\sigma_r &:= \inf\left\{ t \geq 0 : \wt L_t \leq - \frk l_L + r \: \op{or} \: \wt R_t \leq -\frk l_R+r\right\} \notag \\
 \sigma_r &:= \inf\left\{ t\geq 0 : L_t \leq  -\frk l_L + r \: \op{or} \:  R_t \leq -\frk l_R+r\right\}  \\
 \sigma_r^\infty &:= \inf\left\{ t \geq 0 : L_t^\infty \leq - \frk l_L + r \: \op{or} \: R_t^\infty \leq -\frk l_R+r\right\} . \notag 
\alle
By Theorem~\ref{thm-bdy-process-law}, the time $\sigma_0$ is the terminal time of $Z$, i.e.\ the total quantum natural time length of the corresponding $\SLE_6$ in a quantum disk. Furthermore, a.s.\ $Z_t = (-\frk l_L ,-\frk l_R)$ for each $t \geq  \sigma_0$. 

The following lemma is an analog of Lemma~\ref{lem-perc-rn} for the process $\wt Z$.

\begin{lem} \label{lem-perc-limit-rn}
For each $t\geq 0$, the law of $\wt Z|_{[0,t]}$ restricted to the event $\{t < \wt\sigma_0\}$ is absolutely continuous with respect to the law of $Z^\infty|_{[0,t]}$, with Radon-Nikodym derivative
\eqbn
\left( \frac{L_t^{\infty } + R_t^{\infty }}{\frk l_L  + \frk l_R }  + 1    \right)^{-5/2} \BB 1_{(t < \sigma_0^{\infty })} . 
\eqen 
\end{lem}
\begin{proof} 
First note that the statement with $Z$ in place of $\wt Z$ follows from Theorem~\ref{thm-bdy-process-law}. 

By Proposition~\ref{prop-stable-conv}, if we define $\sigma_r^{\infty,n}$ as in~\eqref{eqn-discrete-hit-time}, then for each $r>0$ we have $(Z^{\infty,n} , \tau^{\infty,n}_r) \rta (Z^\infty,\sigma_r^\infty)$ in law as $n\rta\infty$. By Lemma~\ref{lem-perc-rn}, we find that for each $r > 0$ and each $t\geq 0$, the Radon-Nikodym derivative of the law of $Z^n|_{[0,t]}$ restricted to the event $\{t < \sigma_r^n\}$ with respect to the law of $Z^{\infty,n}|_{[0,t]}$ is given by
\eqbn
(1+o(1)) \left( \frac{L_t^{\infty,n} + R_t^{\infty,n}}{\frk l_L^n + \frk l_R^n}  + 1    \right)^{-5/2} \BB 1_{(t < \sigma_r^{\infty,n})} , 
\eqen
where the $o(1)$ tends to zero as $n\rta\infty$, at a deterministic rate (note that $L_t^{\infty,n} + \frk l_L^n$ and $R_t^{\infty,n} + \frk l_R^n$ are bounded away from 0 for $t < \sigma_r^{\infty,n}$). 
This Radon-Nikodym derivative is bounded above by a deterministic constant depending only on $r$, and $Z^\infty$ a.s.\ does not have a jump at time $t$. From this, we infer that for each $t\geq 0$ and each $r>0$, the law of $\wt Z|_{[0,t]}$ restricted to the event $\{t < \wt\sigma_r\}$ is absolutely continuous with respect to the law of $Z^\infty|_{[0,t]}$, with Radon-Nikodym derivative
\eqbn
\left( \frac{L_t^{\infty } + R_t^{\infty }}{\frk l_L  + \frk l_R }  + 1    \right)^{-5/2} \BB 1_{(t < \sigma_r^{\infty })}. 
\eqen
Since $\{t < \wt\sigma_0\} = \bigcap_{r > 0} \{t < \wt\sigma_r\}$, sending $r\rta 0$ shows that the same is true with $r = 0$. 
\end{proof}

We next establish some basic properties of the times $\wt\sigma_r$ from~\eqref{eqn-continuum-hit-time}. 

\begin{lem} \label{lem-endpoint-cont}
The times $\wt\sigma_r$ satisfy the following properties.
\begin{enumerate}
\item \label{item-endpoint-law} The times $\wt\sigma_0$ and $\sigma_0$ have the same law. 
\item \label{item-endpoint-const} Almost surely, $\wt Z_t = (-\frk l_L , -\frk l_R)$ for each $t\geq \wt\sigma_0$.  
\item \label{item-endpoint-strict} Almost surely, $\wt\sigma_r < \wt\sigma_0$ for each $r>  0$.  
\item \label{item-endpoint-lim} Almost surely, $\lim_{r\rta 0} \wt\sigma_r = \wt\sigma_0$. 
\end{enumerate}
Assertions~\ref{item-endpoint-const} through~\ref{item-endpoint-lim} are also true with $Z$ in place of $\wt Z$. 
\end{lem} 
\begin{proof}
By Lemma~\ref{lem-perc-limit-rn} and Theorem~\ref{thm-bdy-process-law}, for each $t\geq 0$, 
\eqbn
\BB P\left[t <  \wt\sigma_0 \right] = \BB E\left[ \left( \frac{L_t^{\infty } + R_t^{\infty }}{\frk l_L  + \frk l_R }  + 1    \right)^{-5/2} \BB 1_{(t < \sigma_0^{\infty })} \right] = \BB P\left[ t < \sigma_0 \right] 
\eqen
which yields assertion~\ref{item-endpoint-law}. 

To prove the other assertions, fix $\ep \in (0,1)$. By Lemma~\ref{lem-short-bdy-time} and Lemma~\ref{lem-max-bdy-length} (applied to the unexplored quadrangulation at time $I_{\alpha_0,\alpha_1}^n$), for each $r>0$ there exists $0 < \alpha_0 < \alpha_1 \leq r$ and $A > 0$ such that with $\tau_{\alpha_0,\alpha_1}^n$ as in~\eqref{eqn-discrete-close-time-rescale}, it holds for large enough $n\in\BB N$ that
\eqbn
\BB P\left[ \tau_{\alpha_0,\alpha_1}^n \leq A \:\op{and} \: \sup_{t\geq \tau_{\alpha_0,\alpha_1}^n} |Z_t^n - Z_{\tau_{\alpha_0,\alpha_1}^n}^n | \leq \ep \right] \geq 1-\ep .
\eqen
Passing to the scaling limit along the subsequence $\mcl N$ and recalling~\eqref{eqn-bdy-process-inf} shows that with $\wt\tau_{\alpha_0,\alpha_1}$ the smallest $t \geq 0$ for which
\alb
 \frk l_L  + \inf_{s\in [0,t]} \wt L_s \geq \alpha_0 ,\quad \wt L_t +\frk l_L \leq \alpha_1 , \quad
  \frk l_R  + \inf_{s\in [0,t]} \wt R_s \geq\alpha_0 ,\quad \op{and} \quad \wt R_t +\frk l_R \leq \alpha_1 ,
\ale 
it holds that
\eqbn
\BB P\left[ \wt\tau_{\alpha_0,\alpha_1}  < \infty\:\op{and} \: \sup_{t\geq \wt\tau_{\alpha_0,\alpha_1} } |\wt Z_t  - \wt Z_{\wt\tau_{\alpha_0,\alpha_1}} | \leq \ep      \right] \geq 1-\ep .
\eqen
Since $\alpha_1 \leq r$ and $\alpha_0 > 0$, we have $\sigma_r \leq  \wt\tau_{\alpha_0,\alpha_1} < \wt\sigma_0$ so
\eqbn
\BB P\left[ \wt\sigma_r  < \wt\sigma_0 ,\, \sup_{t\geq \wt\sigma_0 } |\wt Z_t  - \wt Z_{\wt\tau_{\alpha_0,\alpha_1}} | \leq \ep \right]  \geq 1-\ep .
\eqen
Since $r>0$ and $\ep\in (0,1)$ can be made arbitrarily small, we obtain assertions~\ref{item-endpoint-const} and~\ref{item-endpoint-strict}.  

To prove assertion~\ref{item-endpoint-lim}, let $\wt\sigma_0' := \lim_{r\rta 0} \wt\sigma_r$. By monotonicity $\wt\sigma_0'$ exists and is at most $ \wt\sigma_0$. On the other hand, either $\inf_{t\in [0,\wt\sigma_0']} \wt L_t = -\frk l_L$ or $\inf_{t\in [0,\wt\sigma_0']} \wt R_t = -\frk l_R$. Since $\wt L$ and $\wt R$ have no upward jumps, each of these functions attains its minimum on $[0,\wt\sigma_0']$ whence $\wt\sigma_0' = \wt\sigma_0$. 

Assertion~\ref{item-endpoint-const} with $Z$ in place of $\wt Z$ is true by definition, and assertion~\ref{item-endpoint-strict} with $Z$ in place of $\wt Z$ follows from Theorem~\ref{thm-bdy-process-law}. Assertion~\ref{item-endpoint-lim} with $Z$ in place of $\wt Z$ follows from the same argument as in the case of $\wt Z$. 
\end{proof}

We now transfer the Radon-Nikodym derivative formula from Lemma~\ref{lem-perc-limit-rn} from deterministic times to the stopping times of~\eqref{eqn-continuum-hit-time}.

\begin{lem} \label{lem-perc-limit-rn-stopping}
For each $r > 0$, the law of $\wt Z|_{[0,\wt\sigma_r]}$ is absolutely continuous with respect to the law of $Z^\infty|_{[0,\sigma_r^\infty]}$, with Radon-Nikodym derivative
\eqb \label{eqn-perc-limit-rn-stopping}
\left( \frac{  L_{\sigma_r^\infty}^{\infty } + R_{\sigma_r^\infty}^{\infty}    }{\frk l_L  + \frk l_R }  + 1    \right)^{-5/2}  \BB 1_{(\sigma_r^\infty < \sigma_0^\infty)} . 
\eqe 
The same is true with $Z$ in place of $\wt Z$. 
\end{lem}
\begin{proof} 
For $k \in\BB N$, let $t_k = 2^{-k} \lceil 2^k \wt\sigma_r \rceil$ and $t_k^\infty = 2^{-k} \lceil 2^k \sigma_r^\infty \rceil$. Then each $t_k $ (resp.\ $t_k^\infty$) is a stopping time for $\wt Z$ (resp.\ $Z^\infty$) and $t_k$ (resp.\ $t_k^\infty$) a.s.\ decreases to $\wt\sigma_r$ (resp.\ $\sigma_r^\infty$) as $k\rta\infty$. By Lemma~\ref{lem-perc-limit-rn}, for each $k \in \BB N$ the law of $\wt Z|_{[0,t_k]}$ restricted to the event $\{t_k < \wt\sigma_0\}$ is absolutely continuous with respect to the law of $Z^\infty|_{[0,t_k^\infty ]}$, with Radon-Nikodym derivative
\eqb \label{eqn-stopping-approx-rn}
\left( \frac{L_{t_k^\infty}^{\infty } + R_{t_k^\infty}^{\infty }}{\frk l_L  + \frk l_R }  + 1    \right)^{-5/2} \BB 1_{(t_k^\infty < \sigma_0^{\infty })} . 
\eqe 

Now let $\ep  > 0$. On the event $\{t_k^\infty < \sigma_{\ep}^\infty  \}$, the Radon-Nikodym derivative~\eqref{eqn-stopping-approx-rn} is bounded above by a constant $c = c(\frk l_L , \frk l_R)>0$ times $\ep^{-5/2} $. By right continuity, a.s.\
\eqbn
\lim_{k\rta\infty} \left( \frac{L_{t_k^\infty}^{\infty } + R_{t_k^\infty}^{\infty }}{\frk l_L  + \frk l_R }  + 1    \right)^{-5/2} \BB 1_{(t_k^\infty < \sigma_{\ep}^\infty)}
=   \left( \frac{L_{\sigma_r^\infty}^{\infty } + R_{\sigma_r^\infty}^{\infty }}{\frk l_L  + \frk l_R }  + 1    \right)^{-5/2} \BB 1_{(\sigma_r^\infty < \sigma_{\ep}^\infty)} .
\eqen
By the discussion just above~\eqref{eqn-stopping-approx-rn} and the dominated convergence theorem, for each bounded, measurable, non-negative function $F$ on the state space for $\wt Z|_{[0,\wt\sigma_r]}$, 
\eqbn
\BB E\left[ F(\wt Z|_{[0,\wt\sigma_r]}) \BB 1_{(\wt\sigma_r  < \wt \sigma_{\ep})} \right] 
= \lim_{k\rta\infty} \BB E\left[ F(\wt Z|_{[0,\wt\sigma_r]}) \BB 1_{(t_k < \wt \sigma_{\ep})} \right]
= \BB E\left[F( Z^\infty|_{[0, \sigma_r^\infty]})   \left( \frac{L_{\sigma_r^\infty}^{\infty } + R_{\sigma_r^\infty}^{\infty }}{\frk l_L  + \frk l_R }  + 1    \right)^{-5/2} \BB 1_{(\sigma_r^\infty < \sigma_{\ep}^\infty )} \right].
\eqen
Sending $\ep \rta 0$ and applying the monotone convergence theorem shows that the law of $\wt Z|_{[0,\wt\sigma_r]}$ restricted to the event $\{\wt\sigma_r < \wt\sigma_0\}$ is absolutely continuous with respect to the law of $Z^\infty|_{[0,\sigma_r^\infty]}$, with Radon-Nikodym derivative as in~\eqref{eqn-perc-limit-rn-stopping}. By Lemma~\ref{lem-endpoint-cont} the event $\{\wt\sigma_r < \wt\sigma_0\}$ has probability 1. We thus obtain the statement of the lemma for $\wt Z$. 

The statement for $Z$ is proven in exactly the same manner. 
\end{proof}

 \begin{proof}[Proof of Proposition~\ref{prop-stable-conv-finite}]
By Lemma~\ref{lem-perc-limit-rn-stopping}, for each $r \geq 0$ the laws of $\wt Z|_{[0,\wt\sigma_r]}$ and $Z|_{[0,\sigma_r]}$ agree, where here $\wt\sigma_r$ and $\sigma_r$ are as in~\eqref{eqn-continuum-hit-time}. By assertion~\ref{item-endpoint-lim} of Lemma~\ref{lem-endpoint-cont}, $\wt Z|_{[0,\wt\sigma_0)}$ and $Z|_{[0,\sigma_0)}$ have the same law.
By Theorem~\ref{thm-bdy-process-law}, $\lim_{t \rta \sigma_0^-} Z_t = (-\frk l_L , -\frk l_R)$ and by assertion~\ref{item-endpoint-const} of Lemma~\ref{lem-endpoint-cont}, $\wt Z$ (resp.\ $Z$) is identically equal to $(-\frk l_L , -\frk l_R)$ after time $\wt\sigma_0$ (resp.\ $\sigma_0$). Hence $\wt Z \eqD Z$. Since our initial choice of subsequence was arbitrary we infer that $Z^n\rta Z$ in law.
\end{proof}

\section{Tightness in the GHPU topology}
\label{sec-ghpu-tight}

Throughout this section we assume we are in the setting of Theorem~\ref{thm-perc-conv}, so that $\frk l_L , \frk l_R > 0$ and $\{(\el_L^n ,\el_R^n)\}_{n\in\BB N}$ is a sequence of pairs of positive integers  such that $ \el_L^n + \el_R^n$ is always even, $ \bcon^{-1} n^{-1/2} \el_L^n  \rta\frk l_L$, and $    \bcon^{-1} n^{-1/2} \el_R^n   \rta \frk l_R$. 

For $n\in\BB N$ let $(Q^n ,\BB e^n , \theta^n)$ be a free Boltzmann quadrangulation with simple boundary with a critical face percolation configuration as in Theorem~\ref{thm-perc-conv}, and define the clusters $\{\dot Q_j^n\}_{j\in\BB N_0}$, the unexplored quadrangulations $\{\ol Q_j^n\}_{j\in\BB N_0}$, the peeled edges $\{\dot e_j^n\}_{j\in\BB N}$, the filtration $\{\mcl F_j^n\}_{j \in \BB N_0}$, and the terminal time $\mcl J^n$ for the percolation peeling process of $(Q^n,\BB e^n,\theta^n)$ with $\el_L^n$-white/$\el_R^n$-black boundary conditions as in Section~\ref{sec-perc-peeling}. 
Also define the boundary length processes $X^{L,n} , X^{R,n}   ,Y^{L,n} ,Y^{R,n}$, and $W^n = (W^{L,n} , W^{R,n})$ as in Definition~\ref{def-bdy-process} and the rescaled boundary length process $Z^n = (L^n,R^n)$ as in~\eqref{eqn-bdy-process-rescale}.

As in the discussion just above Theorem~\ref{thm-perc-conv}, let $\beta^n$ (resp.\ $\lambda^n$) be the boundary path (resp.\ percolation exploration path), and recall that $\lambda^n$ is defined on $\frac12 \BB N_0$. Also let $d^n$, $\mu^n$, $\xi^n$, and $\eta^n$, respectively, be the rescaled graph metric, area measure, boundary path, and percolation exploration path.

\begin{prop} \label{prop-ghpu-tight}
The laws of the doubly curve-decorated metric measure spaces $\frk Q^n = (Q^n , d^n , \mu^n ,  \xi^n , \eta^n)$ for $n\in\BB N$ are tight in the 2-curve GHPU topology.
\end{prop}

We already know from~\cite[Theorem~1.4]{gwynne-miller-simple-quad} that the laws of the curve-decorated metric measure spaces $\{(Q^n , d^n , \mu^n ,  \xi^n)\}_{n\in\BB N}$ are tight in the GHPU topology. By the 2-curve variant of the GHPU compactness criterion~\cite[Lemma~2.6]{gwynne-miller-uihpq} (which is proven in the same manner), we only need to check that the curves~$\eta^n$ are equicontinuous. For this purpose it suffices to prove the following proposition.

\begin{prop}
\label{prop-equicont}
For each $\ep \in (0,1)$, there exists $\delta > 0$ and $n_* \in \BB N$ such that for $n\geq n_*$, it holds with probability at least $1-\ep$ that  
\eqbn
\op{dist}\left( \lambda^n(i) ,\lambda^n(j) ; Q^n \right) \leq \ep n^{1/4} , \quad \forall i,j \in \frac12 \BB N_0 \: \op{with} \: |i-j| \leq \delta n^{3/4} 
\eqen
and with $\BB e_*^n = \beta^n(\el_R^n-1)$ the target edge,
\eqb \label{eqn-equicont-end}
\op{dist}\left( \lambda^n(i) ,  \BB e_*^n ; Q^n \right) \leq \ep n^{1/4} , \quad \forall i \in [ \delta^{-1} n^{3/4} , \infty)_{\frac12 \BB Z}  .
\eqe 
\end{prop}

We expect, but do not prove, that when $\delta$ is small it holds with high probability that the terminal time $\mcl J^n$ is smaller than $\delta^{-1} n^{3/4}$, which implies that in fact $\lambda^n(i) = \BB e_*^n$ for each $i \geq \delta^{-1} n^{3/4}$. The slightly weaker statement~\eqref{eqn-equicont-end} is sufficient for our purposes.

The remainder of this section will be devoted to the proof of Proposition~\ref{prop-equicont}. We start in Section~\ref{sec-length-diam} by proving several estimates which reduce the problem of estimating diameters of segments of $\lambda^n$ to the problem of estimating boundary lengths of certain sub-quadrangulations of $Q^n$. In particular, we prove in Lemma~\ref{lem-fb-bdy-holder} a bound for the diameter of a boundary arc of a free Boltzmann quadrangulation with simple boundary in terms of its length; and an estimate for the maximal diameter of a sub-graph of such a quadrangulation in terms of the diameter of its boundary. These estimates are easy consequences of the convergence of free Boltzmann quadrangulations to the Brownian disk. 

In Section~\ref{sec-uihpq-jump-estimate}, we will use the basic estimates for peeling which we reviewed in Section~\ref{sec-peeling-estimate} to prove estimates for the boundary length processes for the percolation peeling process on the UIHPQ$_{\op{S}}$. In particular, we will show that the maximum of the magnitude of the boundary length process $W^\infty = (W^{L,\infty} ,W^{R,\infty})$ over an interval can only be unusually large if either $W^{L,\infty}$ or $W^{R,\infty}$ has a big downward jump in this interval; and that there cannot be too many such big downward jumps. 
 
In Section~\ref{sec-equicont-proof}, we will transfer the estimates of Section~\ref{sec-uihpq-jump-estimate} to estimates for the boundary length process of the percolation peeling process on $Q^n$ using the Radon-Nikodym derivative estimate of Lemma~\ref{lem-perc-rn}, then deduce Proposition~\ref{prop-equicont} from these estimates together with the estimates of Section~\ref{sec-length-diam}. A more detailed outline of the argument of this subsection appears at the beginning of the subsection.

\subsection{Estimates for distances in terms of boundary length}
\label{sec-length-diam}

In this subsection we will prove some basic estimates for the graph distances in free Boltzmann quadrangulations with simple boundary which are straightforward consequences of the GHPU convergence of these quadrangulations to the Brownian disk~\cite[Theorem~1.4]{gwynne-miller-uihpq}.
We first prove a quantitative bound for distances along the boundary in a free Boltzmann quadrangulation with simple boundary, which will follow from the following estimate for the Brownian disk.

 \begin{lem} \label{lem-bd-bdy-holder}
Let $(H,d,\mu,\xi)$ be a free Boltzmann Brownian disk with unit boundary length equipped with its natural metric, area measure, and boundary path. 
For $\zeta  >0$ and $C>0$,
\eqb \label{eqn-bd-bdy-holder}
\BB P\left[ \sup_{ 0 \leq s < t \leq \frk l}  \frac{ d(\xi(s) , \xi(t)) }{   (t-s)^{1/2}(| \log(t-s)| + 1 )^{7/4 +\zeta}   } >  C  \right] = o_C^\infty(C)
\eqe 
as $C\rta\infty$, at a rate depending only on $\zeta$. 
\end{lem}
\begin{proof}
The analogous statement for a unit boundary length free Boltzmann Brownian disk weighted by its area (which is the random area Brownian disk in~\cite[Definition~3.1]{gwynne-miller-gluing}) follows from the proof of~\cite[Lemma~3.2]{gwynne-miller-gluing}.
That is, if we let $A  =\mu(H)$ be the total area of our given Brownian disk and we let $E_C$ be the event whose probability we are trying to bound in~\eqref{eqn-bd-bdy-holder}, then $\BB E[A \BB 1_{E_C}] = o_C^\infty(C)$. The law of $A$ is given by $\frac{1}{ \sqrt{2\pi a^5 } } e^{-\frac{1}{2 a} } \BB 1_{(a\geq 0)} \, da$, so $\BB E[A^{-1}] < \infty$.
By the Cauchy-Schwarz inequality,  
\eqbn
\BB P[E_C] = \BB E[A^{1/2} A^{-1/2} \BB 1_{E_C} ] \leq \BB E[A \BB 1_{E_C}]^{1/2} \BB E[A^{-1}]^{1/2} = o_C^\infty(C) . \qedhere
\eqen 
\end{proof}

\begin{lem} \label{lem-fb-bdy-holder}
Let $\el \in \BB N$ and let $(Q , \BB e)$ be a free Boltzmann quadrangulation with simple boundary of perimeter~$2\el$. Let $\beta : [0,2\el]_{\BB Z} \rta\mcl E(\bdy Q )$ be its boundary path.  For each $\zeta , \ep  >0$ and $C>0$, the probability that there exists $i , j \in [0,2\el]_{\BB Z}$ with $i < j$ such that 
\alb
\el^{-1/2} \op{dist} \left( \beta(i) , \beta (j) ; Q   \right)  >  
C  \left(\frac{j-i}{\el } \right)^{1/2}\left( \left| \log \left(\frac{j-i}{\el } \right)  \right| + 1 \right)^{7/4 +\zeta}    + \ep   
\ale
is at most $o_C^\infty(C) + o_\el(1)$ as $C\rta\infty$, with the rate of the $o_C^\infty(C)$ depending only on $\zeta$ and the rate of the $o_\el(1)$ depending only on $\zeta$ and $\ep$. 
\end{lem}
\begin{proof}
This follows from Lemma~\ref{lem-bd-bdy-holder} and the fact that free Boltzmann quadrangulations with simple boundary converge in the scaling limit to the Brownian disk in the GHPU topology~\cite[Theorem~1.4]{gwynne-miller-simple-quad}. 
\end{proof}

We also record analogs of the two preceding lemmas in the infinite-volume setting.

\begin{lem} \label{lem-bhp-bdy-holder}
Let $(H^\infty , d^\infty ,\mu^\infty,\xi^\infty)$ be a Brownian half-plane equipped with its natural metric, area measure, and boundary path (with $\xi^\infty(0)$ the marked boundary point). 
For $\zeta  >0$, $A>0$, and $C>0$,
\eqbn
\BB P\left[ \sup_{ -A \leq s < t \leq A } \frac{ d^\infty(\xi^\infty(s) , \xi^\infty(t))}{   (t-s)^{1/2}(| \log(t-s)| + 1 )^{7/4 +\zeta}   } >  C A^{1/2} \right] = o_C^\infty(C)
\eqen
as $C\rta\infty$, at a rate depending only on $\zeta$. 
\end{lem}
\begin{proof}
This follows from the same argument used to prove~\cite[Lemma~3.2]{gwynne-miller-gluing}, but with the encoding functions for the Brownian half-plane from~\cite[Section 1.5]{gwynne-miller-uihpq} used in place of the encoding functions for the Brownian disk. 
Note that the factor of $A^{1/2}$ comes from Brownian scaling. 
\end{proof}

\begin{lem} \label{lem-uihpq-bdy-holder}
Let $(Q^\infty,\BB e^\infty)$ be a UIHPQ$_{\op{S}}$. Let $\beta^\infty : \BB Z\rta\mcl E(\bdy Q^\infty)$ be its boundary path with $\beta^\infty(0) = \BB e^\infty$. 
For each $\zeta , \ep  , A  > 0$ and each $C>0$, the probability that there exists $i , j \in [-A \el , A \el]_{\BB Z}$ with $i < j$ such that 
\alb
 \el^{-1/2} \op{dist} \left( \beta^\infty (i) , \beta^\infty(j) ; Q^\infty   \right)  >  
C A^{1/2} \left(\frac{j-i}{\el } \right)^{1/2}\left( \left| \log \left(\frac{j-i}{\el } \right)  \right| + 1 \right)^{7/4 +\zeta}    + \ep   
\ale
is at most $o_C^\infty(C) + o_\el(1)$ as $C\rta\infty$, with the rate of the $o_C^\infty(C)$ depending only on $\zeta$ and the rate of the $o_\el(1)$ depending only on $A$, $\zeta$, and $\ep$. 
\end{lem}
\begin{proof}
This follows by combining Lemma~\ref{lem-bhp-bdy-holder} with the scaling limit result for the UIHPQ$_{\op{S}}$ in the local GHPU topology~\cite[Theorem~1.12]{gwynne-miller-uihpq}. 
\end{proof}

Lemma~\ref{lem-fb-bdy-holder} together with the Markov property of peeling will eventually enable us to prove estimates for the diameters of the boundaries of certain subsets of the quadrangulations $Q^n$. In order to deduce estimates for the diameters of the sets themselves, we will use Lemma~\ref{lem-fb-pinch} below, which says that a free Boltzmann quadrangulation with simple boundary does not have small bottlenecks which separate sets of macroscopic diameter and which follows from the fact that the Brownian disk has the topology of a disk.

\begin{lem} \label{lem-fb-pinch}
For each $\ep   \in (0,1)$ there exists $\delta >0$ such that the following is true. 
Let $\el\in\BB N$ and let $(Q , \BB e)$ be a free Boltzmann quadrangulation with simple boundary of perimeter $2\el$.
The probability that there exists a subgraph $S$ of $Q$ with 
\eqbn
\op{diam}\left(   S ; Q \right)  \geq \ep \el^{1/2} \quad \op{and} \quad 
\op{diam}\left( \bdy S ; Q \right) \leq \delta \el^{1/2} 
\eqen
is at most $ \ep$, where here $\bdy S$ is the boundary of $S$ relative to $Q$, as in Section~\ref{sec-graph-notation}.  
\end{lem}
\begin{proof}
We will extract the statement of the lemma from the fact that the free Boltzmann quadrangulation with simple boundary converges in the scaling limit to the random-area Brownian disk, which has the topology of a disk. The proof is similar to that of~\cite[Lemma~4.10]{gwynne-miller-uihpq}. 

For $\el\in\BB N$, let $(Q^\el , \BB e^\el)$ be a free Boltzmann quadrangulation with simple boundary of perimeter $2\el$.  Let $d^\el$ be the graph metric on $Q^\el$, rescaled by $(2\el )^{-1/2}$. Also let $(H,d)$ be a free Boltzmann Brownian disk with unit boundary length. 

By~\cite[Theorem~1.4]{gwynne-miller-simple-quad}, $(Q^\el,d^\el) \rta (H,d)$ in law in the Gromov-Hausdorff topology.  By the Skorokhod representation theorem, we can find a coupling of $\{(Q^\el,d^\el)\}_{\el \in \BB N}$ with $(H,d)$ such that this convergence occurs almost surely.  By~\cite[Lemma~A.1]{gpw-metric-measure}, we can a.s.\ find a random compact metric space $(W,D)$ and isometric embeddings $(Q^\el , d^\el) \rta (W,D)$ and $(H,d)\rta (W,D)$ such that if we identify $Q^\el$ and $H$ with their images under these embeddings, then a.s.\ $Q^\el \rta H$ in the $D$-Hausdorff distance as $\el\rta\infty$.

Now, suppose by way of contradiction that the statement of the lemma is false. Then we can find $\ep > 0$ and a sequence $\el_k\rta\infty$ such that for each $k\in\BB N$, it holds with probability at least $\ep$ that there exists a subgraph $S^{\el_k} \subset Q^{\el_k}$ with $\op{diam}(S^{\el_k} , d^{\el_k}) \geq \ep$ and $\op{diam} (\bdy S^{\el_k} , d^{\el_k}) \leq 1/k$. Let $E^{\el_k}$ be the event that this is the case, so that $\BB P[E^{\el_k}] \geq \ep$.  Also let $E$ be the event that $E^{\el_k}$ occurs for infinitely many $k\in\BB N$, so that also $\BB P[E] \geq \ep$. 

On $E$, we can find a random sequence $\mcl K$ of positive integers tending to $\infty$ such that $E^{\el_k}$ occurs for each $k\in\mcl K$.  For $k\in\mcl K$, let $S^{\el_k} \subset Q^{\el_k}$ be as in the definition of $E^{\el_k}$ and choose $y^{\el_k} \in \bdy S^{\el_k}$.  It is clear that a.s.\ $\liminf_{\el \rta\infty} \op{diam}( \bdy Q^\el ; d^{\el }) > 0$, so there a.s.\ exists $\zeta >0$ such that for large enough $k \in \mcl K$, the $d^{\el_k}$-diameter of $Q^{\el_k} \setminus S^{\el_k}$ is at least $\zeta$. 

For $\delta \in (0, (\ep \wedge \zeta)/100 )$, define 
\eqbn
V_\delta^{\el_k} := S^{\el_k} \setminus B_{4\delta}(y^{\el_k}; d^{\el_k}) \quad \op{and} \quad
U_\delta^{\el_k} := Q^{\el_k} \setminus \left( S^{\el_k} \cup B_{4\delta}(y^{\el_k}; d^{\el_k})  \right) .
\eqen
By definition of $E^{\el_k}$, for large enough $k\in \mcl K$ the set $V_\delta^{\el_k}$ (resp.\ $U_\delta^{\el_k}$) has $d^{\el_k}$-diameter at least $\ep/2$ (resp.\ $\zeta/2$).
Furthermore, since $\op{diam}(\bdy S^{\el_k} , d^{\el_k} ) \leq 1/k$, it follows that for large enough $k\in\mcl K$ the sets $\bdy S^{\el_k} \subset B_{2\delta}(y^{\el_k} ; d^{\el_k})$ so the sets $V_\delta^{\el_k}$ and $U_\delta^{\el_k}$ lie at $d^{\el_k}$-distance at least $\delta$ from each other. 

By possibly passing to a further subsequence, we can find $y\in H$ closed sets $U_\delta , V_\delta \subset H$ for each rational $\delta \in (0, (\ep\wedge \zeta) /100)$ such that as $\mcl K\ni k \rta\infty$, a.s.\ $y^{\el_k} \rta y$ and $U_\delta^{\el_k} \rta U_\delta$ and $V_\delta^{ \el_k} \rta V_\delta$ in the $D$-Hausdorff metric. 
Then $U_\delta$ and $V_\delta$ lie at $d$-distance at least $\delta$ from each other and have $d$-diameters at least $\ep/2$ and $\zeta/2$, respectively. 
Furthermore, we have $H = U_\delta\cup V_\delta \cup B_{4\delta}(y ; d)$. Sending $\delta\rta 0$, we see that removing $y$ from $H$ disconnects $H$ into two components. But, $H$ a.s.\ has the topology of a disk~\cite{bet-disk-tight}, so we obtain a contradiction. 
\end{proof}

\subsection{Jumps of the UIHPQ$_{\mathrm{S}}$ boundary length process}
\label{sec-uihpq-jump-estimate}

In this subsection we consider the boundary length processes for the percolation peeling process on the UIHPQ$_{\op{S}}$ $(Q^\infty,\BB e^\infty)$, which we recall are defined in Definition~\ref{def-bdy-process} and denoted by a superscript $\infty$. Our main goal is to prove Lemma~\ref{lem-equicont-reg-event-infty} just below, which gives a regularity statement for the macroscopic downward jumps of the total net boundary length process $ W^{\infty,L} + W^{\infty,R} $. Roughly speaking, the lemma tells us that with very high probability there are at most $\delta^{o_\delta(1)}$ jumps of size at least $\delta^{2/3} n^{1/2}$ in any time interval of length $\delta n^{3/4}$ and $|W^{\infty,L}|$ fluctuates by at most $\delta^{2/3 - o_\delta(1)} n^{1/2}$ between the times of these jumps. 
This lemma is the only statement from this subsection which is needed in the proof of Proposition~\ref{prop-equicont}, and will be transferred to the setting of free Boltzmann quadrangulations in the next subsection.

\begin{lem} \label{lem-equicont-reg-event-infty}
There is a universal constant $c>0$ such that the following is true. 
For $n\in\BB N$, $\delta \in (0,1)$, and $k \in \BB N $, let $T^{\infty,n}_{k,0}(\delta) = \lfloor (k-1) \delta n^{3/4} \rfloor$ and for $r\in \BB N$ inductively define
\eqb \label{eqn-equicont-reg-times-infty}
T^{\infty,n}_{k,r}(\delta)  := \lfloor k \delta n^{3/4} \rfloor \wedge \inf\left\{ j \geq T_{k, r-1}^{\infty,n}(\delta)+1 : W_j^{\infty,L} + W_j^{\infty,R} -  W_{j-1}^{\infty,L} - W_{j-1}^{\infty,R} \leq - c \delta^{2/3} n^{1/2} \right\} , 
\eqe  
so that $T^{\infty,n}_{k,r}(\delta)$ is the $r$th smallest time in $[\lfloor (k-1) \delta n^{3/4} \rfloor , \lfloor k \delta n^{3/4} \rfloor ]_{\BB Z}$ at which the total boundary length process has a big downward jump, or $T^{\infty,n}_{k,r}(\delta) = \lfloor k \delta n^{3/4} \rfloor$ if there are fewer than $r$ such jumps.  

For $A>0$ and $\zeta \in (0,1)$, let $E^{\infty,n}(\delta) = E^{\infty,n}(\delta,A,c,\zeta)$ be the event that the following holds.
\begin{enumerate}
\item For each $k \in [1, 2A \delta^{-1}   ]_{\BB Z}$, the number of jump times in $[\lfloor (k-1) \delta n^{3/4} \rfloor , \lfloor k \delta n^{3/4} \rfloor ]_{\BB Z}$ satisfies $\# \{ r \in \BB N  : T^{\infty,n}_{k,r}(\delta) <  \lfloor k \delta n^{3/4} \rfloor   \} \leq \delta^{-\zeta}$. \label{item-equicont-reg-times}
\item For each $k\in [1,2A \delta^{-1}]_{\BB Z}$ and each $r\in\BB N$, 
\eqbn
\max_{j \in [T_{k,r-1}^{\infty,n}(\delta) , T_{k ,r }^{\infty,n}(\delta) -1 ]_{\BB Z}} |W_j^\infty - W_{T_{k,r-1}^{\infty,n}(\delta)}^\infty |    \leq \delta^{2/3-\zeta}  n^{1/2} .
\eqen \label{item-equicont-reg-sup}
\end{enumerate}
Then
\eqbn
\BB P\left[ E^{\infty,n}(\delta)  \right] \geq 1 - o_\delta^\infty(\delta) 
\eqen 
at a rate depending only on $A$ and $\zeta$. 
\end{lem}

To prove Lemma~\ref{lem-equicont-reg-event-infty}, we will need several further lemmas which are each straightforward consequences of the fact that $W^{\infty,L}+W^{\infty,R}$ has independent, stationary increments (by~\eqref{eqn-peel-inc-cases} and the Markov property of peeling) and the tail asymptotics~\eqref{eqn-cover-tail} for the law of these increments.
Our first lemma gives a tail bound for the total number of covered edges of $Q^\infty$ before the first large downward jump of $W^{\infty,L}+W^{\infty,R}$. 
  
\begin{lem} \label{lem-before-jump} 
For $r \in \BB N$, let $T^\infty(r)$ be the smallest $j \in\BB N$ for which $ W_j^{\infty,L} + W_j^{\infty,R} -  W_{j-1}^{\infty,L} - W_{j-1}^{\infty,R} \leq  -r $. 
There is a universal constant $a_0 >0$ such that for $C \geq 2$, the number of covered edges satisfies
\eqbn
\BB P\left[ Y_{T^\infty(r)-1}^{\infty,L} + Y_{T^\infty(r)-1}^{\infty,R} > C r \right] \leq e^{-a_0 C}. 
\eqen
\end{lem}
\begin{proof} 
Let $I_0 = 0$ and for $k\in\BB N$ inductively let $I_k$ be the smallest $j \geq I_{k-1}+1$ for which
\eqbn
Y_j^{\infty,L} + Y_j^{\infty,R} - Y_{I_{k-1}}^{\infty,L} - Y_{I_{k-1}}^{\infty,R} \geq r .
\eqen  
By~\eqref{eqn-peel-inc-cases} and the strong Markov property, the walk increments $(W^\infty-W^\infty_{I_{k-1}})|_{[I_{k-1}+1 , I_k]_{\BB Z}}$ for $k\in\BB N$ are i.i.d.  

Let $E_k$ be the event that there is a $j\in [I_{k-1} +1 , I_k]_{\BB Z}$ for which $W_j^{\infty,L} + W_j^{\infty,R} -  W_{j-1}^{\infty,L} - W_{j-1}^{\infty,R} \leq  -r$ and let $K$ be the smallest $k\in\BB N$ for which $E_k$ occurs. 
By~\eqref{eqn-bdy-process-inf}, Proposition~\ref{prop-stable-conv}, and the independence of the increments $(W^{\infty,L} ,W^{\infty,R})|_{[I_{k-1}+1 , I_k]_{\BB Z}}$ we infer that there is a universal constant $p  \in (0,1)$ such that $\BB P[E_k \,|\, \mcl F_{I_{k-1}}^\infty ] \geq p$ for each $k\in \BB N$. 
Consequently, $K$ is stochastically dominated by a geometric random variable with success probability $p$. 

We have $ T^\infty(r) \in [I_{K-1} , I_K]$. Since $Y_{I_k}^{\infty,L} + Y_{I_k}^{\infty,R} - Y_{I_{k-1}}^{\infty,L} - Y_{I_{k-1}}^{\infty,R} \leq 2r$ for $k \leq K-1$, we have $Y_{T^\infty(r)-1}^{\infty,L} + Y_{T^\infty(r)-1}^{\infty,R} \leq 2r K$.  
By combining this with the preceding paragraph,  
\eqbn
\BB P\left[ Y_{T^\infty(r)-1}^{\infty,L} + Y_{T^\infty(r)-1}^{\infty,R} > C r \right] \leq \BB P\left[ K >   C/2 \right] \leq (1-p)^{\lfloor C/2\rfloor}
\eqen 
which yields the statement of the lemma.
\end{proof}

We next bound the total number of large downward jumps of $W^{\infty,L} +W^{\infty,R}$ in a given interval.

\begin{lem} \label{lem-big-jump-count}
For $c > 0$ and $m \in \BB N$, let $N_m^\infty(c)$ be the number of $j \in [1,m]_{\BB Z}$ for which 
\eqb \label{eqn-big-jump-def}
W_j^{\infty,L} + W_j^{\infty,R} -  W_{j-1}^{\infty,L} - W_{j-1}^{\infty,R} \leq -c m^{2/3}  .
\eqe 
There is a universal constant $a_1  > 0$ such that for $k \in \BB N$, 
\eqb \label{eqn-big-jump-count}
\BB P\left[ N_m^\infty(c) \geq k \right] \leq  (a_1 c^{-3/2})^k   .
\eqe  
\end{lem}
\begin{proof} 
Let $T^\infty_0 = 0$ and for $k\in\BB N$, let $T^\infty_k = T^\infty_k(c m^{2/3})$ be the $k$th smallest $j \in\BB N$ for which~\eqref{eqn-big-jump-def} holds.
By~\eqref{eqn-peel-inc-cases} and the Markov property of peeling, the increments $T^\infty_k-T^\infty_{k-1}$ for $k\in\BB N$ are i.i.d. 
By~\eqref{eqn-cover-tail}, 
\eqbn
\BB P\left[ T^\infty_k -T^\infty_{k-1} \leq m   \right] 
\leq \sum_{j=1}^m \BB P\left[ W_j^{\infty,L} + W_j^{\infty,R} -  W_{j-1}^{\infty,L} - W_{j-1}^{\infty,R} \leq -c m^{2/3} \right] 
\leq  a_1 \sum_{j=1}^m (cm^{2/3})^{-3/2}
\leq a_1 c^{-3/2}  
\eqen
for $a_1 > 0$ a universal constant. Therefore,
\eqbn
\BB P\left[ N_m^\infty(c) \geq k  \right] = \BB P\left[ T^\infty_k \leq m \right] \leq \BB P\left[ T^\infty_r - T^\infty_{r-1} \leq m,\: \forall r \in [1,k]_{\BB Z} \right] \leq (a_1 c^{-3/2})^k . \qedhere
\eqen
\end{proof}

Next we bound the maximum magnitude of the two-dimensional boundary length process $W^\infty$ before the time of the first large downward jump.

\begin{lem} \label{lem-bdy-bound-infty}
For $r \in \BB N$, let $T^\infty(r)$ be the smallest $j \in\BB N$ for which $ W_j^{\infty,L} + W_j^{\infty,R} -  W_{j-1}^{\infty,L} - W_{j-1}^{\infty,R} \leq  -r $, as in Lemma~\ref{lem-before-jump}. 
There are universal constants $b_0,b_1 > 0$ such that for each $r  ,m  \in \BB N$ and each $C> 0$,  
\eqbn
\BB P\left[ \max_{j\in [1,(T^\infty(r)-1) \wedge m]_{\BB Z}} |W_j^\infty |    > C (r \vee m^{2/3}) \right] \leq b_0 e^{-b_1 C}. 
\eqen 
\end{lem}
\begin{proof} 
By Definition~\ref{def-bdy-process}, for $j \in\BB N$ we have 
\eqb \label{eqn-bdy-bound-split}
| W_j^\infty | \leq  | W_j^{\infty,L} | + | W_j^{\infty, R} | \leq W_j^{\infty,L} + W_j^{\infty,R} + 2\left(  Y_j^{\infty,L} + Y_j^{\infty,R} \right) .
\eqe  
By~\eqref{eqn-peel-inc-cases} and the Markov property of peeling, the increments $W_j^{\infty,L} + W_j^{\infty,R} -  W_{j-1}^{\infty,L} - W_{j-1}^{\infty,R}$ are i.i.d. By the analog of~\eqref{eqn-cover-tail} for the total number of covered edges, the probability that one of these increments is smaller than $-k$ is $\sim k^{-3/2}$. 
By a straightforward estimate for heavy-tailed walks with no upward jumps (see, e.g.,~\cite[Lemma~5.8]{gwynne-miller-saw}), there exist universal constants $b_0' , b_1' > 0$ such that for each $m\in\BB N$ and each $C>0$, 
\eqb \label{eqn-peel-mart-upper}
\BB P\left[ \max_{j \in [0,m]_{\BB Z}} (W_j^{\infty,L} + W_j^{\infty,R})  > C m^{2/3} \right] \leq b_0' e^{-b_1' C}.
\eqe  
By combining this with~\eqref{eqn-bdy-bound-split} and Lemma~\ref{lem-before-jump} and recalling that $j \mapsto Y_j^{\infty,L}$ and $j\mapsto Y_j^{\infty,R}$ are monotone non-decreasing,  
\alb
&\BB P\left[\max_{j\in [1,(T^\infty(r)-1) \wedge m]_{\BB Z}} |W_j^\infty | > C (r \vee m^{2/3}) \right]\\
&\qquad \leq \BB P\left[  \max_{j\in[1,m]_{\BB Z} } \left(  W_j^{\infty,L} + W_j^{\infty,R} \right) + 2(Y_{(T^\infty(r)-1) \wedge m}^{\infty,L} + Y_{(T^\infty(r)-1) \wedge m}^{\infty,R})  > C (r \vee m^{2/3})  \right]  \\
&\qquad \leq \BB P\left[ \max_{j \in [1,m]_{\BB Z}} ( W_j^{\infty,L} + W_j^{\infty,R})  >  \frac12 C m^{2/3} \right] + \BB P\left[ Y_{(T^\infty(r)-1) \wedge m}^{\infty,L} + Y_{(T^\infty(r)-1) \wedge m}^{\infty,R} > \frac14 C r \right] \leq b_0 e^{-b_1 C}
\ale
for appropriate $b_0,b_1>0$ as in the statement of the lemma. 
\end{proof}

\begin{proof}[Proof of Lemma~\ref{lem-equicont-reg-event-infty}]
By Lemma~\ref{lem-big-jump-count} (applied with $m= \lfloor \delta n^{3/4} \rfloor$) and since $W$ has stationary increments, for each fixed $k\in \BB N$ one has
\eqbn
\BB P\left[  \# \{ r \in \BB N  : T^{\infty,n}_{k,r}(\delta) <  \lfloor k \delta n^{3/4} \rfloor   \} >  \delta^{-\zeta}  \right] \leq (a_1 c^{-3/2})^{\delta^{-\zeta}}
\eqen
for $a_1 > 0$ a universal constant. Choose $c > 0$ for which $a_1 c^{-3/2} < 1/2$. Then this last probability is of order $o_\delta^\infty(\delta)$ for each fixed $k \in \BB N$ so by a union bound the probability that condition~\ref{item-equicont-reg-times} in the definition of $E^n(\delta)$ fails to occur is at most $o_\delta^\infty(\delta)$. 

By Lemma~\ref{lem-bdy-bound-infty} (applied with $m = \lfloor \delta n^{3/4} \rfloor$, $r = \lfloor c  \delta^{2/3} n^{1/2} \rfloor$, and $C = \delta^{-\zeta}$) and the strong Markov property, for each $k , r \in \BB N$,
\eqb \label{eqn-equicont-reg-sup-estimate}
\BB P\left[ \max_{j \in [T_{k,r-1}^{\infty,n}(\delta) , T_{k ,r }^{\infty,n}(\delta) -1 ]_{\BB Z}}  |W_j^\infty - W_{T_{k,r-1}^{\infty,n}(\delta)}^\infty |   \right] = o_\delta^\infty(\delta)  .
\eqe 
If condition~\ref{item-equicont-reg-times} in the definition of $E^n(\delta)$ occurs, then there are at most $2 A^{-1} \delta^{-1 -\zeta}$ pairs $(k,r) \in \BB N^2$ for which $k \leq 2 A^{-1}$ and $T_{k,r-1}^{n,\infty} \not= T_{k,r}^{n,\infty}$. We conclude by applying~\eqref{eqn-equicont-reg-sup-estimate}, taking a union bound over all such pairs $(k,r)$, and recalling the previous paragraph. 
\end{proof}

We end this subsection by recording the following straightforward consequence of the above estimates, which is not needed for the proof of tightness but which will be used in Section~\ref{sec-crossing}.
 
\begin{lem} \label{lem-X-tail}
For $m\in\BB N$ and $C>1$, the left/right outer boundary length processes satisfy
\eqbn
\BB P\left[ \max_{j\in [1,m]_{\BB Z}} (X_j^{\infty,L} \vee X_j^{\infty,R}) > C m^{2/3} \right] = o_C^\infty(C)
\eqen
uniformly over all $m\in\BB N$. 
\end{lem}
\begin{proof}
Let $a_1$ be as in Lemma~\ref{lem-big-jump-count} and fix $c>0$ such that $a_1 c^{-3/2} \leq 1/2$. 
As in the proof of Lemma~\ref{lem-big-jump-count}, let $T_0^\infty =0$ and for $k\in \BB N$ let $T_k^\infty$ be the $k$th smallest $j \in \BB N$ for which $W_j^{\infty,L} + W_j^{\infty,R} -  W_{j-1}^{\infty,L} - W_{j-1}^{\infty,R} \leq -c m^{2/3} $. 
Also let $K = N_m^\infty(c) + 1$ be the smallest $k\in \BB N_0$ for which $T_k^\infty \geq m$. 

By~\eqref{eqn-bdy-process-inf} and since $X_j^{\infty,L} - X_{j-1}^{\infty,L} \leq 2$ for each $j\in\BB N_0$, we infer that
\eqb \label{eqn-X-tail-decomp}
 \max_{j\in [1,m]_{\BB Z}} X_j^{\infty,L}   \leq  2 \sum_{k=1}^K \max_{j \in [T_{k-1}^\infty \wedge m , ( T_k^\infty-1 )\wedge m]_{\BB Z}} |W_j^{\infty,L} - W_{T_{k-1}^\infty \wedge m}^{\infty,L}|  + 8 K .
\eqe 
By Lemma~\ref{lem-big-jump-count}, $\BB P[K > \frac{1}{100} C^{1/2}] = o_C^\infty(C)$. 
By Lemma~\ref{lem-bdy-bound-infty} and the strong Markov property, for each $k\in\BB N_0$, 
\eqbn
\BB P\left[ \max_{j \in [T_{k-1}^\infty \wedge m , ( T_k^\infty-1 )\wedge m]_{\BB Z}} |W_j^{\infty,L} - W_{T_{k-1}^\infty \wedge m}^{\infty,L}| > C^{1/2} m^{2/3} \right] = o_C^\infty(C) .
\eqen
Taking a union bound over all $k \in \left[1, \frac{1}{100}C^{1/2}\right]_{\BB Z}$ and recalling~\eqref{eqn-X-tail-decomp} shows that $\BB P[\max_{j\in [1,m]_{\BB Z}} X_j^{\infty,L} > C m^{2/3} ] = o_C^\infty(C)$. We also have the analogous bound with $X^{\infty,R}$ in place of $X^{\infty,L}$.
\end{proof}

\subsection{Proof of Proposition~\ref{prop-equicont}}
\label{sec-equicont-proof}

In this subsection we conclude the proof of Proposition~\ref{prop-equicont} and thereby the proof of Proposition~\ref{prop-ghpu-tight}. 
Throughout this subsection, for $n\in\BB N$ and $0 < \alpha_0 < \alpha_1$ we define the stopping time $I_{\alpha_0,\alpha_1}^n$ as in~\eqref{eqn-discrete-close-time}. 
Our main aim is to prove the following statement, which will be combined with Lemma~\ref{lem-fb-pinch} to obtain Proposition~\ref{prop-equicont}.

\begin{prop}
\label{prop-bdy-equicont}  
Let $\ep \in (0,1)$ and $\zeta \in (0,1/3)$.  There exists $0 < \alpha_0 < \alpha_1$, $A>0$, and $\delta_* \in (0,1)$ depending on $\ep$ and $\zeta$, such that for $\delta \in (0,\delta_*]$ there exists $n_* = n_*(\delta,\ep,\zeta) \in \BB N$ such that for $n\geq n_*$ the following holds with probability at least $1-\ep$.
\begin{enumerate}
\item $I_{\alpha_0,\alpha_1}^n \leq  A n^{3/4}$ and the unexplored quadrangulation $\ol Q_{I_{\alpha_0,\alpha_1}^n}^n$ has internal graph distance diameter at most $\ep n^{1/4}$. \label{item-bdy-equicont-end} 
\item For each $j\in  [0,I_{\alpha_0,\alpha_1}^n-1]_{  \lfloor \delta n^{3/4} \rfloor \BB Z}$, the percolation path increment $\lambda^n([j , (j+ \delta n^{3/4}) \wedge I_{\alpha_0,\alpha_1}^n]_{\frac12\BB Z})$ is contained in a subgraph of $\ol Q_j^n$ whose boundary relative to $\ol Q_j^n$ (Section~\ref{sec-graph-notation}) contains $\dot e_j^n$ and has $\ol Q_j^n$-graph distance diameter at most $ \frac12 \delta^{1/3- \zeta} n^{1/4}$. \label{item-bdy-equicont}
\end{enumerate}
\end{prop}

\makeatletter
\newcommand{\mylabel}[2]{#2\def\@currentlabel{#2}\label{#1}}
\makeatother

\begin{remark}
\label{remark-weaker-cont}
Since internal $\ol Q_j^n$-graph distances are dominated by $Q^n$-graph distances and since each interval $[j_1,j_2]_{\frac12\BB Z}$ with $0\leq j_2-j_1 \leq \delta n^{3/4}$ and $j_1,j_2 \in [0,I_{\alpha_0 , \alpha_1}^n-1]_{\frac12\BB Z}$ is contained in $ [j , (j+ 2\delta n^{3/4}) \wedge I_{\alpha_0,\alpha_1}^n]_{\frac12\BB Z} $ for some $j\in  [0,I_{\alpha_0,\alpha_1}^n-1]_{  \lfloor \delta n^{3/4} \rfloor \BB Z}$, condition~\ref{item-bdy-equicont} in Proposition~\ref{prop-bdy-equicont} implies the following slightly weaker condition: 
\begin{enumerate}
\item[\mylabel{item-bdy-equicont'}{2'}] For each $j_1,j_2 \in [0,I_{\alpha_0,\alpha_1}^n-1]_{\frac12\BB Z}$ with $0\leq j_2-j_1 \leq \delta n^{3/4}$, the percolation path increment $\lambda^n([j_1,j_2]_{\frac12\BB Z})$ is contained in a subgraph of $Q^n$ whose boundary relative to $Q^n$ has $Q^n$-graph distance diameter at most $ \delta^{1/3- \zeta} n^{1/4}$. 
\end{enumerate}
When we apply Proposition~\ref{prop-bdy-equicont}, we will typically use condition~\ref{item-bdy-equicont'}. But, on one occasion (in Section~\ref{sec-crossing-proof}) we will need the stronger condition~\ref{item-bdy-equicont}. 
\end{remark}
 
Proposition~\ref{prop-bdy-equicont} only gives an upper bound for the diameter of the \emph{boundary} of a set containing each $\delta n^{3/4}$-length increment of $\lambda^n$, but this will be enough for our purposes due to Lemma~\ref{lem-fb-pinch}. 
 
The idea of the proof of Proposition~\ref{prop-bdy-equicont} is as follows. When $\alpha_1$ is small, the diameter of the final small unexplored quadrangulation $\ol Q_{ I_{\alpha_0,\alpha_1}^n}^n$ can be bounded using the scaling limit result for free Boltzmann quadrangulations with simple boundary (Lemma~\ref{lem-end-bdy-choice}), so we can restrict attention to $[0,I_{\alpha_0,\alpha_1}^n]_{\BB Z}$. Using Lemmas~\ref{lem-perc-rn} and~\ref{lem-equicont-reg-event-infty}, we show that with high probability when $\delta >0$ is small, no $\delta n^{3/4}$-length interval of time which is contained in $[0,I_{\alpha_0,\alpha_1}^n-1]_{\BB Z}$ contains more than $\delta^{o_\delta(1)}$ downward jumps of $W^{L,n} + W^{R,n}$ of size larger than a constant times $\delta^{2/3} n^{1/2}$, and the supremum of $|W^n|$ over the intervals of time between these downward jumps is at most of order $\delta^{2/3 + o_\delta(1)} n^{1/2}$ (Lemma~\ref{lem-equicont-reg-event}). 

This gives us an upper bound for the outer boundary lengths of the sub-quadrangulations of $Q^n$ discovered by the percolation peeling process between the times corresponding to the downward jumps of $W^{L,n}+W^{R,n}$ of size at least $\delta^{2/3} n^{1/2}$ which happen before time $I_{\alpha_0,\alpha_1}^n$. Combining this with the estimates of Section~\ref{sec-length-diam} and the Markov property of peeling gives an upper bound for the diameters of the boundaries of these sub-quadrangulations (Lemma~\ref{lem-quad-step-bdy}), which leads to the estimate of Proposition~\ref{prop-bdy-equicont}. 

See Figure~\ref{fig-ghpu-tight} for an illustration of the proof.

\begin{figure}[ht!]
 \begin{center}
\includegraphics[scale=1]{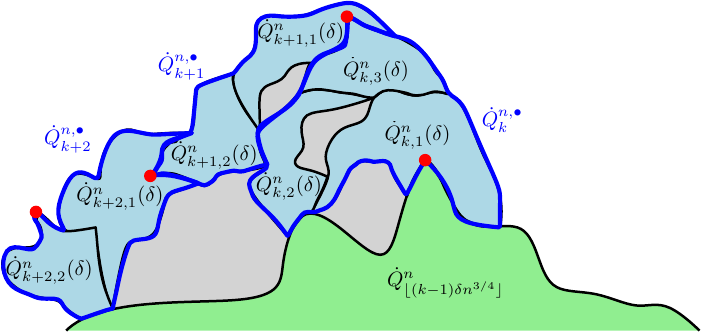} 
\caption[Illustration of the proof of Proposition~\ref{prop-bdy-equicont}]{\label{fig-ghpu-tight}Illustration of the proof of Proposition~\ref{prop-bdy-equicont}.  The quadrangulations $\dot Q_{k,r}^n(\delta)$ (light blue) correspond to increments of time contained in $[ \lfloor (k-1)\delta n^{3/4} \rfloor , \lfloor k \delta n^{3/4} \rfloor \wedge  I_{\alpha_0,\alpha_1}^n ]_{\BB Z}$ between the times $T_{k,r}^n(\delta)$ at which percolation peeling process cuts off a bubble with boundary length larger than $c\delta^{2/3} n^{1/2}$ (these bubbles are colored grey).  The boundary lengths of these quadrangulations and the maximal number of such quadrangulations for each value of $k$ are bounded in Lemma~\ref{lem-end-bdy-choice}.  This leads to an upper bound to the graph-distance diameters of their boundaries in Lemma~\ref{lem-quad-step-bdy}.  Each percolation path increment $\lambda^n([ \lfloor (k-1)\delta n^{3/4} \rfloor , \lfloor k \delta n^{3/4} \rfloor \wedge  I_{\alpha_0,\alpha_1}^n ]_{\frac12\BB Z} )$ is contained in the corresponding filled quadrangulation $\dot Q_k^{n,\bullet}$ (outlined in blue), which consists of the union of the quadrangulations $\dot Q_{k,r}^n$ over all possible values of $r$ and the set of vertices and edges which it disconnects from $\bdy Q^n$. We bound the graph-distance diameter of~$\bdy \dot Q_k^{n,\bullet}$ by summing over~$r$.}
\end{center}
\end{figure}
 
We now proceed with the details. Fix $\ep \in (0,1)$ and $\zeta \in (0,1/3)$. 
We start by making a suitable choice of~$\alpha_1$. 

\begin{lem}
\label{lem-end-bdy-choice}
There exists $\ol\alpha_1= \ol\alpha_1(\ep) > 0$ such that for each $\alpha_1\in (0,\ol\alpha_1]$, there exists $n_0 = n_0(\alpha_1 ,\ep)  > 0$ such that for each $\alpha_0 \in (0,\alpha_1)$ and each $n\geq n_0$, the internal diameter of the unexplored quadrangulation at time $I_{\alpha_0,\alpha_1}^n$ satisfies 
\eqbn
\BB P\left[ \op{diam}\left( \ol Q_{I_{\alpha_0,\alpha_1}^n}^n \right) \leq \ep n^{1/4} \,|\, \mcl F_{I_{\alpha_0,\alpha_1}^n}^n \right] \geq 1-\ep .
\eqen
\end{lem}
\begin{proof}
Since $I_{\alpha_0,\alpha_1}^n$ is a stopping time for the filtration $\{\mcl F_j^n\}_{j\in\BB N_0}$ of~\eqref{eqn-peel-filtration}, the conditional law given $\mcl F_{I_{\alpha_0,\alpha_1}^n}^n$ of the quadrangulation $ \ol Q_{I_{\alpha_0,\alpha_1}^n}^n  $ is that of a free Boltzmann quadrangulation with simple boundary and perimeter $W_{I_{\alpha_0,\alpha_1}^n }^{L,n} + W_{I_{\alpha_0,\alpha_1}^n}^{R,n} + \el_L^n + \el_R^n$, which by the definition~\eqref{eqn-discrete-close-time} is at most $2  \alpha_1  \bcon n^{1/2}$. The statement of the lemma now follows since a free Boltzmann quadrangulation with simple boundary converges in the scaling limit to the Brownian disk when we rescale distances by a factor proportional to the square root of the boundary length~\cite[Theorem~1.4]{gwynne-miller-simple-quad}. 
\end{proof}

Henceforth fix $\alpha_1 \in (0,\ol\alpha_1]$ and $n_0$ as in Lemma~\ref{lem-end-bdy-choice}. 
By Lemma~\ref{lem-short-bdy-time}, there exists $ \alpha_0 = \alpha_0(\alpha_1,\ep) \in (0,\alpha_1)$, $A = A(\alpha_1,\ep)>0$, and $n_1  = n_1(\alpha_1,\ep) \geq n_0$ such that for $n\geq n_1$, 
\eqb \label{eqn-use-short-time-equicont}
\BB P\left[ I_{\alpha_0,\alpha_1}^n \leq A n^{3/4} \right] \geq 1-\ep .
\eqe  
From now on fix such an $\alpha_0$, $A$, and $n_1$. 

We next transfer the regularity statement for the boundary length processes from Lemma~\ref{lem-equicont-reg-event-infty} to the setting of this subsection.  Note that the statement of the following lemma is essentially identical to that of Lemma~\ref{lem-equicont-reg-event-infty}, except that we work with finite quadrangulations, we replace $\zeta$ with $\zeta/2$, and we truncate at the time $I_{\alpha_0,\alpha_1}^n$. 
 
\begin{lem} \label{lem-equicont-reg-event}
Let $\ep,\zeta$, $\alpha_0$, $\alpha_1$, and $A$ be as above.
There is a universal constant $c>0$ such that the following is true. 
For $n\in\BB N$, $\delta \in (0,1)$, and $k\in\BB N$, let $T^n_{k,0}(\delta) = I_{\alpha_0,\alpha_1}^n \wedge \lfloor (k-1) \delta n^{3/4} \rfloor$ and for $r \in \BB N$ inductively define
\eqb \label{eqn-equicont-reg-times}
T^n_{k,r}(\delta)  := I_{\alpha_0,\alpha_1}^n \wedge \lfloor k \delta n^{3/4}  \rfloor \wedge \inf\left\{ j \geq T_{k,r-1}^n(\delta)+1 : W_j^{L,n} + W_j^{R,n} -  W_{j-1}^{L,n} - W_{j-1}^{R,n} \leq - c \delta^{2/3} n^{1/2} \right\} .
\eqe 
Let $E^{n}(\delta) = E^{n}(\delta,c,\alpha_0,\alpha_1,A,\zeta)$ be the event that the following hold.
\begin{enumerate}
\item $I_{\alpha_0,\alpha_1}^n \leq A n^{3/4}$. 
\item For each $k \in \BB N$, one has $\# \{ r \in \BB N  : T^{ n}_{k,r}(\delta) < I_{\alpha_0,\alpha_1}^n \wedge \lfloor k \delta n^{3/4} \rfloor   \} \leq \delta^{-\zeta/2}$.  \label{item-equicont-reg-times'}
\item For each $k\in \BB N $ and each $r \in \BB N$, 
\eqbn
\max_{j \in [T^n_{k,r-1}(\delta) , T^n_{k,r}(\delta)-1]_{\BB Z}} |W_j^n - W_{T^n_{k,r-1}(\delta)}^n| \leq \delta^{2/3-\zeta/2}  n^{1/2} .
\eqen 
\end{enumerate}
There is a $\delta_0 = \delta_0(\alpha_0,\alpha_1,\zeta,\ep) \in (0,1)$ and an $n_2 = n_2(\alpha_0,\alpha_1,\zeta,\ep) \geq n_1$ such that for $\delta \in (0,\delta_0)$ and $n\geq n_2$, 
\eqbn
\BB P\left[ E^{n}(\delta)  \right] \geq 1 - 2\ep .
\eqen  
\end{lem} 
\begin{proof}
Let $I_{\alpha_0, \alpha_1}^{\infty,n}$ be the analogous stopping time for the percolation peeling process on the UIHPQ$_{\op{S}}$, as in~\eqref{eqn-discrete-close-time-infty}.
As in the proof of Lemma~\ref{lem-sk-tight}, there is a $\wt n_2 = \wt n_2(\alpha_1 ,\ep) \geq n_1$ such that for $n\geq \wt n_2$, the law of $W^n|_{[0, I_{\alpha_0,\alpha_1}^n]_{\BB Z}}$ on the event $\{ I_{\alpha_0,\alpha_1}^n  <\infty\}$ is absolutely continuous with respect to the law of $W^\infty|_{[0, I_{\alpha_0,\alpha_1}^{\infty,n}]_{\BB Z}}$, with Radon-Nikodym derivative bounded above by $b \alpha_0^{-5/2} \BB 1_{  (I_{\alpha_0,\alpha_1}^{\infty,n}  < \mcl J^{\infty,n} ) } $, where here $b >0$ is a constant depending only on $(\frk l_L , \frk l_R)$ and $\mcl J^{\infty,n}$ is as in~\eqref{eqn-infty-terminal-time}.  

Let $c > 0$ be the constant from Lemma~\ref{lem-equicont-reg-event-infty}.
By Lemma~\ref{lem-equicont-reg-event-infty} (applied with $\zeta/2$ in place of $\zeta$ and $b^{-1} \alpha_0^{5/2} \ep$ in place of $\ep$), there exists $\delta_0 = \delta_0(A,\zeta,\alpha_0, \ep) \in (0,1)$ such that for $\delta\in(0,\delta_0]$, the event of that lemma defined with $\zeta/2$ in place of $\zeta$ and $A$ as in~\eqref{eqn-use-short-time-equicont} satisfies $\BB P\left[ E^{\infty,n}(\delta)  \right] \geq 1-b^{-1} \alpha_0^{5/2} \ep$. By combining this with the above Radon-Nikodym derivative estimate and~\eqref{eqn-use-short-time-equicont}, we obtain the statement of the lemma.
\end{proof}

Henceforth fix $\delta_0$ and $n_2$ as in Lemma~\ref{lem-equicont-reg-event}. We will now define some sub-quadrangulations of $Q^n$ which are illustrated in Figure~\ref{fig-ghpu-tight}. 
For $n\in\BB N$ and $\delta \in (0,1)$, define the time $T^n_{k,r}(\delta)$ for $k\in\BB N_0$ as in~\eqref{eqn-equicont-reg-times} and define the terminal indices for these stopping times by
\eqb \label{eqn-reg-times-terminal}
k_*^n(\delta) := \min\left\{ k\in \BB N : T_{k ,0}^n(\delta) = I_{\alpha_0,\alpha_1}^n \right\} 
\quad \op{and} \quad r_{*,k}^n(\delta) := \min\left\{ r \in \BB N_0 : T_{k,r }^n(\delta) = \lfloor k \delta n^{3/4} \rfloor \right\} . 
\eqe
On the event $E^n(\delta)$ of Lemma~\ref{lem-equicont-reg-event}, we have $k_*^n(\delta) \leq A \delta^{ -1 } $ and $r_{*,k}^n(\delta) \leq \delta^{-\zeta/2}$ for each $k \in \BB N$. 
For $k\in [1,k_*^n(\delta)]_{\BB Z}$ and $r\in [1,r_{*,k}^n(\delta)]_{\BB Z}$, define
\eqb \label{eqn-quad-step-def}
\dot Q_{k,r}^n(\delta) := \left(  \dot Q_{T^n_{k,r}(\delta)-1}^n \cap \ol Q^n_{T^n_{k,r-1} (\delta)} \right) \cup \frk f\left(\ol Q^n_{T^n_{k,r}(\delta)-1} , \dot e_{T_{k,r}^n(\delta)} \right) .
\eqe 
That is, $\dot Q_{k,r}^n(\delta)$ is obtained from the peeling cluster increment $\dot Q_{T^n_{k,r}(\delta) }^n \cap \ol Q^n_{T^n_{k,r-1} (\delta)}$ by removing the quadrangulation(s) which are disconnected from the target edge at time $T_{k,r}^n(\delta)$ (one of these quadrangulations might be quite large since $W^{L,n} + W^{R,n}$ may have a large downward jump at time $T_{k,r}^n(\delta)$).
The chordal percolation exploration path $\lambda^n$ satisfies
\eqb \label{eqn-equicont-path-contain}
\lambda^n\left([T_{k,r-1}^n(\delta) , T_{k,r}^n(\delta)]_{\frac12\BB Z} \right) \subset \dot Q_{k,r}^n(\delta) ,\quad \forall k \in [1,k_*^n(\delta)]_{\BB Z}, \, \forall r \in  [1,r_{*,k}^n(\delta)]_{\BB Z} .
\eqe  

Let $\dot Q_k^{n,\bullet}(\delta)$ be the union of $\bigcup_{r=1}^{r_{*,k}^n(\delta)} \dot Q_{k,r}^n(\delta)$ and the set of all vertices and edges which are disconnected from $\bdy Q^n$ by this union. Then $\dot Q_k^{n,\bullet}(\delta)$ is a quadrangulation with simple boundary and by~\eqref{eqn-equicont-path-contain},  
\eqb \label{eqn-equicont-path-contain'}
\lambda^n([ \lfloor (k-1)\delta n^{3/4} \rfloor , \lfloor k \delta n^{3/4} \rfloor \wedge  I_{\alpha_0,\alpha_1}^n ]_{ \frac12 \BB Z} )  \subset \dot Q_k^{n,\bullet} ,\quad \forall k\in [1,k_*^n(\delta)]_{\BB Z} .
\eqe
Furthermore,
\eqb \label{eqn-bdy-relation} 
\bdy \dot Q_k^{n,\bullet}(\delta) \subset \bigcup_{r=1}^{r_{*,k}^n(\delta)} \bdy \dot Q_{k,r}^n(\delta).
\eqe 
Hence we are led to estimate the diameters of the boundaries of the quadrangulations $\dot Q_{k,r}^n(\delta)$.

\begin{lem} \label{lem-quad-step-bdy}
For each $\zeta \in (0,1/3)$, there exists $\delta_1 =\delta_1(\alpha_0,\alpha_1,\zeta,\ep) \in (0,\delta_0]$ such that for each $\delta \in (0,\delta_1]$, there exists $n_3 = n_3(\delta,\alpha_0,\alpha_1,\zeta,\ep) \geq n_2$ such that for each $n\geq n_3$, it holds with probability at least $1- 3\ep$ that the event $E^n(\delta)$ of Lemma~\ref{lem-equicont-reg-event} occurs and
\eqbn
\op{diam}\left(\bdy \dot Q_{k,r}^n(\delta) ; \ol Q_{T_{k,r-1}^n(\delta)}^n \right) \leq \frac12 \delta^{1/3-\zeta/2} n^{1/4} ,\quad \forall k \in [1,k_*^n(\delta)]_{\BB Z}, \, \forall r \in  [1,r_{*,k}^n(\delta)]_{\BB Z} .
\eqen 
\end{lem} 
\begin{proof}
The idea of the proof is as follows. We first use the definition of the event of Lemma~\ref{lem-equicont-reg-event} to bound the number of edges in the boundary of each of the $\dot Q_{k,r}^n(\delta)$'s. All but a constant-order number of the edges of $\dot Q_{k,r}^n(\delta)$ are contained in the union of the boundaries of the unexplored quadrangulations at times $ T^n_{k,r-1}(\delta) $ and $ T^n_{k,r }(\delta) $. We will apply Lemma~\ref{lem-fb-bdy-holder} to these quadrangulations to bound the diameter of $\bdy \dot Q_{k,r}^n(\delta)$.

By Definition~\ref{def-bdy-process} and~\eqref{eqn-bdy-process-inf}, for $k \in [1,k_*^n(\delta)]_{\BB Z}$ and $r \in  [1,r_{*,k}^n(\delta)]_{\BB Z} $, 
\alb
\# \mcl E\left(\bdy \dot Q_{k,r}^n(\delta) \right)  
&\leq 2 \max_{j \in [T^n_{k,r-1}(\delta) , T^n_{k,r}(\delta)-1]_{\BB Z}} |W_j^n - W_{T^n_{k,r-1}(\delta)}^n| + 12 .
\ale 
Hence on the event $E^n(\delta)$ of Lemma~\ref{lem-equicont-reg-event},
\eqb \label{eqn-bdy-length-on-reg}
\# \mcl E\left(\bdy \dot Q_{k,r}^n(\delta) \right)  \leq 2 \delta^{2/3-\zeta/2} n^{1/2} + 12 ,\quad \forall k \in [1,k_*^n(\delta)]_{\BB Z}, \, \forall r \in  [1,r_{*,k}^n(\delta)]_{\BB Z} .
\eqe

For $\delta \in (0,1)$ and $(k,r) \in \BB N \times \BB N_0 $, let $A_{k,r}^n(\delta)$ be the boundary arc of the unexplored quadrangulation $\ol Q_{T_{k,r }^n(\delta)}^n$ which contains $\lceil 2\delta^{2/3-\zeta/2} n^{1/2} + 12 \rceil$ edges of $\bdy \ol Q_{T_{k ,r}^n(\delta)}^n$ lying to the left and to the right of the root edge $\dot e_{T_{k,r}^n(\delta)+1}$. By~\eqref{eqn-bdy-length-on-reg}, if $E^n(\delta)$ occurs then 
\eqb \label{eqn-reg-bdy-contain}
\bdy \dot Q_{k,r}^n(\delta) \subset A_{k,r-1}^n(\delta) \cup A_{k,r}^n(\delta) \cup \bdy \frk f\left(\ol Q^n_{T^n_{k,r}(\delta)-1} , \dot e_{T_{k,r}^n(\delta)} \right) ,\quad \forall k \in [1,k_*^n(\delta)]_{\BB Z}, \, \forall r \in  [1,r_{*,k}^n(\delta)]_{\BB Z} .
\eqe  
Here we recall that $\bdy\frk f\left(\ol Q^n_{T^n_{k,r}(\delta)-1} , \dot e_{T_{k,r}^n(\delta)} \right)  $ is the boundary of the peeled quadrilateral at time $T_{k,r}^n(\delta)$, which contains at most 4 edges.

We will now estimate $\op{diam}(A_{k,r}^n ; \ol Q^n_{T^n_{k,r}(\delta)} )$.   
By, e.g., Lemma~\ref{lem-sk-tight}, there is a $C = C(\ep) >0$ such that 
\eqb \label{eqn-equicont-sup-event}
\BB P[G^n] \geq 1- \frac{\ep}{2} ,\: \forall n\in\BB N \quad \op{where}\quad G^n := \left\{ \max_{j \in \BB N_0} |W_j^n| \leq C n^{1/2} \right\} . 
\eqe 
Each of the times $T_{k,r}^n(\delta)$ is a $\{\mcl F_j^n\}_{j\in\BB N_0}$-stopping time. By the Markov property of peeling, for each $(k,r) \in \BB N^2$ the conditional law given $\mcl F_{T^n_{k,r}(\delta)}$ of the unexplored quadrangulation $( \ol Q_{T_{k,r }^n(\delta)}^n , \dot e_{T_{k,r}^n(\delta)+1})$ is that of a free Boltzmann quadrangulation with simple boundary and perimeter $W_{T_{k,r}^n(\delta)}^{L,n} + W_{T_{k,r}^n(\delta)}^{R,n} + \el_L^n + \el_R^n$. Since $T_{k,r}^n(\delta) \leq I_{\alpha_0,\alpha_1}^n$ by definition, this perimeter lies in $[   \alpha_0 \bcon n^{1/2} , 2 C n^{1/2} ]_{\BB Z}$ provided the event $G^n$ of~\eqref{eqn-equicont-sup-event} occurs and $I_{\alpha_0,\alpha_1}^n < \infty$.
 
Lemma~\ref{lem-fb-bdy-holder} implies that there is a $\delta_1 =\delta_1(\alpha_0,\alpha_1,\zeta,\ep) \in (0,\delta_0]$ such that for each $\delta \in (0,\delta_1]$, there exists $n_3 = n_3(\delta,\alpha_0,\alpha_1,\zeta,\ep) \geq n_2$ as in the statement of the lemma such that for $n\geq n_3$ and $(k,r) \in\BB N \times \BB N_0$,
\eqbn
\BB P\left[ \op{diam}\left( A_{k,r}^n(\delta) ; \ol Q_{T_{k,r }^n(\delta)}^n   \right)  > \frac14 \delta^{1/3- \zeta/2} n^{1/4} -4 ,\: G^n ,\: I_{\alpha_0,\alpha_1}^n < \infty \right] \leq \frac12 A^{-1} \delta^{ 1+\zeta/2} \ep .
\eqen
Since $ I_{\alpha_0,\alpha_1}^n < \infty $, $k_*^n(\delta) \leq A\delta^{-1 }$, and  $r_{*,k}^n(\delta) \leq \delta^{-\zeta/2}$ for each $k\in [1,k_*^n(\delta)]_{\BB Z}$ on $E^n(\delta)$ 
and since $\BB P[G^n \cap E^n(\delta)] \geq 1- 5\ep/2$, we can take a union bound over at most $O_\delta(\delta^{-1-\zeta/2})$ values of $k$ to get
\eqb \label{eqn-quad-arc-diam}
\BB P\left[ \op{diam}\left( A_{k,r}^n(\delta) ; \ol Q_{T_{k ,r}^n(\delta)}^n   \right)  \leq \frac14 \delta^{1/3- \zeta/2} n^{1/4} -4 ,\, \forall k \in [1,k_*^n(\delta)]_{\BB Z}, \, \forall r \in  [0,r_{*,k}^n(\delta)]_{\BB Z}  ,\: E^n(\delta)      \right] \geq 1- 3\ep .
\eqe  
The statement of the lemma now follows by combining~\eqref{eqn-reg-bdy-contain} with~\eqref{eqn-quad-arc-diam}. 
\end{proof}

\begin{proof}[Proof of Proposition~\ref{prop-bdy-equicont}]
Let $\delta_1 =\delta_1(\alpha_0,\alpha_1,\zeta,\ep) \in (0,\delta_0]$ and $n_3 = n_3(\delta,\alpha_0,\alpha_1,\zeta,\ep) \geq n_2$ for $\delta \in (0,\delta_1]$ be as in Lemma~\ref{lem-quad-step-bdy}.   

By Lemma~\ref{lem-quad-step-bdy}, if $\delta \in (0,\delta_1]$ and $n\geq n_3$ it holds with probability at least $1-3\ep$ that $E^n(\delta)$ occurs and each of the boundaries $\bdy\dot Q_{k,r}^n(\delta)$ for $k\in [1,k_*^n(\delta)]_{\BB Z}$ and $r \in  [1,r_{*,k}^n(\delta)]_{\BB Z}$ has $ \ol Q_{T_{k,r-1}^n(\delta)}^n$- and hence also $\ol Q_{\lfloor (k-1) \delta n^{3/4}\rfloor}^n$-diameter at most $\frac12 \delta^{1/3-\zeta/2} n^{1/4}$. Henceforth assume that this is the case.

By condition~\ref{item-equicont-reg-times'} in the definition of $E^n(\delta)$, we have $r_{*,k}^n(\delta) \leq \delta^{-\zeta/2}$ for each $k\in [1,k_*^n(\delta)]_{\BB Z}$. By combining this with~\eqref{eqn-bdy-relation} and using that $ \dot Q_k^{n,\bullet}$ is connected, we see that each $\bdy \dot Q_k^{n,\bullet}$ has $\ol Q_{\lfloor (k-1) \delta n^{3/4}\rfloor}^n$-diameter at most $\frac12 \delta^{1/3- \zeta } n^{1/4}$.  By~\eqref{eqn-equicont-path-contain'}, we infer that condition~\ref{item-bdy-equicont} in the proposition statement is satisfied. 
 
On $E^n(\delta)$ we have $I_{\alpha_0,\alpha_1}^n \leq A n^{3/4}$ and by Lemma~\ref{lem-end-bdy-choice}, for $n\geq n_3$ it holds with probability at least $1-\ep$ that $\op{diam}\left( \ol Q_{I_{\alpha_0,\alpha_1}^n}^n \right) \leq \ep n^{1/4} $, i.e.\ condition~\ref{item-bdy-equicont-end} in the proposition statement is satisfied.  Hence the statement of the proposition is satisfied with $4\ep$ in place of $\ep$. Since $\ep \in (0,1)$ is arbitrary, we conclude. 
\end{proof}

\begin{proof}[Proof of Proposition~\ref{prop-equicont}] 
Fix $\ep \in (0,1)$ and $\zeta\in (0,1/3)$ and let $\delta_* = \delta_*(\ep,\zeta) \in (0,1)$ be as in Proposition~\ref{prop-bdy-equicont}. 
By Lemma~\ref{lem-fb-pinch}, there is a $\delta_*' = \delta_*(\ep) \in (0,\delta_*]$ such that for $\delta \in (0,\delta_*']$, there exists $n_*' = n_*'(\delta,\ep) \geq n_*$ such that for $n\geq n_*'$, it holds with probability at least $1-\ep$ that the following is true. For each subgraph $S$ of $Q^n$ with $\op{diam}(\bdy S ; Q^n) \leq  \delta^{1/3- \zeta} n^{1/4}$, it holds that $\op{diam}(S ; Q^n) \leq \ep$.   

By combining the preceding paragraph with Proposition~\ref{prop-bdy-equicont} (c.f.\ Remark~\ref{remark-weaker-cont}), we find that if $n\geq n_*'$, then with probability at least $1-2\ep$, it holds that $I_{\alpha_0,\alpha_1}^n < \infty$, 
\eqb \label{eqn-use-step-bdy} 
\op{diam}\left( \lambda^n([j_1,j_2]_{\frac12 \BB Z}) ; Q^n \right) \leq \ep n^{1/4} ,\quad \forall  j_1, j_2 \in [0,I_{\alpha_0,\alpha_1}^n-1]_{\frac12\BB Z} \: \op{with} \: 0\leq j_2-j_1 \leq \delta n^{3/4} , 
\eqe 
and $\op{diam}\left( \ol Q_{I_{\alpha_0,\alpha_1}^n}^n \right) \leq \ep n^{1/4} $. 
Since $\lambda^n([I_{\alpha_0,\alpha_1}^n,\infty)_{\frac12\BB Z}) \subset \ol Q_{I_{\alpha_0,\alpha_1}^n}^n$, we find that the statement of the proposition is satisfied with $2\ep$ in place of $\ep$. Since $\ep \in (0,1)$ is arbitrary, we conclude.
\end{proof}

We end by recording an analog of Proposition~\ref{prop-bdy-equicont} for face percolation on the UIHPQ$_{\op{S}}$, whose proof is a subset of a proof of Proposition~\ref{prop-bdy-equicont}. In contrast to the statement of Proposition~\ref{prop-bdy-equicont}, we get a quantitative estimate for the probability of our regularity event, since we do not need to deal with any analog of the time~$I_{\alpha_0,\alpha_1}^n$.

\begin{prop} \label{prop-bdy-equicont-uihpq}
Let $\zeta \in (0,1/3)$ and $A>0$.  
For each $\delta \in (0,1)$, there exists $n_* = n_*(\delta,\zeta , A ) \in\BB N$ such that for $n\geq n_*$, it holds with probability at least $1-o_\delta^\infty(\delta)$ (at a rate which is uniform for $n\geq n_*$) such that the following is true. 
For each $j\in  [0, A n^{3/4}]_{  \lfloor \delta n^{3/4} \rfloor \BB Z}$, the percolation path increment $\lambda^\infty([j ,  j+ \delta n^{3/4} ]_{\frac12\BB Z})$ is contained in a subgraph of the unexplored quadrangulation $\ol Q_j^\infty$ whose boundary relative to $\ol Q_j^\infty$ contains $\dot e_j^\infty$ and has $\ol Q_j^\infty$-graph distance diameter at most $ \frac12 \delta^{1/3- \zeta} n^{1/4}$.  
\end{prop}
\begin{proof}
From Lemmas~\ref{lem-equicont-reg-event-infty} and~\ref{lem-uihpq-bdy-holder}, we obtain an analog of Lemma~\ref{lem-quad-step-bdy} for face percolation on the UIHPQ$_{\op{S}}$ which holds on an event of probability $1-o_\delta^\infty(\delta)$ rather than $1-3\ep$ using exactly the same argument used to prove Lemma~\ref{lem-quad-step-bdy}. 
From this, we deduce the statement of the proposition via the same argument used to conclude the proof of Proposition~\ref{prop-bdy-equicont}. 
\end{proof}

We note that the condition of Proposition~\ref{prop-bdy-equicont-uihpq} implies a weaker condition with $Q^\infty$-graph distances in place of $\ol Q_j^\infty$-graph distances, analogous to the one in Remark~\ref{remark-weaker-cont}.

\section{Crossings between filled metric balls}
\label{sec-crossing}

Suppose $\el_L,\el_R \in \BB N \cup \{\infty\}$ are such that $\el_L$ and $\el_R$ are either both even, both odd, or both infinite and $(Q , \BB e , \theta)$ is a free Boltzmann quadrangulation with simple boundary of perimeter $\el_L+\el_R$.  Consider the percolation peeling process of $(Q,\BB e ,\theta)$ with $\el_L$-white/$\el_R$-black boundary conditions as defined in Section~\ref{sec-perc-peeling}. 

The goal of this section is to bound the number of times that the associated percolation exploration path $\lambda$ can cross an annulus between two filled graph metric balls in $Q$. This estimate will allow us to conclude in Section~\ref{sec-homeo} that a subsequential scaling limit of the percolation exploration paths can hit any single point at most 7 times. See Section~\ref{sec-outline} for a discussion of why we need this fact. 
 
Before stating the main results of this section, we introduce some notation.

\begin{defn} \label{def-crossing} 
For sub-graphs $S_0 \subset S_1 \subset Q$, an \emph{inside-outside crossing} of $S_1\setminus S_0$ by the percolation peeling process of $(Q ,\BB e  , \theta )$ is a discrete time interval $[i_0,i_1]_{\BB Z}\subset \BB N_0$ such that the peeled edges satisfy $\dot e_{i_0}  \in S_0 $, $\dot e_{i_1} \notin  S_1$, and $\dot e_j \in S_1 \setminus S_0$ for each $j\in [i_0+1 ,i_1-1]_{\BB Z}$.   
We write $\op{cross}(S_0,S_1) $ for the set of all crossings of $S_1\setminus S_0$.     
\end{defn}
 
We will sometimes include an extra superscript $\infty$ in the notation $\op{cross} (S_0,S_1 )$ to denote the UIHPQ$_{\op{S}}$ case or an extra superscript $n$ when we take $\el_L$ and $\el_R$ to depend on $n$.

Since the percolation exploration path satisfies $\lambda(j) =\dot e_j$ for $j\in\BB N_0$, an element of $\op{cross}(S_0,S_1)$ is (essentially) the same as a crossing of $S_1 \setminus S_0$ by the path $\lambda$. We work with crossings of the percolation peeling process instead mostly for the notational convenience of not having to worry about the values of $\lambda$ at half-integer times.
 
We will primarily be interested in inside-outside crossings of annular regions between filled graph metric balls, which are defined as follows.

\begin{defn} \label{def-filled-ball}
For a subgraph $Q'$ of $Q$ containing the target edge $\BB e_*$ (or an unbounded subgraph in the case $\el = \infty$), $r\geq 0$, and a subset $S$ of $Q'$ consisting of vertices and edges, the \emph{filled metric ball} $B_r^\bullet(S ; Q')$ is the subgraph of $Q'$ consisting of the graph metric ball $B_r (S ; Q')$ and the set of all vertices and edges of $Q'$ which it disconnects from $\BB e_*$ (or $\infty$ if $\el = \infty$). 
\end{defn}

Note that $B_r^\bullet(S;Q') = Q'$ if $\op{dist}(S , \BB e_* ; Q') < r$. 

Suppose now that we are in the setting of Theorem~\ref{thm-perc-conv}, so that $\{(\el_L^n , \el_R^n)\}_{n\in\BB N}$ is a sequence of pairs of positive integers such that $\el_L^n + \el_R^n$ is always even, $\bcon^{-1} n^{-1/2} \el_L^n \rta \frk l_L> 0$, and $\bcon^{-1} n^{-1/2} \el_R^n\rta \frk l_R  > 0$.
Let $(Q^n , \BB e^n , \theta^n)$ be a free Boltzmann quadrangulation with simple boundary of perimeter $\el_L^n + \el_R^n$ and, as per usual, denote the objects from Section~\ref{sec-perc-peeling} with $\el_L^n$-white/$\el_R^n$-black boundary conditions and the crossing sets of Definition~\ref{def-crossing} with an additional superscript $n$. 
The main result of this section is the following proposition.

\begin{prop}
\label{prop-fb-crossing}
For each $\ep \in (0,1)$, there exists $\delta_* = \delta_*(\ep) \in (0,1)$ such that for each $\delta \in (0,\delta_*]$, there exists $n_* = n_*(\delta,\ep) \in \BB N$ such that for $n\geq n_*$, it holds with probability at least $1-\ep$ that the following is true.
For each vertex $v\in \mcl V(Q^n)$ with $\op{dist}(v  , \bdy Q^n ; Q^n) \geq \ep n^{1/4}$, the number of inside-outside crossings (Definition~\ref{def-crossing}) of the filled metric ball annulus $ B_{\ep n^{1/4}}^\bullet(v ; Q^n)  \setminus B_{\delta n^{1/4}}^\bullet(v ; Q^n)$ satisfies
\eqbn
\#\op{cross}^n\left( B_{\delta n^{1/4}}^\bullet(v ; Q^n) , B_{\ep n^{1/4}}^\bullet(v ; Q^n) \right) \leq   7 .
\eqen
\end{prop}

Proposition~\ref{prop-fb-crossing} only treats vertices which are at macroscopic distance from $\bdy Q^n$. 
We have the following variant of Proposition~\ref{prop-fb-crossing} which deals with boundary vertices. 

\begin{prop}
\label{prop-fb-crossing-bdy}
For each $\ep \in (0,1)$, there exists $\delta_* = \delta_*(\ep) \in (0,1)$ such that for each $\delta \in (0,\delta_*]$, there exists $n_* = n_*(\delta,\ep) \in \BB N$ such that for $n\geq n_*$, it holds with probability at least $1-\ep$ that the following is true.
For each vertex $v\in \mcl V(\bdy Q^n)$ with $\op{dist}(v  , \BB e_*^n ; Q^n) \geq \ep n^{1/4}$, the number of inside-outside crossings (Definition~\ref{def-crossing}) of the filled metric ball annulus $ B_{\ep n^{1/4}}^\bullet(v ; Q^n)  \setminus B_{\delta n^{1/4}}^\bullet(v ; Q^n)$ satisfies
\eqbn
\#\op{cross}^n\left( B_{\delta n^{1/4}}^\bullet(v ; Q^n) , B_{\ep n^{1/4}}^\bullet(v ; Q^n) \right) \leq   4 .
\eqen
\end{prop}

The important point in Propositions~\ref{prop-fb-crossing} and~\ref{prop-fb-crossing-bdy} is that we have some constant finite upper bound for the number of crossings; the particular numbers 7 and 3 are not important and we do not try to optimize them. Since chordal $\SLE_6$ does not have any triple points~\cite[Remark 5.3]{miller-wu-dim}, once Theorem~\ref{thm-perc-conv} is established we will obtain that for each fixed $\ep > 0$, each small enough $\delta > 0$, and each large enough $n\in\BB N$, it holds with probability at least $1-\ep$ that $\#\op{cross}^n\left( B_{\delta n^{1/4}}^\bullet(v ; Q^n) , B_{\ep n^{1/4}}^\bullet(v ; Q^n) \right) \leq 2$ for all $v\in \mcl V(Q^n)$.  

Although Propositions~\ref{prop-fb-crossing} and~\ref{prop-fb-crossing-bdy} are statements about inside-outside crossings by the percolation peeling process of filled metric balls centered at general vertices of $Q^n$, for most of the proof of the proposition we will instead consider a closely related quantity which is defined precisely in Section~\ref{sec-crossing-perturb}.  In particular, we will bound the number of percolation interfaces which cross an annulus between two filled metric balls centered at edges on the boundary of a free Boltzmann quadrangulation with simple boundary or a UIHPQ$_{\op{S}}$; as we will see in Lemma~\ref{lem-crossing-ip}, this quantity provides an upper bound for the quantity considered in Proposition~\ref{prop-fb-crossing}. 

In Section~\ref{sec-uihpq-crossing} we will prove Proposition~\ref{prop-uihpq-crossing}, which is the key estimate needed for the proof of Propositions~\ref{prop-fb-crossing} and~\ref{prop-fb-crossing-bdy} and which is formulated in the setting of face percolation on the UIHPQ$_{\op{S}}$. Roughly speaking, this proposition gives for each $N\in\BB N$ and $\zeta\in (0,1/100)$ an $n$-independent upper bound for the probability that there are more than $N$ face percolation interface paths which cross an annulus between two filled metric balls of respective radii $\delta n^{1/4}$ and $\delta^\zeta n^{1/4}$ centered at the root edge, provided we truncate on the event that the boundary length of the inner filled metric ball is not unusually large.  By the union bound, for $M\in\BB N$ one gets the same estimate simultaneously for the unexplored quadrangulation at all times in $[0, n^{3/4}]_{ \lfloor \delta^M n^{3/4} \rfloor \BB Z}$ except that we lose a factor of $\delta^{-M}$. 

In Section~\ref{sec-pbl-bdy}, we prove an upper bound for the boundary lengths of filled metric balls centered at edges on the boundary in a free Boltzmann quadrangulation with simple boundary, which will be used to remove the truncation in the estimate of Section~\ref{sec-uihpq-crossing}. In Section~\ref{sec-crossing-proof}, we conclude the proof of Proposition~\ref{prop-fb-crossing} using the estimates of the preceding subsections, the modulus of continuity estimate for $\lambda^n$ in Proposition~\ref{prop-bdy-equicont}, and a triangle inequality argument. In Section~\ref{sec-crossing-proof-bdy}, we conclude the proof of Proposition~\ref{prop-fb-crossing-bdy}.

\subsection{Interface path crossings and peeling-by-layers clusters}
\label{sec-crossing-perturb}

Instead of bounding inside-outside crossings by the percolation peeling process of an annulus between filled metric balls, throughout most of this section we will focus on bounding the number of interface paths which cross an annulus between two peeling-by-layers clusters, a closely related quantity which has nicer probabilistic properties. We emphasize that we will be considering percolation interfaces (Section~\ref{sec-interface}) instead of percolation exploration paths.  The purpose of this brief subsection is to define the above quantity precisely and explain how it is related to the crossing count considered in Proposition~\ref{prop-fb-crossing}. 

Throughout this subsection, we suppose we are in the general setting described at the very beginning of this section, so that $\el_L,\el_R \in \BB N \cup \{\infty\}$ are such that $\el_L$ and $\el_R$ are either both even, both odd, or both infinite, $(Q , \BB e , \theta)$ is a free Boltzmann quadrangulation with simple boundary of perimeter $\el_L+\el_R$, and we define the percolation peeling process with $\el_L$-white/$\el_R$-black boundary conditions as in Section~\ref{sec-perc-peeling}. 

For $r,j\in\BB N_0$ and a finite connected set $\BB A$ of edges on the boundary of the unexplored quadrangulation $\ol Q_j$, we let $B_r^\pbl(\BB A ; \ol Q_j)$ be the radius-$r$ peeling-by-layers cluster of $\ol Q_j$ with initial edge set $\BB A$ targeted at $\BB e_*$ as in~\cite[Section~4.1]{gwynne-miller-simple-quad} (where it is denoted by $\dot Q_{J_r}$). The set $B_r^\pbl(\BB A ; \ol Q_j)$ is essentially the same as the filled metric ball of radius $r$ (Definition~\ref{def-filled-ball}) in the sense that 
\eqb \label{eqn-filled-ball-contain}
B_r^\bullet(\BB A ; \ol Q_j) \subset B_r^\pbl(\BB A ; \ol Q_j) \subset B_{r+2}^\bullet(\BB A ; \ol Q_j) .
\eqe  
We also define the outer boundary arc of $B_r^\pbl(\BB A ; \ol Q_j)$ by
\eqb \label{eqn-pbl-bdy-length-def}
\mcl A_r^\pbl(\BB A ; \ol Q_j) := \mcl E\left( \bdy B_r^\pbl(\BB A ; \ol Q_j) \setminus \bdy \ol Q_j  \right)    .
\eqe  
As explained in~\cite[Section~4.1]{gwynne-miller-simple-quad}, $B_r^\pbl(\BB A ; \ol Q_j)$ is the cluster of a peeling process of $\ol Q_j$ (the so-called peeling-by-layers process) run up to a stopping time, so by the Markov property of peeling the unexplored quadrangulation $(\ol Q_j \setminus B_r^\pbl(\BB A ; \ol Q_j ) )  \cup \mcl A_r^\pbl(\BB A ; \ol Q_j) $ is conditionally independent from $B_r^\pbl(\BB A ; \ol Q_j )$ given its perimeter and its conditional law is that of a free Boltzmann quadrangulation with simple boundary of given perimeter. We will not need the precise definition of $B_r^\pbl(\BB A ;\ol Q_j)$ here.

\begin{defn} \label{def-ip-crossing}
For $j\in\BB N_0$ and $e\in \bdy \ol Q_j \setminus \{\BB e_*\}$, let $\op{IP}_j(e; r_0, r_1)$ be the set of interface paths (Section~\ref{sec-interface}) which cross the semi-annular region $B_{r_1}^\pbl(e ; \ol Q_j) \setminus B_{r_0}^\pbl(e  ; \ol Q_j)$ from inside to outside, i.e.\ the number of interface paths $\rng\lambda : [0, b]_{\BB Z} \rta \mcl E(Q)$ such that $\rng\lambda(0) \in B_{r_0}^\pbl(e ; \ol Q_j) $, $\rng\lambda(b) \notin  B_{r_1}^\pbl(e ; \ol Q_j)$, and $\rng\lambda([1,b-1]_{\BB Z}) \subset B_{r_1}^\pbl(e ; \ol Q_j) \setminus B_{r_0}^\pbl(e ; \ol Q_j) $. We abbreviate
\eqb
\op{IP}_j(r_0,r_1) := \op{IP}_j(\dot e_j ; r_0,r_1) .
\eqe
\end{defn}

The following lemma is the main reason for our interest in the set of interface paths $\op{IP}_j(e;r_0,r_1)$. 

\begin{lem} \label{lem-crossing-ip}
For each $j\in\BB N_0$, each edge $e\in \bdy \ol Q_j$, and each $0\leq r_0  < r_1$ with $r_1 - r_0 \geq 5$, one has, in the notation of Definitions~\ref{def-crossing} and~\ref{def-ip-crossing}, 
\eqbn
\# \op{cross} \left(    B_{r_0}^\pbl(e  ; \ol Q_j)   , B_{r_1}^\pbl(e  ; \ol Q_j)  \right)  \leq \# \op{IP}_j(e;r_0+2 ,r_1-2)  .
\eqen
\end{lem}
\begin{proof}
By stationarity it suffices to prove the statement of the lemma for $j=0$. 
Let $\rng\lambda_*$ and $\dot v_j$ and $s_j$ for $j\in\BB N_0$ be the interface path, vertices, and times, respectively, from  Lemma~\ref{lem-peel-interface} so that $\dot v_j$ is an endpoint of $\dot e_j$ and is the terminal endpoint of $\rng\lambda_*(s_j) $. 

Suppose $[i_0,i_1]_{\BB Z} $ is an inside-outside crossing of $B_{r_1}^\pbl(e ; \ol Q_j) \setminus B_{r_0}^\pbl(e ; \ol Q_j)$. 
Then $\rng\lambda_*(s_{i_0}) \in B_{r_0}^\pbl(e  ; \ol Q_j)$ and $\rng\lambda_*(s_{i_1}) \notin B_{r_1}^\pbl(e  ; \ol Q_j)$ so there exists a discrete interval $[\rng k_0 , \rng k_1]_{\BB Z}\subset [s_{i_0} , s_{i_1}]_{\BB Z}$ such that $\rng\lambda_*|_{[\rng k_0 , \rng k_1]_{\BB Z}}$ belongs to $\op{IP}_0(r_0+2 , r_1-2)$. 
Distinct inside-outside crossings $[i_0,i_1]_{\BB Z}$ are disjoint, so (since $\rng\lambda_*$ does not hit any edge more than once) must correspond to intervals $[s_{i_0} , s_{i_1}]_{\BB Z}$ which can possibly intersect only at their endpoints. Hence the inside-outside crossings of $B_{r_1}^\pbl(e ; \ol Q_j) \setminus B_{r_0}^\pbl(e ; \ol Q_j)$ correspond to distinct elements of $\op{IP}_0(r_0+2,r_1-2)$. 
\end{proof}

\subsection{Crossings of semi-annuli in the UIHPQ$_{\op{S}}$}
\label{sec-uihpq-crossing}

Throughout this subsection we assume we are in the setting of Section~\ref{sec-perc-peeling} with $\el_L = \el_R = \infty$ (so in particular $(Q^\infty, \BB e^\infty)$ is a UIHPQ$_{\op{S}}$) and as per usual we denote the objects of that subsection with a superscript $\infty$. 
Our goal is to prove the following estimate for the number of interface paths which cross an annulus between two peeling-by-layers clusters centered at the root edge, which is the main input in the proof of Proposition~\ref{prop-fb-crossing}. 

\begin{prop} \label{prop-uihpq-crossing} 
Let $\zeta \in (0,1/100)$, $M > 0$, and $A>0$. 
For each $\delta \in (0,1)$ there exists $n_* = n_*(\delta ,\zeta ,A) \in \BB N$ such that for $n\geq n_*$ and $N\in\BB N$, it holds with probability $1 - O_\delta(\delta^{(N-1) (1-2\zeta) - M}) $ (at a rate which is uniform for $n\geq n_*$) that, in the notation of Definition~\ref{def-ip-crossing} and~\eqref{eqn-pbl-bdy-length-def}, 
\alb
&\#\op{IP}_j^\infty( \delta  n^{1/4} , \delta^\zeta n^{1/4}   ) \leq N ,\notag\\ 
&\qquad\forall j \in [0, A n^{3/4}]_{ \lfloor \delta^M n^{3/4} \rfloor \BB Z  }   \quad\text{such that} \quad     \# \mcl A_{\delta  n^{1/4}}^\pbl(\dot e_j^\infty ; \ol Q^\infty_j) \leq  \delta^{2-\zeta} n^{1/2}   .
\ale
\end{prop}

\begin{figure}[ht!]
 \begin{center}
\includegraphics[scale=1]{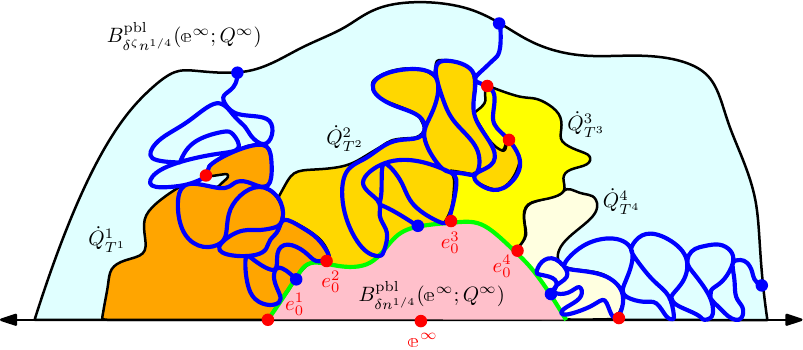} 
\caption[Illustration of the proof of Proposition~\ref{prop-uihpq-crossing}]{\label{fig-uihpq-crossing} Illustration of the proof of Proposition~\ref{prop-uihpq-crossing}. The inner (resp.\ outer) peeling-by-layers cluster $B_{\delta n^{1/4}}^\pbl(\BB e^\infty ; Q^\infty )$ (resp.\ $B_{\delta^\zeta n^{1/4}}^\pbl(\BB e^\infty ; Q^\infty )$) is shown in light red (resp.\ light blue). The arc $\mcl A_{\delta n^{1/4}}^\pbl(  \BB e^\infty ; Q^\infty )$ is shown in green. 
Here $K$, the total number of auxiliary percolation peeling clusters needed before we disconnect $\mcl A_{\delta n^{1/4}}^\pbl(  \BB e^\infty ; Q^\infty )$ from $\infty$, is equal to~$4$. The initial edges $  e_0^k$ for these clusters as well as the time-$(T^k+1)$ peeled edge for each cluster is denoted with a red dot. Each of these clusters intersects at most one percolation interface (blue) which crosses the ``annulus" between the peeling-by-layers clusters (here there are 3 such interfaces). The endpoints of the interfaces are shown with blue dots. Consequently, $K$ gives an upper bound for the number $\#\op{IP}_0^\infty( \delta  n^{1/4} , \delta^\zeta n^{1/4}   )$ of such interfaces (Lemma~\ref{lem-ip-time-count}) and hence an upper bound for the number of times that the percolation exploration path $\lambda^\infty$ can cross the annulus (Lemma~\ref{lem-crossing-ip}). 
}
\end{center}
\end{figure}

To prove Proposition~\ref{prop-uihpq-crossing}, we will first prove an estimate for $j=0$ then take a union bound.  See Figure~\ref{fig-uihpq-crossing} for an illustration of the proof. The idea is to iteratively grow percolation peeling clusters in the unbounded complementary connected component of $B_{\delta n^{1/4}}^\pbl( \BB e^\infty ; Q^\infty  )$ up to the first time $T^k$ that either the right outer boundary length of the current cluster exceeds $\delta^{3\zeta} n^{1/2}$ or the union of the previous clusters disconnects the arc $\mcl A_{\delta n^{1/4}}^\pbl( \BB e^\infty ; Q^\infty  )$ from $\infty$. The total number~$K$ of such clusters which we need to grow before disconnecting $\mcl A_{\delta n^{1/4}}^\pbl( \BB e^\infty ; Q^\infty )$ from $\infty$ provides an upper bound for $\#\op{IP}_j^\infty( \delta  n^{1/4} , \delta^\zeta n^{1/4}   )$ on the event that none of the clusters exits $B_{\delta^\zeta n^{1/4}}^\pbl( \BB e^\infty ;  Q^\infty  )$ before time~$T^k$ (Lemma~\ref{lem-ip-time-count}). 

On the other hand, if $\# \mcl A_{\delta n^{1/4}}^\pbl( \BB e^\infty ; Q^\infty  )$ is at most $\delta^{2-\zeta} n^{1/2}$, then peeling estimates give us an upper bound for $K$ (Lemma~\ref{lem-aux-count-bound}); and since distances along the boundary of the UIHPQ$_{\op{S}}$ can be bounded above in terms of boundary length, it is very unlikely that any of our clusters exits $B_{\delta^\zeta n^{1/4}}^\pbl( \BB e^\infty ; Q^\infty )$ before time~$T^k$ (Lemma~\ref{lem-excursion-dist-length}). Combining these facts and taking a union bound over $j\in [1,An^{3/4}]_{\lfloor \delta^M n^{3/4} \rfloor \BB Z}$ proves Proposition~\ref{prop-uihpq-crossing} 

We now proceed with the details. Our first task is to define the auxiliary percolation peeling clusters mentioned above.  Fix $\zeta \in (0,1/100)$, $n\in\BB N$, and $\delta \in (0,1)$.  We will define nested quadrangulations with infinite simple boundary $Q_0^1 \supset Q_0^2 \supset \dots$, root edges $e_0^k \in \mcl E(\bdy Q_0^k)$ for $k\in\BB N$, and $\sigma$-algebras $\mcl F_0^k$ for $k\in\BB N$ such that the conditional law of $(Q_0^k , e_0^k)$ given $\mcl F_0^k$ is that of a UIHPQ$_{\op{S}}$. Each of these UIHPQ$_{\op{S}}$'s will be equipped with a percolation peeling process started from the root edge and targeted at $\infty$, and the associated objects will be denoted by a superscript $k$. 

Let $Q_0^1$ be the unexplored quadrangulation for the peeling-by-layers process of $Q^\infty$ started from $\BB e^\infty$ grown up to radius $\delta n^{1/4}$, so that
\eqbn
\mcl E(Q_0^1)  = \mcl E\left( Q^\infty \setminus B_{\delta  n^{1/4}}^\pbl(\BB e^\infty ; Q^\infty )  \right) \cup \mcl A_{\delta  n^{1/4}}^\pbl(\BB e^\infty ; Q^\infty) .
\eqen
Let $\mcl F_0^1$ be the $\sigma$-algebra generated by this peeling-by-layers process grown up to radius $\delta n^{1/4}$. 
Also let $e_0^1$ be the edge of $\bdy Q_0^1$ immediately to the left of $\mcl A_{\delta  n^{1/4}}^\pbl(\BB e^\infty ; Q^\infty)$, equivalently the edge of $\bdy Q^\infty$ immediately to the left of $\bdy  B_{\delta  n^{1/4}}^\pbl(\BB e^\infty ; Q^\infty )\cap \bdy Q^\infty$.  By the Markov property of peeling, the conditional law of $(Q_0^1 , e_0^1)$ given $\mcl F_0^1$ is that of a UIHPQ$_{\op{S}}$. 

Inductively, suppose $k\in\BB N$ and we have defined a $\sigma$-algebra $\mcl F_0^k$ and an infinite rooted quadrangulation $(Q_0^k ,e_0^k)$ with simple boundary whose conditional law given $\mcl F_0^{k }$ is that of a UIHPQ$_{\op{S}}$.  Let $\{\dot Q_j^k \}_{j\in\BB N_0}$, $\{\ol Q_j^k\}_{j\in\BB N_0}$, and $\{\dot e_j^k\}_{j\in\BB N}$ be the clusters, unexplored quadrangulations, peeled edges, respectively, for the percolation peeling process of $(Q_0^{k } ,e_0^{k })$ targeted at $\infty$ (i.e., with white/black boundary conditions).  Also let $\{\mcl F_j^k\}_{j\in\BB N_0}$ be the filtration generated by $\mcl F_0^k$ and the clusters and peeling steps of this process up to time $j$. 

We note that the boundary conditions used to define the percolation peeling process in $Q_0^{k }$ are \emph{not} the boundary conditions inherited from the inclusion $Q_0^k \subset Q^\infty$; rather, for the purpose of defining this peeling process we ignore the colors of the quadrilaterals in $  \mcl F(Q^\infty) \setminus \mcl F(Q_0^k)$ and pretend that all quadrilaterals adjacent to edges of $\bdy Q_0^k$ lying to the left (resp.\ right) of $e_0^k$ are white (resp.\ black). 

Define the boundary length processes $X^{k,L} ,X^{k,R} , Y^{k,L} , Y^{k,R}$, and $W^k = ( W^{k,L}, W^{k,R})$ as in Definition~\ref{def-bdy-process} for this peeling process and define the stopping times
\eqb \label{eqn-length-time-def}
T^k := \min\left\{ j \in \BB N_0 : X_j^{k,R} \geq \delta^{3\zeta} n^{1/2} \: \op{or} \:   \bdy \ol Q_j^k \cap \mcl A_{\delta  n^{1/4}}^\pbl(\BB e^\infty ; Q^\infty )  =\emptyset\right\} .
\eqe 
Let $Q_0^{k+1} := \ol Q_{T^k}^k$, let $ \mcl F_0^{k+1} := \mcl F_0^k \vee \mcl F_{T^k}^k $, and let $e_0^{k+1}$ be the edge of $\bdy Q_0^{k+1}$ immediately to the right of $\bdy \dot Q_{T^k}^k \cap \bdy Q_0^{k+1}$, equivalently the edge of $\bdy Q_0^1$ immediately to the right of $ \dot Q_{T^k}^k$.  The Markov property of peeling implies that the conditional law of $(Q_0^{k+1} , e_0^{k+1})$ given $\mcl F_0^{k+1}$ is that of a UIHPQ$_{\op{S}}$. 

Let
\eqb \label{eqn-total-aux-def}
K = K^n(\delta) := \min\left\{ k \in \BB N : \bdy Q_0^{k+1} \cap \mcl A_{\delta  n^{1/4}}^\pbl(\BB e^\infty ; Q^\infty) = \emptyset \right\}
\eqe 
be the smallest $k\in\BB N$ for which the union of the clusters $\dot Q_{T^{k'}}^{k'}$ for $k' \leq k$ disconnects $B_{\delta  n^{1/4}}^\pbl(\BB e^\infty ; Q^\infty)$ from~$\infty$ in~$Q^\infty$. The reason for our interest in the integer~$K$ is contained in the following lemma.

\begin{lem} \label{lem-ip-time-count}
On the event 
\eqb \label{eqn-aux-subset-event}
\left\{ \dot Q_{T^k}^k \subset B_{\delta^\zeta n^{1/4}}^\pbl(\BB e^\infty , Q^\infty) ,\: \forall k \in [1,K ]_{\BB Z} \right\}  ,
\eqe
the cardinality of the set of interface paths from Definition~\ref{def-ip-crossing} satisfies
\eqb \label{eqn-ip-time-count}
 \#\op{IP}_0^\infty( \delta  n^{1/4} , \delta^\zeta n^{1/4}   ) \leq K +1.
\eqe 
\end{lem}
\begin{proof}
Essentially, this follows from the fact that the peeling exploration path associated with each of the percolation peeling processes $\{\dot Q_j^k\}_{j\in\BB N}$ is ``close" to an interface path (in the sense of Lemma~\ref{lem-peel-interface}), and this interface path cannot cross any of the interface paths in $\op{IP}_0^\infty( \delta  n^{1/4} , \delta^\zeta n^{1/4}   )$. We now explain the necessary geometric argument.

Let $N = \#\op{IP}_0^\infty( \delta  n^{1/4} , \delta^\zeta n^{1/4})$ and let $\rng \lambda_1 , \dots \rng \lambda_N$ be the elements of $\op{IP}_0^\infty( \delta  n^{1/4} , \delta^\zeta n^{1/4})$, ordered from left to right (we can order the paths in this manner since distinct interface paths cannot cross).  For $k \in \BB N$, also let $\rng\lambda_*^k : \BB N \rta \bdy \mcl E(Q_0^k)$ be the interface path associated with the percolation peeling clusters $\{\dot Q_j^k\}_{j\in\BB N_0}$ as in Lemma~\ref{lem-peel-interface}, i.e.\ the one which traces the outer boundary of the white cluster in $Q_0^k$ which contains all of the external white quadrilaterals.  

If~\eqref{eqn-aux-subset-event} occurs, the graph $\bigcup_{k=1}^K \dot Q_{T^k}^k$ is connected and disconnects $B_{\delta  n^{1/4}}^\pbl(\BB e^\infty ; Q^\infty) $ from $Q^\infty \setminus B_{\delta^\zeta  n^{1/4}}^\pbl(\BB e^\infty ; Q^\infty) $ in $Q^\infty$, so must cross each of the paths $\rng\lambda_m$ for $m\in [1,N]_{\BB Z}$ which does not hit a vertex of $\bdy Q^\infty$ lying to the right of $\BB e^\infty$.  Only $\rng\lambda_N$ can possibly hit such a vertex since the sets of vertices hit by the paths $\rng\lambda_m$ for $m\in [1,N]_{\BB Z}$ are disjoint (the paths are separated in $B_{\delta^\zeta n^{1/4}}^\pbl(\BB e^\infty , Q^\infty) \cap \dot Q_0^1$ by white clusters) and these paths are ordered from left to right.  Hence $\bigcup_{k=1}^K \dot Q_{T^k}^k$ contains a quadrilateral of $\dot Q_0^1$ lying strictly to the right of~$\rng\lambda_m$ for each $m\in [1,N-1]_{\BB Z}$. 

For $m\in [1,N-1]_{\BB Z}$, let $k_m$ be the smallest $k \in [1,K]_{\BB Z}$ for which $\dot Q_{T^k}^k $ contains a quadrilateral lying strictly to the right of $\rng\lambda_m$ and let $j_m$ be the smallest $j \in [2,T^k ]_{\BB Z}$ for which $\dot Q_{j }^{k_m}$ contains such a quadrilateral.

We will show that
\eqb \label{eqn-k-mono}
k_1 < \dots < k_{N-1}  
\eqe
 on the event~\eqref{eqn-aux-subset-event}, which will prove~\eqref{eqn-ip-time-count}.  Indeed, suppose by way of contradiction that $k_{m'} \leq k_m$ for some $1\leq m < m' \leq N-1$. It is clear from the definition of the growth processes $\{\dot Q_j^k\}_{j\in\BB N_0}$ for $k \in \BB N$ that $k_{m} \leq k_{m'}$, so $k_{m } = k_{m'}$.  Hence the times $j_m$ and $j_{m'}$ from the preceding paragraph are both defined with respect to $\{\dot Q_j^{k_m}\}_{m\in\BB N_0}$ and satisfy $j_m \leq j_{m'}$. 

The $j_{m'} $th peeled quadrilateral $\frk f(\ol Q_{j_{m'}-1}^k ; \dot e_{j_{m'} }^k )$ lies immediately to the right of $\rng\lambda_{m'}$, so must be black. Hence the right endpoint of the peeled edge $\dot e_{j_{m'}}^k$ is one of the vertices hit by $\rng\lambda_{m'}$ and in particular lies strictly to the right of $\rng \lambda_m$ in $Q_0^1$ (here we use that the paths $\rng\lambda_m$ and $\rng\lambda_{m'}$ are separated by a white cluster in $Q_0^1$, so cannot share a vertex).  By Lemma~\ref{lem-peel-interface}, there is a time $s \in \BB N_0$ such that $\rng\lambda_*^{k_m}(s)$ shares an endpoint with an edge of $\rng\lambda_m$ and $\rng\lambda_*^{k_m}([1,s ]_{\BB Z}) \subset \dot Q_{j_m }^{k_m}$.  

On the event~\eqref{eqn-aux-subset-event}, one has $\rng\lambda_*^{k_m}([1,s ]_{\BB Z})  \subset B_{\delta^\zeta n^{1/4}}^\pbl(\BB e^\infty , Q^\infty)$.  Since~$\rng\lambda_*^{k_m}(1)$ shares an endpoint with the root edge~$e_0^{k_m}$, which lies to the left of~$\rng \lambda_m$ in~$Q_0^1$, it follows that either $\rng\lambda_m$ contains the first edge of~$\rng\lambda_*^{k_m}$ or~$\rng\lambda_*^{k_m}$ crosses~$\rng\lambda_m$. In the former case, $\rng\lambda_m$ is a sub-path of $\rng\lambda_*^{k_m}$ so~$\rng\lambda_*^{k_m}$ cannot hit a vertex to the right of~$\rng\lambda_m$ before exiting $B_{\delta^\zeta n^{1/4}}^\pbl(\BB e^\infty , Q^\infty)$ and we arrive at a contradiction. The latter case is impossible since two percolation interfaces cannot cross.  We conclude that~\eqref{eqn-k-mono} holds. 
\end{proof}

In light of Lemma~\ref{lem-ip-time-count}, to prove Proposition~\ref{prop-uihpq-crossing} we need to prove an upper bound for $\BB P[K \geq N]$ and show that the event~\eqref{eqn-aux-subset-event} is very likely to occur. We start with the upper tail bound for $K$.

\begin{lem}
\label{lem-aux-count-bound}
Let $K$ be as in~\eqref{eqn-total-aux-def}. 
For each $\delta \in (0,1)$, there exists $n_*  = n_*(\delta,\zeta) \in \BB N$ such that for $n\geq n_*$ and $N\in\BB N$, 
\eqbn
\BB P\left[ K  >  N ,\,  \#\mcl A_{\delta  n^{1/4}}^\pbl(\BB e^\infty ; Q^\infty) \leq \delta^{2-\zeta} n^{1/2} \right]  \preceq \delta^{N (1-  2\zeta) }
\eqen
with the implicit constant depending only on $\zeta$. 
\end{lem}

The idea of the proof of Lemma~\ref{lem-aux-count-bound} is to show that it is unlikely that the right exposed boundary length process $X^{k,R}$ reaches $\delta^{3\zeta} n^{1/2}$ before the right covered boundary length process $Y^{k,R}$ reaches $\delta^{2-\zeta} n^{1/2}$, at which time the clusters $\{\dot Q_j^k\}_{j\in \BB N_0}$ disconnect $\mcl A_{\delta  n^{1/4}}^\pbl(\BB e^\infty ; Q^\infty) \cap \bdy Q_0^k$ from $\infty$ in $Q_0^k$, then multiply over all $k$.  This will be accomplished by means of the scaling limit result for the boundary length processes from Proposition~\ref{prop-stable-conv} and the following estimate for stable processes with no upward jumps.  We emphasize here that $\zeta$ is very small, so $\delta^{3\zeta}$ is much larger than $\delta^{2-\zeta}$.

\begin{lem} \label{lem-stable-hit}
Let $\alpha \in (1,2)$ and let $R$ be a totally asymmetric $\alpha$-stable process started from~$0$ with no upward jumps. 
For $\delta >0$, let $\tau_\delta = \inf\{t\geq 0 : R_t < -\delta \}$. Then for $\zeta \in (0,1)$ and $\delta \in (0,1)$, 
\eqbn
\BB P\left[ \sup_{s \in [0, \tau_\delta]} R_s \geq \delta^\zeta \right] \preceq \delta^{(\alpha-1)(1-\zeta)} 
\eqen 
with the implicit constant depending only on $\alpha,\zeta$. 
\end{lem} 
\begin{proof}
The probability that $R$ hits $-\delta$ before $\delta^\zeta$ is given by $W(\delta) / W(\delta+\delta^\zeta)$, where $W$ is the scale function for $R$ (see, e.g.,~\cite[Equation (4)]{hk-scale-function}). We have $W(x) = x^{\alpha-1}/\Gamma(\alpha)$ (see, e.g.,~\cite[Example 2, Section 2]{hk-scale-function}) which gives the lemma statement.  
\end{proof}

\begin{proof}[Proof of Lemma~\ref{lem-aux-count-bound}]
For $k\in \BB N$, let $S^k$ be the smallest $j \in \BB N_0$ for which the number of right covered edges satisfies $Y_{j}^{k,R} \geq \delta^{2-\zeta} n^{1/2}$. 
If $\#\mcl A_{\delta  n^{1/4}}^\pbl(\BB e^\infty ; Q^\infty) \leq \delta^{2-\zeta} n^{1/2}$, then since $\mcl A_{\delta  n^{1/4}}^\pbl(\BB e^\infty ; Q^\infty) \cap \bdy Q_0^k$ lies entirely to the right of $e_0^k$, it follows that $\mcl A_{\delta  n^{1/4}}^\pbl(\BB e^\infty ; Q^\infty) \cap \ol Q_{S^k}^k  =\emptyset$. 
Consequently, in this case the definition~\eqref{eqn-length-time-def} of $T^k$ implies that $T^k\leq S^k$ and if in fact $S^k = T^k$, it holds that $K \leq k $. 

If $T^k < S^k$, then by Definition~\ref{def-bdy-process},
\eqbn
\max_{j\in [0,S^k]_{\BB Z}} W_j^{k,R} \geq \max_{j\in [0,S^k]_{\BB Z}} X_j^{k,R} - \delta^{2-\zeta} n^{1/2} \geq (\delta^{3\zeta} - \delta^{2-\zeta}) n^{1/2} .
\eqen
Since the conditional law of $ W^{k,R}$ given $\mcl F_0^k$ converges weakly to that of a $3/2$-stable process in the Skorokhod topology as $n\rta\infty$ under an appropriate scaling limit (Proposition~\ref{prop-stable-conv}), it follows from Lemma~\ref{lem-stable-hit} (with $\alpha=3/2$, $\delta^{2-\zeta}$ in place of $\delta$, and $3\zeta/(2-\zeta)$  in place of $\zeta$) that there is an $n_*  = n_*(\delta,\zeta) \in \BB N$ such that for $n\geq n_*$ and $k\in\BB N$,  
\eqbn
\BB P\left[ K > k  \,|\, \mcl F_0^k \right] \leq    \BB P\left[ T^k < S^k  \,|\, \mcl F_0^k \right] 
\preceq \delta^{ (2-\zeta)(3/2 - 1) ( 1 - 3\zeta/(2-\zeta))      }  
= \delta^{1-2\zeta} .
\eqen
In particular, $\BB P[K > k \,|\, K > k-1] \preceq \delta^{1-2\zeta}$. Multiplying over $k \in [1,N ]_{\BB Z}$ yields the statement of the lemma.
\end{proof}

Our next lemma tells us that the event~\eqref{eqn-aux-subset-event} is very likely to occur when $n$ is large, $\delta$ is small, and $\#\mcl A_{\delta  n^{1/4}}^\pbl(\BB e^\infty ; Q^\infty) \leq \delta^{2-\zeta} n^{1/2}$.

\begin{lem} \label{lem-excursion-dist-length}
For each $  \delta  \in (0,1)$ there exists $n_* = n_*( \delta ,\zeta) \in \BB N$ such that for each $n\geq n_*$ and each $k\in \BB N$, 
\eqbn
\BB P\left[ \dot Q_{T^k}^k \not\subset B_{\delta^\zeta n^{1/4}}^\pbl(\BB e^\infty , Q^\infty) ,\, \#\mcl A_{\delta  n^{1/4}}^\pbl(\BB e^\infty ; Q^\infty) \leq \delta^{2-\zeta} n^{1/2} \right] = o_\delta^\infty(\delta)
\eqen
at a rate which is uniform for $n\geq n_*$ and $k\in\BB N$. 
\end{lem} 

We will prove that $\dot Q_{T^k}^k \subset B_{\delta^\zeta n^{1/4}}^\pbl(\BB e^\infty , Q^\infty)$ provided $\#\mcl A_{\delta  n^{1/4}}^\pbl(\BB e^\infty ; Q^\infty) \leq \delta^{2-\zeta} n^{1/2}$ and the regularity event defined in the following lemma occurs. 

\begin{lem} \label{lem-crossing-reg}
For $k\in \BB N$, let $G^k = G^{k,n}(\delta,\zeta)$ be the event that the following hold. 
\begin{enumerate}
\item \label{item-crossing-reg-time} $T^k \leq  n^{3/4}$. 
\item \label{item-crossing-reg-levy} For each $j_1,j_2 \in [0, 2 n^{3/4}]_{\BB Z}$ with $0 \leq j_2-j_1 \leq \delta n^{3/4}$,  
\eqbn
\#\mcl E\left( \bdy ( \dot Q_{j_2}^k \cap \ol Q_{j_1}^k) \setminus \bdy \ol Q_{j_1}^k  \right)  \leq \delta^{2/3 - \zeta} n^{1/2} .
\eqen
\item \label{item-crossing-reg-bdy}  For $j\in\BB N_0$, let $\beta_j^k$ be the boundary path of the UIHPQ$_{\op{S}}$ $\ol Q_j^k$ with $\beta_j^k(0)  = \dot e_j^k$ if $\dot e_j^\infty \in \bdy \ol Q_j^k$ or $\beta_j^k(0)$ equal to the leftmost edge of $\mcl A_{\delta n^{1/4}}^\pbl(\BB e^\infty ; Q^\infty) \cap \bdy \ol Q_j^k$ if $\dot e_j^\infty \notin \bdy \ol Q_j^k$.  For each $j\in [0, 2n^{3/4}]_{  \lfloor \delta  n^{3/4} \rfloor \BB Z}$ and each $i_1,i_2 \in \bdy \ol Q_j^k$ with $|i_1-i_2| \leq  2\delta^{3\zeta} n^{1/2}$, 
\eqbn
\op{dist}\left( \beta_j^k (i_1) , \beta_j^k(i_2) ; \ol Q_j^k \right) \leq  \frac12 \delta^{ \zeta} n^{1/4} .
\eqen
\item \label{item-crossing-reg-equicont} For each $j_1,j_2 \in [0,  2n^{3/4} ]_{\frac12\BB Z}$ with $0 \leq j_2-j_1 \leq \delta  n^{3/4}$, the percolation path increment $\lambda^k([j_1,j_2]_{\frac12\BB Z})$ is contained in a subgraph of $Q_0^k$ whose boundary relative to $Q_0^k$ has $Q_0^k$-graph distance diameter at most $\delta^{1/3-\zeta} n^{1/4}$. 
\end{enumerate}
For each $\delta \in (0,1)$, there exists $n_*  = n_*(\delta,  \zeta  ) \in \BB N$ such that for $n\geq n_*$ and $k\in\BB N$, 
\eqbn
\BB P\left[ (G^k)^c ,\,  \#\mcl A_{\delta  n^{1/4}}^\pbl(\BB e^\infty ; Q^\infty) \leq \delta^{2-\zeta} n^{1/2}  \right] =  o_{\delta}^\infty(\delta)
\eqen
at a rate which is uniform for $n\geq n_*$ and $k\in\BB N$. 
\end{lem}
\begin{proof}
We will bound the probability that each condition in the definition of $G^k$ fails separately. Throughout, all $o_\delta^\infty(\delta)$ errors are required to be uniform in $n$ and $k$.
\medskip 

\noindent\textit{Condition~\ref{item-crossing-reg-time}.} If $T^k > n^{3/4}$ and $\#\mcl A_{\delta  n^{1/4}}^\pbl(\BB e^\infty ; Q^\infty) \leq \delta^{2-\zeta} n^{1/2}$, then by~\eqref{eqn-bdy-process-inf} and the definition~\eqref{eqn-length-time-def} of $T^k$, 
\eqbn
-\delta^{2-\zeta} n^{1/2} - 3 \leq W_j^{k,R} \leq \delta^{3 \zeta} n^{1/2} ,\quad \forall j \in [0, n^{3/4}]_{\BB Z}.
\eqen 
 By Proposition~\ref{prop-stable-conv}, $W^{k,R}$ converges in law to a $3/2$-stable process in the Skorokhod topology as $n\rta\infty$ under an appropriate scaling limit.  By multiplying over $\lfloor t \ep^{-2/3} \rfloor$ i.i.d.\ increments of time length $\ep^{2/3}$, we see that for $\ep > 0$, the probability that a $3/2$-stable process stays in the $\ep$-neighborhood of the origin for $ t \gg \ep $ units of time is at most $c_0 \exp(- c_1 \ep^{-2/3})$ for appropriate constants $c_0,c_1 > 0$ depending on $t$.  If $\#\mcl A_{\delta  n^{1/4}}^\pbl(\BB e^\infty ; Q^\infty) \leq \delta^{2-\zeta} n^{1/2}$ and condition~\ref{item-crossing-reg-time} in the definition of $G^k$ fails to occur, then $|W_j^{k,R}| \leq \delta^{3\zeta} n^{1/2}$ for $j \in [0,n^{3/4}]_{\BB Z}$.  Hence there exists $n_0 = n_0(\delta,\zeta) \in \BB N$ such that for $n\geq n_0$ and $k\in\BB N$, the probability that this is the case is of order $o_\delta^\infty(\delta)$.
\medskip

\noindent\textit{Condition~\ref{item-crossing-reg-levy}.} 
In the notation of Definition~\ref{def-bdy-process},  
\eqbn
\#\mcl E\left( \bdy ( \dot Q_{j_2}^k \cap \ol Q_{j_1}^k) \setminus \bdy \ol Q_{j_1}^k  \right) \leq \#\mcl E\left( \bdy \dot Q_{j_2}^k \setminus \bdy Q_0^k \right) =   X^{\infty ,R}_{j_2} . 
\eqen
Hence Lemma~\ref{lem-X-tail} implies that for $\delta \in (0,1)$, there exists $n_1= n_1(\delta , \zeta ) \in \BB N$ such that for $n\geq n_1$, it holds with probability  $1-o_{\delta}^\infty(\delta)$ that $\#\mcl E\left( \bdy ( \dot Q_{j_2}^k \cap \ol Q_{j_1}^k) \setminus \bdy \ol Q_{j_1}^k  \right) \leq \frac12 \delta^{2/3 - \zeta} n^{1/2}$ for each $0\leq j_1 \leq j_2 \leq \delta n^{1/2}$. By the Markov property of the percolation peeling process and a union bound over $\lceil 2   \delta^{-1} \rceil$ approximately evenly spaced values of $j \in [1 , 2   n^{3/4}]_{\BB Z}$, we see that for $n\geq n_1$ and $k\in\BB N$, the probability that condition~\ref{item-crossing-reg-levy} in the definition of $G^k$ fails to hold is $ o_{\delta}^\infty(\delta)$.  
\medskip

\noindent\textit{Condition~\ref{item-crossing-reg-bdy}.} 
By Lemma~\ref{lem-uihpq-bdy-holder} and a union bound over all $j\in [0, 2  n^{3/4}]_{  \lfloor \delta n^{3/4} \rfloor \BB Z}$, we infer that for each $\delta \in (0,1)$, there exists $n_2  = n_2 (\delta, \zeta    ) \geq n_1$ such that for $n\geq n_2$ and $k\in\BB N$, the probability that condition~\ref{item-crossing-reg-bdy} in the definition of~$G^k$ fails to hold is of order $o_{\delta}^\infty(\delta)$.
\medskip

\noindent\textit{Condition~\ref{item-crossing-reg-equicont}.} 
By Proposition~\ref{prop-bdy-equicont-uihpq}, for $\delta \in (0,1)$, there exists $n_* = n_*(\delta , \zeta  )\geq n_2$ such that for $n\geq n_*$, the probability that condition~\ref{item-crossing-reg-equicont} in the definition of $G^k$ fails to hold is $o_{\delta}^\infty(\delta)$.
\end{proof}

\begin{proof}[Proof of Lemma~\ref{lem-excursion-dist-length}]
Fix $k\in\BB N$ and let $G^k = G^{k,n}(\delta,\zeta )$ be the event of Lemma~\ref{lem-crossing-reg}.  By Lemma~\ref{lem-crossing-reg}, it suffices to show that if $\delta$ is sufficiently small (depending only on $\zeta$) and
\eqb \label{eqn-excursion-dist-reg-event}
  G^k \cap \{\# \mcl A_{\delta n^{1/4} }^\pbl(\BB e^\infty; Q^\infty) \leq \delta^{2-\zeta } n^{1/2}\}
\eqe
occurs, then $\dot Q_{T^k}^k \subset B_{\delta^\zeta n^{1/4}}^\pbl(\BB e^\infty , Q^\infty)$. 

Suppose to the contrary that the event of~\eqref{eqn-excursion-dist-reg-event} occurs but $\dot Q_{T^k}^k \not\subset B_{\delta^\zeta n^{1/4}}^\pbl(\BB e^\infty ; Q^\infty)$.  Let $j_*$ be the smallest $j\in [0,T^k]_{\BB Z}$ for which $\dot e_{j_*}^k \notin B_{\delta^\zeta n^{1/4} }^\pbl(\BB e^\infty; Q^\infty)$ and let $j_*'$ be the smallest integer multiple of $\lfloor \delta  n^{3/4} \rfloor$ which is at least $j_*$.  By condition~\ref{item-crossing-reg-time} in the definition of $G^k$, we have $j_* \leq  n^{3/4}$ and $j_*' \leq 2 n^{3/4}$.  We seek a contradiction to condition~\ref{item-crossing-reg-bdy} in the definition of $G^k$ with this choice of $j_*'$. 

We first claim that if $\delta$ is sufficiently small (depending only on $\zeta$) then
\eqb \label{eqn-excursion-dist}
\op{dist}\left(  \dot e_{j_*'}^k ,  \mcl A_{\delta n^{1/4} }^\pbl(\BB e^\infty ; Q^\infty) \cap \bdy \ol Q_{j_*'}^k    ; Q_0^k \right) > \frac12 \delta^{ \zeta } n^{1/4} .
\eqe 
By condition~\ref{item-crossing-reg-equicont} in the definition of $G^k$, the percolation exploration path increment $\lambda^k([j_* , j_*']_{\frac12\BB Z})$ is contained in a subgraph $S$ of $Q_0^k$ whose boundary relative to $Q_0^k$ has $Q_0^k$-graph distance diameter at most $\delta^{1/3-\zeta} n^{1/4}$. This implies that the edges $\dot e_{j_*}^k$ and $\dot e_{j_*'}^k$ are also contained in $S$.  Hence to prove~\eqref{eqn-excursion-dist} it suffices to show that for small enough $\delta$ (depending only on $\zeta$) the set $S$ is disjoint from $B_{ \frac12 \delta^\zeta n^{1/4} }^\pbl(\BB e^\infty ; Q^\infty)$. To see this, we observe that if $S$ intersects $B_{ \frac12 \delta^{ 2 \zeta} n^{1/4} }^\pbl(\BB e^\infty ; Q^\infty)$, then in fact $\bdy S$ intersects $B_{ \frac12 \delta^{  \zeta } n^{1/4} }^\pbl(\BB e^\infty ; Q^\infty)$ so since $\bdy S$ has $Q_0^k$- and hence also $Q^\infty$-diameter at most $\delta^{1/3-\zeta} n^{1/4}$, which is smaller than $\frac12 \delta^\zeta n^{1/4}$ for small enough $\delta$, it follows that $\bdy S \subset B_{  \delta^{ \zeta} n^{1/4}}(\BB e^\infty ; Q^\infty)$ for small enough $\delta$. This implies that $S \subset B_{ \delta^{\zeta } n^{1/4}}^\pbl(\BB e^\infty ; Q^\infty)$ and in particular $e_{j_*}^k \in B_{  \delta^{  \zeta } n^{1/4}}^\pbl(\BB e^\infty ; Q^\infty)$, which contradicts our choice of $j_*$. 
 
We next argue that
\eqb \label{eqn-excursion-length}
\op{dist}\left(  \dot e_{j_*'}^k ,  \mcl A_{\delta n^{1/4} }^\pbl(\BB e^\infty ;    Q^\infty) \cap \bdy \ol Q_{j_*'}^k    ; \bdy \ol Q_{j_*'}^k \right) \leq    2 \delta^{3\zeta} n^{1/2}   , 
\eqe  
equivalently $X_{j_*'}^{k,R} \leq 2\delta^{3\zeta} n^{1/2}$. Indeed, this is immediate from condition~\ref{item-crossing-reg-levy} in the definition of $G^k$, the fact that $X_{j_*}^{k,R} \leq \delta^{3\zeta} n^{1/2}$ (which follows since $j_* \leq T^k$) and the estimate
\eqbn
X_{j_*'}^{k,R} \leq X_{j_*}^{k,R} +  \#\mcl E\left( \bdy ( \dot Q_{j_*'}^k \cap \ol Q_{j_*}^k) \setminus \bdy \ol Q_{j_*}^k  \right) .
\eqen

The relations~\eqref{eqn-excursion-dist} and~\eqref{eqn-excursion-length} contradict condition~\ref{item-crossing-reg-bdy} in the definition of~$G^k$, applied with $j = j_*'$, $\beta_j^k(i_1) = \dot e_{j_*'}^k$, and $\beta_j^k(i_2)$ an edge of $\mcl A_{\delta n^{1/4} }^\pbl(\BB e^\infty ;    Q^\infty) \cap \ol Q_{j_*'}^k$ which is at minimal $\bdy \ol Q_{j_*'}^k$-graph distance from~$\dot e_{j_*'}^k$.
\end{proof}

\begin{proof}[Proof of Proposition~\ref{prop-uihpq-crossing}]
By Lemmas~\ref{lem-ip-time-count},~\ref{lem-aux-count-bound}, and~\ref{lem-excursion-dist-length}, for each $\delta\in (0,1)$ there exists $n_* = n_*(\delta , \zeta)\in\BB N$ such that for $n\geq n_*$ and $N\in\BB N$ one has
\alb
&\BB P\left[ \#\op{IP}_0^\infty( \delta  n^{1/4} , \delta^\zeta n^{1/4}   ) > N  ,\, \#\mcl A_{\delta  n^{1/4}}^\pbl(\BB e^\infty ; Q^\infty) \leq \delta^{2-\zeta} n^{1/2} \right]  \\
&\qquad \leq \BB P[K > N -1 ,\, \#\mcl A_{\delta  n^{1/4}}^\pbl(\BB e^\infty ; Q^\infty) \leq \delta^{2-\zeta} n^{1/2} ] \\
&\qquad \qquad + \BB P\left[\exists k \in [1,N]_{\BB Z} \: \text{with} \: \dot Q_{T^k}^k \not\subset B_{\delta^\zeta n^{1/4}}^\pbl(\BB e^\infty , Q^\infty) ,\, \#\mcl A_{\delta  n^{1/4}}^\pbl(\BB e^\infty ; Q^\infty) \leq \delta^{2-\zeta} n^{1/2}  \right] \\
&\qquad \leq O_\delta(\delta^{(N-1)(1-2\zeta)} ) + o_\delta^\infty(\delta)  =  O_\delta(\delta^{(N-1) (1-2\zeta)} ) .
\ale
By stationarity of the law of the unexplored quadrangulations $\ol Q_j^\infty$ and the restriction of the percolation peeling process to these quadrangulations, we can now take a union bound over all $j \in [0, A  n^{3/4}]_{ \lfloor \delta^M n^{3/4} \rfloor \BB Z}$ to obtain the statement of the proposition.
\end{proof}

\subsection{Boundary lengths of peeling-by-layers clusters}
\label{sec-pbl-bdy}

Proposition~\ref{prop-uihpq-crossing} is not by itself sufficient to deduce Proposition~\ref{prop-fb-crossing} since the estimate of the former proposition only holds for peeling-by-layers clusters whose outer boundary length $\# \mcl A_{\delta  n^{1/4}}^\pbl(\dot e_j^\infty ; \ol Q^\infty_j) $ is at most $\delta^{2-\zeta} n^{1/2}$.  In this subsection, we will prove that in the setting of a free Boltzmann quadrangulation with simple boundary, this outer boundary length is very unlikely to be larger than  $\delta^{2-\zeta} n^{1/2}$.

\begin{lem} \label{lem-pbl-bdy-length}
For each $\zeta \in (0,1)$ and each $\delta  \in (0,1)$, there exists $\el_* = \el_*(\delta  ,\zeta) \in \BB N$ such that the following holds for each $\el \geq \el_*$.  Let $(Q , \BB e)$ be a free Boltzmann quadrangulation with simple boundary of perimeter $2\el$.  Also let $\BB e_* \in \mcl E(\bdy Q)$ be chosen in a deterministic manner, let $B_{\delta \el^{1/2}}^\pbl(\BB e ; Q)$ be the radius-$  \delta \el^{1/2} $ peeling-by-layers cluster of $Q$ started from $\BB e$ and targeted at $\BB e_*$, and let $\mcl A_{\delta \el^{1/2}}^\pbl(\BB e ; Q) :=  \mcl E(\bdy B_{\delta \el^{1/2}}^\pbl(\BB e ; Q) \setminus \bdy Q)$, as in Section~\ref{sec-crossing-perturb}.  Then 
\eqb \label{eqn-pbl-bdy-length}
\BB P\left[ \# \mcl A_{\delta \el^{1/2}}^\pbl(\BB e ; Q) > \delta^{2-\zeta} \el  \right] = o_\delta^\infty(\delta) 
\eqe 
at a rate which is uniform for $\el\geq \el_*$. 
\end{lem}

Lemma~\ref{lem-max-bdy-length}---which gives an upper bound for boundary lengths in an arbitrary peeling process---tells us that $\# \mcl A_{\delta \el^{1/2}}^\pbl(\BB e ; Q)  \leq \delta^{2-\zeta} \el $ with high probability, but we need a stronger probabilistic estimate since we eventually want to take a union bound over peeling-by-layers clusters in a large number of different unexplored quadrangulations. For this purpose we will need to use the relationship between peeling-by-layers clusters and metric balls.

The idea of the proof of Lemma~\ref{lem-pbl-bdy-length} is to show that the number of edges in the $Q$-metric ball $B_{2 \delta \el^{1/2}   }(\BB e ; Q )$ is very unlikely to be larger than $\delta^{4-\zeta} \el^2$; and that if the arc $\mcl A_{\delta \el^{1/2}}^\pbl(\BB e ; Q)$ contains too many edges, then this $Q$-metric ball contains more than $\delta^{4-\zeta} \el^2$ edges with uniformly positive probability. Both of these statements are proven using the GHPU convergence of free Boltzmann quadrangulations with simple boundary toward the free Boltzmann Brownian disk~\cite[Theorem~1.4]{gwynne-miller-simple-quad} and estimates for the metric, area measure, and boundary length measure of the free Boltzmann Brownian disk which are similar to those found in~\cite[Section~3.2]{gwynne-miller-gluing}. In particular, we will make crucial use of~\cite[Corollary 6.2]{legall-geodesics}, which gives an upper bound for the area of a metric ball in the Brownian map.  

We start by estimating the maximal area of a metric ball in a random-area Brownian map. This estimate will be transferred to an estimate for the area of a metric ball in a free Boltzmann Brownian disk momentarily. 

Following~\cite{tbm-characterization}, we let $\BB m_{\op{BM}}^2$ be the infinite measure on doubly marked Brownian maps $(M , d_M ,\mu_M , z_0 , z_1)$, which is obtained as follows: first ``sample" an area $A$ from the infinite measure $a^{-3/2} \, da$; then, conditional on~$A$, sample a Brownian map $(M,d_M ,\mu_M)$ with area $A$ (equipped with its natural metric and area measure); finally, conditional on $(M,d_M ,\mu_M)$, sample two conditionally independent points $z_0,z_1\in M$ from $\mu_M$ (normalized to be a probability measure). 

\begin{lem}
\label{lem-bm-ball-sup}
Let $\frk a > 0$ and let $(M , d_M ,\mu_M , z_0 , z_1)$ be sampled from the infinite measure $\BB m_{\op{BM}}^2$ conditioned on $\{\mu_M(M) \geq \frk a\}$ (which is a finite measure). 
For each $p \geq 1$ and each $\zeta \in (0,2/p)$, one has 
\eqbn
\BB E\left[ \left( \sup_{\delta > 0} \sup_{z\in M} \frac{\mu_M(B_\delta(z;d_M))}{\delta^{4-\zeta}} \right)^p \right] \preceq \frk a^{p\zeta/4}
\eqen
with universal implicit constant. 
\end{lem}
\begin{proof}
Let $A$ be sampled from the law of $\mu_M(M)$ (conditioned on $\{\mu_M(M) \geq \frk a\}$), which is equal to $2\frk a^{1/2} a^{-3/2}  \BB 1_{(a \geq \frk a)} \, da$ and let $(M ,d_M^1,\mu_M^1, z_0  , z_1 )$ be a unit area doubly marked Brownian map, independent from $A$. By the scaling property of the Brownian map,
\eqbn
(M , d_M,\mu_M , z_0 , z_1) := (M  , A^{1/4} d_M^1 , A \mu_M^1 , z_0 , z_1)
\eqen
has the law described in the statement of the lemma. By~\cite[Corollary~6.2]{legall-geodesics} (and H\"older's inequality to deal with non-integer values of $p$), for each $p\geq 1$ and each $\zeta \in (0,1)$,
\eqbn
\BB E\left[ \left( \sup_{\delta > 0} \sup_{z\in M} \frac{\mu_M^1(B_\delta(z;d_M^1))}{\delta^{4-\zeta}} \right)^p \right] < \infty. 
\eqen
If $\zeta \in (0,2/p)$, then $\BB E[A^{p\zeta/4}] \preceq \frk a^{p\zeta/4}$. Hence 
\alb
\BB E\left[ \left( \sup_{\delta > 0} \sup_{z\in M} \frac{\mu_M(B_\delta(z;d_M))}{\delta^{4-\zeta}} \right)^p \right] 
 =  \BB E\left[ A^{p\zeta/4} \right] \BB E\left[  \left( \sup_{\delta  > 0} \sup_{z\in M} \frac{   \mu_M^1(B_{ \delta } (z;  d_M^1 ))}{   \delta^{4-\zeta}} \right)^p \right] 
\preceq \frk a^{p\zeta/4}, 
\ale
as required.
\end{proof}

Next we estimate the maximal area of a metric ball in a free Boltzmann Brownian disk.

\begin{lem} \label{lem-bd-ball-sup}
Let $(H , d , \mu )$ be a free Boltzmann Brownian disk with unit perimeter, equipped with its natural metric and area measure. 
For each $p\geq 1$ and each $\zeta \in (0,2/p)$, one has
\eqb  \label{eqn-bd-ball-sup}
\BB E\left[ \left( \sup_{\delta > 0} \sup_{z\in H} \frac{\mu (B_\delta(z;d))}{\delta^{4-\zeta}} \right)^p \right] < \infty .
\eqe 
\end{lem}
\begin{proof}
Let $(M , d_M ,\mu_M , z_0 , z_1)$ be sampled from the infinite measure $\BB m_{\op{BM}}^2$ conditioned on $\{\mu_M(M) \geq 1 \}$. The idea of the proof is to show using the results of~\cite{tbm-characterization} that with positive probability, there is a subset of $M$ (in particular, a complementary connected component of a metric ball) which, when equipped with the internal metric of $d_M$ and the restriction of $\mu_M$, has the law of a free Boltzmann Brownian disk with random boundary length bounded away from~$0$ and $\infty$. The statement of the lemma will then follow from Lemma~\ref{lem-bm-ball-sup} and a scaling argument. The argument is slightly more subtle than one might expect since $M$ is conditioned on an event involving its area but we need to find a Brownian disk without any conditioning on its area.
  
To lighten notation set $R := d_M(z_0,z_1)$. 
Let $\{\Gamma_r \}_{r \in [0,R] }$ be the metric net of $M$ from $z_0$ to $z_1$, i.e.\ $\Gamma_r$ for $r \in [0,R)$ is the union over all $s \in [0,r]$ of the boundary of the connected component of $M\setminus B_r(z_0;d_M)$ containing $z_1$.
 
As explained in~\cite{tbm-characterization}, there is a cadlag process $L = \{L_r\}_{r\in [0,R]}$ with no upward jumps which can be interpreted as the length of the outer boundary of $\Gamma_r$ at each time $r$ (this process coincides with the $\sqrt{8/3}$-LQG length measure if we identify $M$ with an instance of the Brownian sphere as in~\cite[Corollary~1.5]{lqg-tbm2}).  Let $\{r_j\}_{j\in\BB N }$ be an enumeration of the times when $L$ makes a downward jump, chosen in a manner which depends only on~$L$.  At each time $r_j$, the metric net disconnects a bubble $U_{r_j}$ from $z_1$, and the collection of all such bubbles is precisely the set of connected components of $M\setminus \Gamma_R$. The reason why we know that there is exactly one bubble disconnected from $z_1$ at each of the times $r_j$ is that the metric net is equivalent (as an increasing family of topological spaces) to the so-called \emph{L\'evy net} associated with $L$, and by definition this L\'evy net forms exactly one bubble each time $L$ has a downward jump; see~\cite[Sections  3.3 and 4.2]{tbm-characterization}. 

For $j\in\BB N$, let $d_{U_{r_j}}$ be the internal metric of $d_M$ on $U_{r_j}$ and let $\mu_{U_{r_j}} := \mu_M|_{U_{r_j}}$. 
Also let $\Delta L_{r_j} :=    \lim_{s\rta r_j^-} L_s - L_{r_j}$ be the magnitude of the downward jump of $L$ at time $r_j$, which gives the boundary length of $U_{r_j}$. 
By~\cite[Proposition~4.4]{tbm-characterization}, if we condition on $L$, then the conditional law of the collection of metric measure spaces $(U_{r_j} , d_{U_{r_j}} , \mu_{U_{r_j}})$ for $j\in\BB N$ is that of a collection of independent free Boltzmann Brownian disks with respective boundary lengths $\Delta L_{r_j}$, conditioned on the event that the sum of their areas is at least~$1$. 
 
Fix $\frk l_2 > \frk l_1 > 0$ to be chosen later. Let $T$ be the smallest $r\geq 0$ for which the $\mu_M$-mass of the region disconnected from $z_1$ by $\Gamma_r$ is at least $1$ and let $S$ be the smallest time $r \geq T$ for which the downward jump $\Delta L_r$ lies in $[\frk l_1,\frk l_2]$, or $S = R$ if no such $r$ exists. Almost surely, $\mu_M(M)  > 1$, the boundary of the metric net converges to $z_1$ in the Hausdorff distance as $r\rta R^-$, and the metric net $\{\Gamma_r \}_{r \in [0,d_M(z_0,z_1)] }$ disconnects a bubble from $z_1$ in every positive-length interval of times.  Hence for an appropriate universal choice of~$\frk l_1$ and~$\frk l_2$, it holds that $\BB P\left[ S < R \right] \geq \frac12$.  Henceforth assume we have chosen $\frk l_1$ and $\frk l_2$ in this manner.

Note that $S$ is one of the jump times $r_j$ on the event $\{S < R\}$. The above description of the law of the bubbles $\{U_{r_j}\}_{j\in \BB N_0}$ implies that if we condition on $L$ and $\{(U_{r_j} , d_{U_{r_j}} , \mu_{U_{r_j}}) : r_j \leq T\}$ then on the event $\{S  <R\}$ the conditional law of $(U_S , d_{U_S}  , \mu_{U_S})$ is that of a free Boltzmann Brownian disk with boundary length $\Delta L_S$ (note that the sum of the areas of the bubbles before time $T$ is at least $1$). Each $d_{U_S}$-metric ball is contained in a $d_M$-metric ball of the same radius. By Lemma~\ref{lem-bm-ball-sup} and our choice of $\frk l_1$ and $\frk l_2$, 
\eqb \label{eqn-bd-ball-sup0}
\BB E\left[ \left( \sup_{\delta > 0} \sup_{z\in U_S} \frac{\mu_{U_S} (B_\delta(z;d_{U_S}))}{\delta^{4-\zeta}} \right)^p \,\big|\, S < R \right] \leq 2 \BB E\left[ \left( \sup_{\delta > 0} \sup_{z\in M} \frac{\mu_M(B_\delta(z;d_M))}{\delta^{4-\zeta}} \right)^p \right]    < \infty .
\eqe 
By scale invariance of the Brownian disk, if we set
\eqb
(H , d , \mu ) := (\ol U_S , (\Delta L_S)^{-1/2} d_{U_S} , (\Delta L_S)^{-2} \mu_{U_S}) 
\eqe 
with $\ol U_S$ the closure and $d_{U_S}$ extended to $\bdy \ol U_S$ by continuity, then the conditional law of $(H,d,\mu)$ given $\{S < R\}$ is that of a free Boltzmann Brownian disk with unit boundary length. Hence~\eqref{eqn-bd-ball-sup0} implies~\eqref{eqn-bd-ball-sup}.
\end{proof}

Together with Lemma~\ref{lem-bd-ball-sup}, the following lemma is the other key input in the proof of Lemma~\ref{lem-pbl-bdy-length}.

\begin{lem} \label{lem-bdy-interval-mass}
Let $(H , d , \mu ,\xi )$ be a free Boltzmann Brownian disk with unit perimeter, equipped with its natural metric, area measure, and 1-periodic boundary path. 
For each $\ep , \zeta \in (0,1)$, there exists $c = c(\ep , \zeta) >0$ such that 
\eqbn
\BB P\left[ \mu\left( B_\delta\left(\xi([u_1,u_2]) ; d \right) \right)  \geq c \delta^{2+\zeta} (u_2-u_1) ,\: \forall \delta \in (0,1), \: \forall \: u_1,u_2 \in \BB R \: \op{with} \: u_1<u_2 \right] \geq 1-\ep .
\eqen
\end{lem}
\begin{proof}
This is an immediate consequence of~\cite[Lemma~3.4]{gwynne-miller-gluing}.
\end{proof}

\begin{proof}[Proof of Lemma~\ref{lem-pbl-bdy-length}]
By~\cite[Theorem~1.4]{gwynne-miller-simple-quad}, the free Boltzmann quadrangulation with simple boundary, appropriately rescaled, converges in law to the unit boundary length free Boltzmann Brownian disk in the GHPU topology as $\el\rta \infty$. 
From this convergence and Lemma~\ref{lem-bd-ball-sup}, we infer that there exists $\el_1 \in \BB N$ such that for $c >0$, $\el\geq \el_1$, and $\delta \in [\delta_0 ,1)$, 
\eqb \label{eqn-fb-bdy-mass-upper}
\BB P\left[ \# \mcl E\left( B_{2 \delta \el^{1/2}  + 2 }(\BB e ; Q ) \right)  \geq c\delta^{4 - \zeta/2} \el^2 \right] = o_\delta^\infty(\delta) ,
\eqe 
at a rate depending on $c$ and $\zeta$. We will now use Lemma~\ref{lem-bdy-interval-mass} to argue that the event in~\eqref{eqn-fb-bdy-mass-upper} occurs with uniformly positive conditional probability if we condition on the event that $ \# \mcl A_{\delta \el^{1/2}}^\pbl(\BB e ; Q) > \delta^{2-\zeta} \el$, which will give our desired estimate~\eqref{eqn-pbl-bdy-length}. 

Let $\ol Q^\pbl $ be the radius-$\delta \el^{1/2}$ unexplored region for the peeling-by-layers process of $Q$ starting from $\BB e$ and targeted at $\BB e_*$ and let $\mcl F^\pbl $ be the $\sigma$-algebra generated by the clusters and peeling steps up to radius $\delta \el^{1/2}$. 
By the Markov property of peeling (see~\cite[Lemma~4.1]{gwynne-miller-simple-quad}), the conditional law of $\ol Q^\pbl $ given $\mcl F^\pbl $ is that of a free Boltzmann quadrangulation with given perimeter. The edge set $ \mcl A_{\delta \el^{1/2}}^\pbl(\BB e ; Q) $ is a connected arc of~$\bdy \ol Q^\pbl $. 
 
By the aforementioned GHPU scaling limit result for free Boltzmann quadrangulations with simple boundary applied to $\ol Q^\pbl $ together with Lemma~\ref{lem-bdy-interval-mass} (applied with $\zeta/2$ in place of $\zeta$) we infer that there is a $c  = c(\zeta) >0$ and an $\el_* \geq \el_1$ such that for $\el \geq \el_*$, 
\eqb \label{eqn-fb-bdy-mass-pos}
\BB P\left[  \# \mcl E\left( B_{  \delta \el^{1/2}} \left(    \mcl A_{\delta \el^{1/2}}^\pbl(\BB e ; Q)  ;     \ol Q^\pbl   \right)  \right)   \geq c \delta^{ 4 - \zeta/2} \el^2 \,\big|\,    \# \mcl A_{\delta \el^{1/2}}^\pbl(\BB e ; Q) > \delta^{2-\zeta} \el     \right] \geq \frac12 .
\eqe 
By~\eqref{eqn-filled-ball-contain}, $B_{\delta \el^{1/2}}^\pbl(\BB e ; Q)  \subset B_{\delta \el^{1/2} + 2}(\BB e ; Q)$, whence 
\eqbn
B_{  \delta \el^{1/2}} \left(    \mcl A_{\delta \el^{1/2}}^\pbl(\BB e ; Q)  ;     \ol Q^\pbl   \right) 
\subset B_{  2 \delta \el^{1/2} +2} \left(    \BB e ; Q  \right) .
\eqen
Hence~\eqref{eqn-fb-bdy-mass-upper} and~\eqref{eqn-fb-bdy-mass-pos} together imply that
\eqbn
\BB P\left[   \# \mcl A_{\delta \el^{1/2}}^\pbl(\BB e ; Q) > \delta^{2-\zeta} \el  \right] 
\leq \frac{ \BB P\left[ \# \mcl E\left( B_{2 \delta \el^{1/2}  + 2 }(\BB e ; Q ) \right)  \geq c \delta^{4 - \zeta/2} \el^2 \right]  }{ \BB P\left[ \# \mcl E\left( B_{2 \delta \el^{1/2}  + 2 }(\BB e ; Q ) \right)  \geq c \delta^{4 - \zeta/2} \el^2  \,|\,  \# \mcl A_{\delta \el^{1/2}}^\pbl(\BB e ; Q) >  \delta^{2-\zeta} \el   \right] } = o_\delta^\infty(\delta) ,
\eqen
as required.
\end{proof}

\subsection{Proof of Proposition~\ref{prop-fb-crossing}}
\label{sec-crossing-proof}

Throughout this subsection, we assume we are in the setting of Proposition~\ref{prop-fb-crossing}. To deduce the proposition from Proposition~\ref{prop-uihpq-crossing}, it remains to transfer our estimates from the UIHPQ$_{\op{S}}$ to the free Boltzmann quadrangulation with simple boundary $Q^n$; and to transfer for an estimate for crossings between boundaries of peeling-by-layers clusters in $\ol Q_j^n$ to an estimate for crossings between boundaries of filled $Q^n$-metric balls centered at general vertices of $Q^n$. Transferring from the UIHPQ$_{\op{S}}$ to $Q^n$ can be done via a straightforward local absolute continuity estimate, see Lemma~\ref{lem-fb-crossing0} below. 

Transferring from peeling-by-layers clusters to filled metric balls is more involved. The main difficulty is that the peeling-by-layers cluster $B_{\delta^\zeta n^{1/4}}^\pbl(\dot e_j^n ; \ol Q_j^n)$ can in principle be much larger than $B_{\delta^\zeta n^{1/4}}^\bullet(\dot e_j^n ; Q^n)$. The reason is that there can be large regions of $\ol Q_j^n$ which are disconnected from the target edge $\BB e_*^n$ by subsets of $\ol Q_j^n$ with small diameter which are close to $\dot e_j^n$. Such regions will be part of $B_{\delta^\zeta n^{1/4}}^\pbl(\dot e_j^n ; \ol Q_j^n)$ but not $B_{\delta^\zeta n^{1/4}}^\bullet(\dot e_j^n ; Q^n)$. 
To avoid this difficulty, for a given $v\in\mcl V(Q^n)$ we will make a very careful choice of $j\in\BB N$ with the property that $\dot e_j^n$ is ``close" to $v$ and $B_{\delta^\zeta n^{1/4}}^\pbl(\dot e_j^n ; \ol Q_j^n)$ can  only intersect $\bdy \ol Q_j^n$ in a certain arc of $\bdy \ol Q_j^n$ which is contained in a set whose boundary has small $Q^n$-diameter. 

In order to choose this special value of $j$, we will consider the hitting times by the percolation path $\lambda^n$ of several different $Q^n$-graph distance balls centered at $v$. More precisely, we fix a small constant $\zeta \in (0,1)$ and for $k\in\BB N$, we let
\eqb \label{eqn-crossing-scale-def}
\delta_k := 2^{-(2/\zeta)^k}  .
\eqe
Note that $\delta_{k-1} = \delta_k^{\zeta/2}$. 
 
To start off the proof of Proposition~\ref{prop-fb-crossing}, let us first record what we get from Proposition~\ref{prop-uihpq-crossing} and Lemma~\ref{lem-pbl-bdy-length}. 
In what follows, for $n\in\BB N$ and $0 <\alpha_0 < \alpha_1$, we define the stopping time $I_{\alpha_0,\alpha_1}^n$ for the percolation peeling process as in~\eqref{eqn-discrete-close-time}.
  
\begin{lem} \label{lem-fb-crossing0}
Let $0 < \alpha_0 < \alpha_1$ and $A>0$. 
For each $m  , K \in \BB N$ with $K \geq 2m$, there exists $n_* = n_*(K , m ,  \zeta  ,  \alpha_0 , \alpha_1 , A) \in \BB N $ such that for $n\geq n_*$, the probability that the following is true is at most $O_K(\delta_{\lfloor K /2 \rfloor}^{6(1-2\zeta)  - 5 })$, uniformly for all $n\geq n_*$. 
We have $I_{\alpha_0,\alpha_1}^n \leq A n^{3/4}$ and there exists $k \in [K - m  ,K]_{\BB Z}$ and $j \in [0,  I_{\alpha_0,\alpha_1}^n  ]_{ \lfloor \delta_k^5 n^{3/4} \rfloor \BB Z  }$ such that $\# \op{IP}_j^n\left( 2 \delta_k n^{1/4} , \delta_k^\zeta n^{1/4}   \right)  > 7$ and $B_{\delta^\zeta n^{1/4}}^\pbl(\dot e_j^n ; \ol Q_j^n) \cap \ol Q_{I_{\alpha_0,\alpha_1}}^n = \emptyset$. 
\end{lem}

Note that the exponent $6 (1-2\zeta) - 5$ is positive provided $\zeta$ is chosen to be sufficiently small. We henceforth assume that this is the case. 

\begin{proof}[Proof of Lemma~\ref{lem-fb-crossing0}]
For each $k\in\BB N$, the event that there exists $j \in [0,  I_{\alpha_0,\alpha_1}^n  ]_{ \lfloor \delta_k^5 n^{3/4} \rfloor \BB Z  }$ such that $\# \op{IP}_j^n\left( \delta_k n^{1/4} , \delta_k^\zeta n^{1/4}   \right)  >7$ and $B_{\delta^\zeta n^{1/4}}^\pbl(\dot e_j^n ; \ol Q_j^n) \cap \ol Q_{I_{\alpha_0,\alpha_1}}^n = \emptyset$ is measurable with respect to the $\sigma$-algebra generated by $\dot Q_{I_{\alpha_0,\alpha_1}^n}$ and the percolation peeling process up to time $I_{\alpha_0,\alpha_1}^n$ (here we use that peeling-by-layers clusters are generated by a peeling process). 
By Proposition~\ref{prop-uihpq-crossing} (applied with $N = 7$ and $M = 5$) and Lemma~\ref{lem-perc-rn} (applied at the time $I_{\alpha_0,\alpha_1}^n$), for each $k\in\BB N$, we can find $n_*  = n_* (k , \zeta ,  \alpha_0 , \alpha_1 , A) \in \BB N $ such that for $n \geq n_* $, the probability that $I_{\alpha_0,\alpha_1}^n \leq A n^{3/4} $ and there exists  $j \in [0,  I_{\alpha_0,\alpha_1}^n  ]_{ \lfloor \delta_k^5 n^{3/4} \rfloor \BB Z }$ satisfying the conditions in the statement of the lemma plus the additional condition that 
\eqb \label{eqn-fb-crossing-ball-length}
\# \mcl A_{\delta_k n^{1/4}}^\pbl\left( \dot e_j^n ; \ol Q_j^n   \right)  \leq \delta_k^{2-\zeta} n^{1/2}
\eqe 
 is at most $O_k(\delta_k^{6(1-2\zeta) - 5})$, at a rate which is uniform for $n\geq n_*$.  
 
By Lemma~\ref{lem-pbl-bdy-length} and a union bound, by possibly increasing $n_*$ we can arrange that for $n\geq n_*$, the probability that there exists $j \in [0, An^{3/4}]_{\lfloor \delta_k^5 n^{3/4} \rfloor \BB Z  }$ such that~\eqref{eqn-fb-crossing-ball-length} fails to hold is of order $o_k^\infty(\delta_k)$, uniformly for $n\geq n_*$. 
We now obtain the lemma statement by means of a union bound over all $k\in [K - m,K]_{\BB Z}$. 
\end{proof}

We now combine Lemma~\ref{lem-fb-crossing0} and the continuity condition of Proposition~\ref{prop-bdy-equicont} to get an event on which (as we will see below) the condition of Proposition~\ref{prop-fb-crossing} holds.

\begin{lem} \label{lem-fb-crossing1}
For each $\ep \in (0,1)$, there exists $0  <\alpha_0 < \alpha_1 $, $A > 0$, and $m_*   \in \BB N$ (depending only on $\ep  $) such that for each $m \geq m_*$, there exists $K_* = K_*(m,\ep,\zeta) \geq m$ such that for each $K \geq K_*$, there exists $n_* = n_*(K,m,\ep , \zeta ) \in \BB N$ such that for each $n\geq n_*$, the following is true. 
Let $E^n =E^n(K,m,\ep , \zeta )$ be the event that the following is true.
\begin{enumerate}
\item \label{item-fb-crossing-time} $I_{\alpha_0,\alpha_1}^n \leq A n^{3/4}$ and the unexplored quadrangulation $\ol Q_{I_{\alpha_0,\alpha_1}^n}^n$ has internal graph distance diameter at most $\frac{1}{8} \ep n^{1/4}$. 
\item \label{item-fb-crossing-cont} If $k \in [K-m,K]_{\BB Z}$ and $j \in \lfloor \delta_k^5 n^{3/4} \rfloor \BB N_0$ satisfies  $\op{dist}(\dot e_j^n  , \BB e_*^n ; Q^n) \geq \frac12 \ep n^{1/4}$, then $\lambda^n([j   , j + \delta_k^5 n^{3/4}]_{\frac12\BB Z})$ is contained in a subgraph $S$ of $\ol Q_j^n$ such that $\dot e_j^n \in \bdy S$ and the $\ol Q_j^n$-graph distance diameter of $\bdy S$ is at most $ \delta_k^{5/4} n^{1/4}$. The same is true with $A  m^{-1}$ used in place of $\delta_k^5$ throughout. 
\item \label{item-fb-crossing-endpoint} If $k$ and $j$ are as as in condition~\ref{item-fb-crossing-cont}, then $\# \op{IP}_j^n\left( 2 \delta_k n^{1/4} , \delta_k^\zeta n^{1/4}   \right) \leq 7$. 
\end{enumerate}
Then $\BB P[E^n] \geq 1-\ep$. 
\end{lem} 

We comment briefly on the definition of $E^n$. 
Condition~\ref{item-fb-crossing-endpoint} is the most important condition in the definition; this condition together with Lemma~\ref{lem-crossing-ip} will eventually lead to the bound for crossings of annuli between filled metric balls in Proposition~\ref{prop-fb-crossing}.
The bounds for $j \in \lfloor \delta_k^5 n^{3/4} \rfloor \BB N_0$ in condition~\ref{item-fb-crossing-cont} are needed for the following purpose.
Condition~\ref{item-fb-crossing-endpoint} only holds for $j \in \lfloor \delta_k^5 n^{3/4} \rfloor \BB N_0$, so we need condition~\ref{item-fb-crossing-cont} to approximate an edge $\dot e_j^n$ for a general $j\in \BB N_0$ by an edge of the form $\dot e_{j'}^n$ for $j' \in \lfloor \delta_k^5 n^{3/4} \rfloor \BB N_0$. 

We now explain the purpose of condition~\ref{item-fb-crossing-time} and the $Am^{-1}$-variant of condition~\ref{item-fb-crossing-cont}. In the proof of Lemma~\ref{lem-good-ball} below, we will consider a vertex $v\in Q^n$ which is at macroscopic distance from $\bdy Q^n$ and for $k\in [K-m,K]_{\BB Z}$ we will let $j_k$ be the first time when $\lambda^n$ enters $B_{\delta_k n^{1/4}}(v ; Q^n)$. The times $j_k$ all have to be less than or equal to $I_{\alpha_0,\alpha_1}^n$ (otherwise, $v$ would have to be close to $\BB e_*^n \in \bdy Q^n$). Condition~\ref{item-fb-crossing-time} therefore tells us that $j_K - j_{K-m} \leq I_{\alpha_0,\alpha_1}^n \leq A n^{3/4}$, so there must be some $k \in [K-m,K]_{\BB Z}$ for which $j_k - j_{k-1} \leq A m^{-1} n^{3/4}$. Condition~\ref{item-fb-crossing-cont} for $A m^{-1}$ will allow us to say that the set $\lambda^n([j_{k-1} , j_k]_{\frac12 \BB Z})$ is small (in a certain sense) when $m$ is large. This will eventually allow us to show that $B_{\delta_k^\zeta n^{1/4}}^{\op{pbl}}(\dot e_J^n ; \ol Q_J^n)$ is contained in $B_{\ep n^{1/4}}(v ; Q^n)$ for $J \in  \lfloor \delta_k^5 n^{3/4} \rfloor \BB N_0$ slightly smaller than $j_k$.

\begin{proof}[Proof of Lemma~\ref{lem-fb-crossing1}]
To deal with conditions~\ref{item-fb-crossing-time} and~\ref{item-fb-crossing-cont}, we will apply Proposition~\ref{prop-bdy-equicont} several times. 
By Proposition~\ref{prop-bdy-equicont}, there exists $0  <\alpha_0 < \alpha_1 $ and $A > 0$, depending only on $\ep $, such that for each large enough $n\in\BB N$, the probability of condition~\ref{item-fb-crossing-time} in the definition of $E^n$ is at least $1-\ep/4$. Henceforth fix such a choice of $\alpha_0,\alpha_1,A$. 

By a second application of Proposition~\ref{prop-bdy-equicont} (with $\delta = A m^{-1}$ and $\zeta$ replaced by a sufficiently small universal constant), there exists $m_* = m_*(\ep) \in \BB N$ such that for $m\geq m_*$ and each large enough $n\in\BB N$, it holds with probability at least $1-\ep/4$ that
\begin{itemize}
\item For each $j\in  [0,I_{\alpha_0,\alpha_1}^n-1]_{ \lfloor A m^{-1} n^{3/4} \rfloor \BB Z}$, the percolation path increment $\lambda^n([j , (j+ A m^{-1} n^{3/4}) \wedge I_{\alpha_0,\alpha_1}^n]_{\frac12\BB Z})$ is contained in a subgraph $S$ of $\ol Q_j^n$ such that $\dot e_j^n \in \bdy S$ and the $\ol Q_j^n$-graph distance diameter of $\bdy S$ is at most $ A^{1/4} m^{-1/4} n^{1/4}$. 
\end{itemize}

We now apply Proposition~\ref{prop-bdy-equicont} $m+1$ more times, with each of $\delta_k^5$ for $k\in [K-m,K]_{\BB Z}$ in place of $\delta$, $\ep / (4 (m+1) )$ in place of $\ep$, and a sufficiently small universal constant in place of $\zeta$. We then take a union bound over $k \in [K-m,K]_{\BB Z}$ to get that if $K$ is sufficiently large (depending on $m,\ep,\zeta$), then there exists $n_* = n_*(K,m,\ep,\zeta) \in \BB N$ such that for each $n\geq n_*$, it holds with probability at least $1-\ep/4$ that the following is true. 
\begin{itemize}
\item For each $k\in [K-m,K]_{\BB Z}$ and each $j\in  [0,I_{\alpha_0,\alpha_1}^n-1]_{ \lfloor \delta_k^5 n^{3/4} \rfloor \BB Z}$, the percolation path increment $\lambda^n([j , (j+ \delta_K^5 n^{3/4}) \wedge I_{\alpha_0,\alpha_1}^n]_{\frac12\BB Z})$ is contained in a subgraph $S$ of $\ol Q_j^n$ such that $\dot e_j^n \in \bdy S$ and the $\ol Q_j^n$-graph distance diameter of $\bdy S$ is at most $   \delta_k^{5/4} n^{1/4}$.    
\end{itemize}
We emphasize that here we are using only condition~\ref{item-bdy-equicont} of Proposition~\ref{prop-bdy-equicont}. In particular, the parameters $\alpha_0,\alpha_1,A$ are still chosen as above (depending only on $\ep$).
 
By Lemma~\ref{lem-fb-crossing0}, by possibly increasing $K_*$ and $n_*$, we can arrange that for $n\geq n_*$, it holds with probability at least $1-\ep$ that condition~\ref{item-fb-crossing-time} in the definition of $E^n$ holds, the above two indented conditions hold, and also
\begin{itemize}
\item For each $k\in [K - m,K]_{\BB Z}$ and each $j \in [0,  I_{\alpha_0,\alpha_1}^n  ]_{\lfloor \delta_k^5 n^{3/4} \rfloor \BB Z  }$ with $B_{\frac{1}{8} \ep n^{1/4}}^\pbl(\dot e_j^n ; \ol Q_j^n) \cap \ol Q_{I_{\alpha_0,\alpha_1}}^n = \emptyset$,  we have $\# \op{IP}_j^n\left( \delta_k n^{1/4} , \delta_k^\zeta n^{1/4}   \right) \leq 7$. 
\end{itemize}
Henceforth assume that this is the case. We will show that $E^n$ occurs. 
By assumption, condition~\ref{item-fb-crossing-time} holds. 

We now deal with condition~\ref{item-fb-crossing-endpoint}.
Since $\op{diam} \left( \ol Q_{I_{\alpha_0,\alpha_1}^n}^n \right)  \leq \frac{1}{8} \ep  n^{1/4}$, we have $B_{\frac{1}{8} \ep n^{1/4}}^\pbl(\dot e_j^n ; \ol Q_j^n) \cap \ol Q_{I_{\alpha_0,\alpha_1}}^n = \emptyset$ whenever $j\in\BB N_0$ with $\op{dist}(\dot e_j^n  , \BB e_*^n ; Q^n) \geq \frac{1}{4} \ep n^{1/4}$; and $\op{dist}(\dot e_j^n , \BB e_*^n ; Q^n) \leq \frac{1}{8} \ep n^{1/4}$ whenever $j \geq I_{\alpha_0,\alpha_1}^n$.  In particular, if $\op{dist}(\dot e_j^n  , \BB e_*^n ; Q^n) \geq \frac{1}{2} \ep n^{1/4}$ then $j \leq I_{\alpha_0,\alpha_1}^n$ and $B_{\frac{1}{8} \ep n^{1/4}}^\pbl(\dot e_j^n ; \ol Q_j^n) \cap \ol Q_{I_{\alpha_0,\alpha_1}}^n = \emptyset$. So, we can apply the last indented condition above to get that condition~\ref{item-fb-crossing-endpoint} in the definition of $E^n$ holds. 

To check condition~\ref{item-fb-crossing-cont}, we need to show that if $k\in [K-m,K]_{\BB Z}$ and $j \in  \lfloor \delta_k^5 n^{3/4} \rfloor \BB N_0$ with $\op{dist}(\dot e_j^n  , \BB e_*^n ; Q^n) \geq \frac12 \ep n^{1/4}$ then $j + \delta_k^5 n^{3/4} \leq I_{\alpha_0,\alpha_1}^n-1$.  Indeed, we know that $\lambda^n([j , (j+ \delta_k^5 n^{3/4}) \wedge I_{\alpha_0,\alpha_1}^n]_{\frac12\BB Z})$ is contained in a subgraph $S$ of $Q^n$ whose boundary contains $\dot e_j^n$ and has $Q^n$-graph distance diameter at most $\delta_k^{5/4} n^{1/4}$.  Such a subgraph cannot intersect $\ol Q_{I_{\alpha_0,\alpha_1}^n}^n$ since otherwise we would have $\op{dist}(\dot e_j^n  , \BB e_*^n ; Q^n) \leq \frac18 \ep n^{1/4} + \delta_k^{5/4} n^{1/4}  < \frac12 \ep n^{1/4}$. Hence $ (j+ \delta_k^5 n^{3/4}) \wedge I_{\alpha_0,\alpha_1}^n  < I_{\alpha_0,\alpha_1}^n $, which means that $j+ \delta_k^5 n^{3/4} < I_{\alpha_0,\alpha_1}^n$, as required.  We similarly check the condition with $  A  m^{-1}$ used in place of $\delta_k^5$. 
\end{proof}

\begin{figure}[t!]
 \begin{center}
\includegraphics[scale=1]{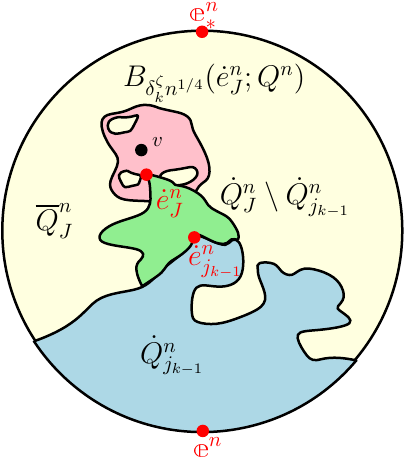}
\vspace{-0.01\textheight}
\caption{Illustration of the proof of~\eqref{eqn-good-ball-out} of Lemma~\ref{lem-good-ball} (not shown to scale). The integer $k$ is chosen so that $j_k  -j_{k-1} \leq A m^{-1}n^{3/4}$, which means that also $J-j_{k-1} \leq A m^{-1} n^{3/4}$. By condition~\ref{item-fb-crossing-cont} in the definition of $E^n$, the set $\bdy \ol Q_J^n \setminus \bdy \ol Q_{j_{k-1}}^n$ (i.e., the boundary of the green region minus the boundary of the blue region) is contained in a subgraph $S''$ of $Q^n$ whose boundary has $Q^n$-diameter at most a constant times $m^{-1/4} n^{1/4}$ (this set is not pictured). This allows us to upper-bound the size of $B_{\delta_k^\zeta n^{1/4}}^{\pbl}(\dot e_J^n ; \ol Q_J^n)$ in terms of $B_{\delta_k^\zeta n^{1/4}}(\dot e_J^n ; Q^n)$ and $S''$. 
}\label{fig-fb-crossing}
\end{center}
\vspace{-1em}
\end{figure} 

The following lemma will allow us to transfer from crossings of $B_{\delta_k^\zeta n^{1/4} }^\pbl\left( \dot e_j^n ; \ol Q_j^n \right)  \setminus B_{2\delta_k n^{1/4}    }^\pbl\left( \dot e_j^n ; \dot Q_j^n \right)  $ for $j \in \lfloor \delta_k^5 n^{3/4} \rfloor \BB N_0$ (which are bounded by Lemma~\ref{lem-fb-crossing1}) to crossings of $B_{\ep  n^{1/4} }^\bullet\left( v ; Q^n \right) \setminus B_{\delta n^{1/4} }^\bullet\left( v ; Q^n \right)$ for $v\in\mcl V(Q^n)$ (which we want to bound in Proposition~\ref{prop-fb-crossing}).

\begin{lem} \label{lem-good-ball}
Let $\ep \in (0,1)$ and let $K , m\in\BB N$ with $K  >m$.
Assume that the event $E^n = E^n(K,m,\ep)$ occurs, that $m$ is sufficiently large (depending on $\ep$) and that $K$ is sufficiently large (depending only on $m,\ep,\zeta$). 
For each $v\in\mcl V(Q^n) $ such that $\op{dist}(v , \bdy Q^n ; Q^n) \geq   \ep n^{1/4}$ and $\lambda^n$ hits $B_{\delta_K n^{1/4}}(v ; Q^n)$, there exists $k \in [K - m, K]_{\BB Z}$ and $J \in \lfloor \delta_k^5 n^{3/4} \rfloor \BB N_0$ with $\op{dist}(\dot e_J^n , \BB e_*^n ; Q^n ) \geq \frac12 \ep n^{1/4}$ (both $J$ and $k$ are random) such that 
\eqb \label{eqn-good-ball-in}
B_{(\delta_k/2) n^{1/4}  }^\bullet\left( v ; Q^n \right) \subset  B_{2\delta_k n^{1/4}   -2  }^\pbl\left( \dot e_J^n ; \ol Q_J^n \right)   
\eqe 
and
\eqb \label{eqn-good-ball-out}
B_{\delta_k^\zeta n^{1/4} + 2}^\pbl\left( \dot e_J^n ; \ol Q_J^n \right)  
\subset B_{\ep  n^{1/4} }^\bullet\left( v ; Q^n \right)   . 
\eqe 
\end{lem}
\begin{proof}
Throughout the proof, we assume that $m$ is sufficiently large (depending on $\ep$) and $K$ is sufficiently large (depending on $m,\ep,\zeta$) so that the needed estimates are satisfied. 
\medskip

\noindent\textit{Step 1: choice of $k$ and $J$.}
For $k\in [K-m,K]_{\BB Z}$, let $j_k$ be the smallest $j\in \BB N_0$ such that $\dot e_j^n \in B_{ \delta_k n^{1/4} }(v ; Q^n)$.
Since $\op{dist}(v , \BB e_*^n ; Q^n) \geq \ep n^{1/4}$ and $\op{dist}(\dot e_{j_k}^n ; Q^n ) \leq \delta_k n^{1/4} \leq \frac{1}{2} \ep n^{1/4}$, the triangle inequality implies that $\op{dist}(\dot e_{j_k}^n ,  \BB e_*^n ; Q^n) \geq \frac12 \ep n^{1/4}$. In particular, condition~\ref{item-fb-crossing-time} of Lemma~\ref{lem-fb-crossing1} implies that $j_K \leq I_{\alpha_0,\alpha_1}^n \leq A n^{3/4}$ (note that $\lambda^n([I_{\alpha_0,\alpha_1}^n , \infty)_{\frac12\BB Z}) \subset \ol Q_{\alpha_0,\alpha_1}^n$).  

By Markov's inequality, there exists $k\in [K-m,K]_{\BB Z}$ such that
\eqbn
j_k - j_{k-1} \leq  A m^{-1}  n^{3/4}.
\eqen
Henceforth fix such a $k$ and let $J$ be the largest $j \in \lfloor \delta_k^5 n^{3/4} \rfloor \BB N_0$ for which $J \leq j_k$. We claim that the lemma statement holds for this choice of $k$ and $J$. 

Note that, by our choice of $J$, 
\eqb \label{eqn-good-ball-intervals}
j_k - J \leq \delta_k^5 n^{3/4} \quad \text{and} \quad J - j_{k-1} \leq A m^{-1} n^{3/4} .
\eqe
The bounds in~\eqref{eqn-good-ball-intervals} will allow us to control the sizes of the sets $\lambda^n([J,j_k]_{\frac12\BB Z})$ and $\lambda^n([j_{k-1} , J]_{\frac12 \BB Z})$ via condition~\ref{item-fb-crossing-cont} of Lemma~\ref{lem-fb-crossing1}. This is the reason for our choice of $k$ and $J$. 
\medskip

\noindent\textit{Step 2: proof of~\eqref{eqn-good-ball-in}.}
By condition~\ref{item-fb-crossing-cont} of Lemma~\ref{lem-fb-crossing1}, the set $\lambda^n([J , j_k]_{\frac12 \BB Z})$ is contained in a subgraph $S$ of $\ol Q_J^n$ such that $\dot e_J^n \in \bdy S$ and 
\eqb  \label{eqn-good-ball-S}
\op{diam}(\bdy S ; \ol Q_J^n) \leq  \delta_k^{5/4} n^{1/4}  .
\eqe 
By the definition of $j_k$, there is a path in $Q^n$ from $v$ to $\dot e_{j_k}^n$ of length at most $\delta_k n^{1/4}$. In fact, such a path must be contained in $\ol Q_{j_k}^n \subset \ol Q_J^n$ since $j_k$ is the \emph{smallest} $j\in \BB N_0$ such that $\dot e_j^n \in B_{ \delta_k n^{1/4} }(v ; Q^n)$.
This path must hit $\bdy S$, so 
\eqb \label{eqn-good-ball-dist}
\op{dist}\left( v ,  \dot e_J^n ; \ol Q_J^n \right) 
\leq \left( \delta_k + \delta_k^{5/4} \right) n^{1/4} \leq  \frac{3}{2} \delta_k  n^{1/4}    - 4 .
\eqe
Since $J \leq j_k$, we also have
\eqb \label{eqn-good-ball-dist'}
\op{dist}\left( v ,  Q^n \setminus \ol Q_J^n ; Q^n  \right)  \geq \delta_k n^{1/4} .
\eqe
By~\eqref{eqn-good-ball-dist} and~\eqref{eqn-good-ball-dist'}, 
\eqbn
B_{(\delta_k/2) n^{1/4}}(v ; Q^n) = B_{(\delta_k/2) n^{1/4}}(v ; \ol Q_J^n)  \subset B_{2\delta_k n^{1/4} - 2}^\bullet\left( \dot e_J^n ; \ol Q_J^n \right) . 
\eqen
This last inclusion, combined with~\eqref{eqn-filled-ball-contain}, gives~\eqref{eqn-good-ball-in}. We also note that since $\op{dist}(v , \bdy Q^n ; Q^n) \geq \ep n^{1/4}$ and $\delta_k$ is much smaller than $\ep$, the relation~\eqref{eqn-good-ball-dist} implies that $\op{dist}(\dot e_J^n , \bdy Q^n ; Q^n) \geq \frac12 \ep n^{1/4}$. 
\medskip

\noindent\textit{Step 3: proof of~\eqref{eqn-good-ball-out}.} 
See Figure~\ref{fig-fb-crossing} for an illustration. 
By condition~\ref{item-fb-crossing-cont} of Lemma~\ref{lem-fb-crossing1} (the $ A m^{-1}$ case, see Remark~\ref{remark-weaker-cont}) together with~\eqref{eqn-good-ball-intervals}, there is a subgraph $S'$ of $Q^n$ such that
\eqb \label{eqn-good-ball-S'}
\lambda^n([j_{k-1} , J]_{\frac12 \BB Z} ) \subset S' \quad \text{and} \quad
\op{diam}\left( \bdy S' ; Q^n \right) \leq  2 A^{1/4}  m^{-1/4} n^{1/4} .
\eqe

The set $\bdy \ol Q_J^n \setminus \bdy \ol Q_{j_{k-1}}^n$ is a connected arc of $\bdy \ol Q_J^n$. Each edge of this arc is a side of one of the peeled quadrilaterals in the percolation peeling process between times $j_{k-1} $ and $J$, so lies at $Q^n$-graph distance at most 2 from $\lambda^n([j_{k-1} , J]_{\frac12 \BB Z} )$. By considering the graph distance neighborhood of $S'$ of radius 2, we obtain a subgraph $S''$ of $Q^n$ such that 
\eqb \label{eqn-good-ball-S''}
\bdy \ol Q_J^n \setminus \bdy \ol Q_{j_{k-1}}^n \subset S'' \quad \text{and} \quad
\op{diam}\left( \bdy S'' ; Q^n \right) \leq C m^{-1/4} n^{1/4}  
\eqe
where $C =   2 A^{ 1/4} + 2$ is a constant depending only on $\ep,\zeta$. 

By the triangle inequality, followed by~\eqref{eqn-good-ball-dist} and the definition of $j_{k-1}$, then the definition~\eqref{eqn-crossing-scale-def} of $\delta_k$,
\alb
 \op{dist}\left( \dot e_J^n ,  \dot Q_{j_{k-1}}^n ;  Q^n \right)  
&\geq  \op{dist}\left( v ,   \dot Q_{j_{k-1}}^n ; Q^n \right) -  \op{dist}\left( v ,  \dot e_J^n ; Q^n \right) \notag\\
&\geq \left( \delta_{k-1} - \frac32 \delta_k \right) n^{1/4} + 4  > \delta_k^\zeta n^{1/4} + 4 .
\ale
Hence 
\eqb \label{eqn-good-ball-past}
     B_{\delta_k^\zeta n^{1/4} + 4 } \left( \dot e_J^n ; Q^n \right) \cap \dot Q_{j_{k-1}}^n = \emptyset .
\eqe 
Since $\op{dist}(\dot e_J^n ,\bdy Q^n ; Q^n) \geq \frac12 \ep n^{1/4}$, if $k$ is large enough so that $\delta_k^\zeta   <  \frac12 \ep$, then we also have 
\eqb \label{eqn-good-ball-bdy}
B_{\delta_k^\zeta n^{1/4} +4}\left( \dot e_J^n ; Q^n \right) \cap \bdy Q^n =\emptyset . 
\eqe

By~\eqref{eqn-good-ball-past} and~\eqref{eqn-good-ball-bdy}, the ball $B_{\delta_k^\zeta n^{1/4}+4}\left( \dot e_J^n ; Q^n \right)$ can intersect $\bdy \ol Q_J^n$ only in the connected arc $\bdy \ol Q_J^n \setminus \bdy \ol Q_{j_{k-1}}^n$.
Since $ B_{\delta_k^\zeta n^{1/4}+4}\left( \dot e_J^n ; \ol Q_J^n \right) \subset B_{\delta_k^\zeta n^{1/4}+4}\left( \dot e_J^n ; Q^n \right)$, the same is true with $B_{\delta_k^\zeta n^{1/4}+4}\left( \dot e_J^n ; \ol Q_J^n \right)$ in place of $B_{\delta_k^\zeta n^{1/4}+4}\left( \dot e_J^n ; Q^n \right)$.  
Since $\bdy \ol Q_J^n \setminus \bdy \ol Q_{j_{k-1}}^n$ is a connected arc of $\bdy \ol Q_J^n$ which does not contain $\BB e_*^n$, it follows that the boundary of every connected component of $\ol Q_J^n \setminus B_{\delta_k^\zeta n^{1/4}+4}\left( \dot e_J^n ; \ol Q_J^n \right)$ which does not contain $\BB e_*^n$ is the union of a subgraph of $\bdy  B_{\delta_k^\zeta n^{1/4}+4}\left( \dot e_J^n ; \ol Q_J^n \right)$ and a subgraph of $\bdy \ol Q_J^n \setminus \bdy \ol Q_{j_{k-1}}^n$. 
Therefore, $B_{\delta_k^\zeta n^{1/4}+4}^\bullet\left( \dot e_J^n ; \ol Q_J^n \right)$ can intersect $\bdy \ol Q_J^n$ only in $\bdy \ol Q_J^n \setminus \bdy \ol Q_{j_{k-1}}^n$.

Thus, with $S''$ as in~\eqref{eqn-good-ball-S''}, the set $ B_{\delta_k^\zeta n^{1/4} +4}^\bullet \left( \dot e_J^n ; \ol Q_J^n \right)$ is disconnected from $\bdy Q^n$ by the union of $B_{\delta_k^\zeta n^{1/4}+4}^\bullet\left(\dot e_J^n ; Q^n\right)$ and the set $S''$. 
By~\eqref{eqn-good-ball-S''}, the boundary (relative to $Q^n$) of this last union has $Q^n$-diameter at most $(\delta_k^\zeta + C m^{-1/4})n^{1/4} \leq 2C m^{-1/4} n^{1/4}$. 
Since $v\in  B_{\delta_k^\zeta n^{1/4} + 4 }^\bullet \left( \dot e_J^n ; \ol Q_J^n \right)$, we obtain
\eqbn
B_{\delta_k^\zeta n^{1/4} +4}^\bullet \left( \dot e_J^n ; \ol Q_J^n \right) \subset B_{4 C m^{-1/4} n^{1/4}}\left( v ; Q^n \right) .
\eqen
Due to the relation between filled metric balls and peeling-by-layers clusters~\eqref{eqn-filled-ball-contain}, this gives~\eqref{eqn-good-ball-out} provided $m$ is large enough so that $4 C m^{-1/4} \leq \ep$. 
\end{proof}

\begin{proof}[Proof of Proposition~\ref{prop-fb-crossing}]
For $\ep > 0$, let $\alpha_0,\alpha_1,A,m,K,n$ be as in Lemma~\ref{lem-fb-crossing1} and define the event $E^n = E^n(K,m,\ep,\zeta)$ as in that lemma, so that $\BB P[E^n] \geq 1-\ep$. We now assume that $E^n$ occurs and show that the event in the proposition statement occurs. By Lemma~\ref{lem-good-ball}, for each $v\in \mcl V(Q^n)$  such that $\op{dist}(v , \bdy Q^n ; Q^n) \geq   \ep n^{1/4}$ and $\lambda^n$ hits $B_{\delta_K n^{1/4}}(v ; Q^n)$, there exists $k\in [K-m,K]_{\BB Z}$ and $J \in \lfloor \delta_k^5 n^{3/4} \rfloor \BB N_0$ with $\op{dist}(\dot e_J^n , \BB e_*^n ; Q^n ) \geq \frac12 \ep n^{1/4}$ such that~\eqref{eqn-good-ball-in} and~\eqref{eqn-good-ball-out} hold. For such a choice of $k$ and $J$, 
\eqbn
\# \op{cross}^n\left( B_{(\delta_k/2) n^{1/4}}^\bullet(v ; Q^n)  ,  B_{\ep n^{1/4}}^\bullet(v ; Q^n)     \right) \leq    
\# \op{cross}^n\left( B_{2\delta_k n^{1/4}-2}^\pbl(\dot e_J^n ; \ol Q_J^n) ,    B_{ \delta_k^\zeta n^{1/4}+2}^\pbl(  \dot e_J  ; \ol Q_J^n)   \right)   .
\eqen
By Lemma~\ref{lem-crossing-ip}, this last quantity is at most $\# \op{IP}^n_J\left( 2\delta_k ,  \delta_k^\zeta n^{1/4} \right)$ which by condition~\ref{item-fb-crossing-endpoint} in the definition of $E^n$ is at most~$7$. We thus obtain the proposition statement for any $\delta \leq \delta_K/2$. 
\end{proof}

\subsection{Proof of Proposition~\ref{prop-fb-crossing-bdy}}
\label{sec-crossing-proof-bdy}

The proof of Proposition~\ref{prop-fb-crossing-bdy} is in some ways similar to the proof of Proposition~\ref{prop-fb-crossing}, but the argument is substantially simpler since we can work exclusively with distances in $Q^n$ rather than in the quadrangulations $\ol Q_j^n$. This means that we do not need to worry about comparing filled metric balls in different quadrangulations, which was the hardest part of Section~\ref{sec-crossing-proof}. 

Due to the similarity to previous arguments, we will be terse. We first record an analog of Proposition~\ref{prop-uihpq-crossing} (Lemma~\ref{lem-uihpq-crossing-bdy}) where we take a union bound over boundary edges of $Q^\infty$ instead of a union bound over times $j$ for the percolation peeling process. We then transfer to an estimate for $Q^n$ using absolute continuity (Lemma~\ref{lem-fb-crossing-bdy0}). Finally, we use this together with a continuity estimate for distances along the boundary of $Q^n$ (Lemma~\ref{lem-fb-bdy-holder}) to conclude the proof.  
The starting point of the proof of Proposition~\ref{prop-fb-crossing-bdy} is the following straightforward consequence of Proposition~\ref{prop-uihpq-crossing}. 

\begin{lem} \label{lem-uihpq-crossing-bdy}
Let $\zeta \in (0,1/100)$, $C > 0$, and $A>0$. 
Also let $\beta^\infty$ be the boundary path of the UIHPQ$_{\op{S}}$ $Q^\infty$. 
For each $\delta \in (0,1)$ there exists $n_* = n_*(\delta ,\zeta ,A , C) \in \BB N$ such that for $n\geq n_*$ and $N\in\BB N$, it holds with probability $1 - O_\delta(\delta^{(N-1) (1-2\zeta)  - (2+\zeta) }) $ (at a rate which is uniform for $n\geq n_*$) that, in the notation of Definition~\ref{def-ip-crossing} and~\eqref{eqn-pbl-bdy-length-def}, 
\alb
&\#\op{IP}_0^\infty( \beta^\infty(i)  ; \delta  n^{1/4} , \delta^\zeta n^{1/4}   ) \leq N ,\notag\\ 
&\qquad\forall i \in [-C n^{1/2} , C n^{1/2}]_{\lfloor \delta^{2+\zeta}  n^{1/2} \rfloor \BB Z} \quad \text{such that} \quad    \# \mcl A_{\delta  n^{1/4}}^\pbl(\beta^\infty(i)  ;  Q^\infty ) \leq  \delta^{2-\zeta} n^{1/2}   .
\ale
\end{lem}
\begin{proof}
Since $(Q^\infty , \beta^\infty(i)) \eqD (Q^\infty , \BB e^\infty)$ for each $i \in \BB Z$, this follows from the $M=0$ case of Proposition~\ref{prop-uihpq-crossing} and a union bound over all $i \in [-C n^{1/2} , C n^{1/2}]_{\lfloor \delta^{2+\zeta}  n^{1/2} \rfloor \BB Z}$.
\end{proof}

We can now prove an analog of Lemma~\ref{lem-fb-crossing0} in the setting of Proposition~\ref{prop-fb-crossing-bdy}. 

\begin{lem} \label{lem-fb-crossing-bdy0}
Assume we are in the setting of Proposition~\ref{prop-fb-crossing-bdy} and for $n\in\BB N$ let $\beta^n : [0,l^n]_{\BB Z} \rta \mcl E(\bdy Q^n)$ be the boundary path of $Q^n$. 
Let $0 < \alpha_0 < \alpha_1$, $A>0$, and $\zeta\in (0,1)$. 
For each $\delta \in (0,1)$, there exists $n_* = n_*(\delta ,  \zeta  ,  \alpha_0 , \alpha_1 , A) \in \BB N $ such that for $n\geq n_*$, the probability that the following is true is at most $O_\delta(\delta^{3(1-2\zeta)  - (2+\zeta) })$, uniformly for all $n\geq n_*$. 
We have $I_{\alpha_0,\alpha_1}^n \leq A n^{3/4}$ and there exists $i \in [0,l^n]_{\lfloor \delta^{2+\zeta} n^{1/2} \rfloor \BB Z}$ such that $\# \op{IP}_0^n\left( \beta^n(i) ;  2 \delta  n^{1/4} , \delta^\zeta n^{1/4}   \right)  > 4$ and $B_{\delta^\zeta n^{1/4}}^\pbl(\beta^n(i) ; Q^n) \cap \ol Q_{I_{\alpha_0,\alpha_1}}^n = \emptyset$. 
\end{lem}

Note that the exponent $3(1-2\zeta) - (2+\zeta)$ is positive provided $\zeta$ is chosen to be sufficiently small. We henceforth assume that this is the case. 

\begin{proof}[Proof of Lemma~\ref{lem-fb-crossing-bdy0}]
This follows by combining Lemma~\ref{lem-uihpq-crossing-bdy} (applied with $N=4$) with Lemmas~\ref{lem-perc-rn} and~\ref{lem-pbl-bdy-length}, via the same argument as in the proof of Lemma~\ref{lem-fb-crossing1}. 
\end{proof}

The following lemma gives us a high-probability event on which we will check the condition of Proposition~\ref{prop-fb-crossing-bdy}.

\begin{lem} \label{lem-fb-crossing-bdy1} 
Let $\ep\in (0,1)$ and $\zeta\in (0,1)$. 
There exists $\delta_* = \delta_*(\ep) \in (0,\ep^{1/\zeta}/100]$ such that for each $\delta \in (0,\delta_*]$, there exists $n_* = n_*(\delta ,  \zeta  ) \in \BB N $ such that for $n\geq n_*$, the following is true.
Let $E^n = E^n(\delta,\ep,\zeta)$ be the event that the following is true.
\begin{enumerate}
\item For each $i,j \in [0,l^n]_{\BB Z}$ with $|i-j| \leq \delta^{2+\zeta} n^{1/2}$, we have $\op{dist}(\beta^n(i) , \beta^n(j) ; Q^n) \leq \delta^{1+\zeta/4} n^{1/4}$. \label{item-bdy-cont}
\item For each  $i \in [0,l^n]_{\lfloor \delta^{2+\zeta} n^{1/2} \rfloor \BB Z}$ such that $\op{dist}(\beta^n(i) , \BB e_*^n ; Q^n) \geq \frac12 \ep n^{1/4}$, we have $\# \op{IP}_0^n\left( \beta^n(i) ;  2 \delta  n^{1/4} , \delta^\zeta n^{1/4}   \right)  \leq 4$. \label{item-bdy-cross}
\end{enumerate}
Then $\BB P[E^n] \geq 1-\ep$. 
\end{lem} 
\begin{proof} 
This follows from Lemma~\ref{lem-fb-crossing-bdy0} combined with condition~\ref{item-bdy-equicont-end} of Proposition~\ref{prop-bdy-equicont} and Lemma~\ref{lem-fb-bdy-holder}, via a very similar argument to the one in the proof of Lemma~\ref{lem-fb-crossing1}. 
\end{proof}

\begin{proof}[Proof of Proposition~\ref{prop-fb-crossing-bdy}]
Fix $\zeta\in (0,1)$. For $\ep > 0$, let $\delta_*$ be as in Lemma~\ref{lem-fb-crossing-bdy1}, let $\delta \in (0,\delta_*]$, and let $E^n$ be as in Lemma~\ref{lem-fb-crossing-bdy1}, so that $\BB P[E^n] \geq 1-\ep$ provided $n$ is sufficiently large.
We assume that $E^n$ occurs and show that the condition in the lemma statement is satisfied. 

Suppose $ v \in\bdy Q^n$ such that $\op{dist}(v , \BB e_*^n ; Q^n) \geq \ep n^{1/4}$. Let $i\in [0,l^n]_{\BB Z}$ be chosen so that $v$ is one of the endpoints of $\beta^n(i)$. We can find $i' \in  [0,l^n]_{\lfloor \delta^{2+\zeta} n^{1/2} \rfloor \BB Z}$ such that $|i-i'| \leq \delta^{2+\zeta} n^{1/2}$. By condition~\ref{item-bdy-cont} of Lemma~\ref{lem-fb-crossing-bdy1}, for this choice of $i'$ we have $\op{dist}(v , \beta^n(i') ; Q^n) \leq \delta^{1+\zeta/4} n^{1/4}$. 
Since $\op{dist}(v , \BB e_*^n ; Q^n) \geq \ep n^{1/4}$, it follows that $\op{dist}(\beta^n(i') , \BB e_*^n ; Q^n) \geq \frac12 \ep n^{1/4}$. Furthermore, by the triangle inequality, 
\eqb \label{eqn-crossing-bdy-in}
B_{\delta n^{1/4}}\left( v ; Q^n \right) 
\subset B_{(\delta + \delta^{1+\zeta/4}) n^{1/4}}\left(\beta^n(i') ; Q^n \right)
\subset B_{2\delta n^{1/4} - 4}\left(\beta^n(i') ; Q^n \right) 
\eqe
and
\eqb \label{eqn-crossing-bdy-out}
B_{\ep n^{1/4}}\left( v ; Q^n \right) 
\supset B_{(\ep - \delta^{1+\zeta/4}) n^{1/4}}\left(\beta^n(i') ; Q^n \right)
\supset B_{\delta^\zeta n^{1/4} + 4}\left(\beta^n(i') ; Q^n \right)  .
\eqe
The inclusions~\ref{eqn-crossing-bdy-in} and~\eqref{eqn-crossing-bdy-out} also hold with filled metric balls instead of ordinary metric balls.  
By~\eqref{eqn-crossing-bdy-in} and~\eqref{eqn-crossing-bdy-out},
\eqb \label{eqn-crossing-bdy-final}
\# \op{cross}^n\left( B_{\delta n^{1/4}}^\bullet(v ; Q^n)  ,  B_{\ep n^{1/4}}^\bullet(v ; Q^n)     \right) \leq    
\# \op{cross}^n\left( B_{2\delta n^{1/4}-4}^\bullet( \beta^n(i')  ; Q^n) ,    B_{ \delta^\zeta n^{1/4}+4}^\bullet( \beta^n(i') ; Q^n )   \right)   .
\eqe 
By~\eqref{eqn-filled-ball-contain} followed by Lemma~\ref{lem-crossing-ip}, the right side of~\eqref{eqn-crossing-bdy-final} is at most $\op{IP}_0(\beta^n(i') ; 2\delta n^{1/4} , \delta^\zeta n^{1/4})$. By condition~\ref{item-bdy-cross} of Lemma~\ref{lem-fb-crossing-bdy1}, this last quantity is at most 4. 
\end{proof}

\section{Proof of main theorems}
\label{sec-identification}

Throughout this section, we assume we are in the setting of Theorem~\ref{thm-perc-conv}.  In particular, we fix $\frk l_L , \frk l_R > 0$ and a sequence of pairs of positive integers $(\el_L^n ,\el_R^n)_{n\in\BB N}$ such that $ \el_L^n + \el_R^n$ is always even and
\eqbn
\frk l_L^n :=  \bcon^{-1} n^{-1/2} \el_L^n \rta \frk l_L \quad \op{and} \quad \frk l_R^n :=  \bcon^{-1}  n^{-1/2} \el_R^n  \rta \frk l_R . 
\eqen 

We recall that $(Q^n ,\BB e^n , \theta^n)$ is a free Boltzmann quadrangulation with simple boundary of perimeter $\el_L^n+\el_R^n$, viewed as a connected metric space by replacing each edge with an isometric copy of the unit interval and that $d^n$, $\mu^n$, $\xi^n$, and $\eta^n$, respectively, denote the rescaled metric, area measure, and boundary path. Recall also that $\frk Q^n$ denotes the doubly curve-decorated metric measure space $\left( Q^n , d^n,\mu^n, \xi^n,\eta^n \right)$.
   
As per usual, we define the clusters $\{\dot Q_j^n\}_{j\in\BB N_0}$, the unexplored quadrangulations $\{\ol Q_j^n\}_{j\in\BB N_0}$, the peeled edges $\{\dot e_j^n\}_{j\in\BB N}$, the filtration $\{\mcl F_j^n\}_{j \in \BB N_0}$, and the terminal time $\mcl J^n$ for the percolation peeling exploration of $(Q^n,\BB e^n,\theta^n)$ with $\el_L^n$-white/$\el_R^n$-black boundary conditions as in Section~\ref{sec-perc-peeling}.  We also define the rescaled boundary length process $Z^n = (L^n,R^n)$ as in~\eqref{eqn-bdy-process-rescale}. 
 
By Propositions~\ref{prop-stable-conv-finite} and~\ref{prop-ghpu-tight}, for any sequence of positive integers tending to $\infty$ there is a subsequence $\mcl N$ and a coupling of a doubly curve-decorated metric measure space $\wt{\frk H} = (\wt H , \wt d , \wt \mu , \wt\xi , \wt\eta)$ with a two-dimensional cadlag process $Z = (L,R)  : [0,\infty) \rta \BB R$ such that  
\eqb \label{eqn-ssl-conv0}
(\frk Q^n , Z^n) \rta (\wt{\frk H} ,   Z ) \quad \op{as} \: \mcl N\ni n \rta\infty
\eqe 
in the two-curve GHPU topology on the first coordinate and the Skorokhod topology on the second coordinate. 

In fact, we know the limit of the laws of $(Q^n,d^n,\mu^n,\xi^n)  $ and $Z^n$ in the GHPU and Skorokhod topologies, respectively: $(\wt H , \wt d , \wt \mu , \wt\xi) $ is a free Boltzmann Brownian disk with boundary length $\frk l_L + \frk l_R$ equipped with its natural metric, area measure, and boundary path and $Z$ has the law of the left/right boundary length process for a chordal $\SLE_6$ on such a free Boltzmann Brownian disk between two points at counterclockwise boundary length distance $\frk l_R$ from each other (although we do not know a priori that $Z$ is the left/right boundary length process of any curve on $\wt H$).  We extend $\xi $ to $\BB R$ by periodicity, so that in particular $\xi(-\frk l_L) = \xi(\frk l_R)$.

We will prove Theorem~\ref{thm-perc-conv} by showing that the pair $(\wt{\frk H} , \wt\eta )$ satisfies the conditions of Theorem~\ref{thm-bead-mchar}.  This will then imply that $\wt \eta $ is a chordal $\SLE_6$ from $\wt\xi(0)$ to $\wt\xi(\frk l_R)$ in $\wt H$ parameterized by quantum natural time and that $Z$ is its left/right boundary length process.  As we will see in Section~\ref{sec-main-proof}, Theorem~\ref{thm-perc-conv-uihpq} is a straightforward consequence of Theorem~\ref{thm-perc-conv}.

We start in Section~\ref{sec-ssl} by introducing some notation and passing to a further subsequence of $\mcl N$ along which not only $\frk Q^n$ and $Z^n$ but also the internal metric spaces corresponding to the complementary connected components of the curve $\eta^n([0,t])$ for each rational time $t\geq 0$ converge in law. We also establish that the limits of these internal metrics are free Boltzmann Brownian disks conditional on the limiting boundary length process $Z$ and express their boundary lengths in terms of $Z$.

In Section~\ref{sec-map-limit}, we prove several relationships among our subsequential limiting objects which eventually lead to the statement that the limits of the internal metric spaces corresponding to the complementary connected components of $\eta^n([0,t])$ are the same as the internal metric spaces corresponding to the complementary connected components of $\wt\eta([0,t])$ (Lemma~\ref{lem-map-conclude}). The arguments in this subsection are similar to those found in~\cite[Section~7.3]{gwynne-miller-saw}. The results of Section~\ref{sec-map-limit} allow us to prove Proposition~\ref{prop-component-law}, which says that $(\wt{\frk H} , \wt\eta )$ satisfies condition~\ref{item-bead-mchar-wedge} of Theorem~\ref{thm-bead-mchar}.

The next two subsections are devoted to checking condition~\ref{item-bead-mchar-homeo} of Theorem~\ref{thm-bead-mchar}, the topology and consistency condition.  In Section~\ref{sec-bdy-process-id} we will describe the boundary length measures on the complementary connected components of $\wt\eta([0,t])$ in terms of the limiting boundary length process $Z$. The results of this subsection imply in particular that if $t_1,t_2 \geq 0$ are such that an $\SLE_6$ with left/right boundary length process $Z$ satisfies $ \eta(t_1) =  \eta(t_2)$, then also $\wt\eta(t_1) = \wt\eta(t_2)$. In other words, the curve~$\wt\eta$ has at least as many self-intersections as we would expect from the process~$Z$.

As explained at the beginning of Section~\ref{sec-homeo}, the results of Section~\ref{sec-bdy-process-id} show that there exists an $\SLE_6$-decorated Brownian disk $(H,d,\mu,\xi,\eta)$ and a continuous, measure-preserving, curve-preserving surjective map $\Phi : H\rta \wt H$ which is injective on $H\setminus \eta$ (Lemma~\ref{lem-surjection}). To show that $\Phi$ is a homeomorphism, we will use the results of Section~\ref{sec-crossing} to argue that $\wt\eta$ can hit each point in $\wt H$ at most~$6$ times (Lemma~\ref{lem-ssl-curve-hit}). This statement together with the topological theorem~\cite[Main Theorem]{almost-inj} will imply that $\Phi$ is in fact a homeomorphism, whence condition~\ref{item-bead-mchar-homeo} of Theorem~\ref{thm-bead-mchar} is satisfied.

In Section~\ref{sec-main-proof} we will conclude the proofs of Theorems~\ref{thm-perc-conv} and~\ref{thm-perc-conv-uihpq}.

\subsection{Subsequential limits}
\label{sec-ssl}

Suppose we are given a subsequence $\mcl N$ and a limiting coupling $(\wt{\frk H} , Z)$ as in~\eqref{eqn-ssl-conv0}. 
Before proving any additional statements about this coupling, we will pass to a further subsequence of $\mcl N$ along which several additional curve-decorated metric measure spaces, corresponding to (roughly speaking) the complementary connected components of $\eta^n$ at each time $t\in \BB Q_+ := \BB Q \cap [0,\infty)$, equipped with their internal metrics and boundary paths, converge in the GHPU topology. See Figure~\ref{fig-map-limit} for an illustration of these objects.

\begin{figure}[ht!]
 \begin{center}
\includegraphics[scale=.8]{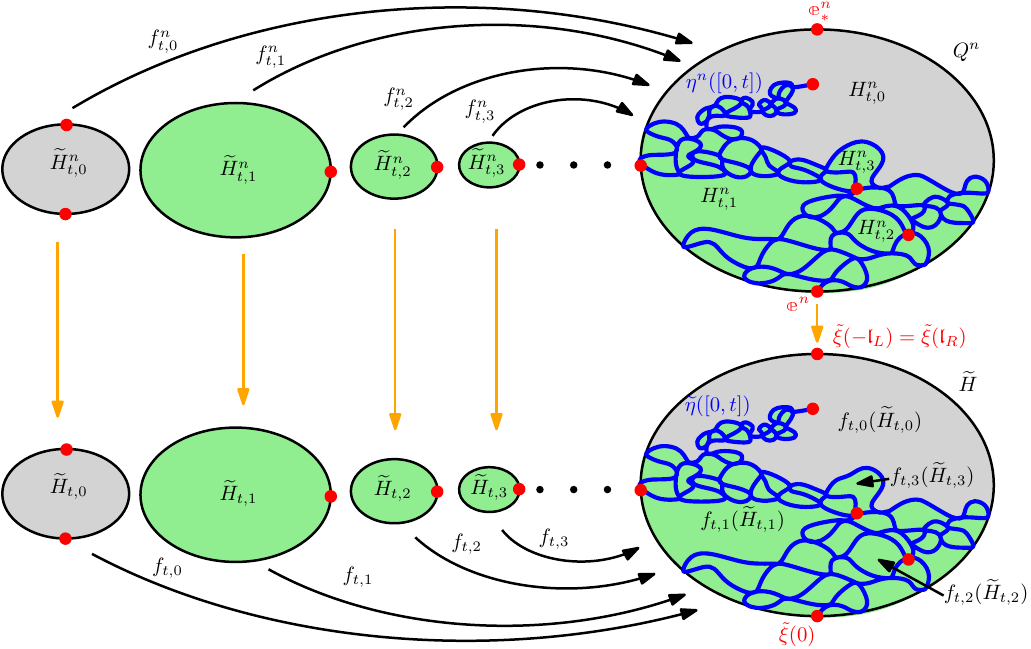} 
\caption[Illustration of the spaces and functions involved in Section~\ref{sec-identification}]{Illustration of the curve-decorated metric measure spaces and their subsequential limits considered in Section~\ref{sec-identification}. On the right are the main doubly curve-decorated metric measure spaces $\frk Q^n = (Q^n ,d^n , \mu^n , \xi^n , \eta^n)$ corresponding to a free Boltzmann quadrangulation equipped with its natural metric, area measure, and boundary path and a face percolation exploration path (appropriately rescaled); and their subsequential limit $\wt{\frk H} = (\wt H ,\wt d , \wt\mu,\wt\xi,\wt\eta)$, which is a free Boltzmann Brownian disk equipped with a curve $\wt\eta$ whose law we seek to identify. The spaces $\frk H_{t,k}^n = (H_{t,k}^n , d_{t,k}^n , \mu_{t,k}^n , \xi_{t,k}^n)$ for $(t,k) \in \BB Q_+ \times \BB N_0$ and $n\in\BB N$ are free Boltzmann quadrangulations equipped with their internal metrics, natural area measures, and boundary paths (appropriately rescaled): $\frk H_{t,0}^n$ is the unexplored quadrangulation at time $\lfloor \tcon t n^{3/4} \rfloor$ and $\frk H_{t,k}^n$ for $k\geq 1$ are the bubbles cut out by the percolation peeling process run up to time $\lfloor \tcon t n^{3/4} \rfloor$, in decreasing order of boundary length. The spaces $\wt{\frk H}_{t,k} = (\wt H_{t,k} , \wt d_{t,k} , \wt \mu_{t,k} , \wt\eta_{t,k} )$ are their subsequential limits. The orange arrows indicate HPU convergence in the compact metric spaces $\{(W_{t,k} , D_{t,k})\}_{k \in \BB N_0}$ and $(W,D)$ into which these spaces are isometrically embedded (see Section~\ref{sec-map-limit}). The maps $f_{t,k}$ are the limits of the inclusion maps $f_{t,k}^n$ in the sense of Lemma~\ref{lem-map-limit}. We show in Proposition~\ref{prop-component-law} that the sets $f_{t,k}(\wt H_{t,k} \setminus \bdy \wt H_{t,k})$ are the connected components of $\wt H\setminus \wt\eta([0,t])$. 
}\label{fig-map-limit}
\end{center}
\end{figure}

We first introduce a system for indexing the complementary connected components of the re-scaled boundary path $\eta^n(\cdot) = \lambda^n(\tcon n^{3/4} \cdot)$ run up to a specified time $t$.
For $n\in\BB N$ and $t \in \BB Q_+$, let 
\eqb \label{eqn-unexplored-metric-space}
H_{t,0}^n := \ol Q_{\lfloor \tcon t n^{3/4} \rfloor}^n \quad \op{and} \quad \BB e_{t,0}^n := \dot e_{\lfloor \tcon t n^{3/4} \rfloor+1}^n
\eqe
be the time-$\lfloor\tcon t n^{3/4} \rfloor$ unexplored region for the percolation peeling exploration, as in Section~\ref{sec-perc-peeling} (viewed as a connected metric space as in Remark~\ref{remark-ghpu-graph}) and the $(\lfloor \tcon tn^{3/4} \rfloor+1)$th peeled edge, respectively.

By the Markov property of peeling the conditional law of $( H_{t,0}^n , \BB e_{t,0}^n )$ given the percolation peeling $\sigma$-algebra~$\mcl F_{\lfloor \tcon t n^{3/4} \rfloor}^n$ is that of a free Boltzmann quadrangulation with simple boundary and perimeter $ \bcon n^{1/2} \left(  L_t^n +   R_t^n + \frk l_L^n + \frk l_R^n \right) $.
Let
\eqb \label{eqn-unexplored-perimeter-rescale}
\Delta_{t,0}^n :=   L_t^n +  R_t^n + \frk l_L^n + \frk l_R^n
\eqe 
be the rescaled perimeter of $H_{t,0}^n$. 

For $n\in\BB N$ and $(t,k) \in \BB Q_+ \times \BB N$, let $H_{t,k}^n$ be the bubble disconnected from the target edge $\BB e_*^n$ by the percolation peeling exploration of $(Q^n , \BB e^n , \theta^n)$ run up to time $\lfloor \tcon t n^{3/4} \rfloor$ whose perimeter is the $k$th largest among the perimeters of all such bubbles (with ties broken in some arbitrary deterministic manner), or let $H_{t,k}^n = \emptyset$ if there are fewer than $k$ such bubbles. 

Let $J_{t,k}^n$ be the time at which $  H_{t,k}^n$ is disconnected from $\infty$ by the percolation peeling exploration and let 
\eqb \label{eqn-bubble-t-time}
\tau_{t,k}^n := \tcon^{-1} n^{-3/4} J_{t,k}^n .
\eqe 
 We define the root edge for $ H_{t,k}^n$ to be the edge $\BB e_{t,k}^n \in \bdy H_{t,k}^n$ which is the leftmost edge of the peeled quadrilateral $\frk f\left( \ol Q_{J_{t,k}^n-1} , \dot e_{J_{t,k}^n} \right) \cap \bdy \rng H_{t,k}^n$ which belongs to $\bdy H_{t,k}^n$.

Let
\eqb \label{eqn-bubble-jump}
\Delta_{t,k}^n :=   \bcon^{-1} n^{-1/2} \#\mcl E(\bdy H_{t,k}^n)  ,
\eqe 
so that $\Delta_{t,k}^n$ differs from the downward jump of the re-scaled boundary length process $L^n$ (resp.\ $R^n$) at time $\tau_{t,k}^n$ by at most a universal constant times $n^{-1/2}$ if $H_{t,k}^n$ lies to the left (resp.\ right) of $\lambda^n$.  

We will view the bubbles $H_{t,k}^n$ as curve-decorated metric measure spaces.
For $n\in\BB N$ and $(t,k) \in \BB Q_+ \times \BB N_0$, 
let $d_{t,k}^n$ be the internal metric of $d^n$ on $ H_{t,k}^n$ (i.e., the graph metric on $H_{t,k}^n$ rescaled by $\bcon^{-1/2} n^{-1/4}$) and let $\mu_{t,k}^n := \mu^n|_{\ H_{t,k}^n}$. 
Let $ \beta_{t,k}^n : \BB Z \rta \bdy H_{t,k}^n$ be the periodic counterclockwise boundary path of $ H_{t,k}^n$ with $\beta_{t,k}^n(0)$ equal to the root edge $\BB e_{t,k}^n$, extended by linear interpolation in the manner of Remark~\ref{remark-ghpu-graph}, and let $ \xi_{t,k}^n(s) := \beta_{t,k}^n\left( \bcon n^{1/2} s\right)$ for $s\in \BB R$. Define the curve-decorated metric measure spaces 
\eqb \label{eqn-bubble-space}
\frk H_{t,k}^n := \left( H_{t,k}^n , d_{t,k}^n, \mu_{t,k}^n, \xi_{t,k}^n \right) ,\quad \forall (t,k) \in \BB Q_+ \times \BB N_0 .
\eqe 
We remind the reader that the case $k=0$ is special: $H_{t,0}^n$, defined as in~\eqref{eqn-unexplored-metric-space}, is the unexplored region at time $\lfloor \tcon  t n^{3/4} \rfloor$ whereas $H_{t,k}^n$ for $k\in\BB N$ is one of the bubbles disconnected from the target edge $\BB e_*^n$ by the percolation peeling exploration run up to time $\lfloor \tcon  t n^{3/4} \rfloor$ (or $\emptyset$).  

By the Markov property of peeling, if $n\in\BB N$ and $t\in \BB Q_+$ and we condition on the $\sigma$-algebra 
\eqb \label{eqn-smaller-filtration}
 \sigma\left( \frk P(\ol Q_{j-1}^n , \dot e_j^n ) , \theta_j^n : j \in[0, t n^{3/4}]_{\BB Z} \right) \subset \mcl F_{\lfloor \tcon t n^{3/4} \rfloor}^n
\eqe 
generated by the peeling indicators and colors of the peeled quadrangulations (but not the peeling clusters $\dot Q_j^n$) up to time $\lfloor \tcon t n^{3/4} \rfloor$
then the conditional law of the quadrangulations $\{( H_{t,k}^n , \BB e_{t,k}^n )\}_{k\in\BB N_0}$ is that of a collection of independent free Boltzmann quadrangulations with simple boundary and given perimeter.  

Since the laws of the processes $Z^n$ converge in the Skorokhod topology (Proposition~\ref{prop-stable-conv-finite}), the rescaled boundary lengths of the quadrangulations $\{H_{t,k}^n\}_{n\in\BB N}$ are tight for each $(t , k) \in \BB Q_+ \times \BB N_0$. By~\cite[Theorem~1.4]{gwynne-miller-simple-quad} the laws of the curve-decorated metric measure spaces $\{\frk H_{t,k}^n\}_{n\in\BB N}$ of~\eqref{eqn-bubble-space} are tight in the GHPU topology. 

By the Prokhorov theorem, after possibly passing to a subsequence of $\mcl N$ we can find a coupling of our original subsequential limiting pair $(\wt{\frk H} , Z)$ with curve-decorated metric measure spaces
\eqb \label{eqn-bubble-ssl}
\wt{\frk H}_{t,k}  := \left(  \wt H_{t,k} ,  \wt d_{t,k} ,  \wt\mu_{t,k}  ,  \wt\xi_{t,k} \right) ,\quad \forall (t,k) \in  \BB Q_+ \times \BB N_0
\eqe
such that the following convergence of joint laws holds as $\mcl N \ni n \rta\infty$: 
\eqb \label{eqn-ssl-conv}
\left( \frk Q^n , Z^n ,   \{ \frk H_{t,k}^n\}_{(t,k) \in \BB Q_+ \times \BB N_0} \right) \rta 
\left( \wt{\frk H} , Z , \{ \wt{\frk H}_{t,k} \}_{(t,k) \in \BB Q_+ \times \BB N} \right)  
\eqe
in the two-curve GHPU topology on the first coordinate, the Skorokhod topology on the second coordinate, and the countable product of the GHPU topology on the third coordinate.

By the Skorokhod representation theorem, we can couple the objects of~\eqref{eqn-ssl-conv} for $n\in\mcl N$ together in such a way the convergence~\eqref{eqn-ssl-conv} occurs a.s. In the remainder of this section we fix such a sequence $\mcl N$ and such a coupling.

In the next two lemmas, we identify the conditional law of the curve-decorated metric measure spaces~$\wt{\frk H}_{t,k} $ given the limiting boundary length process~$Z$.  For this purpose we define continuum analogs of some of the above objects.  For $t\in \BB Q_+$, define
\eqb \label{eqn-unexplored-perimeter-ssl}
\Delta_{t,0}  :=   L_t   +   R_t + \frk l_L  + \frk l_R .
\eqe 
For $(t,k) \in \BB Q_+ \times \BB N$, let $\tau_{t,k} $ be the time of the downward jump of either $L $ or $R $ before time $t$ with the $k$th largest magnitude. Also let  $\Delta_{t,k} := (L_{\tau_{t,k}^-} - L_{\tau_{t,k} }) \vee (R_{\tau_{t,k}^-} - R_{\tau_{t,k} })$ be the size (in absolute value) of this downward jump.

\begin{lem} \label{lem-bubble-time-conv}
Let $(t,k) \in \BB Q_+ \times \BB N_0$. In the notation introduced above, almost surely $\Delta_{t,k}^n \rta \Delta_{t,k} $ and, if $k \geq 1$, then almost surely $\tau_{t,k}^n \rta \tau_{t,k} $ and for large enough $n\in\mcl N$, the bubble $H_{t,k}^n$ lies to the left (resp.\ right) of $\lambda^n$ if and only if $L$ (resp.\ $R$) has a downward jump at time $\tau_{t,k}$. 
\end{lem}
\begin{proof}
It follows from local absolute continuity with respect to a pair of independent $3/2$-stable processes with no upward jumps (Lemma~\ref{lem-perc-limit-rn}) that a.s.\ the two coordinates $L$ and $R$ of $Z$ do not have any simultaneous downward jumps and neither of these coordinates has two downward jumps of the same magnitude or a downward jump at time $t$. Hence the Skorokhod convergence $Z^n\rta Z$ immediately implies the convergence conditions in the statement of the lemma.  
\end{proof}
 
\begin{lem} \label{lem-bubble-cond-ssl}
Let $t\in\BB Q_+$.  If we condition on $Z|_{[0,t]}$, then the curve-decorated metric measure spaces $\{ \wt{\frk H}_{t,k} \}_{k \in \BB N_0}$ are conditionally independent free Boltzmann Brownian disks with respective boundary lengths $\{ \Delta_{t,k} \}_{k \in \BB N_0}$, each equipped with its natural metric, area measure, and boundary path. 
\end{lem} 
\begin{proof}
Let $\mcl G_t^n$ be the $\sigma$-algebra of~\eqref{eqn-smaller-filtration} and note that $Z^n|_{[0,t]}$ and each $\Delta_{t,k}^n$ is $ \mcl G_t^n$-measurable. By the Markov property of peeling, if we condition on $ \mcl G_t^n$ then the curve-decorated metric measure spaces $\{ \frk H_{t,k}^n \}_{k\in\BB N_0}$ are conditionally independent free Boltzmann quadrangulations with respective perimeters $\{n^{1/2} \Delta_{t,k}^n\}_{k\in \BB N_0}$, each equipped with its rescaled metric, area measure, and boundary path. By~\cite[Theorem~1.4]{gwynne-miller-simple-quad} and Lemma~\ref{lem-bubble-time-conv}, the above described conditional laws given $\mcl G_t^n$, which are the same as the conditional laws given only $Z^n|_{0,t]}$ and $\{ \Delta_{t,k}\}_{k \in \BB N_0}$, converge as $\mcl N \ni n\rta\infty$ to the conditional laws described in the statement of the lemma. Since $\{ \Delta_{t,k} \}_{k \in \BB N_0}$ is a.s.\ determined by $Z |_{[0,t]}$, we obtain the statement of the lemma. 
\end{proof}

\subsection{Laws of complementary connected components}
\label{sec-map-limit}

Suppose we have fixed a subsequence $\mcl N$ and a coupling as in Section~\ref{sec-ssl}.  In this subsection we will establish several facts concerning the relationship between the main curve-decorated metric measure space $\wt{\frk H}$ and the curve-decorated metric measure spaces $\{\wt{\frk H}_{t,k}\}_{(t,k) \in \BB Q_+ \times \BB N_0}$ which are the subsequential limits of the complementary connected components of the curves $\eta^n([0,t])$.  This will in particular lead to the following proposition, which will be used to check condition~\ref{item-bead-mchar-wedge} of Theorem~\ref{thm-bead-mchar} for $\wt{\frk H}$.

\begin{prop}[Laws of complementary connected components] \label{prop-component-law}
For $t\geq 0$, let $\wt{\mcl U}_t$ be the collection of singly marked metric measure spaces of the form $(U , \wt d_U  , \wt\mu_U  , \wt x_U)$ where $U$ is a connected component of $\wt H \setminus \wt\eta ([0,t])$, $\wt d_U $ is the internal metric of $\wt d $ on $U$, and $\wt x_U$ is the point where $\wt\eta $ finishes tracing $\bdy U$.  If we condition on $Z |_{[0,t]}$, then the conditional law of $\wt{\mcl U}_t$ is that of a collection of independent singly marked free Boltzmann Brownian disks with boundary lengths specified as follows. The elements of $\wt{\mcl U}_t$ corresponding to the connected components of $\wt H \setminus \wt\eta ([0,t])$ which do not have the target point $\wt\xi(\frk l_R)$ on their boundaries are in one-to-one correspondence with the downward jumps of the coordinates of $Z|_{[0,t]}$, with boundary lengths given by the magnitude of the corresponding jump. The element of $\wt{\mcl U}_t$ corresponding to the connected component of $\wt H \setminus \wt\eta ([0,t])$ with $\wt\xi(\frk l_L)$ on its boundary has boundary length $L_t   +   R_t + \frk l_L  + \frk l_R $. 
\end{prop}

Proposition~\ref{prop-component-law} will follow from Lemma~\ref{lem-bubble-cond-ssl} once we establish that the curve-decorated metric measure spaces $\wt{\frk H}_{t,k}$ of~\eqref{eqn-bubble-ssl} are related to the main curve-decorated metric measure space $\wt{\frk H}$ in the appropriate manner, i.e., the $\wt{\frk H}_{t,k}$'s are the connected components of $\wt H\setminus \wt\eta([0,t])$, each equipped with the internal metric of $\wt d$ and the restriction of $\wt\mu$. This will be checked using an elementary metric space argument similar to~\cite[Section~7.3]{gwynne-miller-saw}. 

It will be convenient to view each of our sequences of convergent metric spaces and its limit as a sub-space of a common metric space.
By Proposition~\ref{prop-ghpu-embed}, there a.s.\ exist random compact metric spaces $(W,D)$ and $\{(W_{t,k},  D_{t,k})\}_{(t,k) \in \BB Q_+ \times \BB N_0}$ and isometric embeddings 
\allb \label{eqn-ghpu-embeddings}
&\iota : (\wt H , \wt d) \rta (W,D) ,\quad 
\iota^n : (Q^n,d^n) \rta (W,D) ,\: \forall n \in \mcl N , \notag \\ 
&\iota_{t,k} : ( \wt H_{t,k}  ,  \wt d_{t,k} ) \rta ( W_{t,k} ,  d_{t,k} )   , \quad \op{and} \quad 
\iota_{t,k}^n : ( H_{t,k}^n ,  d_{t,k}^n) \rta ( W_{t,k} ,  D_{t,k} )  , \: \forall n \in \mcl N 
\alle 
such that a.s.\ $\iota^n(\frk Q^n) \rta \iota(\wt{\frk H})$ in the $D$-HPU topology (Definition~\ref{def-hpu}) and $ \iota_{t,k}^n( \frk H_{t,k}^n) \rta  \iota_{t,k} ( \wt{\frk H}_{t,k} )$ in the $D_{t,k}$-HPU topology for each $(t,k) \in \BB Q_+ \times \BB N_0$.

We henceforth identify the doubly curve-decorated metric measure space $\wt{\frk H}$ with its image under $\iota$ and each of the doubly curve-decorated metric measure spaces $\frk Q^n$ for $n\in\BB N$ with its image under $\iota^n$.  We also identify $\wt{\frk H}_{t,k} $ for $(t,k) \in \BB Q_+ \times \BB N_0$ with its image under $\iota_{t,k}$.  Since $\frk H_{t,k}^n$ for $n\in\mcl N$ and $(t,k) \in \BB Q_+ \times \BB N_0$ has already been identified with its image under $\iota^n$ (recall that $H_{t,k}^n \subset Q^n$) we write
\eqb \label{eqn-embed-bubble}
\wt{\frk H}_{t,k}^n = \left(  \wt H_{t,k}^{n} ,  \wt d_{t,k}^{n} ,  \wt \mu_{t,k}^{ n} ,  \wt\xi_{t,k}^{ n}      \right) :=  \iota_{t,k}^n( \frk H_{t,k}^n) .
\eqe 
We also define maps
\eqb \label{eqn-embed-maps} 
f_{t,k}^n = (\iota_{t,k}^n)^{-1} : \wt H_{t,k}^{ n} \rta H_{t,k}^n \subset Q^n .
\eqe  
See Figure~\ref{fig-map-limit} for an illustration of the above maps.

We now check that the maps $f_{t,k}^n$ admit subsequential limits $f_{t,k} : \wt H_{t,k} \rta \wt H$ (in an appropriate sense) and establish some basic properties of the maps $f_{t,k}$. We eventually aim to show that the sets $f_{t,k}(\wt H_{t,k} \setminus \bdy \wt H_{t,k})$ are the connected components of $\wt H\setminus (\wt\eta([0,t]) \cup \bdy\wt H)$, which will be established in Lemma~\ref{lem-map-conclude}. This statement together with Lemma~\ref{lem-bubble-cond-ssl} will imply Proposition~\ref{prop-component-law}.
On a first read, the reader may want to read only the statements of the next two lemmas and skip their proofs.

\begin{lem}
\label{lem-map-limit}
Almost surely there is a (random) subsequence $\mcl N'\subset \mcl N$ and maps $f_{t,k} : \wt H_{t,k}  \rta \wt H$ for $(t,k) \in \BB Q_+ \times \BB N_0$ such that the following hold for each $(t,k) \in \BB Q_+ \times \BB N_0$.
\begin{enumerate}
\item\label{item-map-limit-conv}  The maps $\{f_{t,k}^n\}_{n \in \mcl N'}$ converge to $f_{t,k}$ in the following sense.
For each $x \in \wt H_{t,k} $, each subsequence $\mcl N''\subset \mcl N'$, and each sequence of points $x^n \in \wt H_{t,k}^{n}$ for $n\in\mcl N''$ such that $D_{t,k} (x^n,x) \rta 0$ as $\mcl N''\ni n \rta\infty$, we have $D(f_{t,k}^n(x^n) , f_{t,k}(x)) \rta 0$.  
\item\label{item-map-limit-haus} $f_{t,k}^n(\wt H_{t,k}^n) \rta f_{t,k}(\wt H_{t,k})$ in the $D$-Hausdorff distance. 
\item \label{item-map-limit-dist}For each sequence $x^n\rta x$ as in condition~\ref{item-map-limit-conv}, 
\allb \label{eqn-map-limit-dist}
&\lim_{\mcl N'' \ni n \rta\infty} \wt d_{t,k}^n\left(x^n ,\bdy \wt H_{t,k}^n \right)
= \lim_{\mcl N'' \ni n\rta\infty} d^n\left(f_{t,k}^n(x^n) , \eta^n([0,t]) \cup \bdy Q^n \right)  \notag \\
&\qquad = \wt d_{t,k}\left(x , \bdy \wt H_{t,k}\right)  
= \wt d\left(f_{t,k}(x) , \wt\eta([0,t]) \cup \bdy \wt H\right) .
\alle
In particular, $f_{t,k}(\bdy \wt H_{t,k}) \subset \wt\eta([0,t]) \cup\bdy \wt H$ and $f_{t,k}(\wt H_{t,k} \setminus \bdy \wt H_{t,k}) \cap (\wt\eta([0,t]) \cup \bdy \wt H) = \emptyset$. 
\item\label{item-map-limit-internal} For each $x\in \wt H_{t,k} \setminus \bdy \wt H_{t,k}$ and $\rho \in \left( 0, \frac13 \wt d_{t,k}(x , \bdy \wt H_{t,k}) \right)$, the map $f_{t,k}|_{B_\rho(x;\wt d_{t,k})}$ is an isometry from $(B_\rho(x; \wt d_{t,k}) , \wt d_{t,k})$ to $(B_\rho(f_{t,k}(x) ; \wt d) , \wt d)$.
\item\label{item-map-limit-measure} For each $x$ and $\rho$ as in condition~\ref{item-map-limit-internal}, we have $\wt\mu_{t,k}(A) = \wt\mu(f_{t,k}(A))$ for each Borel set $A\subset B_\rho(x; \wt d_{t,k})$. 
\end{enumerate} 
\end{lem}
\begin{proof}
\noindent\textit{Proof of condition~\ref{item-map-limit-conv}.} Each of the maps $f_{t,k}^n$ is $1$-Lipschitz from $(\wt H_{t,k}^{ n} , \wt d_{t,k}^{ n} )$ to $(W,D)$, so the existence of a subsequence $\mcl N'\subset \mcl N$ and maps $\{f_{t,k}\}_{(t,k) \in \BB Q_+ \times \BB N_0}$ satisfying condition~\ref{item-map-limit-conv} is immediate from~\cite[Lemma~2.1]{gwynne-miller-uihpq} (plus a diagonalization argument to get a subsequence which works for all $(t,k) \in \BB Q_+ \times \BB N_0$ simultaneously). Henceforth fix such a subsequence $\mcl N'$ and maps $f_{t,k}$.
\medskip

\noindent\textit{Proof of condition~\ref{item-map-limit-haus}.} It is clear from condition~\ref{item-map-limit-conv} that any subsequential limit of the sets $f_{t,k}^n(\wt H_{t,k}^n) = H_{t,k}^n$ for $n\in\mcl N'$ in the $D$-Hausdorff distance must coincide with $f_{t,k}(\wt H_{t,k} )$, so since $(W,D)$ is compact we infer that condition~\ref{item-map-limit-haus} holds.
\medskip


\noindent\textit{Proof of condition~\ref{item-map-limit-dist}.} Fix a subsequence $\mcl N''\subset \mcl N'$ and a sequence $x^n\rta x$ as in condition~\ref{item-map-limit-conv}. 
For each $n\in\mcl N''$, we have $f_{t,k}^n(\bdy \wt H_{t,k}^n) = \bdy H_{t,k}^n \subset \eta^n([0,t]) \cup \bdy Q^n$, $f_{t,k}^n(\wt H_{t,k}^n \setminus \bdy \wt H_{t,k}^n) \cap (\eta^n([0,t]) \cup \bdy Q^n) =\emptyset$, and any path in $H_{t,k}^n = f_{t,k}^n(\wt H_{t,k}^n)$ which hits $\eta^n([0,t]) \cup \bdy Q^n$ must pass through $\bdy H_{t,k}^n$. Hence 
\eqb \label{eqn-bdy-dist-n}
  \wt d_{t,k}^n\left(x^n ,\bdy \wt H_{t,k}^n \right) =  d^n\left(f_{t,k}^n(x^n) , \eta^n([0,t]) \cup \bdy Q^n \right) + o_n(1)   
\eqe 
where here the $o_n(1)$ is a deterministic rounding error coming from the fact that $\eta^n$ does not trace every edge of each peeled quadrilateral.
Since $x^n\rta x$ and $\bdy \wt H_{t,k}^n$ is parameterized by the path $\xi_{t,k}^n $, which converges uniformly to the parameterization $\wt \xi_{t,k}$ of $\bdy \wt H_{t,k}$, we infer that the left side of~\eqref{eqn-bdy-dist-n} converges a.s.\ to $\wt d_{t,k}\left(x , \bdy \wt H_{t,k}\right)$. On the other hand, condition~\ref{item-map-limit-conv} implies that $f_{t,k}^n(x^n) \rta f_{t,k}(x)$. Since $\eta^n \rta \wt\eta$ and the boundary parameterizations $\xi^n \rta \wt\xi$ uniformly, we infer that the right side of~\eqref{eqn-bdy-dist-n} converges a.s.\ to $\wt d\left(f_{t,k}(x) , \wt\eta([0,t]) \cup \bdy \wt H\right) $. Thus~\eqref{eqn-map-limit-dist} holds. 

The last statement of condition~\ref{item-map-limit-dist} follows since by~\eqref{eqn-map-limit-dist}, if $x\in \wt H_{t,k}$ then $x \in \bdy \wt H_{t,k}$ if and only if $\wt d_{t,k}\left(x , \bdy \wt H_{t,k}\right) = 0$ if and only if $\wt d\left(f_{t,k}(x) , \wt\eta([0,t]) \cup \bdy \wt H\right) = 0$ if and only if $f_{t,k}(x) \in \wt\eta([0,t]) \cup \bdy \wt H$. 
\medskip

\noindent\textit{Proof of condition~\ref{item-map-limit-internal}.} Let $x\in \wt H_{t,k}$ and $\rho \in \left( 0, \frac13 \wt d_{t,k}(x , \bdy \wt H_{t,k}) \right)$. Also fix $0 < \ep < \frac13 \wt d_{t,k}(x , \bdy \wt H_{t,k}) - \rho$ and points $y_1,y_2 \in B_\rho(x ; \wt d_{t,k})$. 

Since $\wt H_{t,k}^n \rta \wt H_{t,k}$ in the $D_{t,k}$-Hausdorff distance, we can find points $y_1^n,y_2^n , x^n\in \wt H_{t,k}^n$ for $n\in\mcl N'$ such that a.s.\ $x^n\rta x$, $y_1^n\rta y_1$, and $y_2^n \rta y_2$ as $\mcl N' \ni n \rta\infty$. By condition~\ref{item-map-limit-conv}, a.s.\ $f_{t,k}^n(x^n) \rta f_{t,k}(x)$, $f_{t,k}^n(y_1^n) \rta f_{t,k}(y_1)$, and $f_{t,k}^n(y_2^n) \rta f_{t,k}(y_2)$. 
By our choice of $\rho$ and $\ep$ together with condition~\ref{item-map-limit-dist}, for large enough $n\in\mcl N'$ we have 
\eqbn
\wt d_{t,k}^n\left(x^n , \bdy \wt H_{t,k}^n \right) > 3\rho +3\ep \quad \op{and} \quad \wt d_{t,k}^n\left( x^n , y_i^n \right) < \rho + \ep ,\quad \forall i \in \{1,2\} .
\eqen 
If this is the case, then $y_1^n$ and $y_2^n$ are $\wt d_{t,k}^n$-closer to each other than to $\bdy \wt H_{t,k}^n$, so since every path in $H_{t,k}^n = f_{t,k}^n(\wt H_{t,k}^n)$ which exits $H_{t,k}^n$ must pass through $\bdy H_{t,k}^n$, 
\eqbn
\wt d_{t,k}^n(y_1^n , y_2^n) = d_{t,k}^n \left(  f_{t,k}^n(y_1^n) , f_{t,k}^n(y_2^n) \right) = d^n\left(  f_{t,k}^n(y_1^n) , f_{t,k}^n(y_2^n) \right)   . 
\eqen
Taking a limit as $n\rta\infty$ shows that $\wt d_{t,k}(y_1,y_2) = d(f_{t,k}(y_1) , f_{t,k}(y_2))$. 
Therefore $f_{t,k}$ is distance-preserving on $B_\rho(x ; \wt d_{t,k} )$. 

We still need to show that $f_{t,k}(B_\rho(x ; \wt d_{t,k})) = B_\rho( f_{t,k}(x) ; \wt d)$. It is clear from the preceding paragraph that $f_{t,k}(B_\rho(x ; \wt d_{t,k})) \subset B_\rho( f_{t,k}(x) ; \wt d)$, so we just need to prove the reverse inclusion.
Since $f_{t,k}^n(x^n) \rta f_{t,k}(x)$ and $Q^n \rta \wt H $ in the $D$-local Hausdorff distance,  
\eqb \label{eqn-map-ball-conv}
B_\rho\left( f_{t,k}^n(x^n)  ; d^n \right) \rta B_\rho\left( f_{t,k}(x) ; \wt d \right) 
\eqe
in the $D$-Hausdorff distance. By~\eqref{eqn-map-ball-conv}, for each $z\in  B_\rho\left( f_{t,k}(x) ; \wt d \right) $, there exists a sequence of points $z^n \in B_\rho\left( f_{t,k}^n(x^n)  ; d^n \right)$ for $n\in\mcl N'$ such that $z^n \rta z$. 
By condition~\ref{item-map-limit-dist} and our choice of $\rho$, for large enough $n\in\mcl N' $, $z^n$ is $d^n$-closer to $f_{t,k}^n(x^n)$ than to $\bdy H_{t,k}^n$, so $z^n\in H_{t,k}^n$ and 
\eqb \label{eqn-ball-dist-compare}
d^n\left( f_{t,k}^n(x^n) , z^n \right) = d_{t,k}^n(f_{t,k}^n(x^n) , z^n) = \wt d_{t,k}^n\left( x^n , (f_{t,k}^n)^{-1}(z^n)  \right) .
\eqe 
By compactness of $(W_{t,k} , D_{t,k})$, there is a subsequence $\mcl N''$ of $\mcl N'$ and a $y \in \wt H_{t,k}$ such that $(f_{t,k}^n)^{-1}(z^n) \rta y$ as $\mcl N'' \ni n \rta\infty$. By condition~\ref{item-map-limit-conv}, $f_{t,k}(y) = z$. 
The left side of~\eqref{eqn-ball-dist-compare} converges to $d^n(f_{t,k}(x) , z) \leq \rho$ and the right side converges to $\wt d_{t,k}(x,y)$. Therefore $y\in B_\rho(x ; \wt d_{t,k})$ so since our initial choice of $z \in B_\rho(f_{t,k}(x) ; \wt d)$ was arbitrary, we see that $f_{t,k}(B_\rho(x ; \wt d_{t,k})) =  B_\rho( f_{t,k}(x) ; \wt d)$. 
\medskip

\noindent\textit{Proof of condition~\ref{item-map-limit-measure}.} Let $x \in \wt H_{t,k}$, $\rho > 0$, and $x^n \in \wt H_{t,k}^n$ for $n\in\mcl N'$ be as above and choose $\rho' > \rho$ such that $\rho' <   \frac13 \wt d_{t,k}(x , \bdy \wt H_{t,k})$ and 
\eqbn
\wt\mu_{t,k}\left( \bdy B_{\rho'}\left(x ; \wt d_{t,k} \right) \right) =     \wt\mu\left( \bdy B_{\rho'}\left( f_{t,k}(x) ; \wt d \right) \right) = 0 .
\eqen
By this condition together with the HPU convergence $\wt{\frk H}_{t,k}^n \rta \wt{\frk H}_{t,k}$ and $\frk Q^n \rta \wt{\frk H}$ as $\mcl N'\ni n \rta\infty$, 
\eqb \label{eqn-ball-measure-conv}
\wt\mu_{t,k}^n|_{ B_{\rho'}(x^n ; \wt d_{t,k}^n )} \rta \wt\mu_{t,k} |_{ B_{\rho'}\left(x ; \wt d_{t,k} \right)}
\eqe 
in the $D_{t,k}$-Prokhorov metric and 
\eqb \label{eqn-ball-measure-conv'}
 \mu^n|_{ B_{\rho'}(f_{t,k}^n(x) ; d^n )} \rta  \wt\mu|_{ B_{\rho'}\left( f_{t,k}(x) ;  \wt d\right)} 
\eqe 
in the $D $-Prokhorov metric.  

Conditional on everything else, for $n\in \mcl N'$ let $w^n$ be sampled uniformly from $\wt\mu_{t,k}^n|_{ B_{\rho'}(x^n ; \wt d_{t,k}^n )} $ (normalized to be a probability measure) and let $w$ be sampled uniformly from $\wt\mu_{t,k}|_{ B_{\rho'}\left(x ; \wt d_{t,k} \right)}$ (normalized to be a probability measure). By~\eqref{eqn-ball-measure-conv} $w^n\rta w$ in law, so by the Skorokhod representation theorem we can couple together $\{ w^n\}_{n\in\mcl N'} $ and $w$ in such a way that a.s.\ $w^n \rta w$ as $\mcl N' \ni n \rta\infty$. By condition~\ref{item-map-limit-conv}, $f_{t,k}^n(w^n) \rta f_{t,k}(w)$. By our choice of $\rho'$, for each sufficiently large $n \in \mcl N'$, 
\eqbn
B_{\rho'}(f_{t,k}(x^n) ; d^n ) = B_{\rho'}(f_{t,k}(x^n) ; d_{t,k}^n ) .
\eqen
For such an $n$ the law of $f_{t,k}^n(w^n)$ is that of a uniform sample from $ \mu^n|_{ B_{\rho'}(f_{t,k}(x^n) ; d^n )}$. By~\eqref{eqn-ball-measure-conv'}, the law of $f_{t,k}(w)$ is that of a uniform sample from $ \wt\mu|_{ B_{\rho'}\left( f_{t,k}(x) ;  \wt d\right)}$. We similarly infer from~\eqref{eqn-ball-measure-conv} and~\eqref{eqn-ball-measure-conv'} that  
\eqbn
\wt\mu_{t,k}\left(  B_{\rho'}\left(x ; \wt d_{t,k}\right) \right) =  \wt \mu\left( B_{\rho'}\left( f_{t,k}(x) ;  \wt d\right) \right) .
\eqen
Therefore,  
\eqbn
(f_{t,k})_*\left(   \wt\mu_{t,k}|_{ B_{\rho'}\left(x ; \wt d_{t,k} \right)} \right) =  \wt \mu|_{ B_{\rho'}\left( f_{t,k}(x) ; \wt d\right)} ,
\eqen
which implies condition~\ref{item-map-limit-measure}.
\end{proof}

From Lemma~\ref{lem-map-limit}, we can deduce some further properties of the maps $f_{t,k}$.

\begin{lem} \label{lem-map-property}
Suppose we are in the setting of Lemma~\ref{lem-map-limit} and fix a subsequence $\mcl N'$ and maps $\{f_{t,k} : (t,k) \in \BB Q_+ \times \BB N_0\}$ satisfying the conditions of that lemma. Almost surely, the following is true for each $t\in \BB Q_+$.
\begin{enumerate}
\item For each distinct $k_1,k_2 \in \BB N_0$, we have $f_{t,k_1} (\wt H_{t,k_1} \setminus \bdy \wt H_{t,k_1}) \cap f_{t,k_2}(\wt H_{t,k_2})  = \emptyset$. \label{item-map-disjoint} 
\item For $k\in\BB N_0$, let $\rng d_{t,k}$ be the internal metric of $\wt d$ on $f_{t,k}(\wt H_{t,k} \setminus \bdy \wt H_{t,k} )$. Then $f_{t,k}$ is an isometry from $( \wt H_{t,k}  \setminus \bdy \wt H_{t,k} , \wt d_{t,k})$ to $(f_{t,k}(\wt H_{t,k}   \setminus \bdy \wt H_{t,k} ) , \rng d_{t,k})$. \label{item-map-iso}
\item For each $k \in \BB N_0$ and each Borel set $A\subset \wt H_{t,k} \setminus \bdy \wt H_{t,k}$, we have $\wt\mu_{t,k}(A) = \wt\mu(f_{t,k}(A))$. \label{item-map-measure'}
\item $\wt\mu\left( \wt H \setminus \bigcup_{k=0}^\infty f_{t,k}(\wt H_{t,k} \setminus \bdy \wt H_{t,k}) \right) = 0$. \label{item-map-mass}
\item $\bigcup_{k=0}^\infty f_{t,k}(\wt H_{t,k} \setminus \bdy \wt H_{t,k}) = \wt H\setminus ( \wt\eta([0,t]) \cup \bdy \wt H)$. \label{item-map-cover}
\end{enumerate}
\end{lem}
\begin{proof}
Throughout the proof, we assume we are working on the full probability event that the conditions of Lemma~\ref{lem-map-limit} are satisfied, and omit the qualifier ``a.s.". 
\medskip

\noindent\textit{Proof of condition~\ref{item-map-disjoint}.} Suppose $k_1,k_2\in \BB N_0$ with $k_1\not= k_2$ and $z \in f_{t,k_1}(\wt H_{t,k_1}) \cap f_{t,k_2}(\wt H_{t,k_2})$. 
By condition~\ref{item-map-limit-haus} of Lemma~\ref{lem-map-limit}, we can choose $x_1^n \in \wt H_{t,k_1}^n$ and $x_2^n\in \wt H_{t,k_2}^n$ for $n\in\mcl N'$ such that $f_{t,k_1}^n(x_1^n)$ and $f_{t,k_2}^n(x_2^n)$ each converge to $z$ as $\mcl N'\ni n \rta\infty$. By compactness of $(W_{t,k_i} , D_{t,k_i})$ for $i\in \{1,2\}$, we can find  a subsequence $\mcl N''$ of $\mcl N'$ and points $x_1 \in \wt H_{t,k_1}$ and $x_2 \in \wt H_{t,k_2}$ such that $x_1^n\rta x_1$ and $x_2^n\rta x_2$ as $\mcl N'' \ni n \rta\infty$. By condition~\ref{item-map-limit-conv} of Lemma~\ref{lem-map-limit}, $f_{t,k_1}(x_1) = f_{t,k_2}(x_2) = z$. Since $d^n(f_{t,k_1}^n(x_1^n) , f_{t,k_2}^n(x_2^n) ) \rta 0$, also
\eqbn
d^n(f_{t,k_1}^n(x_1^n) , \bdy H_{t,k_1}^n ) = \wt d_{t,k_1}^n( x_1^n ,\bdy \wt H_{t,k_1}^n) \rta 0 .
\eqen
By condition~\ref{item-map-limit-dist} of Lemma~\ref{lem-map-limit}, $x_1 \in \bdy \wt H_{t,k_1} $ so (by the last statement of that condition) $z = f_{t,k_1}(x_1) \notin f_{t,k_1}(\wt H_{t,k_1} \setminus \bdy \wt H_{t,k_1})$.  
\medskip

\noindent\textit{Proof of condition~\ref{item-map-iso}.} The interior of the Brownian disk $(\wt H_{t,k} \setminus \partial \wt H_{t,k} , \wt d_{t,k})$ is homeomorphic to an open Euclidean disk (see, e.g.,~\cite{bet-disk-tight}), so is $\sigma$-compact. Hence we can choose countably many balls $B_j = B_{\rho_j}(x_j; \wt d_{t,k})$ for $j\in\BB N$ such that $x_j \in \wt H_{t,k} \setminus \bdy \wt H_{t,k}$ and $\rho_j \in \left(0 , \wt d_{t,k}(x_j , \bdy \wt H_{t,k} ) \right)$ whose union covers $\wt H_{t,k} \setminus \bdy \wt H_{t,k}$. By condition~\ref{item-map-limit-internal} of Lemma~\ref{lem-map-limit}, $f|_{B_j}$ is an isometry from $(B_j , \wt d_{t,k})$ to $(f_{t,k}(B_j ) , \wt d)$ for each $j\in \BB N_0$.

If $\gamma : [0,T] \rta \wt H_{t,k} \setminus \bdy \wt H_{t,k}$ is a path with finite $\wt d_{t,k}$-length, then the image of $\gamma$ is compact so we can find finitely many times $ 0 = s_0 < \dots < s_N = T$ such that $\gamma([s_{i-1} , s_i])$ is contained in a single one of the balls $B_j$ for each $i \in [1,N]_{\BB Z}$. 
Hence the $\wt d$-length of $f_{t,k}(\gamma)$ coincides with the $\wt d_{t,k}$-length of $\gamma$. Since this holds for every such path $\gamma$, we obtain condition~\ref{item-map-iso}. 
\medskip

\noindent\textit{Proof of condition~\ref{item-map-measure'}.} Define the balls $\{B_j\}_{j\in\BB N}$ as above. 
For $j\in \BB N$, let $C_j := B_j\setminus \bigcup_{i=1}^{j-1} B_i$ so that the $C_j$'s are disjoint and their union covers $\wt H_{t,k} \setminus \bdy \wt H_{t,k}$. By condition~\ref{item-map-limit-measure} of Lemma~\ref{lem-map-limit}, for each Borel set $A\subset \wt H_{t,k} \setminus \bdy \wt H_{t,k}$ it holds that $\wt\mu_{t,k}(C_j \cap A) = \wt\mu(f_{t,k}(C_j\cap A))$ for each $j\in\BB N$. Summing over all $j$ yields condition~\ref{item-map-measure'}. 
\medskip

\noindent\textit{Proof of condition~\ref{item-map-mass}.}
It follows from Lemma~\ref{lem-bubble-cond-ssl} that a.s.\ $\wt\mu_{t,k}(\bdy \wt H_{t,k}) = 0$ for each $k\in\BB N_0$. Hence conditions~\ref{item-map-disjoint} and~\ref{item-map-measure'} together imply that
\eqb \label{eqn-mass-dominate}
\sum_{k=0}^\infty \wt\mu_{t,k}(\wt H_{t,k} ) \leq \wt\mu(\wt H) .
\eqe 
We will now argue that $\sum_{k=0}^\infty \wt\mu_{t,k}(\wt H_{t,k} ) \eqD  \wt\mu(\wt H)$, so that~\eqref{eqn-mass-dominate} is a.s.\ an equality, which in turn will imply condition~\ref{item-map-mass}. 
 
Recall that the law of the area of a free Boltzmann Brownian disk with unit boundary length is given by 
\eqb \label{eqn-area-law}
\frac{1}{ \sqrt{2\pi a^5 } } e^{-\frac{1}{2 a} } \BB 1_{(a\geq 0)}\, da
\eqe
 where $da$ denotes Lebesgue measure on $[0,\infty)$. The law of the area of a free Boltzmann Brownian disk with boundary length $\frk l$ can be obtained by sampling a random variable from the law~\eqref{eqn-area-law} and then multiplying it by $\frk l^2$. 
Let $\{X_k\}_{k\in\BB N_0}$ be a collection of i.i.d.\  random variables with the law~\eqref{eqn-area-law}, independent from everything else, and recall the boundary length $  \Delta_{t,k}$ of $\wt H_{t,k}$ for $k\in\BB N_0$ defined as in~\eqref{eqn-unexplored-perimeter-ssl}. By Lemma~\ref{lem-bubble-cond-ssl},  
\eqb \label{eqn-area-sum-expand}
\sum_{k=0}^\infty \wt\mu_{t,k} (\wt H_{t,k}  ) \eqD \sum_{k=0}^\infty  \Delta_{t,k}^2 X_k .
\eqe  

The process $Z$ has the same law as the left/right boundary length process of a chordal $\SLE_6$ on an independent doubly marked Brownian disk with left/right boundary lengths $\frk l_L$ and $\frk l_R$, parameterized by quantum natural time. 
The quantities $\{\Delta_{t,k}\}_{k \in \BB N_0}$ are precisely the set of boundary lengths of the complementary connected components of such an $\SLE_6$ curve run up to time $t$. If we condition on these boundary lengths, then by Theorem~\ref{thm-sle-bead-nat} the conditional law of the collection of internal metric measure spaces corresponding to these connected components is that of a collection of independent quantum disks with given boundary lengths.
Since the area of a chordal $\SLE_6$ curve on an independent doubly marked Brownian disk is a.s.\ equal to zero, we infer that the sum of the areas of these connected components is a.s.\ equal to the total mass of the Brownian disk. Therefore the right side of~\eqref{eqn-area-sum-expand}, and hence also the left side of~\eqref{eqn-mass-dominate}, has the same law as $\wt \mu(\wt H)$. Thus condition~\ref{item-map-mass} holds.
 \medskip

\noindent\textit{Proof of condition~\ref{item-map-cover}.} Condition~\ref{item-map-limit-dist} of Lemma~\ref{lem-map-limit} implies that $\bigcup_{k=0}^\infty f_{t,k}(\wt H_{t,k} \setminus \bdy \wt H_{t,k}) \subset \wt H \setminus ( \wt\eta([0,t]) \cup \bdy \wt H) $, so we just need to prove the reverse inclusion. 
To this end, suppose $z \in   \wt H\setminus ( \wt\eta([0,t]) \cup \bdy \wt H)$. We seek $k \in \BB N_0$ such that $z\in f_{t,k}(\wt H_{t,k} \setminus \bdy \wt H_{t,k})$. Choose $\rho \in \left( 0, \wt d(z , \wt\eta([0,t]) \cup \bdy \wt H) \right)$.  
We can find a sequence of points $z^n \in Q^n$ for $n\in\mcl N'$ such that $D(z^n,z) \rta 0$. Since $\xi^n\rta \wt\xi$ and $\eta^n \rta \wt\eta$ $D$-uniformly, we have $\liminf_{n\rta\infty} d^n(z^n , \eta^n([0,t]) \cup \bdy Q^n) > \rho$. In particular, $z^n\in Q^n\setminus (\eta^n([0,t]) \cup \bdy Q^n)$, so there exists $k^n\in \BB N_0$ such that $z^n \in  H_{t,k^n}^n$ and
\eqb \label{eqn-bdy-dist}
\liminf_{n\rta\infty}  d^n( z^n , \bdy  H_{t,k^n}^n ) > \rho .
\eqe 

We claim that $\liminf_{\mcl N' \ni n\rta\infty} k^n < \infty$. Suppose by way of contradiction that this is not the case, i.e.\ $k^n \rta \infty$. Any path in $Q^n$ from $z^n$ to a point not in $H_{t,k^n}^n$ must pass through $\bdy H_{t,k^n}^n$. Hence for each fixed $k\in\BB N_0$, 
\eqbn
\liminf_{n\rta\infty}  d^n( z^n ,  H_{t,k}^n ) \geq    \liminf_{n\rta\infty}  d^n( z^n , \bdy  H_{t,k^n}^n ) > \rho .
\eqen
By condition~\ref{item-map-limit-haus} of Lemma~\ref{lem-map-limit}, $\wt d(z , f_{t,k}(\wt H_{t,k}) ) > \rho$ for each such $k$. Hence $B_\rho(z ; \wt d)$ is disjoint from each $f_{t,k}(\wt H_{t,k})$. By condition~\ref{item-map-mass}, $\wt\mu(B_\rho(z;\wt d)) = 0$, which contradicts the fact that the natural area measure $\wt\mu$ of the Brownian disk $\wt H$ a.s.\ assigns positive mass to every open set, so proves our claim.  
 
Hence $\liminf_{\mcl N' \ni n\rta\infty} k^n < \infty$, so there exists $k\in\BB N_0$ and a subsequence $\mcl N''\subset \mcl N'$ such that $k^n = k$ for each $n\in \mcl N''$. By compactness of $(W_{t,k} , D_{t,k})$, after possibly passing to a further subsequence we can arrange that $(f_{t,k}^n)^{-1} (z^n) \rta x \in \wt H_{t,k}$ as $\mcl N'' \ni n \rta\infty$. By condition~\ref{item-map-limit-conv} of Lemma~\ref{lem-map-limit}, $f_{t,k}(x) = z$ and by condition~\ref{item-map-limit-dist} of Lemma~\ref{lem-map-limit}, $z \in \wt H_{t,k} \setminus \bdy \wt H_{t,k}$. Thus $z\in \bigcup_{k=0}^\infty f_{t,k}(\wt H_{t,k} \setminus \bdy \wt H_{t,k}) $, as required.
\end{proof}

The following lemma summarizes two of the most important implications of the preceding lemmas. 

\begin{lem}
\label{lem-map-conclude}
Suppose we are in the setting of Lemma~\ref{lem-map-limit} and fix a subsequence $\mcl N'$ and maps $\{f_{t,k} : (t,k) \in \BB Q_+ \times \BB N_0\}$ satisfying the conditions of that lemma. Almost surely, the following is true.
\begin{enumerate}
\item For each $t\in \BB Q_+$, the sets $f_{t,k}(\wt H_{t,k} \setminus \bdy \wt H_{t,k})$ for $k\in \BB N_0$ are precisely the connected components of $\wt H\setminus (\wt\eta([0,t]) \cup \bdy \wt H)$. \label{item-map-component}
\item The mass of the limiting curve $\wt\eta$ satisfies $\wt\mu(\wt\eta) = 0$. \label{item-map-curve-mass}
\end{enumerate}
\end{lem}
\begin{proof}
By condition~\ref{item-map-limit-internal} of Lemma~\ref{lem-map-limit}, each of the sets $ f_{t,k}(\wt H_{t,k} \setminus \bdy \wt H_{t,k})$ for $k\in\BB N_0$ is open. 
By condition~\ref{item-map-disjoint} of Lemma~\ref{lem-map-property}, these sets are disjoint and by condition~\ref{item-map-cover} of Lemma~\ref{lem-map-property}, their union is $\wt H\setminus (\wt\eta([0,t]) \cup \bdy \wt H)$. These facts together imply condition~\ref{item-map-component}. 

By conditions~\ref{item-map-mass} and~\ref{item-map-cover} of Lemma~\ref{lem-map-property}, for each $t\in\BB Q_+$,
\eqbn
\wt\mu\left( \wt\eta([0,t]) \right) \leq \wt\mu\left( \wt H \setminus \bigcup_{k=0}^\infty f_{t,k}(\wt H_{t,k} \setminus \bdy \wt H_{t,k})   \right) = 0.
\eqen
Sending $t\rta\infty$ gives condition~\ref{item-map-curve-mass}.
\end{proof}

\begin{proof}[Proof of Proposition~\ref{prop-component-law}]
By Lemma~\ref{lem-bubble-cond-ssl} and condition~\ref{item-map-component} of Lemma~\ref{lem-map-conclude}, the statement of the proposition is true for each $t\in \BB Q_+$. 
Since $\wt\eta$ is continuous, $Z$ is right continuous, and the law of a free Boltzmann Brownian disk depends continuously on its boundary length in the GHPU topology (by scaling) the statement for general $t \geq 0$ follows by taking limits.
\end{proof}

\subsection{$Z$ is the boundary length process of $\wt\eta$}
\label{sec-bdy-process-id}

Recall that $ Z$ has the same law as the left/right boundary length process for a chordal $\SLE_6$ on an independent doubly marked free Boltzmann Brownian disk with left/right boundary lengths $\frk l_L$ and $\frk l_R$, parameterized by quantum natural time.  The goal of this subsection and the next is to establish the existence of an $\SLE_6$-decorated Brownian disk and a homeomorphism $\Phi$ satisfying condition~\ref{item-bead-mchar-homeo} of Theorem~\ref{thm-bead-mchar}. 

In the present subsection, we will show that $\wt\eta$ intersects itself at least as often as a chordal $\SLE_6$ whose boundary length process is $Z$ (Lemma~\ref{lem-curve-top}) and that $Z$ determines the boundary length measure, not just the total boundary length, of the complementary connected components of $\wt\eta$ (Lemma~\ref{lem-bubble-bdy}).  These two statements will be used to show the existence of the desired map $\Phi$ in Section~\ref{sec-homeo} below (showing the injectivity of $\Phi$ will also require the estimates of Section~\ref{sec-crossing}). 

The proofs in this section are based on elementary limiting arguments together with the description of the topology of $\eta$ in terms of the left/right boundary length process given at the end of Section~\ref{sec-lqg-bdy-process}. The reader may wish to skip the rest of this section on a first read and refer back to the various lemmas as they are used.

Throughout the remainder of this section, we write
\eqb \label{eqn-continuum-terminal-time}
\sigma_0 := \inf\left\{ t\geq 0 :  L_t = -\frk l_L \: \op{or} \: R_t = -\frk l_R \right\} ,
\eqe 
as in Section~\ref{sec-lqg-bdy-process},
for the terminal time of $Z$, so that $Z_t = (-\frk l_L , -\frk l_R)$ for each $t\geq \sigma_0$. 
 
The following lemma tells us that the self-intersection times of $\wt\eta$ and the times when it hits $\bdy \wt H$, respectively, are at least as frequent as the self-intersection times and boundary intersection times of a chordal $\SLE_6$ on a Brownian disk with boundary length process $Z$ (recall~\eqref{eqn-inf-adjacency0} from Section~\ref{sec-lqg-bdy-process} and the discussion just after).

\begin{lem}
\label{lem-curve-top}
Almost surely, the following is true.
\begin{enumerate}
\item\label{item-curve-top-terminal} $\wt\eta(t) = \wt\xi(\frk l_R)$ for each time $t$ after the terminal time $\sigma_0$. 
\item\label{item-curve-top-double} For each $t_1,t_2 \geq 0$ with $t_1<t_2$ such that either
\eqb
\label{eqn-inf-adjacency}
 L_{t_1}  =  L_{t_2} = \inf_{s\in [t_1,t_2]}  L_s \quad\op{or} \quad
 R_{t_1}  =  R_{t_2} = \inf_{s\in [t_1,t_2]}   R_s
\eqe 
it holds that $\wt\eta(t_1) = \wt\eta(t_2)$. 
\item \label{item-curve-top-bdy} For each time $t\geq 0$ at which $  L$ (resp.\ $R$) attains a record minimum, it holds that $\wt\eta(t) = \wt\xi(L_t)$ (resp.\ $\wt\eta(t) = \wt\xi(- R_t)$). 
\end{enumerate}
\end{lem}

For the proof of Lemma~\ref{lem-curve-top}, we need some elementary regularity properties for the process $Z$. 

\begin{lem}
\label{lem-bdy-connect-reg}
Almost surely, the limiting boundary length process $ Z = (L,R)$ satisfies the following properties. Suppose $0 < t_1<t_2 \leq  \sigma_0$ are such that 
\eqb
\label{eqn-inf-adjacency-reg}
  L_{t_1}  =  L_{t_2} = \inf_{s\in [t_1,t_2]}  L_s.
\eqe
Then
\begin{enumerate}
\item\label{item-bdy-connect-initial} For each $\ep > 0$ there is a time $t_1' \in [t_1-\ep , t_1]$ with $ L_{t_1' } <   L_{t_1}$. 
\item\label{item-bdy-connect-triple} There is no time $t \in (t_1,t_2)$ such that $L_t = L_{t_1} = L_{t_2}$. 
\item\label{item-bdy-connect-jump} For each $\ep > 0$, there are times $t_1' \in [t_1, t_1 +\ep]$ and $t_2' \in [t_2-\ep , t_2]$ such that $ L_{t_1'}  = L_{t_2'} = \inf_{s\in [t_1' , t_2' ]}   L_s$ and $L$ does not attain a local minimum at time $t_1'$ or $t_2'$.
\end{enumerate}
The same holds with $R$ in place of $ L$. 
\end{lem}
\begin{proof}
Let $Z^\infty = (L^\infty,R^\infty)$ be a pair of independent totally asymmetric $3/2$-stable processes with no upward jumps.  Recall from~\cite[Corollary~1.19]{wedges} that $Z^\infty$ is the left/right boundary length for a chordal $\SLE_6$ $\eta^\infty$ from $0$ to $\infty$ on an independent $\sqrt{8/3}$-quantum wedge, parameterized by quantum natural time.  If there are times $t_1, t_2 \in [0,\infty)$ with $t_1<t_2$ such that $L^\infty_{t_1} = L^\infty_{t_2} = \inf_{s \in [t_1,t_2]} L^\infty_s$ and $L^\infty$ attains a running infimum relative to time~$0$ at time $t_1$, then $\eta^\infty(t_1) = \eta^\infty(t_2) $ lies in the boundary of the quantum wedge.  Since chordal $\SLE_6$ a.s.\ does not have a boundary double point~\cite[Remark~5.3]{miller-wu-dim}, we see that a.s.\ no such times $t_1<t_2$ exist. 

Since $L^\infty$ has stationary increments there a.s.\ do not exist times $0 \leq t_0 < t_1 < t_2$ such that $L^\infty_{t_1} = L^\infty_{t_2} = \inf_{s \in [t_1,t_2]} L^\infty_s$ and $L^\infty$ attains a running infimum relative to time $t_0$ at time $t_1$.  Since the law of $Z$ is absolutely continuous with respect to the law of $Z^\infty$ up to any time prior to the terminal time $\sigma_0$ (Theorem~\ref{thm-bdy-process-law}) we infer that there a.s.\ do not exist times $0 \leq t_0 < t_1 < t_2 \leq  \sigma_0$ such that $ L_{t_1} =   L_{t_2} = \inf_{s \in [t_1,t_2]}   L_s$ and $L$ attains a running infimum relative to time $t_0$ at time $t_1$.  Applying this with $t_0 = t_1-\ep$ shows that condition~\ref{item-bdy-connect-initial} is satisfied.

Condition~\ref{item-bdy-connect-triple} is immediate from condition~\ref{item-bdy-connect-initial} applied with $t$ in place of $t_1$. 

We now consider condition~\ref{item-bdy-connect-jump}.  Almost surely, the $3/2$-stable process $L^\infty$ has only countably many downward jumps and no two such jumps have the same size. Hence a.s.\ each time at which $L^\infty$ has a downward jump takes the form $\inf\{t\geq q : L_t^\infty - \lim_{s\rta t^-} L_s^\infty \in I\}$ for some $q\in \BB Q_+$ and some interval $I\subset (-\infty,0)$ with rational endpoints. By the strong Markov property and since there are only countably many such times $q$ and intervals $I$, we infer that a.s.\ there is no time at which $L^\infty$ has a downward jump and attains a local minimum. By local absolute continuity the same is a.s.\ true of $L$.  Hence for the proof of condition~\ref{item-bdy-connect-jump} we can assume without loss of generality that $L$ does not have a downward jump at time $t_2$. 

Now set $t_* := (t_1+t_2)/2$ and for $\delta > 0$, let
\eqbn
t_2^\delta := \inf\left\{ t  \geq t_* : L_t \leq L_{t_1} + \delta \right\}\quad \op{and} \quad
t_1^\delta := \sup\left\{ t  \leq t_* : L_t = L_{t_2^\delta} \right\}.
\eqen
Since $L$ does not have a downward jump at time $t_2$, we have $t_2^\delta  <t_2$ and $t_1^\delta > t_1$.  By definition,~\eqref{eqn-inf-adjacency-reg} holds with $t_1^\delta$ and $t_2^\delta$ in place of $t_1$ and~$t_2$.  By condition~\ref{item-bdy-connect-triple}, $t_1^\delta \rta t_1$ and $t_2^\delta \rta t_2$ as $\delta \rta 0$.  By the preceding paragraph, $L$ does not have a local minimum at any of the times~$t_2^\delta$ at which~$L$ has a downward jump, and such times exist for arbitrarily small~$\delta$. By condition~\ref{item-bdy-connect-initial}, $L$ does not have a local minimum at any of the times~$t_1^\delta$. Thus condition~\ref{item-bdy-connect-jump} holds.

The statement for $R$ follows from symmetry. 
\end{proof}

\begin{proof}[Proof of Lemma~\ref{lem-curve-top}]
If $t \in  \BB Q_+$, then $Z_t = (-\frk l_L , -\frk l_R)$ on the event $\{t > \sigma_0\}$ so since $Z^n\rta Z$ in the Skorokhod topology on this event the rescaled boundary length $\Delta_{t,0}^n$ of the unexplored region~$H_{t,0}^n$ at time~$t$ tends to~$0$ a.s.\ as $n\rta\infty$. Since the conditional law of~$H_{t,0}^n$ given~$\Delta_{t,0}^n$ is that of a free Boltzmann quadrangulation with simple boundary, we infer from~\cite[Theorem~1.4]{gwynne-miller-simple-quad} that $\op{diam} \left( H_{t,0}^n; d_{t,0}^n \right) \rta 0$ in probability on the event $\{t > \sigma_0\}$. Since $\eta^n([t,\infty)) \subset H_{t,0}^n$ (modulo rounding error) and $\eta^n\rta \wt\eta$ uniformly, we infer that a.s.\ $\wt\eta$ is constant on $(\sigma_0 ,\infty) \cap \BB Q_+$ and hence, by continuity, on all of $[\sigma_0,\infty)$. Since the target edges satisfy $\BB e_*^n = \xi^n(\frk l_R^n) + O_n(n^{-1/4}) \rta \wt\xi(\frk l_R)$ in $(W,D)$, we infer that $\wt\eta(\sigma_0 ) = \wt\xi(\frk l_R)$. 

If $t_1,t_2 \geq 0$ are such that $L_{t_1}  = L_{t_2} = \inf_{s\in [t_1,t_2]} L_s$ and $L$ does not have a local minimum at either $t_1$ or $t_2$, then the Skorokhod convergence $L^n\rta L$ implies that we can find sequences $t_1^n \rta t_1$ and $t_2^n\rta t_2$ such that for each $n\in\BB N$,
\eqbn
L^n_{t_1^n}  = L^n_{t_2^n} \leq \inf_{s\in (t_1^n,t_2^n)} L^n_s.
\eqen
By~\eqref{eqn-bdy-process-inf} applied to the percolation peeling process after time $\lfloor n^{3/4} t_1^n \rfloor$, this latter condition implies that $d^n(\eta^n(t_1^n) , \eta^n(t_2^n)) = O_n(n^{-1/4})$ (where here the $O_n(n^{-1/4})$ comes from rounding error).  Since $\eta^n \rta \wt \eta$ uniformly, we infer that $\wt\eta(t_1) = \wt\eta(t_2)$ for each such pair of times $t_1$, $t_2$. By continuity of $\wt\eta$ together with condition~\ref{item-bdy-connect-jump} of Lemma~\ref{lem-bdy-connect-reg}, we find that $\wt\eta(t_1) =\wt\eta(t_2)$ whenever $L_{t_1}  = L_{t_2} = \inf_{s\in [t_1,t_2]} L_s$. We similarly obtain the analogous statement with $R$ in place of $L$. Hence condition~\ref{item-curve-top-double} holds.

To check condition~\ref{item-curve-top-bdy}, for $r \in [0,\frk l_L]$ let $T_r = \inf\{t\geq 0 : L_t \leq - r\}$ and $T_r^n := \inf\{t\geq 0: L_t^n \leq -r\}$.  Also let $T_{r^+} = \lim_{u \rta r^+} T_u$.  By~\eqref{eqn-bdy-process-inf}, $d^n(\eta^n(T_r^n) , \xi^n(r)) = O_n(n^{-1/4})$.  If $T_r = T_{r^+}$ (which by the monotonicity of $r\mapsto T_r$ is the case for all but countably many values of $r$) then by the Skorokhod convergence $Z^n\rta Z$ we have $T_r^n \rta T_r$.  Since $\eta^n \rta \wt\eta$ and $\xi^n\rta\wt\xi$ uniformly, we infer that $\wt\eta(T_r) = \xi(r)$ whenever $T_r = T_{r^+}$.  On the other hand, $  L_{T_r} = L_{T_{r^+}}  = \inf_{s\in [T_r ,T_{r^+}]} L_s$, whence condition~\ref{item-curve-top-double} implies that $\wt\eta(T_r) = \wt\eta(T_{r^+})$. By the continuity of~$\wt\eta$ we infer that a.s.\ $\wt\eta(T_r) = \xi(r)$ for each $r\in [0,\frk l_L]$. We similarly obtain the analogous statement for record minima of~$R$. 
\end{proof}

We next describe the boundary length measures on the connected components $f_{t,k}(\wt H_{t,k} \setminus \bdy \wt H_{t,k})$ of $\wt H\setminus (\wt\eta([0,t]) \cup \bdy \wt H)$  (Lemma~\ref{lem-map-conclude}) in terms of the process $Z$. In particular, this description will be the same as the description of the boundary length measure on the complementary connected components of an $\SLE_6$ on an independent quantum disk given at the end of Section~\ref{sec-lqg-bdy-process}. 

For this purpose we introduce the notation
\eqbn
\ul L_t = \inf_{s\in[0,t]} L_s \quad \op{and} \quad \ul R_t = \inf_{s\in [0,t]} R_s. 
\eqen
We similarly define $\ul L_t^n$ and $\ul R_t^n$ for the rescaled discrete boundary length processes.
To describe the boundary length measure in the case when $k=0$ (which we recall corresponds to the complementary connected component containing the target point), for 
$t\in \BB Q_+$ and $u \in \left[ 0, L_t  - \ul{L}_t \right]$ (resp.\ $u \in \left[ 0 , R_t - \ul{R}_t \right]$) let
\eqb  \label{eqn-backup-time}
T^L_{t,0}(u) := \sup\left\{ s \leq t : L_s \leq L_t  - u \right\} \quad \op{and} \quad 
T^R_{t,0}(u) := \sup\left\{ s \leq t : R_s \leq R_t  - u \right\} .
\eqe 
Since $  L  $ has no upward jumps, we infer that a.s.\ $L_{T^L_{t,0}(u)} = L_t -u$ for each $u \in \left[ 0, L_t  - \ul{L}_t \right]$ and similarly for $R$. 

In the case when $k \in \BB N$, we recall the boundary length $\Delta_{t,k}$ of $\wt H_{t,k}$ from the discussion just above Lemma~\ref{lem-bubble-time-conv}. 
For $(t,k) \in \BB Q_+ \times \BB N$ such that $L$ (resp.\ $R$) has a downward jump at time $\tau_{t,k}$ and $u \in [0, \Delta_{t,k} \wedge (L_{ \tau_{t,k}^-} - \ul{L}_{\tau_{t,k}^-} )]$ (resp.\ $u \in [0, \Delta_{t,k} \wedge ( R_{\tau_{t,k}^-} - \ul{R}_{\tau_{t,k}^-} )]$), define 
\eqb \label{eqn-bubble-backup-time}
T_{t,k}^L(u) := \sup\left\{ s  < \tau_{t,k} : L_s \leq L_{\tau_{t,k}^-}  - u \right\} \qquad
\text{(resp.\ }
T_{t,k}^R(u) := \sup\left\{ s  < \tau_{t,k} : R_s \leq R_{\tau_{t,k}^-}  - u \right\} \text{)} .
\eqe

\begin{lem}
\label{lem-bubble-bdy}
Suppose we are in the setting of Lemma~\ref{lem-map-limit} and fix a subsequence $\mcl N'$ and maps $\{f_{t,k} : (t,k) \in \BB Q_+ \times \BB N_0\}$ satisfying the conditions of that lemma.  In the notation described just above, for each $u\in [ - \frk l_L - L_t , \frk l_R +  R_t   ]$ the image of the boundary path of $\wt H_{t,0}$ satisfies
\eqb \label{eqn-unexplored-bdy-path}
f_{t,0}(\wt\xi_{t,0}(u)) = 
\begin{cases}
\wt\xi\left( u + L_t  \right) ,\quad &u \in \left[ -(L_t + \frk l_L ) ,  - (L_t - \ul{L}_t) \right]\\
\wt\eta\left(  T^L_{t,0}(-u) \right) , \quad &u \in [- (L_t - \ul{L}_t) , 0 ] \\
\wt\eta\left( T^R_{t,0}(u) \right) ,\quad &u \in [0, R_t - \ul{R}_t ] \\
  \wt\xi\left( u -  R_t  \right)   ,\quad &u \in [R_t -  \ul{R}_t ,   R_t + \frk l_R ]  . 
\end{cases}
\eqe 
Furthermore, the image of the boundary paths for the bubbles $\wt H_{t,k}$ for $k\in\BB N$ which lie to the left of $\lambda^n$ satisfy
\eqb \label{eqn-bubble-bdy-path-L}
f_{t,k}(\wt\xi_{t,k}(u)) = 
\begin{cases} 
\wt\eta\left(  T_{t,k}^L( u) \right) , \quad &u \in [ 0 , \Delta_{t,k} \wedge (L_{\tau_{t,k}^-} -  \ul{L}_{\tau_{t,k}^-})      ] \\ 
 \wt\xi\left( u - L_{\tau_{t,k}^-} \right)   ,\quad &u \in [ \Delta_{t,k} \wedge (L_{\tau_{t,k}^-} -  \ul{L}_{\tau_{t,k}^-}),  \Delta_{t,k} ]  ; 
\end{cases}
\eqe 
and a similar formula holds with $R$ in place of $L$ for the bubbles to the right of $\lambda^n$. 
\end{lem}
 
\begin{proof} 
For $n\in\mcl N$, let
\eqbn
T^{L,n}_{t,0}(u) := \sup\left\{ s < t :   L^n_s \leq L^n_t  - u \right\} \quad \op{and} \quad 
T^{R,n}_{t,0}(u) := \sup\left\{ s < t : R^n_s \leq R^n_t  - u \right\}  
\eqen
be the rescaled discrete analogs of the times in~\eqref{eqn-backup-time}. 
Recalling~\eqref{eqn-bdy-process-inf}, we see that for $n\in \mcl N$ and $u\in [-\frk l_L^n - L_t^n , \frk l_R^n + R_t^n]$, the $d^n$-distance from $f_{t,0}^n(\xi_{t,0}^n(u))  $ to $z^n(u)$ is bounded above by a deterministic rounding error of order $O_n(n^{-1/4})$, where here
\eqbn 
z^n(u) := 
\begin{cases}
\xi^n\left( u +   L^n_t  \right) ,\quad &u \in \left[ -(L^n_t + \frk l_L^n ) ,  - ( L^n_t - \ul{L}^n_t) \right]\\
\eta^n\left(  T^{L,n}_{t,0}(-u) \right) , \quad &u \in [- (L^n_t - \ul{ L}^n_t) , 0 ] \\
\eta^n\left( T^{R,n}_{t,0}(u) \right) ,\quad &u \in [0, R^n_t - \ul{ R}^n_t ] \\
  \xi^n\left( u -    R_t^n \right)   ,\quad &u \in [ R_t^n -  \ul{  R}^n_t ,   R_t^n + \frk l_R^n ]   .
\end{cases}
\eqen
 
For $u \in \left[ 0, L_t  - \ul{L}_t \right]$ (resp.\ $u \in \left[ 0 , R_t - \ul{R}_t \right]$) let
\eqbn
  T^L_{t,0}(u^+) :=  \lim_{v\rta u^+} T^L_{t,0}(v) \leq T^L_{t,0}(u) \quad \op{and} \quad 
  T^R_{t,0}(u^+) :=  \lim_{v\rta u^+} T^R_{t,0}(v) \leq T^R_{t,0}(u) .
\eqen   
By the Skorokhod convergence $L^n\rta L$, we infer that a.s.\ $L^n_{T_{t,0}^{L,n}(u)} \rta L_{T_{t,0}^L(u)}$ for each $u \in \left[ 0, L_t  - \ul{L}_t \right]$ such that $ L_{T^L_{t,0}(u^+)}  =  L_{T^L_{t,0}(u )}$. 
By monotonicity of $u\mapsto T_{t,0}^L(u)$, the set of times $u \in \left[ 0, L_t  - \ul{L}_t \right]$ for which this is not the case is a.s.\ countable.
Analogous statement hold for $R$ and $R^n$. Since a.s.\ neither $L$ nor $R$ has a jump at time $t$, we also have the a.s.\ convergence of running infima at time $t$, $\ul L_t^n \rta \ul{L}_t  $ and $\ul R_t^n \rta \ul{R}_t  $. 

By combining the above boundary length process convergence statements with the a.s.\ uniform convergence $\eta^n\rta \wt\eta$ and $\xi^n\rta\wt\xi$, we see that a.s.\ $z^n(u) \rta \wt z(u)$ for all but countably many $u\in [- \frk l_L - L_t , \frk l_R +  R_t   ]$, where here $\wt z(u)$ denotes the right side of~\eqref{eqn-unexplored-bdy-path}. Since also $\xi_{t,0}^n  \rta \wt \xi_{t,0}$ uniformly a.s., we infer from condition~\ref{item-map-limit-conv} of Lemma~\ref{lem-map-limit} and the above relation between $f_{t,0}^n(\xi_{t,0}^n(u))$ and $z^n(u)$ that a.s.~\eqref{eqn-unexplored-bdy-path} holds for all but countably many $u \in [ - \frk l_L - L_t , \frk l_R +  R_t   ]$. 

It is clear that $u\mapsto f_{t,0}(\wt\xi_{t,0}(u)) $ is continuous. We will now argue that a.s.\ $u\mapsto \wt z(u)$ is continuous, so that a.s.~\eqref{eqn-unexplored-bdy-path} holds for all $u \in [ - \frk l_L - L_t , \frk l_R +  R_t   ]$ simultaneously. 
By right continuity of $L$, it is a.s.\ the case that for each $u \in \left[ 0, L_t  - \ul{L}_t \right]$,
\eqbn
L_{T^L_{t,0}(u^+)} = L_{T^L_{t,0}(u )} = L_t - u = \inf_{s \in [ T^L_{t,0}(u^+)  ,  T^L_{t,0}(u ) ]} L_s
\eqen
and similarly with $R$ in place of $L$. By Lemma~\ref{lem-curve-top}, a.s.\  
\eqb \label{eqn-eta-on-excursion}
\wt\eta( T^L_{t,0}(u^+) ) = \wt\eta( T^L_{t,0}(u ) )  ,\quad \forall u \in \left[ 0, L_t  - \ul{L}_t \right]
\eqe 
and similarly with $R$ in place of $L$. Since $\wt\eta$ is continuous and $u \mapsto  T^L_{t,0}(u ) $ and $u\mapsto T^R_{t,0}(u)$ are right-continuous function of $u$, we obtain the desired continuity for $\wt z$. Thus~\eqref{eqn-unexplored-bdy-path} holds.

The formula~\eqref{eqn-bubble-bdy-path-L} is proven via a similar argument (here we recall Lemma~\ref{lem-bubble-time-conv}). 
\end{proof}

\subsection{Existence of a homeomorphism}
\label{sec-homeo}

In this subsection we will establish the following proposition, which implies that the subsequential limit $(\wt{\frk H}, Z)$ satisfies condition~\ref{item-bead-mchar-homeo} of Theorem~\ref{thm-bead-mchar}. 

\begin{prop}[Topology and consistency] \label{prop-homeo}
The topology of $(\wt H ,\wt\eta )$ is determined by $Z $ in the same manner as the topology of a chordal $\SLE_6$ on a free Boltzmann Brownian disk, 
i.e.\ there is a doubly curve-decorated metric measure space $( H ,d , \mu  , \xi  , \eta )$ consisting of a free Boltzmann Brownian disk with left/right boundary lengths $\frk l_L$ and $\frk l_R$ equipped with its natural metric, area measure, and boundary path decorated by an independent chordal $\SLE_6$ from $\xi(0)$ to $\xi(\frk l_R)$, parameterized by quantum natural time, such that the following is true almost surely. 
There is a homeomorphism $\Phi : H \rta \wt H$ such that $\Phi\circ\eta=\wt\eta$, $\Phi_* \mu = \wt\mu$, and for each $t\in \BB Q_+$, $\Phi $ a.s.\ pushes forward the natural length measure on the boundary of the connected component of $H \setminus \eta ([0,t])$ with $\xi(\frk l_R)$ on its boundary to the natural boundary length measure on the connected component of $\wt H  \setminus \wt\eta ([0,t])$ with $\wt\xi(\frk l_R)$ on its boundary (which is well-defined since we know from Proposition~\ref{prop-component-law} the internal metric on this component is that of a Brownian disk). 
\end{prop}

To prove Proposition~\ref{prop-homeo}, we will first establish the existence of the SLE$_6$-decorated Brownian disk $(H,d,\mu,\xi,\eta)$ and a map $\Phi : H\rta \wt H$ satisfying all of the conditions in the proposition statement, except that $\Phi$ is not a priori known to be injective on the range of $\eta$, in the following manner. We know from Proposition~\ref{prop-component-law} that the joint law of $Z$ and the bubbles cut out by $\wt\eta$ is the same as the law of the left/right boundary length process for SLE$_6$ on a Brownian disk and the bubbles it cuts out. Hence we can choose $(H,d,\mu,\xi,\eta)$ in such a way that the left/right boundary length process for $\eta$ is equal to $Z$ and the bubbles cut out by $\eta$ (viewed as curve-decorated metric measure spaces) are the same as the bubbles cut out by $\wt\eta$. The results of Section~\ref{sec-bdy-process-id} tell us that $\wt\eta$ hits itself at least as often as $ \eta$, so since the corresponding bubbles cut out by $\eta$ and $\wt\eta$ agree, this gives us a measure-preserving, curve-preserving surjection $\Phi :  H \rta \wt H$ which is an isometry away from $\eta$ (Lemma~\ref{lem-surjection}). 

To show that $\Phi$ is a homeomorphism, we will use Proposition~\ref{prop-fb-crossing} to show that the pre-image of any point of $\wt H$ under $\Phi$ has cardinality at most 6 (Lemma~\ref{lem-ssl-curve-hit}), then apply the criterion of~\cite[Main Theorem]{almost-inj}, as discussed in Section~\ref{sec-overview}.

Suppose we are in the setting of Lemma~\ref{lem-map-limit} and fix a subsequence $\mcl N' \subset \mcl N$ and maps $\{f_{t,k} :(t,k) \in \BB Q_+ \times \BB N\}$ as in that lemma.  By conditions~\ref{item-curve-top-terminal} and~\ref{item-curve-top-bdy} of Lemma~\ref{lem-curve-top}, we have $\wt\eta(t) = \wt\eta(\sigma_0) = \wt\xi(\frk l_R)$ for each time $t$ after the terminal time $\sigma_0$ of~\eqref{eqn-continuum-terminal-time}.  Let $N$ be the smallest integer which is at least $\sigma_0$ and for $k\in\BB N_0$ let 
\eqbn
\wt{\frk H}_{\infty,k}  = \left(\wt H_{\infty,k} , \wt d_{\infty , k} ,  \wt\mu_{\infty,k }  , \wt\xi_{\infty,k} \right)  := \wt{\frk H}_{N,k}
\eqen
and $f_{\infty,k} : = f_{N,k}$. Then $\wt{\frk H}_{t,k} = \wt{\frk H}_{\infty,k}$ for each rational $t\geq \sigma_0$ and each $k\in\BB N_0$. Furthermore, the unexplored Brownian disk $\wt{\frk H}_{N,0}$ degenerates to the trivial one-point curve-decorated metric measure space and condition~\ref{item-map-component} of Lemma~\ref{lem-map-conclude} implies that and that the sets $f_{\infty,k}(\wt H_{\infty,k} \setminus \bdy \wt H_{\infty,k})$ are precisely the connected components of $\wt H\setminus ( \wt\eta \cup \bdy \wt H)$. 
 
By Lemma~\ref{lem-bubble-cond-ssl} and conditions~\ref{item-map-iso} and~\ref{item-map-measure'} of Lemma~\ref{lem-map-property}, we find that the conditional law given $ Z$ of these connected components, each viewed as a metric measure space equipped with the internal metric of $\wt d$ and the restriction of $\wt\mu$, is that of a collection of independent free Boltzmann Brownian disks with boundary lengths specified by the magnitudes of the downward jumps of the coordinates $ L$ and $ R$. 
By Theorem~\ref{thm-sle-bead-nat}, this conditional law is the same as the conditional law of the collection of singly-marked metric measure spaces corresponding to the bubbles cut out by a chordal $\SLE_6$ on an independent doubly marked free Boltzmann Brownian disk with left/right boundary lengths $\frk l_L$ and $\frk l_R$, given the left/right boundary length process. 

Hence there exists a doubly curve-decorated metric measure space $\frk H = (H , d , \mu  , \xi , \eta)$ such that $(H,d,\mu,\xi)$ is a free Boltzmann Brownian disk with boundary length $\frk l_L + \frk l_R$ equipped with its natural metric, area measure, and boundary path; and $\eta$ is an independent chordal $\SLE_6$ from $\xi(0)$ to $\xi(\frk l_R)$, parameterized by quantum natural time such that $ Z$ is the same as the left/right boundary length process for $\eta$ and the following is true a.s. For $k\in\BB N$ let $H_k$ be the closure of the connected component of $H\setminus  \eta$ with the $k$th largest boundary length; let $d_k$ be the internal metric of $d$ on $H_k\setminus \bdy H_k$, extended by continuity to all of $H_k$; and let~$\xi_k$ be the periodic boundary path of~$H_k$ such that~$\xi_k(0)$ is the point where~$\eta$ finishes tracing~$\bdy H_k$ (which is well-defined since~$(H_k ,d_k)$ is a Brownian disk). 
There is an isometry
\eqb \label{eqn-component-map}
g_k  : (\wt H_{\infty,k} , \wt d_{\infty,k}) \rta (H_k  ,d_k) 
\eqe 
such that $(g_k)_* \wt\mu_{\infty,k} = \mu|_{H_k}$ and $g_k \circ \wt\xi_{\infty,k} = \xi_k$.

For $t\geq 0$, let $H_{t,0}$ be the closure of the connected component of $H\setminus \eta([0,t])$ with the target point $\xi(\frk l_R)$ on its boundary. 
The internal metric of $d$ on $H_{t,0}\setminus \bdy H_{t,0}$ is that of a free Boltzmann Brownian disk, so we can define the periodic boundary path $\xi_{t,0} : \BB R\rta \bdy H_{t,0}$ with $\xi_{t,0}(0) = \eta_{t,0}$. 

The following lemma establishes the existence of a map satisfying all of the properties in Proposition~\ref{prop-homeo} except that $\Phi$ is only injective on $H\setminus (\eta \cup \bdy H)$, not necessarily injective everywhere.

\begin{lem} \label{lem-surjection}
Almost surely, there exists a continuous surjective map $\Phi : (H,d) \rta (\wt H , \wt d)$ such that $\Phi\circ \eta = \wt\eta$, $\Phi_* \mu = \wt\mu$, $\Phi\circ \xi_{t,0} =f_{t,0} \circ \wt\xi_{t,0}$ for each $t\in \BB Q_+$, and for each $k\in\BB N$ it holds that $\Phi|_{H_k\setminus \bdy H_k}$ is an isometry between the internal metric spaces $(H_k\setminus \bdy H_k , d_k)$ and $(f_{\infty,k}(\wt H_{\infty,k}\setminus \bdy \wt H_{\infty,k}) , \rng d_{\infty,k})$, where here $\rng d_{\infty,k}$ is the internal metric of $\wt d$ on $\wt H_{\infty,k}$, as in Lemma~\ref{lem-map-property}.
\end{lem}
\begin{proof}
\noindent \textit{Step 1: definition of $\Phi|_{\eta \cup \bdy H}$.} 
 Let $X $ (resp.\ $\wt X$) be the topological space $\eta\cup \bdy H$ (resp.\ $\wt \eta\cup \bdy \wt H$), equipped with the topology it inherits from $H$ (resp.\ $\wt H$).
Then $X$ (resp.\ $\wt X$) is the image of $[0,\infty) \sqcup [-\frk l_L , \frk l_R]$ under the continuous surjection $\eta \sqcup \xi$ (resp.\ $\wt\eta\sqcup\wt\xi$). Since $\eta$ (resp.\ $\wt\eta$) extends continuously to $[0,\infty]$ (see condition~\ref{item-curve-top-terminal} of Lemma~\ref{lem-curve-top} in the case of $\wt\eta$), it follows by compactness that this continuous surjection is in fact a quotient map. 

By the discussion at the end of Section~\ref{sec-lqg-bdy-process} and a compactness argument as above, $X$ is the topological quotient of $[0,\infty)\times[-\frk l_L , \frk l_R]$ under the equivalence relation which identifies $t_1,t_2\in [0,\infty)$ whenever
\eqbn
\wt L_{t_1}  = \wt L_{t_2} = \inf_{s\in [t_1,t_2]} \wt L_s \quad\op{or} \quad
\wt R_{t_1}  = \wt R_{t_2} = \inf_{s\in [t_1,t_2]} \wt R_s
\eqen
and identifies $t \in [0,\infty)$ with $\wt\xi(  \wt L_t)$ (resp.\ $ \wt\xi(- \wt R_t)$) whenever $\wt L$ (resp.\ $\wt R$) attains a record minimum at time~$t$.  By Lemma~\ref{lem-curve-top} and the universal property of the quotient topology, there is a continuous surjective map $\Phi_X : X \rta \wt X$ such that $\Phi_X\circ \eta = \wt \eta$ and $\Phi_X \circ \xi = \wt\xi$. 
\medskip

\noindent \textit{Step 2: piecing together maps between bubbles.} 
Since each of the maps $g_k$ of~\eqref{eqn-component-map} is a continuous (in fact 1-Lipschitz) bijection between compact spaces $(\wt H_{\infty,k} ,\wt d_{\infty,k})$ and $(H_k , d)$, it follows that each $g_k^{-1}$ is continuous from $(H_k ,d)$ to $(\wt H_{\infty,k} , \wt d_{\infty,k})$.  Hence each $f_{\infty,k}\circ g_k^{-1}$ for $k\in\BB N$ is continuous from $(H_k , d)$ to $(f_{\infty,k}(\wt H_{\infty,k}) , \wt d)$. By~\eqref{eqn-bubble-bdy-path-L} of Lemma~\ref{lem-bubble-bdy}, each of these maps agrees with $\Phi_X$ on $\bdy H_k \subset X$. Therefore, the map $\Phi : H \rta \wt H$ defined by
\eqbn
\Phi(x) := \begin{cases}
\Phi_X(x) ,\quad &x\in X \\
(f_{\infty,k} \circ g_k^{-1})(x) , \quad &x\in H_k 
\end{cases}
\eqen
is well-defined. 

It is clear from surjectivity of $\Phi_X$ together with condition~\ref{item-map-cover} of Lemma~\ref{lem-map-property} that $\Phi$ is surjective. 
Since $\Phi_X \circ \eta = \wt\eta$ and $\Phi_X \circ \xi =\wt\xi$, the same is true for $\Phi$. 
By~\eqref{eqn-unexplored-bdy-path} of Lemma~\ref{lem-bubble-bdy}, a.s.\ $\Phi\circ \xi_{t,0} =f_{t,0} \circ \wt\xi_{t,0}$ for each $t\in \BB Q_+$. 

Each $g_k$ is measure-preserving and each $f_{\infty,k}$ is measure-preserving by condition~\ref{item-map-measure'} of Lemma~\ref{lem-map-property}. Since $\mu(\eta) = \wt\mu(\wt\eta) = 0$ (condition~\ref{item-map-curve-mass} of Lemma~\ref{lem-map-conclude}) we infer that $\Phi_* \mu = \wt\mu$.  Since $g_k^{-1}$ is an isometry $(H_k  , d_k) \rta (\wt H_{\infty,k}  , \wt d_{\infty,k}) $ and each $f_{\infty,k}$ restricts to an isometry $(\wt H_{\infty,k} \setminus \bdy \wt H_{\infty,k} , \wt d_{\infty,k}) \rta (f_{\infty,k}(\wt H_{\infty,k}\setminus \bdy \wt H_{\infty,k}) , \rng d_{\infty,k})$ (condition~\ref{item-map-iso} of Lemma~\ref{lem-map-property}) we infer that each $\Phi|_{H_k\setminus \bdy H_k}$ is an isometry between the internal metric spaces $(H_k\setminus \bdy H_k , d_k)$ and $(f_{\infty,k}(\wt H_{\infty,k}\setminus \bdy \wt H_{\infty,k}) , \rng d_{\infty,k})$. 
\medskip

\noindent \textit{Step 3: continuity.} 
We now check that~$\Phi$ is continuous via a compactness argument.  Suppose that we are given a sequence $\{x_m\}_{m\in\BB N}$ of points in~$H$ which converges to a point $x  \in H$. We must show that $\Phi(x_m) \rta \Phi(x)$.  By compactness of~$\wt H$, for every sequences of positive integers tending to $\infty$ there exists a subsequence $\mcl M$ along which $\Phi(x_m) \rta \wt x \in \wt H$. It suffices to show that $\wt x = \Phi(x)$ for every such subsequence $\mcl M$. 

Since $\Phi_X$ and each $f_{\infty,k} \circ g_k^{-1}$ is continuous, it is clear that $\wt x = \Phi(x)$ if either $x_m \in X$ for infinitely many $m\in\mcl M$ or there is a $k\in\BB N$ such that $x_m \in H_k$ for infinitely many $m \in \mcl M$.  If this is not the case, then after possibly passing to a further subsequence we can arrange that for each $m\in\mcl M$, there exists $k_m \in \BB N$ such that $x_m \in H_{k_m} \setminus \bdy H_{k_m}$ and $k_m\rta\infty$. 

Since $\eta$ (resp.\ $\wt\eta$) extends continuously to $[0,\infty]$, it follows that the set $H\setminus (\eta\cup \bdy H)$ (resp.\ $\wt H\setminus (\wt \eta \cup \bdy \wt H)$) has only finitely many connected components of $d$- (resp.\ $\wt d$-) diameter larger than $\ep$ for any $\ep  >0$. Hence $\op{diam}(H_{k_m} ; d) \rta 0$ as $\mcl M\ni m \rta\infty$ and, by condition~\ref{item-map-component} of Lemma~\ref{lem-map-conclude}, also $\op{diam}(\Phi(H_{k_m}) ; \wt d) \rta 0$ as $m\rta\infty$. If we choose $y_m \in \bdy H_{k_m}$ for $m\in\mcl M$, then $d(x_m , y_m) \rta 0$ and since $\Phi(x_m) , \Phi(y_m) \in \Phi(H_{k_m})$ also $\wt d(\Phi( x_m) ,\Phi(y_m) )   \rta 0$. Therefore $y_m \rta x$ so since $y_m \in X$ for each $m$, also $\Phi(y_m) \rta \Phi(x)$ and hence $\Phi(x_m) \rta \Phi(x)$, i.e.\ $\Phi(x) = \wt x$ as required. 
\end{proof}

To prove Proposition~\ref{prop-homeo}, it remains to show that the map $\Phi$ of Lemma~\ref{lem-surjection} is injective. This will be accomplished by means of the topological theorem~\cite[Main Theorem]{almost-inj}, which says that a continuous map between topological manifolds which is \emph{almost injective}, in the sense that the set of points with multiplicity~$1$ is dense; and \emph{light}, in the sense that the pre-image of every point is totally disconnected, must be an embedding (i.e., a homeomorphism onto its image). The isometry condition in Lemma~\ref{lem-surjection} implies that $\Phi$ is almost injective, so we need to check that $\Phi$ is light. In fact, we will show using the results of Section~\ref{sec-crossing} that the pre-image of any point has cardinality at most $7$.  

\begin{lem} \label{lem-ssl-curve-hit}
Almost surely, the curve $\wt\eta$ hits each point of $\wt H \setminus \bdy \wt H$ at most~$7$ times and each point of $\bdy\wt H \setminus \{\wt\eta(\infty)\}$ at most $4$ times.
\end{lem}

To prove Lemma~\ref{lem-ssl-curve-hit}, we need to study inside-outside crossings of annular regions by the curve $\wt\eta$, which are defined in the following manner (in analogy with Definition~\ref{def-crossing}).

\begin{defn}
\label{def-crossing-cont} 
For a topological space $X$, a curve $\eta : [0,\infty) \rta X$, and sets $Y_0 \subset Y_1 \subset X$, an \emph{inside-outside crossing} of $Y_1\setminus Y_0$ by the path $\eta$ is a time interval $[t_0,t_1] \subset [0,\infty)$ such that $ \eta(t_0) \in Y_0$, $\ \eta(t_1) \in X\setminus Y_1$, and $ \eta((t_0, t_1)) \subset Y_1\setminus Y_0$.  We write $\op{cross}(Y_0 ,Y_1 ; \eta)$ for the set of inside-outside crossings of $Y_1\setminus Y_0$ by $ \eta$. 
\end{defn}

As in Section~\ref{sec-crossing},  we will consider crossings of annular regions between filled metric balls. For $z \in \wt H$ and $\rho \geq 0$, we define the \emph{filled metric ball} $B_\rho^\bullet(z; \wt d)$ to be the union of the closed metric ball $B_\rho (z;\wt d)$ and the set of points which it disconnects from the target point $\wt\xi(\frk l_R) = \wt\eta(\infty)$ in $\wt H$.

\begin{lem}
\label{lem-ssl-crossing}
For each $\ep \in (0,1)$, there exists $\delta_* = \delta_*(\ep ) \in (0,1)$ such that for each $\delta \in (0,\delta_*]$, it holds with probability at least $1-\ep$ that the following is true.  For each $z \in \wt H$ with $\wt d(z , \bdy \wt H) \geq \ep$, the number of inside-outside crossings (Definition~\ref{def-crossing-cont}) satisfies
\eqb \label{eqn-ssl-crossing}
\#\op{cross} \left( B_{\delta}^\bullet(z ;\wt d) , B_{\ep}^\bullet(z ; \wt d) ; \wt \eta \right) \leq 7.
\eqe 
Moreover, for each $z\in \bdy \wt H$ with $\wt d(z,\wt\eta(\infty)) \geq \ep$,  
\eqb \label{eqn-ssl-crossing-bdy}
\#\op{cross} \left( B_{\delta}^\bullet(z ;\wt d) , B_{\ep}^\bullet(z ; \wt d) ; \wt \eta \right) \leq 4 .
\eqe 
\end{lem}
\begin{proof}
Fix $\ep \in (0,1)$.  By Proposition~\ref{prop-fb-crossing}, there exists $\delta_* = \delta_*(\ep) \in (0,1)$ such that for each $\delta \in (0,\delta_*]$, there exists $n_* = n_*(\delta,\ep) \in \BB N$ such that for $n\geq n_*$, it holds with probability at least $1-\ep$ that the following is true.  For each $z^n \in Q^n$ with $d^n(z^n , \bdy Q^n  ) \geq \ep/2$, one has
\eqb \label{eqn-discrete-crossing}
\#\op{cross} \left( B_{2 \delta }^\bullet(z^n ; d^n) , B_{\ep/2}^\bullet(z^n ; d^n ) ;  \eta^n \right) \leq 7 .
\eqe 
Hence with probability at least $1-\ep$, there exists a subsequence $\mcl N'\subset \mcl N$ such that~\eqref{eqn-discrete-crossing} holds for each $n\in\mcl N'$. 
Henceforth assume that this is the case. 
We will show that the condition in the statement of the lemma holds.

To this end, fix $z\in \wt H$ with $\wt d(z , \bdy \wt H ) \geq \ep$. 
Since $Q^n \rta \wt H$ in the $D$-Hausdorff distance, there exists $z^n \in Q^n$ for $n\in\mcl N'$ such that $z^n \rta z$. 
This implies that $d^n(z^n ,\bdy Q^n) \geq \ep/2$ for large enough $n\in\mcl N'$ and that
\eqb \label{eqn-ball-h-conv}
B_{\delta }(z^n ; d^n) \rta B_{\delta }(z ; \wt d)  \quad \op{and} \quad B_{ \ep }(z^n ; d^n) \rta B_{ \ep }(z ; \wt d) 
\eqe 
in the $D$-Hausdorff distance. By the compactness of $(W,D)$, we can find a subsequence $\mcl N'' \subset \mcl N'$ and a subset $Y$ of $\wt H$ such that $B_{ \ep  }^\bullet(z^n ; d^n ) \rta Y$ in the $D$-Hausdorff distance as $\mcl N'' \ni n \rta\infty$. 

We claim that 
\eqb \label{eqn-filled-ball-ssl}
Y\subset B_{ \ep }^\bullet(z ; \wt d)  .
\eqe 
To see this, suppose $w \in \wt H\setminus B_{ \ep }^\bullet(z ; \wt d)$ and let $w^n \in Q^n$ for $n\in\mcl N''$ be chosen so that $w^n\rta w$. There is a $\rho > 0$ and points $w = w_0 , w_1 , \dots, w_N = \wt\eta(\infty)$ in $\wt H$ such that $\wt d(\wt w_{j-1} ,  B_{ \ep  }^\bullet(z ; \wt d) )\geq 2\rho$ and $\wt d(w_{j-1} , w_j) \leq \frac14 \rho$ for each $j \in [1,N]_{\BB Z}$. From this and the Hausdorff convergence $Q^n\rta \wt H$, for large enough $n\in\mcl N''$ there exists points $w^n = w_0^n , w_1^n , \dots, w_N^n = \wt\eta^n(\infty)$ in $Q^n$ such that $d^n(w_j^n  ,  B_{ \ep}(z^n ; d^n) ) \geq \rho$ and $d^n(w_{j-1}^n , w_j^n) \leq \frac13  \rho$ for each $j \in [1,N]_{\BB Z}$.
Therefore, for each such $n$ there exists a path from $w^n$ to $\wt\eta^n(\infty)$ in $Q^n$ which stays at distance at least $\frac13 \rho$ from $B_{\ep}(z^n ; d^n)$, whence $w^n $ lies at $d^n$-distance at least $\frac13\rho$ from $B_{ \ep }^\bullet(z^n ; d^n ) $. Since $w^n\rta w$, we infer that $w\notin Y$. 

It follows from~\eqref{eqn-discrete-crossing},~\eqref{eqn-ball-h-conv}, and~\eqref{eqn-filled-ball-ssl} that 
\begin{align*}
\#\op{cross} \left( B_{\delta} (z ;\wt d) , B_{\ep }^\bullet(z ; \wt d) ; \wt \eta \right)
&\leq \#\op{cross} \left( B_{\delta} (z ;\wt d) ,Y ; \wt \eta \right)\\
&\leq \limsup_{\mcl N'' \ni n \rta\infty} \#\op{cross} \left( B_{2 \delta }^\bullet(z^n ; d^n) , B_{ \ep / 2  }^\bullet(z^n ; d^n ) ;  \eta^n \right) 
\leq 7 .
\end{align*}
The only connected component of $\wt H\setminus B_{\delta} (z ;\wt d)$ which contains a point of $\wt H\setminus B_{ \ep  }^\bullet(z ; \wt d)$ is the one which contains $\wt\eta(\infty)$. Therefore, $\op{cross} \left( B_{\delta} (z ;\wt d) , B_{\ep  }^\bullet(z ; \wt d) ; \wt \eta \right) = \op{cross} \left( B_{\delta}^\bullet (z ;\wt d) , B_{\ep  }^\bullet(z ; \wt d) ; \wt \eta \right)$. We thus obtain~\eqref{eqn-ssl-crossing}.

The proof of~\eqref{eqn-ssl-crossing-bdy} is essentially identical, except that we use Proposition~\ref{prop-fb-crossing-bdy} instead of Proposition~\ref{prop-fb-crossing}. 
\end{proof}

\begin{proof}[Proof of Lemma~\ref{lem-ssl-curve-hit}] 
For $k\in \BB N$ let $\ep_k := 2^{-k}$. Let $\delta_k = \delta_k(\ep_k  ) \in (0,2^{-k}]$ be chosen so that the conclusion of Lemma~\ref{lem-ssl-crossing} holds with $\delta_k$ in place of $\delta_*$ and let $E_k$ be the event that the following is true: for each $z \in \wt H$ with $\wt d(z , \bdy\wt H ) \geq \ep_k$,
\eqbn
\#\op{cross} \left( B_{\delta_k}^\bullet(z ;\wt d) , B_{ \ep_k }^\bullet(z ; \wt d) ; \wt \eta \right) \leq   7 .
\eqen 
Then $\BB P[E_k ] \geq 1-\ep_k$ so by the Borel-Cantelli lemma, a.s.\ $E_k$ occurs for each large enough $k\in \BB N$. 
Henceforth fix $k_0 \in \BB N$ and assume we are working on the event $\bigcap_{k=k_0}^\infty E_k$.  

Let $z \in \wt H\setminus \bdy \wt H $. By condition~\ref{item-map-component} of Lemma~\ref{lem-map-conclude}, it is a.s.\ the case that the set of connected components of $\wt H\setminus \wt\eta([0,t_1])$ is a strict subset of the set of connected components of $\wt H\setminus \wt\eta([0,t_2])$ for each $0 < t_1 < t_2 < \wt\sigma_0$. Hence a.s.\ $\wt\eta$ is not constant on any positive-length interval of times. 

Consequently, if $\wt\eta$ hits $z$ more than 7 times, then there exists times $0 \leq t_1 < s_1 < t_2 < \dots < s_7 < t_8 <  s_8 < \wt\sigma_0$ such that $\wt\eta(t_1) = \dots = \wt\eta(t_8) = z$ and $\wt \eta(s_j) \not= z$ for each $j\in [1,7]_{\BB Z}$. Choose $\rho > 0$ such that each $\wt\eta(s_j)$ lies outside $B_\rho^\bullet(z ; \wt d)$.  
Then for each $0 < \rho_1 < \rho_2 \leq \rho$, 
\eqbn
\#\op{cross} \left( B_{\rho_1}^\bullet(z ;\wt d) , B_{\rho_2 }^\bullet(z ; \wt d) ; \wt \eta \right) \geq 8 .
\eqen
This contradicts the occurrence of $E_k$ for $k \geq k_0$ such that $\ep_k \leq \rho \wedge \wt d(z , \bdy \wt H)$. 
Therefore, $\wt\eta$ hits $z$ at most 7 times, as required. 

We similarly obtain that $\wt\eta$ hits each point of $\bdy\wt H \setminus \{\wt\eta(\infty)\}$ at most 4 times using~\eqref{eqn-ssl-crossing-bdy} of Lemma~\ref{lem-ssl-crossing}.
\end{proof}

\begin{proof}[Proof of Proposition~\ref{prop-homeo}]
Let $\Phi : H \rta \wt H$ be the continuous surjective map from Lemma~\ref{lem-surjection}. 
By Lemma~\ref{lem-surjection} and compactness, it suffices to show that $\Phi$ is injective. 

The map $\Phi$ restricts to a bijection from $H\setminus (\eta\cup \bdy H) = \Phi^{-1}(\wt H\setminus (\wt\eta \cup \bdy \wt H) ) $ to $\wt H\setminus (\wt\eta \cup \bdy \wt H)$. 
In particular, $\Phi^{-1}(\Phi(z)) = \{z\}$ for each $z$ in a dense subset of $H$, i.e.\ $\Phi$ is almost injective. 
Since $\Phi\circ \xi_{0,0} = \wt \xi$, $\Phi|_{\bdy H}$ is a bijection from $\bdy H$ to $\bdy \wt H$. 
By Lemma~\ref{lem-bubble-bdy}, $\Phi$ maps the intersection with $\bdy H$ of each connected component of $H\setminus \eta$ to the intersection with $\bdy\wt H$ of the corresponding connected component of $\wt H\setminus\wt\eta$. In particular, $\wt\eta$ does not hit any points of $\bdy \wt H$ whose pre-images under $\Phi$ are not hit by $\eta$.
Since $\Phi\circ\eta=\wt\eta$, we infer that $\Phi^{-1}(\wt\eta) = \eta$ and $\Phi|_{H\setminus \eta}$ is injective.
Hence for each point $z\in \wt H \setminus \{\wt\eta(\infty)\} $ which is hit by $\wt\eta$, 
\eqbn
\Phi^{-1}(z) = \{ \eta(t) : t \geq 0 \: \op{with} \: \wt\eta(t) = z\} .
\eqen
So, Lemma~\ref{lem-ssl-curve-hit} implies that a.s.\ $\# \Phi^{-1}(z)   \leq 7$ for each $z\in \wt H\setminus \{\wt\eta(\infty)\}$. 

For each rational time $t$ less than the terminal time $\sigma_0$, the boundary lengths of the arcs separating the two marked boundary points of the ``future" Brownian disk $\wt H_{t,0}$ are given by $L_t + \frk l_L$ and $R_t + \frk l_R$, which are positive by the definition~\eqref{eqn-continuum-terminal-time} of $\sigma_0$. 
Hence $\wt\eta$ does not hit $\xi(\frk l_R) = \wt\eta(\infty)$ before time $\sigma_0$. By this and condition~\ref{item-curve-top-terminal} of Lemma~\ref{lem-curve-top}, we obtain $\wt\eta^{-1}(\wt\eta(\infty)) = \eta^{-1}(\eta(\infty)) = [\sigma_0 , \infty)$. In particular, $\Phi^{-1}(\wt\eta(\infty)) = \{\eta(\infty)\}$. 

By combining the conclusions of the preceding two paragraphs, we get that the map $\Phi $ is light, i.e.\ the pre-image of every point is totally disconnected.
By~\cite[Main Theorem]{almost-inj}, $\Phi|_{H\setminus \bdy H} : H\setminus \bdy H \rta \wt H$ is a homeomorphism onto its image. In particular, $\Phi(H\setminus \bdy H) \subset \wt H\setminus \bdy \wt H$. Since $\Phi|_{\bdy H}$ is injective, in fact $\Phi$ is injective, as required.
\end{proof}

\subsection{Proofs of Theorems~\ref{thm-perc-conv} and~\ref{thm-perc-conv-uihpq}}
\label{sec-main-proof}

\begin{proof}[Proof of Theorem~\ref{thm-perc-conv}]
By Proposition~\ref{prop-component-law}, Proposition~\ref{prop-homeo}, and Theorem~\ref{thm-bead-mchar}, we infer that the subsequential limiting space $\wt{\frk H}$ has the law of a free Boltzmann Brownian disk with boundary length $\frk l_L + \frk l_R$ decorated by an independent chordal $\SLE_6$ from $\wt\xi(0)$ to $\wt\xi(\frk l_R)$. Since our initial choice of subsequence was arbitrary, we obtain the convergence in the theorem statement. 
\end{proof}

We now deduce our infinite-volume scaling limit result using local absolute continuity.

\begin{proof}[Proof of Theorem~\ref{thm-perc-conv-uihpq}] 
Fix $\rho > 0$, $\ep \in (0,1)$, and $\frk l > 0$ to be chosen later, in a manner depending only on $\rho$ and $\ep$. 
Also fix a sequence of positive integers $ (\el_L^n ,\el_R^n)_{n\in\BB N}$ such that $\el_L^n + \el_R^n$ is always even and $(n^{-1/2} \el_L^n, n^{-1/2} \el_R^n) \rta (\frk l , \frk l)$ and define the doubly curve-decorated metric measure spaces $\frk Q^n $ and $\frk H$ as in Theorem~\ref{thm-perc-conv} for this choice of $(\el_R^n , \el_L^n)$ and for $(\frk l_L , \frk l_R) = (\frk l , \frk l)$.  

By~\cite[Proposition~4.6]{gwynne-miller-simple-quad} and since the percolation exploration path is determined locally by the quadrangulation and the face percolation configuration, there exists $\frk l_0 = \frk l_0(\ep,\rho)$ such that if $n\geq n_*$, then for $\frk l\geq \frk l_0$ and large enough $n\in\BB N$, there exists a coupling of $\frk Q^n$ and $\frk Q^\infty$ such that with probability at least $1-\ep/2$, the $\rho$-truncations (Definition~\ref{def-ghpu-truncate}) satisfy $\frk B_\rho \frk Q^n = \frk B_\rho \frk Q^\infty$.
 
Since the quantum natural time parameterization of a chordal $\SLE_6$ is determined locally by the underlying field, it follows from~\cite[Proposition~4.2]{gwynne-miller-uihpq} that after possibly increasing $\frk l$, we can find a coupling of the limiting spaces $\frk H$ and $\frk H^\infty$ such that with probability at least $1-\ep/2$, one has $\frk B_\rho \frk H = \frk B_\rho \frk H^\infty$.

The theorem statement follows by combining the above two observations with Theorem~\ref{thm-perc-conv}. 
\end{proof}

\section{The case of triangulations}
\label{sec-triangulation}

In this subsection we state analogs of Theorems~\ref{thm-perc-conv} and~\ref{thm-perc-conv-uihpq} for critical ($p=1/2$;~\cite{angel-peeling,angel-uihpq-perc}) site percolation on a triangulation of type I (self-loops and multiple edges allowed) or II (no self-loops, but multiple edges allowed). Our arguments do not transfer directly to the case of type III triangulations, which have no self-loops or multiple edges, since the Markov property of peeling does not hold for such triangulations. 

Our reason for considering site percolation on a triangulation is that this is perhaps the simplest version of percolation. There are several reasons why this is the case:
\begin{itemize}
\item The percolation threshold for site percolation on the uniform infinite planar triangulation of type I or II is $1/2$~\cite{angel-peeling,angel-uihpq-perc}. 
\item Percolation interfaces take a particularly simple form (they are just paths in the dual map whose edges all have one white and one black vertex) and the peeling exploration path is exactly the percolation interface starting from the root edge~\cite{angel-uihpq-perc,angel-curien-uihpq-perc}. 
\item An instance of the uniform infinite planar triangulation of type II decorated by a site percolation configuration can be encoded by means of a simple random walk with steps which are uniform on $\{(0,1), (1,0) , (-1,-1)\}$, via a discrete analog of the peanosphere construction of~\cite{wedges}; this encoding is described in~\cite{bernardi-dfs-bijection,bhs-site-perc}. 
\item Site percolation on the triangular lattice is the only percolation model on a deterministic lattice known to converge to $\SLE_6$~\cite{smirnov-cardy,camia-newman-sle6,hls-sle6}. 
\end{itemize}
We emphasize, however, that the arguments of the present paper can be adapted to prove analogous scaling limit results for any percolation model on a random planar map with simple boundary which can be explored via peeling. 

\subsection{Preliminary definitions}
\label{sec-tri-prelim}

Before stating our results for triangulations, we recall the definition of the free Boltzmann distribution on triangulations with simple boundary of given perimeter and a peeling interface for site percolation on it.

Recall that a triangulation of type I is a general triangulation (with multiple edges and self-loops allowed) and a triangulation of type II is allowed to have multiple edges but no self-loops.
A triangulation of type I or II with simple boundary and its corresponding boundary path $\beta$ are defined in an analogous manner as in the quadrangulation case (recall Section~\ref{sec-intro-def-quad}). Note that in the triangulation case the perimeter is not constrained to be even although the perimeter of a type II triangulation with simple boundary must be at least $2$. 

For $n, \el\in\BB N $ and $i\in \{\mathrm{I},\mathrm{II}\}$ we write $\mcl T_{ \op{S} ,i }^\shortrightarrow(n,\el)$ for the set of pairs $(T , e)$ where $T$ is a triangulation with simple boundary having $ \el$ boundary edges and $n$ interior vertices and $e$ is an oriented root edge in its boundary.  

We define the free Boltzmann partition functions for triangulations of type I or II by
\eqb \label{eqn-fb-partition-tri}
\frk Z_{\mathrm{I}}^\Delta(\el) := \begin{cases}
\frac{2-\sqrt 3}{4} ,\quad &\el = 1\\
\frac{(2\el-5)!! 6^p}{8\sqrt 3 \el!} ,\quad &\el \geq 2
\end{cases} \quad \op{and} \quad
\frk Z_{\mathrm{II}}^\Delta(\el) := \frac{(2\el -4)!}{(\el -2)! \el !} \left( \frac{9}{4} \right)^{\el -1} ,\quad \el \geq 2  
\eqe 
where here $m !!$ for odd $m\in\BB N$ is the product of the positive odd integers which are less than or equal to $m$ and $(-1)!! = 1$.

\begin{defn} \label{def-fb-tri}
For $\el \in \BB N$ (with $\el\geq 2$ in the type II case) the \emph{free Boltzmann triangulation of type $i \in \{{\mathrm {I}},{\mathrm {II}}\}$ with simple boundary of perimeter $\el$} is the probability measure on 
 $\bigcup_{n=0}^\infty \mcl T_{\op{S}}^\shortrightarrow (n ,  \el)$ which assigns to each element of $ \mcl T_{\op{S} , i}^\shortrightarrow (n, \el)$ a probability equal to $\rho_i^{-1} \frk Z_i^\Delta( \el)^{-1}$, where here $\rho_{\mathrm{I}}= \sqrt{432}$ and $\rho_{\mathrm{II}} = 27/2$. 
 \end{defn}
  
For $i\in \{{\mathrm {I}},{\mathrm {II}}\}$, the \emph{uniform infinite planar trangulation of type $i$} (UIHPT$_{\op{S}}^i$) is the Benjamini-Schramm local limit of the free Boltzmann triangulation of type $i$ with simple boundary of perimeter $\el$ as $\el\rta\infty$~\cite{angel-curien-uihpq-perc}.

Suppose now that we are given $\el \in\BB N$, $\el_L , \el_R \in \BB N$ with $\el_L + \el_R = \el$, and a triangulation $(T,\BB e)$ with a distinguished oriented root edge $\BB e \in \bdy T$. 
Also let $\beta$ be the boundary path of $T$ starting from $\BB e$ and let $\BB e_* := \beta(\el_R+1)$. 
A \emph{critical site percolation configuration on $T$ with $\el_L$-white/$\el_R$-black boundary conditions} is a random function $\theta$ from the vertex set of $T$ to $\{0,1\}$ such that $\theta(v) = \mathsf{white} $ (resp.\ $\theta(v)= \mathsf{black} $) for each vertex in the left (resp.\ right) arc of $\bdy T$ from $\BB e$ to $\BB e_*$, including one endpoint of each of $\BB e$ and $\BB e_*$; and the values $\theta(v)$ for vertices $v \in T\setminus \bdy T$ are i.i.d.\  Bernoulli $1/2$-random variables.  We say that vertices $v$ with $\theta(v) =\mathsf{white} $ (resp.\ $\theta(v) = \mathsf{black} $) are open or white (resp.\ closed or black). 

For such a site percolation configuration $\theta$, there necessarily exists a unique path~$\lambda$ in (the dual map of) $T$ from~$\BB e$ to $\BB e_*$ such that each edge traversed by $\lambda$ has a white vertex to its left and a black vertex to its right; see Figure~\ref{fig-tri-perc} for an illustration.  The path $\lambda$ is called the \emph{interface path}, and can be explored via peeling by iteratively revealing the boundary edge of the unexplored triangulation other than $\BB e_*$ which has one white and one black endpoint. 

In contrast to the case of face percolation on a quadrangulation discussed above, for site percolation on a triangulation the percolation interface is identical to the peeling exploration path. Moreover, in the site percolation case $\lambda$ can be defined on the integers rather than the half-integers since each peeled edge shares an endpoint with the previous peeled edge.

\begin{figure}[ht!]
 \begin{center}
\includegraphics[scale=1]{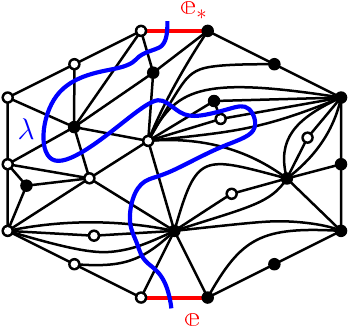} 
\caption[Site percolation interface on a triangulation]{\label{fig-tri-perc} A finite triangulation with simple boundary together with the percolation exploration path, equivalently the percolation interface, with $7$-white/$7$-black boundary conditions. }
\end{center}
\end{figure}

\subsection{Scaling limit result for triangulations}
\label{sec-tri-thm}
 
Our scaling limit results for site percolation on a triangulation rely on the following theorem of Albenque, Holden, and Sun~\cite{aasw-type2}. 

\begin{thm}[\cite{aasw-type2}]
\label{thm-tri-limit}
Let $\BB c_{\mathrm{I}} = 3/4$ and $\BB c_{\mathrm{II}} = 3/2$.  Let $\{\el^n\}_{n\in\BB N}$ be a sequence of positive integers tending to $\infty$ such that $  n^{-1/2} \el^n \rta 1$ as $n\rta\infty$.  For $n\in\BB N$, let $(T^n,\BB e^n)$ be a free Boltzmann triangulation of type $i \in \{\mathrm{I} , \mathrm{II}\}$ with simple boundary of perimeter $\el^n$ viewed as a connected metric space by identifying each edge with an isometric copy of the unit interval. Let $d^n$ be the graph distance on $T^n$, scaled by $\sqrt{3/2} n^{-1/4}$, let $\mu^n$ be the measure on $T^n$ which assigns each vertex a mass equal to $\BB c_i n^{-1}$ times its degree, let $\beta^n :[0,\el^n] \rta \bdy T^n$ be the boundary path of $\el^n$, extended by linear interpolation, and let $\xi^n(s) :=  \beta^n(   n^{1/2} s )$.  The curve-decorated metric measure spaces $\frk T^n := (T^n,d^n,\mu^n,\eta^n)$ converge in the scaling limit in the GHPU topology to the free Boltzmann Brownian disk with unit perimeter (Section~\ref{sec-intro-def-disk}). 
\end{thm}  
 
We note that Theorem~\ref{thm-tri-limit} implies the analogous scaling limit result for the UIHPT$_{\op{S}}^i$ toward the Brownian half-plane in the local GHPU topology by a local coupling argument using~\cite[Proposition 4.2]{gwynne-miller-uihpq} and the triangulation analog of~\cite[Proposition 4.6]{gwynne-miller-simple-quad}. 

Taking Theorem~\ref{thm-tri-limit} as input, an argument which is essentially identical to the one used to prove Theorems~\ref{thm-perc-conv} and~\ref{thm-perc-conv-uihpq} yields exact analogs of these theorems in the case of site percolation on a type $i \in \{\mathrm I, \mathrm{II}\}$ triangulation.  In fact, some of the arguments in the proofs of the theorems in the case of site percolation of a triangulation are slightly simpler due to the simpler form of the peeling process in this case~\cite{angel-uihpq-perc,angel-curien-uihpq-perc} and the exact agreement between the percolation exploration path and the peeling interface.

Define $\BB c_i$ for $i\in \{\mathrm{I}, \mathrm{II}\}$ as in Theorem~\ref{thm-tri-limit}. Following~\eqref{eqn-tcon-choice}, let $\tcon_i = \tfrac32  \ccon(i)^{-1} \lcon $ be the time scaling constant, with $\ccon(i)$ the constant from the type $i$ triangulation analog of~\eqref{eqn-cover-tail} and $\lcon$ the (non-explicit) scaling constant from~\eqref{eqn-levy-process-scaling} (note that the boundary length scaling constant is 1 for triangulations by Theorem~\ref{thm-tri-limit}). 

Fix $\frk l_L , \frk l_R > 0$ and a sequence of pairs of positive integers $\{(\el_L^n ,\el_R^n)\}_{n\in\BB N}$ such that $ n^{-1/2} \el_L^n  \rta \frk l_L$ and $  n^{-1/2} \el_R^n \rta \frk l_R$.  
For $n\in\BB N$, let $(T^n ,\BB e^n )$ be a free Boltzmann triangulation of type $i$ with simple boundary of perimeter $\el_L^n+\el_R^n$ viewed as a connected metric space by replacing each edge with an isometric copy of the unit interval and let $\theta^n$ be a critical face percolation configuration on $T^n$. 
Define the rescaled metric, area measure, and boundary path $d^n$, $\mu^n$, and $\xi^n$ for $T^n$ as in Theorem~\ref{thm-tri-limit}. 
Also let $\lambda^n : [0,\infty) \rta T^n$ be the percolation interface path of $(T^n,\BB e^n , \theta^n)$ with $\el_L^n$-white/$\el_R^n$-black boundary conditions (Section~\ref{sec-tri-prelim}), extended to a continuous path on $[0,\infty)$ which traces the edge $\lambda^n(j)$ during each time interval $[j-1,j]$ for $j\in\BB N$; and for $t\geq 0$ let $\eta^n(t ) := \lambda^n(\tcon_i n^{ 3/4} t)$. 
Define the doubly-marked curve-decorated metric measure spaces
\eqbn
\frk T^n := \left( T^n , d^n , \mu^n , \xi^n , \eta^n \right)  .
\eqen
Also let $\frk H = (H,d,\mu,\xi,\eta)$ be a free Boltzmann Brownian disk with perimeter $\frk l_L + \frk l_R$ decorated by an independent chordal $\SLE_6$ from $\xi(0)$ to $\xi(\frk l_R)$, as in Theorem~\ref{thm-perc-conv}. We have the following triangulation analog of Theorem~\ref{thm-perc-conv}. 

\begin{thm}  \label{thm-perc-conv-tri}
One has $\frk T^n \rta \frk H$ in law with respect to the two-curve Gromov-Hausdorff-Prokhorov-uniform topology. That is, site percolation on a free Boltzmann triangulation of type $i \in \{\mathrm I, \mathrm{II}\}$ with simple boundary converges to chordal $\SLE_6$ on a free Boltzmann Brownian disk. 
\end{thm}

As in the case of Theorem~\ref{thm-perc-conv}, we in fact obtain the joint scaling limit of $\frk T^n$ and the analog of the rescaled boundary length process $Z^n = (L^n,R^n)$ of Definition~\ref{def-bdy-process-rescale} for site percolation on a triangulation toward $\frk H$ and its associated left/right boundary length process in the GHPU topology on the first coordinate and the Skorokhod topology on the second coordinate.  We also obtain from Theorem~\ref{thm-perc-conv-tri} an analogous scaling limit result for site percolation on the UIHPT$_{\op{S}}^i$ toward chordal $\SLE_6$ on the Brownian half-plane, using the same local coupling argument as in Theorem~\ref{thm-perc-conv-uihpq}.

\appendix

\section{Index of notation}
\label{sec-index}
 
Here we record some commonly used symbols in the paper, along with their meaning and the location where they are first defined (notations used only locally are not listed). We emphasize that a superscript $\infty$ denotes objects associated with the infinite-volume setting and a superscript $n$ denotes objects associated with the quadrangulation $Q^n$.

\begin{multicols}{2}
\begin{itemize}
\item $\frk Z$: free Boltzmann partition function; \eqref{eqn-fb-partition}.
\item $\beta$: boundary path of a quadrangulation; Section~\ref{sec-intro-def-quad}.
\item $\BB e$: root edge; Section~\ref{sec-intro-def-quad}.
\item $\BB e_*$: target edge; Section~\ref{sec-intro-def-perc}. 
\item $  (H,d,\mu,\xi)$: Brownian disk with its metric, area measure, and boundary path; Section~\ref{sec-intro-def-disk}.
\item $\theta$: percolation configuration; Section~\ref{sec-intro-def-perc}.
\item $\lambda$: percolation exploration path; Section~\ref{sec-intro-def-perc}.
\item $\bcon = 2^{3/2} / 3$: boundary length scaling constant; \eqref{eqn-normalizing-constant}.
\item $\tcon$: normalizing constant for percolation exploration path time; just after \eqref{eqn-normalizing-constant}.
\item $\frk l_L , \frk l_R$: left/right boundary lengths for Brownian disk; Section~\ref{sec-results}.
\item $\el_L,\el_R$ (or $\el_L^n,\el_R^n$): left/right boundary lengths for quadrangulation; Section~\ref{sec-results}. 
\item $\eta$: SLE$_6$ parameterized by quantum natural time; Section~\ref{sec-lqg-bdy-process}. 
\item $\sigma_0$: time when $\eta$ reaches its target point; Section~\ref{sec-lqg-bdy-process}.
\item $Z=(L,R)$: left/right boundary length process for SLE$_6$ on the Brownian disk; Section~\ref{sec-lqg-bdy-process}.
\item $\frk f(\cdot,\cdot)$: peeled quadrilateral; Section~\ref{sec-general-peeling}. 
\item $\frk P(\cdot,\cdot)$: peeling indicator; Section~\ref{sec-general-peeling}.
\item $\op{Peel}(\cdot,\cdot)$: unexplored quadrangulation when peeling; Section~\ref{sec-general-peeling}.
\item $\frk F(\cdot,\cdot)$: region disconnected from $\infty$ when peeling; Section~\ref{sec-general-peeling}.
\item $\op{Co}_{\BB e_*}, \op{Co}^L_{\BB e_*}, \op{Co}^R_{\BB e_*}$: covered edges after peeling; Section~\ref{sec-general-peeling}.
\item $\op{Ex}_{\BB e_*}$: exposed edges after peeling; Section~\ref{sec-general-peeling}. 
\item $\ol Q_j$: unexplored quadrangulation for peeling process; Section~\ref{sec-perc-peeling}.
\item $\dot Q_j$: peeling cluster; Section~\ref{sec-perc-peeling}.
\item $\dot e_j$: peeled edge; Section~\ref{sec-perc-peeling}.
\item $\mcl J$: terminal time of percolation peeling process; Section~\ref{sec-perc-peeling}. 
\item $\mcl F_j$: filtration of peeling process; Section~\ref{sec-perc-peeling}. 
\item $X_j^L,X_j^R,X_j$: left, right, total exposed boundary length of $\dot Q_j$; Definition~\ref{def-bdy-process}.
\item $Y_j^L,Y_j^R,Y_j$: left, right, total covered boundary length of $\dot Q_j$; Definition~\ref{def-bdy-process}.
\item $W^L_j$ and $W^R_j$: net boundary length $X^L_j-Y^L_j$ and $X^R_j-Y^R_j$; Definition~\ref{def-bdy-process}. 
\item $W_j = (W^L_j , W^R_j)$; Definition~\ref{def-bdy-process}. 
\item $L^n , R^n , Z^n$: re-scaled net boundary length processes; Definition~\ref{def-bdy-process-rescale}. 
\item $J_r^n$: first time that $Y_j^{L,n} \geq \el_L^n - r \BB c n^{1/2}$ or $Y_j^{R,n} \geq \el_R^n - r \BB c n^{1/2}$; \eqref{eqn-discrete-hit-time}.
\item $\sigma_r^n$: re-scaled version of $J_r^n$; \eqref{eqn-discrete-hit-time}
\item $I_{\alpha_0,\alpha_1}^n$: stopping time for left/right boundary length process; \eqref{eqn-discrete-close-time}.
\item $\tau_{\alpha_0,\alpha_1}^n$: re-scaled version of $I_{\alpha_0,\alpha_1}^n$; \eqref{eqn-discrete-close-time-rescale}.
\item $\op{cross}(\cdot,\cdot)$: number of crossings by percolation peeling exploration; Definition~\ref{def-crossing}.
\item $B_r^\bullet(\cdot;\cdot)$: filled metric ball; Definition~\ref{def-filled-ball}.
\item $B_r^\pbl(\cdot;\cdot)$: peeling-by-layers cluster; Section~\ref{sec-crossing-perturb}. 
\item $\mcl A_r^\pbl(\cdot;\cdot)$: edge set $\mcl E\left( \bdy B_r^\pbl(\BB A ; \ol Q_j) \setminus \bdy \ol Q_j  \right)  $; \eqref{eqn-pbl-bdy-length-def}.
\item $\op{IP}_j(\cdot,\cdot)$: number of interface paths which cross an annulus; Definition~\ref{def-ip-crossing}.
\end{itemize}
\end{multicols}

\begin{multicols}{2}
\begin{itemize}
\item $\mcl N$: subsequence along which we have GHPU convergence in law; Section~\ref{sec-identification}.
\item $\wt{\frk H} = (\wt H ,\wt d , \wt \mu , \wt \xi , \wt\eta)$: subsequential limiting space; \eqref{eqn-ssl-conv0}. 
\item $\frk H_{t,k}^n = (H_{t,k}^n , d_{t,k}^n , \mu_{t,k}^n , \xi_{t,k}^n)$: component of $Q^n\setminus \eta^n([0,t])$ with $k$th largest boundary length; \eqref{eqn-bubble-space}. 
\item $\Delta_{t,k}^n$: re-scaled boundary length of $H_{t,k}^n$; \eqref{eqn-bubble-jump}.
\item $\tau_{t,k}^n$: re-scaled time at which $H_{t,k}^n$ is disconnected from the target point; \eqref{eqn-bubble-t-time}.
\item $\wt{\frk H}_{t,k}  = (\wt H_{t,k}  , \wt d_{t,k}  , \wt\mu_{t,k}  , \wt\xi_{t,k} )$: subsequential limit of $\frk H_{t,k}^n$; \eqref{eqn-bubble-ssl}. 
\item $\Delta_{t,k}$: re-scaled boundary length of $\wt{ H}_{t,k}$; \eqref{eqn-unexplored-perimeter-ssl}.
\item $\tau_{t,k}$: time at which $H_{t,k} $ is disconnected from the target point; just below \eqref{eqn-unexplored-perimeter-ssl}.
\item $(W,D)$: space into which $Q^n$ and $H$ are embedded; \eqref{eqn-ghpu-embeddings}.
\item $(W_{t,k} , D_{t,k})$: space into which $H_{t,k}^n$ and $H_{t,k}$ are embedded; \eqref{eqn-ghpu-embeddings}.
\item $\wt{\frk H}_{t,k}^n = (\wt H_{t,k}^n  , \wt d_{t,k}^n  , \wt\mu_{t,k}^n  , \wt\xi_{t,k}^n )$: $\frk H_{t,k}^n$ embedded into $W_{t,k}$; \eqref{eqn-embed-bubble}.
\item $f_{t,k}^n$: identity map $\wt H_{t,k}^n \rta H_{t,k}^n$; \eqref{eqn-embed-maps}.
\item $f_{t,k}$: subsequential limit of $f_{t,k}^n$; Lemma~\ref{lem-map-limit}.
\end{itemize}
\end{multicols}

\bibliography{cibiblong,cibib}
\bibliographystyle{hmralphaabbrv}

\end{document}